\documentclass[11pt]{articlefederico}

\usepackage{fullpage}
\usepackage{todonotes,xypic}
\usepackage{amsthm}
\usepackage{amsmath}
\usepackage{amssymb}
\usepackage{color}\bigskip
\usepackage{enumerate}
\usepackage{bbm}

\usepackage{graphicx}

\newtheorem{theorem}{Theorem}[subsection]
\newtheorem{conjecture}[theorem]{Conjecture}

\newtheorem{proposition}[theorem]{Proposition}
\newtheorem{propdef}[theorem]{Proposition/Definition}
\newtheorem{corollary}[theorem]{Corollary}

\newtheorem{example}[theorem]{Example}

\numberwithin{figure}{section}

\makeatletter
\newtheorem*{rep@theorem}{\rep@title}
\newcommand{\newreptheorem}[2]{%
\newenvironment{rep#1}[1]{%
 \def\rep@title{#2 \ref{##1}}%
 \begin{rep@theorem}}%
 {\end{rep@theorem}}}
\makeatother

\newreptheorem{theorem}{Theorem}
\newreptheorem{corollary}{Corollary}
\newreptheorem{example}{Example}

\usepackage{pgf}
\usepackage{tikz}
\usetikzlibrary{arrows,automata}

\begin{document}

\title{\textsf{Algebraic and geometric methods in \\ enumerative combinatorics}}
\author{\textsf{Federico Ardila\footnote{\textsf{San Francisco State University, San Francisco, USA; Universidad de Los Andes, Bogot\'a, Colombia. federico@sfsu.edu} \newline Partially supported by the US National Science Foundation CAREER Award DMS-0956178 and the SFSU-Colombia Combinatorics Initiative.
This is a close-to-final draft of Chapter 1 of the upcoming \emph{Handbook of Enumerative Combinatorics} from Feb 9, 2015. Please refer to the book for the final version.}}}
\date{}

\maketitle
 
%\begin{abstract}
%We use algebra, geometry, and topology to count. 
%\end{abstract}

%\tableofcontents 

\setcounter{section}{-1}

\section{\textsf{Introduction}}\label{f.sec:intro}

Enumerative combinatorics is about counting. The typical question is to find the number of objects with a given set of properties.

However, enumerative combinatorics is not just about counting. 
%To enumerate a set, we usually need to understand it first. 
In ``real life", when we talk about \emph{counting}, we imagine lining up a set of objects and counting them off: 
%If the objects we wish to count were standing in line, and the line was not too long, we could just count them off: 
$1, 2, 3, \ldots$. However, families of combinatorial objects do not come to us in a natural linear order. To give a very simple example: we do not count the squares in an $m \times n$ rectangular grid linearly. Instead, we use the rectangular structure to understand that the number of squares is $m \cdot n$. Similarly, to count a more complicated combinatorial set, we usually spend most of our efforts understanding
%the most important step is usually to understand 
the underlying structure of the individual objects, or of the set itself. % Once we understand, we count.

Many combinatorial objects of interest have a rich and interesting \textbf{algebraic} or \textbf{geometric} structure, which often becomes a very powerful tool towards their enumeration. In fact, there are many families of objects that we \textbf{only} know how to count using these tools. %, even though the answer suggests there should be a directly combinatorial proof. % number ``looks combinatorial".
This chapter highlights some key aspects of the rich interplay between algebra, discrete geometry, and combinatorics, with an eye towards enumeration.

\bigskip

\noindent{\textsf{\textbf{About this survey.}}}  
Over the last fifty years, combinatorics has undergone a radical transformation. 
Not too long ago, 
%Many of today's experts have stories of how 
%combinatorics was not considered a ``respectable" field of study, with its
%; some called it a ``Mickey Mouse" subject, full of 
combinatorics mostly consisted of ad hoc methods and clever solutions to problems that were fairly isolated from the rest of mathematics. It has since grown to be a central area of mathematics, largely thanks to the discovery of deep connections to other fields. Combinatorics has become an essential tool in many disciplines. Conversely, even though ingenious methods and clever new ideas still abound, there is now a powerful, extensive toolkit of algebraic, geometric, topological, and analytic techniques that can be applied to combinatorial problems. 

It is impossible to give a meaningful summary of the many facets of algebraic and geometric combinatorics in a writeup of this length. I found it very difficult but necessary to omit several beautiful, important directions. In the spirit of a \emph{Handbook of Enumerative Combinatorics}, my guiding principle was to focus on algebraic and geometric techniques that are useful towards the solution of enumerative problems. The main goal of this chapter is to state clearly and concisely some of the most useful tools in algebraic and geometric enumeration, and to give many examples that quickly and concretely illustrate how to put these tools to use.

\newpage

%I COMMENTED THIS OUT
%\tableofcontents
%sSO I COULD ADD IN MY EDITED FILE ALGMETHODS12EDITED.TOC, WHERE I ADDED THE LABELS "PART 1" AND "PART 2"

\input algmethods12edited.toc

\newpage

\begin{Large}
\noindent {\textbf{\textsf{PART 1. ALGEBRAIC METHODS}}}
\end{Large}

\bigskip

The first part of this chapter focuses on algebraic methods in enumeration. In Section \ref{f.sec:answer} we discuss the question: what is a good answer to an enumerative problem? Generating functions are the most powerful tool to unify the different kinds of answers that interest us: explicit formulas, recurrences, asymptotic formulas, and generating functions. In Section \ref{f.sec:genfns} we develop the algebraic theory of generating functions. Various natural operations on combinatorial families of objects correspond to simple algebraic operations on their generating functions, and this allows us to count many families of interest. 
In Section \ref{f.sec:linalg} we show how many problems in combinatorics can be rephrased in terms of linear algebra, and reduced to the problem of computing determinants.
Finally, Section \ref{f.sec:posets} is devoted to the theory of posets. Many combinatorial sets have a natural poset structure, and this general theory is very helpful in enumerating such sets.

\section{\textsf{What is a good answer?}}\label{f.sec:answer}

The main goal of enumerative combinatorics is to count the elements of a finite set. Most frequently, we encounter a family of sets $T_0, T_1, T_2, T_3, \ldots$ and we need to find the number $t_n=|T_n|$ for $n=1, 2, \ldots$. What constitutes a good answer? 

Some answers are obviously good. For example, the number of subsets of $\{1,2, \ldots, n\}$ is $2^n$, and it seems clear that this is the simplest possible answer to this question. Sometimes an answer ``is so messy and long, and so full of factorials and sign alternations and whatnot, that we may feel that the disease was preferable to the cure" \cite{f.Wilf}. Usually, the situation is somewhere in between, and it takes some experience to recognize a good answer. 

A combinatorial problem often has several kinds of answers. Which answer is better depends on what one is trying to accomplish. Perhaps this is best illustrated with an example. Let us count the number $a_n$ of \textbf{domino tilings} of a $2 \times n$ rectangle into $2 \times 1$ rectangles. There are several different ways of answering this question.

\begin{figure}[ht]
 \begin{center}
  \includegraphics[scale=.8]{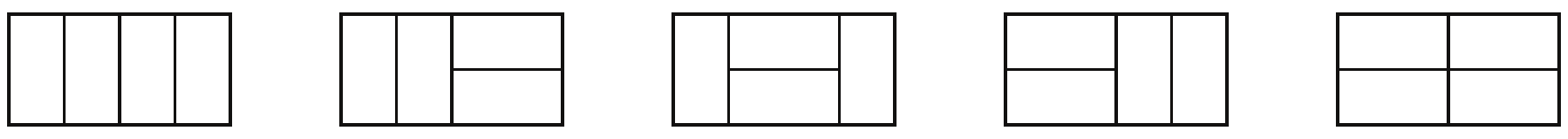}
  \caption{ \label{f.fig:2xn}
The five domino tilings of a $2 \times 4$ rectangle.}
 \end{center}
\end{figure}

\bigskip
\noindent \textbf{\textsf{Explicit formula 1.}} We first look for an explicit combinatorial formula for $a_n$. 
%That usually involves exploring and understanding the structure of the objects we are counting, which in this case is quite simple: 
To do that, we play with a few examples, and quickly notice that these tilings are structurally very simple: they are just a sequence of $2 \times 1$ vertical tiles, and $2 \times 2$ blocks covered by two horizontal tiles. Therefore constructing a tiling is the same as writing $n$ as an ordered sum of $1$s and $2$s. For example, the tilings of Figure \ref{f.fig:2xn} correspond, respectively, to $1+1+1+1, \,  1+1+2, \, 1+2+1, \, 2+1+1, 2+2$. 
These sums are easy to count. If there are $k$ summands equal to $2$ there must be $n-2k$ summands equal to $1$, and there are ${n-2k+k \choose k} = {n-k \choose k}$ ways of  ordering the summands. Therefore 
\begin{equation}\label{f.eq:Fib1}
a_n = \sum_{k=0}^{\lfloor n/2 \rfloor} {n-k \choose k} = {n \choose 0} + {n-1 \choose 1} + {n-2 \choose 2} + \cdots.
\end{equation}
This is a pretty good answer. It is certainly an explicit formula, and it may be used to compute $a_n$ directly for small values of $n$. It does have two drawbacks. Aesthetically, it is certainly not as satisfactory as ``$\,2^n\,$". In practice, it is also not as useful as it seems; after computing a few examples, we will soon notice that computing binomial coefficients is a non-trivial task. In fact there is a more efficient method of computing $a_n$.

\bigskip
\noindent \textbf{\textsf{Recurrence.}} Let $n \geq 2$. In a domino tiling, the leftmost column of a $2 \times n$ can be covered by a vertical domino or by two horizontal dominoes. If the leftmost domino is vertical, the rest of the dominoes tile a $2 \times (n-1)$ rectangle, so there are $a_{n-1}$ such tilings. 
On the other hand, if the two leftmost dominoes are horizontal, the rest of the dominoes tile a $2 \times (n-2)$ rectangle, so there are $a_{n-2}$ such tilings. We obtain the \textbf{recurrence relation}
\begin{equation}\label{f.eq:Fib2}
a_0 = 1, \quad a_1 = 1, \quad a_n = a_{n-1}+a_{n-2} \,\, \textrm{ for } n \geq 2.
\end{equation}
which allows us to compute each term in the sequence in terms of the previous ones. We see that $a_n = F_{n+1}$ is the $(n+1)$th \textbf{Fibonacci number}.

This recursive answer is not as nice as ``$\,2^n\,$" either; it is not even an explicit formula for $a_n$. 
If we want to use it to compute $a_n$, we need to compute all the first $n$ terms of the sequence $1, 1, 2, 3, 5, 8, 13, 21, 34, 55, 89, 144, \ldots$. However, we can  compute those  very quickly; we only need to perform $n-1$ additions. This is an extremely efficient method for computing $a_n$.

%The recurrence relation for $a_n$ also has the advantage that it brings us closer to an extremely powerful tool: generating functions.

%However, this answer is certainly not as satisfactory as ``$2^n$". In particular, it would be nice to be able to compute, say, $a_{100}$ without having to compute the previous $100$ terms of the sequence. Let us try to find an explicit formula for $a_n$.

\bigskip
\noindent \textbf{\textsf{Explicit formula 2.}} There is a well established method that turns linear recurrence relations with constant coefficients, such as (\ref{f.eq:Fib2}), into explicit formulas. We will review it in Theorem \ref{f.th:rational}. In this case, the method gives
\begin{equation}\label{f.eq:Fib4}
a_n = \frac1{\sqrt5}\left(\left(\frac{1+\sqrt5}2\right)^{n+1} - \left(\frac{1-\sqrt5}2\right)^{n+1}\right).
\end{equation}
This is clearly the simplest possible explicit formula for $a_n$; in that sense it is a great formula. 

A drawback is that this formula is really not very useful if we want to compute the exact value of, say, $a_{1000}$. It is not even clear why (\ref{f.eq:Fib4}) produces an integer, and to get it to produce the correct integer would require arithmetic calculations with extremely high precision.

An advantage is that, unlike (\ref{f.eq:Fib1}), (\ref{f.eq:Fib4}) tells us very precisely how $a_n$ grows with $n$.

\bigskip
\noindent \textbf{\textsf{Asymptotic formula.}} 
It follows immediately from (\ref{f.eq:Fib4}) that
\begin{equation}\label{f.eq:Fib5}
a_n \sim c \cdot \varphi^n, \quad
\end{equation}
where $c=\frac{1+\sqrt{5}}{2\sqrt{5}}$ and $\varphi =\frac{1+\sqrt{5}}2 \approx 1.6179\ldots$ is the golden ratio. This notation means that $\lim_{n \rightarrow \infty} a_n/(c \cdot \varphi^n) = 1$. In fact, since $|\frac{1-\sqrt{5}}2| < 1$, $a_n$ is the closest integer to  $c \cdot \varphi^n$.

%The precise rate of growth of our sequence is completely hidden from our other explicit formula (\ref{f.eq:Fib1}). %(This asymptotic formula also follows from (\ref{f.eq:Fib3}) and standard analytic considerations; see Chapter ??.)

\bigskip
\noindent \textbf{\textsf{Generating function.}} The last kind of answer we discuss is the generating function. This is perhaps the strangest kind of answer, but it is often the most powerful one.

Consider the infinite power series $A(x) = a_0 + a_1x + a_2x^2 + \cdots$. We call this the \textbf{generating function} of the sequence $a_0, a_1, a_2, \ldots$.\footnote{For the moment, let us not worry about where this series converges.  The issue of convergence can be easily avoided (as combinatorialists often do, in a way which will be explained in Section \ref{f.sec:fps}) or resolved and exploited to our advantage%\comment{(as Chapter  explains)}
; let us postpone that discussion for the moment.}
We now compute this power series: From (\ref{f.eq:Fib2}) we obtain that $A(x) = 1 + x + \sum_{n \geq 2} (a_{n-1} + a_{n-2}) x^n = 1 + x + x(A(x)-1) + x^2A(x)$ which implies:
\begin{equation}\label{f.eq:Fib3}
A(x) = a_0 + a_1x + a_2x^2 + \cdots = \frac1{1-x-x^2}.
\end{equation}

With a bit of theory and some practice, we will be able to write the equation (\ref{f.eq:Fib3}) immediately, with no further computations (Example 18 in Section \ref{f.sec:ogfs}).
To show this is an \textbf{excellent} answer, let us use it to derive all our other answers, and more. 

\medskip
$\bullet$ Generating functions help us obtain explicit formulas. For instance, rewriting  
\[
A(x) = \frac1{1-(x+x^2)} = \sum_{k \geq 0} (x+x^2)^k
\]
we recover (\ref{f.eq:Fib1}). If, instead, we use the method of partial fractions, we get
\[
A(x) =  \left( \frac{1/{\sqrt5}}{1-\frac{1+\sqrt{5}}2  x} \right) -\left( \frac{1/{\sqrt5}}{1-\frac{1-\sqrt{5}}2 x} \right)
\]
which brings us to our second explicit formula (\ref{f.eq:Fib4}). 

\medskip
$\bullet$ 
Generating functions help us obtain recursive formulas. In this example,
%If we had arrived to (\ref{f.eq:Fib3}) through a different route (and we will), 
we simply compare the coefficients of $x^n$ on both sides of the equation
$A(x)(1-x-x^2)=1$, and we get the recurrence relation (\ref{f.eq:Fib2}).

\medskip
$\bullet$ 
Generating functions help us obtain asymptotic formulas. In this example, (\ref{f.eq:Fib3}) leads to (\ref{f.eq:Fib4}), which gives (\ref{f.eq:Fib5}). In general, almost everything that we know about  the rate of growth of combinatorial sequences comes from their generating functions, because analysis tells us that the asymptotic behavior of $a_n$ is intimately tied to the singularities of the function $A(x)$. 
%\comment{This topic is explored at depth in Chapter ??}

\medskip
$\bullet$ 
Generating functions help us enumerate our combinatorial objects in more detail, and understand some of their statistical properties. 
%can also lead very quickly to a more refined enumeration of our objects.
For instance, say we want to compute the number $a_{m,n}$ of domino tilings of a $2 \times n$ rectangle which use exactly $m$ vertical tiles. Once we really understand (\ref{f.eq:Fib3}) in Section \ref{f.sec:ogfs}, we will get the answer immediately:
\[
\frac1{1-vx-x^2} = \sum_{m, n \geq 0} a_{m,n} v^m x^n.
\]
Now suppose we wish to know what fraction of the tiles is vertical in a large random tiling. Among all the $a_n$ domino tilings of the $2 \times n$ rectangle, there are $\sum_{m \geq 0}  ma_{m,n}$ vertical dominoes. We compute
\[
\sum_{n \geq 0}\left(\sum_{m \geq 0}  ma_{m,n}\right)x^n =\left[ \frac{\partial}{\partial v} \left(\frac1{1-vx-x^2}\right) \right]_{v = 1} = %\left. \frac{x}{(1-vx-x^2)^2} \right|_{v = 1} = 
\frac{x}{(1-x-x^2)^2}.
\]
Partial fractions then tell us that $\sum_{m \geq 0}  ma_{m,n} \sim \frac{n}5\left(\frac{1+\sqrt{5}}2\right)^{n+1} \sim \frac1{\sqrt{5}}na_n$. Hence the fraction of vertical tiles in a random domino tiling of a $2 \times n$ rectangle converges to $1/{\sqrt{5}}$ as $n \rightarrow \infty$.

\bigskip
\noindent
\textbf{\textsf{So what is a good answer to an enumerative problem?}} Not surprisingly, there is no definitive answer to this question. When we count a family of combinatorial objects, we look for explicit formulas, recursive formulas, asymptotic formulas, and generating functions. They are all useful. Generating functions are the most powerful framework we have to relate these different kinds of answers and, ideally, find them all.

\newpage

\section{\textsf{Generating functions}}\label{f.sec:genfns}

In combinatorics, one of the most useful ways of ``determining" a sequence of numbers $a_0, a_1, a_2, \ldots$ is to compute its \textbf{ordinary generating function}
\[
A(x) = \sum_{n \geq 0} a_nx^n = a_0 + a_1 x + a_2 x^2 + a_3 x^3 + a_4x^4 +  \cdots.
\]
or its \textbf{exponential generating function}
\[
A_{\textrm{exp}}(x) = \sum_{n \geq 0} a_n \frac{x^n}{n!} =  %a_0  + a_1 \frac{x^1}{1!} + a_2 \frac{x^2}{2!} + a_3 \frac{x^3}{3!} + \cdots.
a_0  + a_1 x + a_2 \, \frac{x^2}{2} + a_3 \, \frac{x^3}{6} +  a_4 \, \frac{x^4}{24} + \cdots.
\]
This simple idea is extremely powerful because some of the most common algebraic operations on ordinary and exponential generating functions correspond to some of the most common operations on combinatorial objects. This allows us to count many interesting families of objects; this is the content of Section \ref{f.sec:ogfs} (for ordinary generating functions) and Section \ref{f.sec:egfs} (for exponential generating functions). In Section \ref{f.sec:nicegfs} we see how nice generating functions can be turned into explicit, recursive, and asymptotic formulas for the corresponding sequences.

Before we get to this interesting theory, we have to understand what we mean by power series. Section \ref{f.sec:fps} provides a detailed discussion, which is probably best skipped the first time one encounters power series. In the meantime, let us summarize it in one paragraph:

There are two main attitudes towards power series in combinatorics: the analytic attitude and the algebraic attitude. To harness the full power of power series, one should really understand both. Chapter 2 of this \emph{Handbook of Enumerative Combinatorics} is devoted to the analytic approach, which treats $A(x)$ as an honest analytic function of $x$, and uses analytic properties of $A(x)$ to derive combinatorial properties of $a_n$. In this chapter we follow the algebraic approach, which treats $A(x)$ as a formal algebraic expression, and manipulates it using the usual laws of algebra, without having to worry about any convergence issues.

\subsection{\textsf{The ring of formal power series}}\label{f.sec:fps}

Enumerative combinatorics is full of intricate algebraic computations with power series, where justifying convergence is cumbersome, and usually unnecessary. In fact, many natural power series in combinatorics only converge at $0$, such as $\sum_{n \geq 0} n! x^n$; so analytic methods are not available to study them. For these reasons we often prefer to carry out our computations algebraically in terms of \textbf{formal} power series. We will see that even in this approach, analytic considerations are often useful.

In this section we review the definition and basic properties of the \textbf{ring of formal power series} ${\mathbb{C}}[[x]]$. For a more in depth discussion, including the (mostly straightforward) proofs of the statements we make here, see \cite{f.Niven}.

\bigskip
\noindent \textsf{\textbf{Formal power series.}}
A \textbf{formal power series} is an expression of the form
\[
A(x) = a_0 + a_1 x + a_2 x^2 + \cdots, \qquad a_0, a_1, a_2, \ldots \in {\mathbb{C}}.
\]
Formally, this series is just a sequence of complex numbers $a_0, a_1, a_2, \ldots.$ We will find it convenient to denote it $A(x)$, but we do not consider it to be a function of $x$.
%; we will never substitute in a complex number $x$.
%It is implicit in the definition, but worth mentioning, that $\sum_{n \geq 0} a_nx^n = \sum_{n \geq 0} b_nx^n$ if and only if $a_n = b_n$ for all $n \geq 0$.

Let ${\mathbb{C}}[[x]]$ be the \textbf{ring of formal power series}, where the sum and the product of the series $A(x) = \sum_{n \geq 0} a_nx^n$ and $B(x) = \sum_{n \geq 0} b_nx^n$ are defined to mimic analytic functions at $0$:
\[
A(x) + B(x)  = \sum_{n \geq 0} (a_n+b_n)x^n, 
\qquad
A(x)B(x) = \sum_{n \geq 0} \left(\sum_{k=0}^n a_k b_{n-k} \right) x^n.
\]
%In other words, these strictly algebraic objects are defined to have the same algebraic behavior as analytic functions at $0$.
%, algebraically, these formal expressions behave as if they were honest analytic functions. 
It is implicitly understood that
 %but worth mentioning, that 
 $\sum_{n \geq 0} a_nx^n = \sum_{n \geq 0} b_nx^n$ if and only if $a_n = b_n$ for all $n \geq 0$.
%\[
%\sum_{n \geq 0} a_nx^n  + \sum_{n \geq 0} b_nx^n  = \sum_{n \geq 0} (a_n+b_n)x^n \qquad\left(\sum_{n \geq 0} a_nx^n\right) \left(\sum_{n \geq 0} b_nx^n\right)  = \sum_{n \geq 0} \left(\sum_{k=0}^n a_k b_{n-k} \right) x^n.
%\]
%\begin{eqnarray*}
%\sum_{n \geq 0} a_nx^n  + \sum_{n \geq 0} b_nx^n  &=& \sum_{n \geq 0} (a_n+b_n)x^n \\
%%\left(\sum_{n \geq 0} a_nx^n\right) \left(\sum_{n \geq 0} b_nx^n\right)  &=& \sum_{n \geq 0} (a_nb_0 + a_{n-1}b_1 + \cdots + a_0b_n)x^n
%\left(\sum_{n \geq 0} a_nx^n\right) \left(\sum_{n \geq 0} b_nx^n\right)  &=& \sum_{n \geq 0} \left(\sum_{k=0}^n a_k b_{n-k} \right) x^n.
%\end{eqnarray*}

The \textbf{degree} of $A(x) = \sum_{n \geq 0} a_nx^n$ is the \textbf{smallest} $n$ such that $a_n \neq 0$. We also write
\[
[x^n]A(x) := a_n, \qquad A(0) := [x^0]A(x) = a_0.
\]
We also define formal power series inspired by  series from analysis, such as
\[
e^x := \sum_{n \geq 0} \frac{x^n}{n!}, 
\qquad 
- \log(1-x) := \sum_{n \geq 1} \frac{x^n}{n}, 
\qquad
(1+x)^r := \sum_{n \geq 0} {r \choose n} x^n,\qquad
\]
for any  complex number $r$, where ${r \choose n} := r(r-1) \cdots (r-n+1)/n!$.

The ring ${\mathbb{C}}[[x]]$ is commutative with $0 = 0 + 0x +\cdots$ and $1 = 1 + 0x + \cdots$. It is  an integral domain; that is, $A(x)B(x) = 0$ implies that $A(x) = 0$ or $B(x) = 0$. It is easy to describe the units:
\[
\sum_{n \geq 0} a_n x^n \textrm{ is invertible} \quad \Longleftrightarrow \quad a_0 \neq 0.
\] 
For example $\frac1{1-x} = 1+x+x^2+\cdots$ because $(1-x)(1+x+x^2+\cdots) = 1+0x+0x^2+\cdots$.

\bigskip
\noindent \textsf{\textbf{Convergence.}} 
When working in ${\mathbb{C}}[[x]]$, we will not consider convergence of sequences or series of complex numbers. In particular, we will never substitute a complex number $x$ into a formal power series $A(x)$.

However, we do need a notion of convergence for sequences in ${\mathbb{C}}[[x]]$. We say that a sequence $A_0(x), A_1(x), \ldots$ of formal power series \textbf{converges} to $A(x)= \sum_{n \geq 0} a_nx^n $ if $\lim_{n \rightarrow \infty} \deg(A_n(x) - A(x)) = \infty$; that is, if for any $n \in {\mathbb{N}}$, the coefficient of $x^n$ in $A_m(x)$ \textbf{equals} $a_n$ for all sufficiently large $m$. This gives us a useful criterion for convergence of infinite sums and products in ${\mathbb{C}}[[x]]$:
\begin{eqnarray*}
\sum_{j=0}^\infty A_j(x) \textrm{ converges} \quad &\Longleftrightarrow & \quad \lim_{j \rightarrow \infty} \deg A_j(x) = \infty \\ 
\prod_{j=0}^\infty (1+A_j(x)) \textrm{ converges} \quad &\Longleftrightarrow & \quad \lim_{j \rightarrow \infty} \deg A_j(x) = \infty \qquad (A_j(0) = 0)
\end{eqnarray*}
For example, the infinite sum $\sum_{n \geq 0} (x+1)^n/2^n$ does not converge in this topology. Notice that  the coefficient of $x^0$ in this sum cannot be obtained through a finite computation; it would require interpreting the infinite sum $\sum_{n \geq 0} 1/2^n$.
On the other hand, the following infinite sum converges:
\begin{equation}\label{f.eq:1-x}
\sum_{n \geq 0} \frac1{n!} {\left(-\sum_{m \geq 1} \frac{x^m}{m}\right)^n} = 1-x.
\end{equation}
It is clear from the criterion above that this series converges; but why does it equal $1-x$?

\bigskip
\noindent \textsf{\textbf{Borrowing from analysis.}} 
In ${\mathbb{C}}[[x]]$, (\ref{f.eq:1-x}) is an algebraic identity which says that the coefficients of $x^k$ in the left hand side -- for which we can give an ugly but finite formula -- equal $1, -1, 0, 0, 0, \ldots$. If we were to follow a purist algebraic attitude, we would give an algebraic or combinatorial proof of this identity. This is probably possible, but intricate and rather dogmatic. A much simpler approach is to shift towards an analytic attitude, at least momentarily, and recognize that (\ref{f.eq:1-x}) is the Taylor series expansion of 
\[
e^{-\log(1-x)} = 1-x
\]
for $|x| < 1$. Then we can just invoke the following simple fact from analysis.

\begin{theorem}
If two analytic functions are equal in an open neighborhood of $0$, then their Taylor series at $0$ are equal coefficient-by-coefficient; that is, they are equal as formal power series.
\end{theorem}

\bigskip
\noindent \textsf{\textbf{Composition.}} The \textbf{composition} of two series $A(x) = \sum_{n \geq 0} a_nx^n$ and $B(x) = \sum_{n \geq 0} b_n x^n$ with $b_0 = 0$ is naturally defined to be:
\[
A(B(x)) = \sum_{n \geq 0} a_n \left(\sum_{m \geq 0} b_mx^m\right)^n.
\]
Note that this sum converges if and only if $b_0=0$. Two very important special cases in combinatorics are the series $\frac1{1-B(x)}$ and $e^{B(x)}$.

\bigskip
\noindent \textsf{\textbf{``Calculus".}} We define the \textbf{derivative} of $A(x) = \sum_{n \geq 0} a_nx^n$ to be
\[
A'(x) = \sum_{n \geq 0} (n+1)a_{n+1} x^n.
\]
This formal derivative satisfies the usual laws of derivatives, such as:
\[
(A+B)' = A' + B', \qquad (AB)' = A'B + AB', \qquad [A(B(x))]' = A'(B(x)) B'(x).
\]
We can still solve differential equations formally. For example, if we know that $F'(x) = F(x)$ and $F(0) = 1$, then $(\log F(x))' = F'(x)/F(x) = 1$, which gives $\log F(x) = x$ and $F(x) = e^x$. 

\bigskip

This concludes our discussion on the formal properties of power series. Now let us return to combinatorics.

%
%
%\bigskip
%\noindent \textsf{\textbf{Advantages of the algebraic approach.}}
%There are at least two important advantages of the algebraic attitude towards formal power series.
%
%Firstly, in ${\mathbb{C}}[[x]]$ we do not have to constantly worry about convergence. Enumerative combinatorics is full of intricate algebraic computations with power series. Justifying the convergence of our series in the intermediate steps can be tedious, and is usually unnecessary. The final result of our algebraic computation is often a formula $\sum_{n \geq 0} a_n x^n = A(x)$ for the formal power series of the sequence that interests us. At that stage, we can shift back to analysis and study the convergence of $A(x)$, to derive consequences about $a_n$. This is done extensively in Chapter ??.
%
%
%Secondly, and perhaps more importantly, there are numerous examples of power series in combinatorics that converge at $0$; for example, the generating function $\sum_{n \geq 0} n! x^n$ for permutations of $[n]$. In ${\mathbb{C}}[[x]]$ it is allowed, and quite useful, to carry out computations with these formal power series, where analytic techniques are not available to us.  
%
%

\subsection{\textsf{Ordinary generating functions}}\label{f.sec:ogfs}

Suppose we are interested in enumerating a family ${\mathcal{A}}={\mathcal{A}}_0 \sqcup {\mathcal{A}}_1 \sqcup {\mathcal{A}}_2 \sqcup \cdots$ of combinatorial structures, where ${\mathcal{A}}_n$ is a finite set consists of the objects of ``size" $n$. Denote by $|a|$ the size of $a \in {\mathcal{A}}$. The \textbf{ordinary generating function}  of ${\mathcal{A}}$ is
\[
A(x) = \sum_{a \in {\mathcal{A}}} x^{|a|} = a_0 + a_1x + a_2x^2 + \cdots
\]
where $a_n$ is the number of elements of size $n$.

We are not interested in the philosophical question of determining what it means for ${\mathcal{A}}$ to be ``combinatorial"; we are willing to call ${\mathcal{A}}$ a combinatorial structure as long as $a_n$ is finite for all $n$. We consider two structures ${\mathcal{A}}$ and ${\mathcal{B}}$ combinatorially equivalent, and write ${\mathcal{A}} \cong {\mathcal{B}}$, if $A(x) = B(x)$. 

More generally, we may consider a family ${\mathcal{A}}$ where each element $a$ is given a weight ${\mathrm{wt}}(a)$ -- often a constant multiple of $x^{|a|}$, or a monomial in one or more variables $x_1, \ldots, x_n$. Again, we require that there are finitely many objects of any given weight. Then we define the \textbf{weighted} ordinary generating function of ${\mathcal{A}}$ to be the formal power series
\[
A_{{\mathrm{wt}}}(x_1, \ldots, x_n) =  \sum_{a \in {\mathcal{A}}} {\mathrm{wt}}(a)  
\]

Examples of combinatorial structures (with their respective size functions ni parentheses) are  words on the alphabet $\{0,1\}$ (length),  domino tilings of rectangles of height $2$ (width), or Dyck paths (length). We may weight these objects by $t^k$ where $k$ is, respectively, the number of $1$s, the number of vertical tiles, or the number of returns to the $x$ axis.

\subsubsection{\textsf{Operations on combinatorial structures and their generating functions}}\label{f.sec:operations}

There are a few simple but very powerful operations on combinatorial structures, all of which have nice counterparts at the level of ordinary generating functions. 
Many combinatorial objects of interest may be built up from basic structures using these operations.

\begin{theorem}\label{f.th:ogf}
Let ${\mathcal{A}}$ and ${\mathcal{B}}$ be combinatorial structures. 
\begin{enumerate}
\item (${\mathcal{C}} = {\mathcal{A}} + {\mathcal{B}}$: Disjoint union)
If a ${\mathcal{C}}$-structure of size $n$ is obtained by choosing an ${\mathcal{A}}$-structure of size $n$ \textbf{or} a ${\mathcal{B}}$-structure of size $n$, then 
\[
C(x) = A(x) + B(x).
\]
This result also holds for weighted structures if the weight of a ${\mathcal{C}}$-structure is the same as the weight of the respective ${\mathcal{A}}$ or ${\mathcal{B}}$-structures. 

\item (${\mathcal{C}} = {\mathcal{A}} \times {\mathcal{B}}$: Product) 
If a ${\mathcal{C}}$-structure of size $n$ is obtained by choosing an ${\mathcal{A}}$-structure of size $k$ \textbf{and} a ${\mathcal{B}}$-structure of size $n-k$ for some $k$, then
\[
C(x) = A(x) B(x).
\]
This result also holds for weighted structures if the weight of a ${\mathcal{C}}$-structure is the product of the weights of the respective ${\mathcal{A}}$ and ${\mathcal{B}}$-structures.

\item (${\mathcal{C}} = \textsf{Seq}({\mathcal{B}})$: Sequence)
Assume $|{\mathcal{B}}_0|=0$. If a ${\mathcal{C}}$-structure of size $n$ is obtained by choosing a sequence of ${\mathcal{B}}$-structures of total size $n$, then
\[
C(x) = \frac1{1-B(x)}.
\]
This result also holds for weighted structures if the weight of a ${\mathcal{C}}$-structure is the product of the weights of the respective ${\mathcal{B}}$-structures. 
In particular, if $c_k(n)$ is the number of ${\mathcal{C}}$-structures of an $n$-set which decompose into $k$ ``factors" (${\mathcal{B}}$-structures), we have
\[
\sum_{n, k, \geq 0} c_k(n) {x^n} y^k = \frac1{1-yB(x)} = \frac{C(x)}{y+(1-y)C(x)}.
\]

\item (${\mathcal{C}} = {\mathcal{A}} \circ {\mathcal{B}}$: Composition) \textbf{Compositional Formula.}
Assume that $|{\mathcal{B}}_0|=0$. If a ${\mathcal{C}}$-structure of size $n$ is obtained by choosing a sequence of (say, $k$) ${\mathcal{B}}$-structures of total size $n$ and placing an ${\mathcal{A}}$-structure of size $k$ on this sequence of ${\mathcal{B}}$-structures, then 
\[
C(x) = A(B(x)).
\]
This result also holds for weighted structures if the weight of a ${\mathcal{C}}$-structure is the product of the weights of the ${\mathcal{A}}$-structure on its blocks and the weights of the ${\mathcal{B}}$-structures on the individual blocks.

\item (${\mathcal{C}} = {\mathcal{A}}^{-1}$: Inversion) \textbf{Lagrange Inversion Formula.}

(a) \textbf{Algebraic version.} If $A^{<-1>}(x)$ is the compositional inverse of $A(x)$ then
\[
n[x^n] A^{<-1>}(x) = [x^{n-1}] \left(\frac{x}{A(x)}\right)^n.
\]

(b) \textbf{Combinatorial version.} Assume $|{\mathcal{A}}_0|=0$,  $|{\mathcal{A}}_1| = 1$, and let
\[
A(x) = x-a_2x^2-a_3x^3-a_4x^4-\cdots
\]
where $a_n$ is the number of ${\mathcal{A}}$-structures of size $n$ for $n \geq 2$.
\footnote{
Here,
%In Theorem \ref{f.th:ogf}.5(b), 
to simplify matters, we introduced signs into $A(x)$. Instead we could let $A(x)$ be the ordinary generating function for ${\mathcal{A}}$-structures, but we would need to give each ${\mathcal{A}}$-tree the sign $(-1)^m$, where $m$ is the number of internal vertices. Similarly, we could allow $a_1 \neq 1$ at the cost of some factors of $a_1$ on the ${\mathcal{A}}$-trees.}

Let an \textbf{${\mathcal{A}}$-decorated plane rooted tree} (or simply ${\mathcal{A}}$-tree) be a rooted tree $T$ where every internal vertex $v$ has an ordered set $D_v$ of at least two ``children", and each ordered set $D_v$ is given an ${\mathcal{A}}$-structure. The \textbf{size} of $T$ is the number of leaves. 
%Say a tree with $n$ leaves and $m$ internal vertices has \textbf{size} $n$ and \textbf{sign} $(-1)^m$. 

Let $C(x)$ be the 
%signed 
generating function for ${\mathcal{A}}$-decorated plane rooted trees. Then 
\[
C(x) = A^{<-1>}(x).
\]
This result also holds for weighted structures if the weight of a tree is the product of the weights of the ${\mathcal{A}}$-structures at its vertices.

%More generally
%\[
%n[x^n] A^{<-1>}(x)^k = [x^{n-k}] \left(\frac{x}{A(x)}\right)^n.
%\]

%In more explicit terms, let 
%\[
%A(x) = x+ a_2x^2 + a_3x^3 + \cdots, \qquad A^{<-1>}(x) = C(x) = x+ c_2x^2 + c_3x^3 + \cdots
%\]
%be compositional inverses. Then
%\[
%c_n = \sum_{m=2}^{n-1} \sum_{\mathbf{r}}  \frac{(-1)^m}{n+m}{n+m \choose n, r_2, \ldots, r_k} a_2^{r_2} \cdots a_m^{r_m}
%%b_n = \sum_{\mathbf{r}} \frac{1}{n}{n \choose r_0, \ldots, r_m} a_2^{r_2} \cdots a_m^{r_m}
%\]
%where the internal sum is over all sequences $\mathbf{r} = (r_2, \ldots, r_k)$ such that $\sum r_i = m$ and $\sum ir_i = n+m-1$.

\end{enumerate}
\end{theorem}

Theorem \ref{f.th:ogf}.3 is especially useful when we are counting combinatorial objects which ``factor" uniquely into an ordered ``product" of ``irreducible" objects. It tells us that we can count all objects if and only if we can count the irreducible ones.

\begin{proof}
1. is clear. The identity in 2. is equivalent to $c_n = \sum_k a_kb_{n-k}$, which corresponds to the given combinatorial description. Iterating 2., the generating function for $k$-sequences of ${\mathcal{B}}$-structures is $B(x)^k$, so in 3. we have $C(x) = \sum_k B(x)^k = 1/(1-B(x))$ and in 4. we have $C(x) = \sum_k a_k B(x)^k$. The weighted statements follow similarly.

5(b). Observe that, by the Compositional Formula, an ${\mathcal{A}} \circ {\mathcal{C}}$ structure is either
(i) an ${\mathcal{A}}$-tree, or 
(ii) a sequence of $k \geq 2$ ${\mathcal{A}}$-trees $T_1, \ldots, T_k$ with an ${\mathcal{A}}$-structure on $\{T_1, \ldots, T_k\}$. 

The structure in (ii) is equivalent to an ${\mathcal{A}}$-tree $T$, obtained by grafting $T_1, \ldots, T_k$ at a new root and placing the ${\mathcal{A}}$-structure on its offspring (which contributes a negative sign). This $T$ also arises in (i) with a positive sign. 
These two appearances of $T$ cancel each other out in $A(C(x))$, and the only surviving tree is the trivial tree $\bullet$ with one vertex, which only arises once and has weight $x$.

5(a). Let a \textbf{sprig} be a rooted plane tree consisting of a path $r=v_1 v_2 \cdots v_k=l$ starting at the root $r$ and ending at the leaf $l$, and at least one leaf hanging from each $v_i$ and to the right of $v_{i+1}$ for $1 \leq i \leq k-1$. The trivial tree $\bullet$ is an allowable sprig with $k=1$.

An \textbf{${\mathcal{A}}$-sprig} is a sprig where the children of $v_i$ are given an ${\mathcal{A}}$-structure for $1 \leq i \leq k-1$; its \textbf{size} is the number of leaves other than $l$ minus $1$.\footnote{We momentarily allow negative sizes, since the trivial ${\mathcal{A}}$-sprig $\bullet$ has size $-1$. Thus we need to compute with Laurent series, which are power series with finitely many negative exponents.} 
The right panel of Figure \ref{f.fig:treetosprigs} shows several ${\mathcal{A}}$-sprigs. An ${\mathcal{A}}$-sprig is equivalent to a sequence of ${\mathcal{A}}$-structures, with weights shifted by $-1$, so Theorem \ref{f.th:ogf}.3 tells us that
\[
\frac1{A(x)} = \frac1x \cdot \frac1{1-(a_2x + a_3x^2 + \cdots)} = \sum_{n \geq -1} (\# \textrm{ of ${\mathcal{A}}$-sprigs of size } n) \, x^n 
\]
Hence $[x^{n-1}](x/A(x))^n= [x^{-1}](1/A(x))^n$ is the number of sequences of $n$ ${\mathcal{A}}$-sprigs of total size $-1$ by Theorem \ref{f.th:ogf}.2. We need to show that
\[
n \cdot (\# \textrm{ of ${\mathcal{A}}$-trees with $n$ leaves}) = (\# \textrm{ of sequences of $n$ ${\mathcal{A}}$-sprigs of total size $-1$})
\]

\begin{figure}[ht]
 \begin{center}
  \includegraphics[scale=.8]{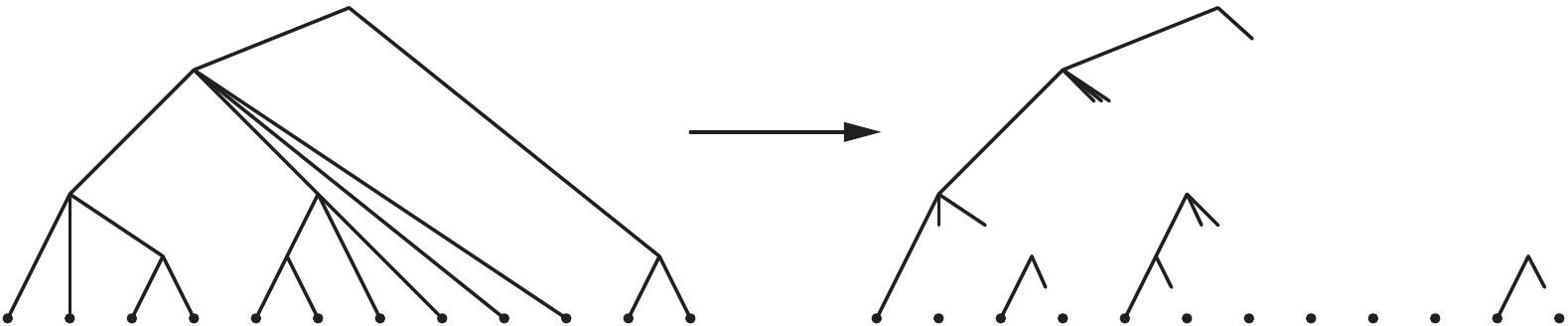}
  \caption{ \label{f.fig:treetosprigs}
The map from ${\mathcal{A}}$-trees to sequences of ${\mathcal{A}}$-sprigs.}
 \end{center}
\end{figure}

An ${\mathcal{A}}$-tree $T$ can be trimmed into a sequence of $n$ ${\mathcal{A}}$-sprigs $S_1, \ldots, S_n$ as follows. At each step, look at the leftmost leaf and the path $P$ to its highest remaining ancestor. Remove $P$ and all the branches hanging directly from $P$ (which form an ${\mathcal{A}}$-sprig), but do not remove any other vertices. Repeat this until the tree is completely decomposed into ${\mathcal{A}}$-sprigs. 
The total size of these sprigs is $-1$. Figure \ref{f.fig:treetosprigs} shows a tree %which gives rise to the monomial $(a_2a_4a_31)1(a_21)1(a_3a_21)11111(a_21)1$ 
of weight $x^{5+(-1)+0+(-1)+2+(-1)+(-1)+(-1)+(-1)+(-1)+0+(-1)} = x^{-1}$. Notice that all the partial sums of the sum $5+(-1)+0+(-1)+2+(-1)+(-1)+(-1)+(-1)+(-1)+0+(-1) = -1$ are non-negative.

Conversely, suppose we wish to recover the ${\mathcal{A}}$-tree corresponding to a sequence of sprigs $S_1, \ldots, S_n$ with $|S_1| + \cdots + |S_n|=-1$. We must reverse the process, adding $S_1, \ldots, S_n$ to $T$ one at a time; at each step we must graft the new sprig at the leftmost free branch. Note that after grafting $S_1, \ldots, S_k$ we are left with $1 + |S_1| + \cdots + |S_k|$ free branches, so a sequence of sprigs corresponds to a tree if and only if the partial sums $|S_1| + \cdots + |S_k|$  are non-negative for $k=1, \ldots, n-1$. Finally, it remains to observe that any sequence $a_1, \ldots, a_n$ of integers adding to $-1$ has a unique cyclic shift $a_i, \ldots, a_n, a_1, \ldots, a_{i-1}$ whose partial sums are all non-negative. Therefore, out of the $n$ cyclic shifts of $S_1, \ldots, S_n$, exactly one of them corresponds to an ${\mathcal{A}}$-tree. The desired result follows.
%The second formulation follows from the fact \cite[Theorem 5.3.10]{EC2} that there are exactly  $\frac{1}{n+m}{n+m \choose n, r_2, \ldots, r_k}$ rooted plane trees with $n$ leaves and $m$ internal vertices having $r_i$ vertices of degree $i$. 
\end{proof}

The last step of the proof above is a special case of the \textbf{Cycle Lemma} of Dvoretsky and Motzkin \cite[Lemma 5.3.7]{f.EC2}, which is worth stating explicitly.
Suppose $a_1, \ldots, a_n$ is a string of $1$s and $-1$s with $a_1+ \cdots + a_n=k>0$. Then there are exactly $k$ cyclic shifts $a_i, a_{i+1}, \ldots, a_n, a_1, \ldots, a_{i-1}$ whose partial sums are all non-negative.

\subsubsection{\textsf{Examples}}\label{f.sec:ogfexamples}

\bigskip
\noindent \textbf{\textsf{Classical applications.}} With practice, these simple ideas give very easy solutions to many classical enumeration problems.

\begin{enumerate}

\item (Trivial classes) It is useful to introduce the trivial class $\circ$ having only one element of size $0$, and the trivial class $\bullet$ having only one element of size $1$. Their generating functions are $1$ and $x$, respectively.

\item (Sequences) The slightly less trivial class $\textsf{Seq} = \{\emptyset, \bullet, \bullet \, \bullet, \bullet \bullet \bullet, \ldots\} =  \textsf{Seq}(\bullet)$  contains one set of each size. Its generating function is $\sum_n  x^n = 1/(1-x)$.

\item (Subsets and binomial coefficients) Let $\textsf{Subset}$ consist of the pairs $([n], A)$ where $n$ is a natural number and $A$ is a subset of $[n]$. Let the size of that pair be $n$. A $\textsf{Subset}$-structure is equivalent to a word of length $n$ in the alphabet $\{0,1\}$, so $\textsf{Subset} \cong \textsf{Seq}(\{0,1\})$ where $|0| = |1| = 1$, and
\[
\textsf{Subset}(x) = \frac{1}{1-(x^1+x^1)}= \sum_{n \geq 0} 2^n x^n.
\]

We can use the extra variable $y$ to keep track of the size of the subset $A$, by giving $([n], A)$ the weight $x^ny^{|A|}$. This corresponds to giving the letters $0$ and $1$  weights $x$ and $xy$ respectively, so we get the generating function
\[
%\textsf{Subset}_{\textsf{size}}(t,x) = 
\textsf{Subset}_{{\mathrm{wt}}}(x) = \frac{1}{1-(x+xy)} = \sum_{n \geq k \geq 0} {n \choose k} x^n y^k
\]
for the \textbf{binomial coefficients} ${n \choose k} = \frac{n!}{k!(n-k)!}$, which count the $k$-subsets of $[n]$.

From this generating function, we can easily obtain the main results about binomial coefficients. 
Computing the coefficient of $x^n y^k$ in $ \left(\sum {n \choose k} x^n y^k\right)({1-x-xy}) = 1$ gives \textbf{Pascal's recurrence}
\[
{n \choose k} = {n-1 \choose k} + {n-1 \choose k-1}, \qquad n \geq k \geq 1
\]
with initial values ${n \choose 0} = {n \choose n} = 1$. Expanding $\textsf{Subset}_{{\mathrm{wt}}}(x) = \frac{1}{1-x(1+y)}  = \sum_{n \geq 0} x^n(1+y)^n$ gives the \textbf{Binomial Theorem}
\[
(1+y)^n = \sum_{k = 0}^n {n \choose k} y^k.
\]

\item (Multinomial coefficients) Let $\textsf{Words}^k \cong \textsf{Seq}(\{1, \ldots, k\})$ consist of the words in the alphabet $\{1, 2, \ldots, k\}$. The words of length $n$ are in bijection with the ways of putting $n$ numbered balls into $k$ numbered boxes.  The placements having $a_i$ balls in box $i$, where $a_1 + \cdots + a_k=n$, are enumerated by the \textbf{multinomial coefficient} ${n\choose a_1, \ldots, a_k} = \frac{n!}{a_1! \cdots a_k!}$.

Giving the letter $i$ weight $x_i$,  we obtain the generating function:
\[
\textsf{Words}^k(x_1, \ldots, x_k) = 
\sum_{a_1, \ldots, a_k \geq 0} {a_1 + \cdots + a_k \choose a_1, \ldots, a_k}  x_1^{a_1} \cdots x_k^{a_k} = \frac{1}{1-x_1-\cdots - x_k}
\]
%where the \textbf{multinomial coefficient} ${a_1 + \cdots + a_k \choose a_1, \ldots, a_k} = \frac{(a_1+\cdots_a_k)!}{a_1! \cdots a_k!}$ 
from which we obtain the recurrence
\[
{n \choose a_1, \ldots, a_k} = 
%\sum_{i=1}^k {n-1 \choose a_1, \ldots a_{i-1}, a_i-1,a_{i+1}, \ldots, a_k} 
{n-1 \choose a_1-1, a_2, \ldots, a_k} + \cdots + {n-1 \choose a_1, \ldots a_{k-1},  a_k-1} 
\]
and the multinomial theorem
\[
(x_1 + \cdots + x_k)^n = \sum_{\stackrel{a_1, \ldots, a_k \geq 0}{a_1+ \cdots + a_k = n}}
{n \choose a_1, \ldots, a_k} x_1^{a_1} \cdots x_k^{a_k}.
\]

\item (Compositions) A \textbf{composition} of $n$ is a way of writing $n=a_1+\cdots+a_k$ as an \textbf{ordered} sum of positive integers $a_1, \ldots, a_k$. For example, $523212$ is a composition of $15$. A composition is just a sequence of positive integers, so $\textsf{Comp} \cong \textsf{Seq}({\mathbb{Z}_{>0}})$ where $|a| = a$. Therefore
\[
\textsf{Comp}(x) = \frac{1}{1-(x+x^2+x^3+\cdots)} = \frac{1-x}{1-2x}= \sum_{n \geq 1} 2^{n-1} x^n.
\]
and there are $2^{n-1}$ compositions of $n$.

If we give a composition of $n$ with $k$ summands the weight $x^ny^k$, the weighted generating function is:
\[
\textsf{Comp}_{{\mathrm{wt}}}(x) = \frac{1}{1-(xy+x^2y+x^3y+\cdots)} = \frac{1-x}{1-x(1+y)} = \sum_{n \geq 1} {n-1 \choose k-1} x^n y^k
\]
so there are ${n-1 \choose k-1}$ compositions of $n$ with $k$ summands.

\item (Compositions into restricted parts) Given a subset $A \subseteq {\mathbb{N}}$, an \textbf{$A$-composition} of $n$ is a way of writing $n$ as an ordered sum $n = a_1 + \cdots + a_k$ where $a_1, \ldots, a_k \in A$. The corresponding combinatorial structure is  $\textsf{$A$-Comp} \cong \textsf{Seq}(A)$ where $|a| = a$, so
\[
\textsf{$A$-Comp}(x) = \frac{1}{1-(\sum_{a \in A} x^a)}
\]
For example, the number of compositions of $n$ into odd parts is the Fibonacci number $F_{n-1}$, because the corresponding generating functions is
\[
\textsf{OddComp}(x) = \frac{1}{1-(x+x^3+x^5+\cdots)} = \frac{1-x^2}{1-x-x^2} = 1+\sum_{n \geq 1} F_{n-1}x^n.
\]

\item (Multisubsets) Let $\textsf{Multiset}^m$ be the collection of multisets consisting of possibly repeated elements of $[m]$. The size of a multiset is the number of elements, counted with repetition. 
For example, $\{1, 2, 2, 2, 3, 5\}$ is a multisubset of $[7]$ of size 6.
Then $\textsf{Multiset}^m \cong \textsf{Seq}(\{1\}) \times \cdots \times \textsf{Seq}(\{m\})$, where $|i| = 1$ for $i = 1, \ldots, m$, 
so the corresponding generating function is
%An $n$-\textbf{multisubset} of $[m]$ is a multiset of (possibly repeated) elements from the set $[m]$ containing $n$ elements, counted with repetitions. %Let $\left({m \choose n})\$ be the number of multisubsets of $[m]$ of size $n$. 
\[
\textsf{Multiset}^m(x) = \left(\frac1{1-x}\right)^m = \sum_{n \geq 0} {-m \choose n} (-x)^n,
\]
and the number of multisubsets of $[m]$ of size $n$ is  $\left({m \choose n}\right):=(-1)^n {-m \choose n} = {m+n-1 \choose n}$.

\item (Partitions)  A \textbf{partition} of $n$ is a way of writing $n=a_1+\cdots+a_k$ as an \textbf{unordered} sum of positive integers $a_1, \ldots, a_k$. We usually write the parts in weakly decreasing order. For example, $532221$ is a partition of $15$ into $6$ parts.
Let $\textsf{Partition}$ be the family of partitions weighted by $x^n y^k$ where $n$ is the sum of the parts and $k$ is the number of parts. Then $\textsf{Partition} \cong \textsf{Seq}(\{1\}) \times \textsf{Seq}(\{2\}) \times \cdots $, where ${\mathrm{wt}}(i) = x^iy$ for $i = 1, 2, \ldots$, 
so the corresponding generating function is
%An $n$-\textbf{multisubset} of $[m]$ is a multiset of (possibly repeated) elements from the set $[m]$ containing $n$ elements, counted with repetitions. %Let $\left({m \choose n})\$ be the number of multisubsets of $[m]$ of size $n$. 
\[
\textsf{Partition}(x,y) = \left(\frac1{1-xy}\right) \left(\frac1{1-x^2y}\right) \left(\frac1{1-x^3y}\right)\cdots 
\]
There is no simple explicit formula for the number $p(n)$ of partitions of $n$, although there is a very elegant and efficient recursive formula. Setting $y=-1$ in the previous identity, and invoking Euler's pentagonal theorem \cite{f.Aigner}
\begin{equation}\label{f.eq:pentagonal}
\prod_{n \geq 0}(1-x^n) = 1 + \sum_{j \geq 1} (-1)^j \left(x^{j(3j-1)/2} + x^{j(3j+1)/2}\right)
\end{equation}
we obtain:
\[
p(n) = p(n-1) + p(n-2) - p(n-5) - p(n-7) + p(n-12) + p(n-15) -  \cdots
\]
where $1, 2, 5, 7, 12, 15, 22, 26, \ldots$ are the pentagonal numbers.

\item (Partitions into distinct parts) Let $\textsf{DistPartition}$ be the family of partitions into distinct parts, weighted by $x^n y^k$ where $n$ is the sum of the parts and $k$ is the number of parts. Then $\textsf{DistPartition} \cong \{1, \overline{1}\} \times \{2, \overline{2}\} \times \cdots $, where ${\mathrm{wt}}(i) = x^iy$ and ${\mathrm{wt}}(\overline{i}) = 1$ for $i = 1, 2, \ldots$, 
so the corresponding generating function is
%An $n$-\textbf{multisubset} of $[m]$ is a multiset of (possibly repeated) elements from the set $[m]$ containing $n$ elements, counted with repetitions. %Let $\left({m \choose n})\$ be the number of multisubsets of $[m]$ of size $n$. 
\[
\textsf{DistPartition}(x,y) = (1+xy)(1+x^2y)(1+x^3y) \cdots
\]

\item (Partitions into restricted parts) It is clear how to adapt the previous generating functions to partitions where the parts are restricted. For example, the identity
\[
\frac1{1-x} = (1+x)(1+x^2)(1+x^4)(1+x^8)(1+x^{16})\cdots
\]
expresses the fact that every positive integer can be written uniquely in binary notation, as a sum of distinct powers of $2$. The identity
\[
(1+x)(1+x^2)(1+x^3)(1+x^4)\cdots = \frac{1}{1-x} \cdot \frac{1}{1-x^3} \cdot \frac{1}{1-x^5 } \cdot \frac{1}{1-x^7} \cdots,
\]
which may be proved by writing $1+x^k = (1-x^{2k})/(1-x^k)$, 
expresses that the number of partitions of $n$ into distinct parts equals the number of partitions of $n$ into odd parts.

\item (Partitions with restrictions on the size and the number of parts) Let $p_{\leq k}(n)$ be the number of partitions of $n$ into at most $k$ parts. This is also the number of partitions of $n$ into parts of size at most $k$. To see this, represent a partition $n=a_1 + \cdots + a_j$ as a left-justified array of squares, where the $i$th row has $a_i$ squares. Each partition $\lambda$ has a conjugate partition $\lambda'$ obtained by exchanging the rows and the columns of the Ferrers diagram. Figure \ref{f.fig:Ferrers} shows the Ferrers diagram of $431$ and its conjugate partition $3221$. It is clear that $\lambda$ has at most $k$ parts if and only if $\lambda'$ has parts of size at most $k$.

From the previous discussion it is clear that 
\[
\sum_{n \geq 0} p_{\leq k}(n) x^n = \frac{1}{1-x} \cdot \frac{1}{1-x^2} \cdots \frac{1}{1-x^k}
\]

\begin{figure}[ht]
 \begin{center}
  \includegraphics[scale=.8]{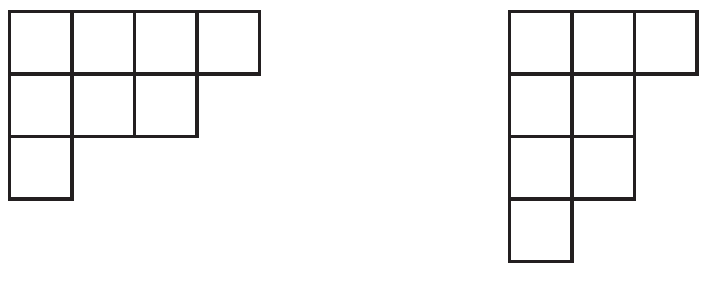}
  \caption{ \label{f.fig:Ferrers}
The Ferrers diagrams of the conjugate partitions 431 and 3221.}
 \end{center}
\end{figure}

Now let $p_{\leq j, \leq k}(n)$ be the number of partitions of $n$ into at most $j$ parts of size at most $k$. Then
\[
\sum_{n \geq 0} p_{\leq j, \leq k}(n) x^n = \frac{(1-x)(1-x^2)\cdots(1-x^{j+k})}{(1-x)(1-x^2)\cdots(1-x^{j}) \cdot (1-x)(1-x^2)\cdots(1-x^{k})}
\]

This is easily proved by induction, using that $p_{\leq j, \leq k}(n) = p_{\leq j, \leq k-1}(n)  + p_{\leq j, \leq k}(n-k)$.

%\comment{Mention connection with $q$-enumeration.}

\item (Even and odd partitions) 
Setting $y=-1$ into the generating function for partitions into distinct parts of Example 9, we get
\[
\sum_{n \geq 0} (\textsf{edp}(n)-\textsf{odp}(n))x^n = 
\prod_{n \geq 1} (1-x^n) = 1 - x - x^2 + x^5 + x^7 - x^{12} - x^{15} + \cdots
\]
where $\textsf{edp}(n)$ (resp. $\textsf{odp}(n)$) counts the partitions of $n$ into an even (resp. odd) number of distinct parts. Euler's pentagonal formula (\ref{f.eq:pentagonal}) says that $\textsf{edp}(n) - \textsf{odp}(n)$ equals $0$ for all $n$ except for the pentagonal numbers, for which it equals $1$ or $-1$. 

There are similar results for partitions into distinct parts coming from a given set $S$.

$\bullet$ When $S$ is the set of Fibonacci numbers, the coefficients of the generating function
\[
\prod_{n \geq 1}(1-x^{F_n}) = 1 - x - x^2 + x^4 + x^7 - x^8 + x^{11} - x^{12} - x^{13} + x^{14} + \cdots
\]
are also equal to $0, 1,$ or $-1$. \cite{f.Robbins, f.ArdilaFibonacci} 

$\bullet$
This is also true for any ``$k$-Fibonacci sequence" $S=\{a_1, a_2, \ldots\}$ given by $a_n=a_{n-1} + \cdots + a_{n-k}$ for $n > k$ and $a_j > a_{j-1}  + \cdots + a_1$ for $1 \leq j \leq k$. \cite{f.Diao} 

$\bullet$
The result also holds trivially for $S=\{2^j \, : \, j \in {\mathbb{N}}\}$ since there is a unique partition of any $n$ into distinct powers of $2$. 

These three results seem qualitatively different from (and increasingly less surprising than) Euler's result, as these sequences $S$ grow much faster than $\{1,2, 3, \ldots\}$, and $S$-partitions are sparser. 
Can more be said about the sets $S$ of positive integers for which the coefficients of $\prod_{n \in S}(1-x^n)$ are all $1, 0$ or $-1$?

\item (Set partitions) A \textbf{set partition} of a set $S$ is an unordered collection of pairwise disjoint sets $S_1, \ldots, S_k$ whose union is $S$. The family of set partitions with $k$ parts is $\textsf{SetPartition}^k \cong \bullet \times \textsf{Seq}(\{1\}) \times \bullet \times \textsf{Seq}(\{1,2\}) \times \cdots \times \bullet \times \textsf{Seq}(\{1,2, \ldots, k\})$, where the singleton $\bullet$ and all numbers $i$ have size $1$. To see this, we regard a word such as
$w = \bullet \, 11 \bullet 1221 \bullet 31$ as an instruction manual to build a set partition $S_1, \ldots, S_k$. The $j$th symbol $w_j$ tells us where to put the number $j$: if $w_j$ is a number $h$, we add $j$ to the part $S_h$;  if $w_j$ is the $i$th $\bullet$, then we add $j$ to a new part $S_i$. The sample word above leads to the partition $\{1, 2, 3, 5, 8, 11\},\{4, 6, 7\}, \{9, 10\}$. This process is easily reversible. It follows that
\[
\sum_{n \geq 0} S(n,k) x^n = \frac{x}{1-x} \cdot \frac{x}{1-2x} \cdot \cdots \cdot \frac{x}{1-kx} ,
\]
where $S(n,k)$ is the number of set partitions of $[n]$ into $k$ parts. These numbers are called the \textbf{Stirling numbers of the second kind}. 

The equation $(1-kx)\sum_{n \geq 0} S(n,k) x^n = x 
\sum_{n \geq 0} S(n,k-1) x^n$ gives the recurrence 
\[
S(n,k) = k S(n-1, k) + S(n-1,k-1), \qquad 1 \leq k \leq n,
\]
with initial values $S(n,0) = S(n,n) = 1$. Note the great similarity with Pascal's recurrence.

\item (Catalan structures) It is often said that if you encounter a new family of mathematical objects, and you have to guess how many objects of size $n$ there are, you should guess ``the Catalan number $C_n= \frac1{n+1}{2n \choose n}$." The Catalan family has more than 200  incarnations in combinatorics and other fields \cite{f.EC2, f.Catalanaddendum}; let us see three important ones.

\begin{enumerate}

\item (Plane binary trees) A \textbf{plane binary tree} is a rooted tree where every internal vertex has a left and a right child. Let $\textsf{PBTree}$ be the family of plane binary trees, where a tree with $n$ internal vertices (and necessarily $n+1$ leaves) has size $n$. 
A tree is either the trivial tree $\circ$ of size $0$, or the grafting of a left subtree and a right subtree at the root $\bullet$, so $\textsf{PBTree} \cong \circ + (\textsf{PBTree} \, \times \, \bullet \, \times \, \textsf{PBTree})$. It follows that the generating function for plane binary trees satisfies
\[
T(x) = 1 + T(x) \, x \, T(x).
\]
We may use the quadratic formula\footnote{Since this is the first time we are using the quadratic formula, let us do it carefully. Rewrite the equation as $(1-2xT(x))^2 = 1-4x$ , or $(1-2xT(x) - \sqrt{1-4x})(1-2xT(x) + \sqrt{1-4x}) =0$. Since ${\mathbb{C}}[[x]]$ is an integral domain, one of the factors must be $0$. From the constant coefficients we see that it must be the first factor.} and the binomial theorem to get
\[
T(x) = \frac{1-\sqrt{1-4x}}{2x} =
\sum_{n \geq 0} \frac1{n+1} {2n \choose n} x^n.
\]
It follows that the number of plane binary trees with $n$ internal vertices (and $n+1$ leaves) is the \textbf{Catalan number} $C_n = \frac1{n+1}{2n \choose n}$.

\item (Triangulations) A \textbf{triangulation} of a convex polygon is a subdivision into triangles using only the diagonals of $P$. A triangulation of an $(n+2)$-gon has $n$ triangles; we say it has size $n$. If we fix an edge $e$ of $P$, then a triangulation of $P$ is obtained by choosing the triangle $T$ that will cover $e$, and then choosing a triangulation of the two polygons to the left and to the right of $T$. Therefore $\textsf{Triang} \cong \circ + (\textsf{Triang} \, \times \, \bullet \, \times \, \textsf{Triang})$ and the number of triangulations of an $(n+2)$-gon is also the Catalan number $C_n$.

\item (Dyck paths) A \textbf{Dyck path} $P$ of length $n$ is a path from $(0,0)$ to $(2n,0)$ which uses the steps $(1,1)$ and $(1,-1)$ and never goes below the $x$-axis. Say $P$ is \textbf{irreducible} if it touches the $x$ axis exactly twice, at the beginning and at the end. Let $D(x)$ and $I(x)$ be the generating functions for Dyck paths and irreducible Dyck paths.

A Dyck path is equivalent to a sequence of \textbf{irreducible} Dyck paths. Also, an irreducible path of length $n$ is the same as a Dyck path of length $n-1$ with an additional initial and final step. Therefore
\[
D(x) = \frac1{1-I(x)}, \qquad I(x) = xD(x)
\]
from which it follows that $D(x) = \frac{1-\sqrt{1-4x}}{2x}$ as well, and the number of Dyck paths of length $n$ is also the Catalan number. 

\end{enumerate}

Generatingfunctionology gives us fairly easy algebraic proofs that these three families are enumerated by the Catalan numbers. Once we have discovered this fact, the temptation to search for nice bijections is hard to resist.

Our algebraic analysis suggests a bijection $\phi$ from (b) to (a). The families of plane binary trees and triangulations grow under the same recursive recipe, and so we can let the bijection grow with them, mapping a triangulation $T \times \bullet \times  T'$ to the tree $\phi(T) \times \bullet \times \phi(T')$. A non-recursive description of the bijection is the following. Consider a triangulation $T$ of the polygon $P$, and fix an edge $e$. Put a vertex inside each triangle of $T$, and a vertex outside $P$ next to each edge other than $e$. Then connect each pair of vertices separated by an edge. Finally, root the resulting tree at the vertex adjacent to $e$. This bijection is illustrated in Figure \ref{f.fig:triangtotrees}.

\begin{figure}[ht]
 \begin{center}
  \includegraphics[scale=.8]{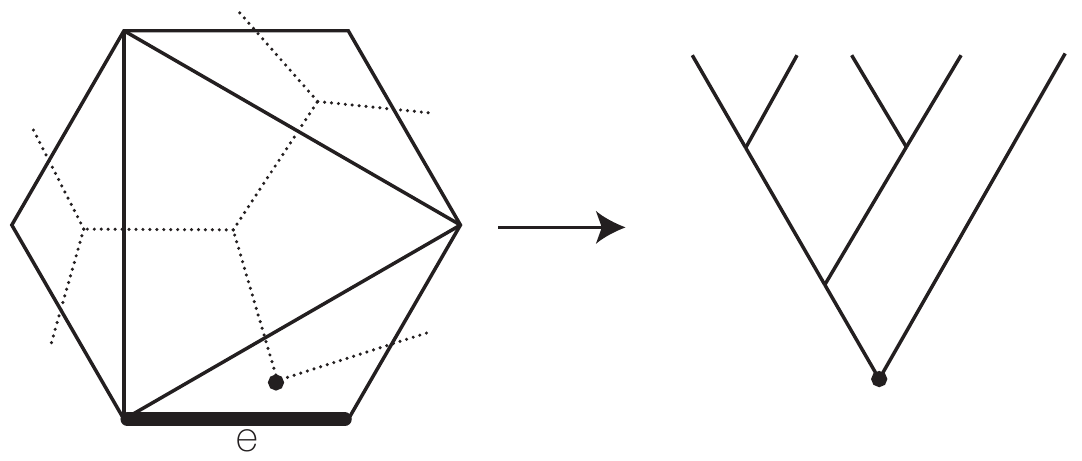}
  \caption{ \label{f.fig:triangtotrees}
The bijection from triangulations to plane binary trees.}
 \end{center}
\end{figure}

\begin{figure}[ht]
 \begin{center}
  \includegraphics[scale=.8]{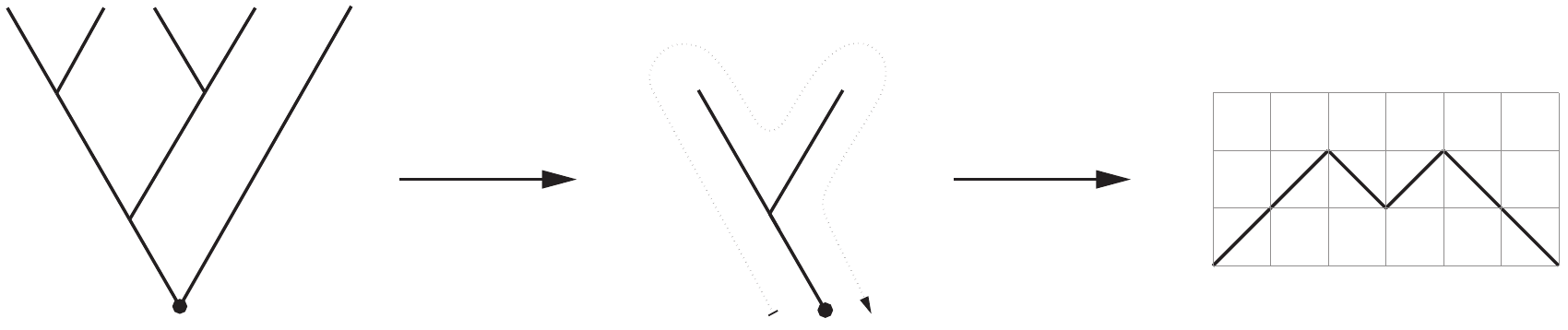}
  \caption{ \label{f.fig:treestoDyck}
The bijection from plane binary trees to Dyck paths.}
 \end{center}
\end{figure}

A bijection from (a) to (c) is less obvious from our algebraic computations, but is still not difficult to obtain. Given a plane binary tree $T$ of size $n$, prune all the leaves to get a tree $T'$ with $n$ vertices. Now walk around the periphery of the tree, starting on the left side from the root, and continuing until we traverse the whole tree. Record the walk in a Dyck path $D(T)$ : every time we walk up (resp. down) a branch we take a step up (resp. down) in $D(T)$. One easily checks that this is a bijection.

Even if it may be familiar, it is striking that two different (and straightforward) algebraic computations show us that two families of objects that look quite different are in fact equivalent combinatorially. 
Although a simple, elegant bijection can often explain the connection between two families more transparently, the algebraic approach is sometimes simpler, and better at discovering such connections.

\item ($k$-Catalan structures) Let $\textsf{PTree}_k$ be the class of \textbf{plane $k$-ary trees}, where every vertex that is not a leaf has $k$ ordered children; let the size of such a tree be its number of leaves. In the sense of Theorem \ref{f.th:ogf}.5(a), this is precisely an ${\mathcal{A}}$-tree, where ${\mathcal{A}} =  \{\bullet, \bullet^k\}$ consists of one structure of size $1$ and one of size $k$. Therefore $\textsf{PTree}_k = (x - x^k)^{<-1>}$. Lagrange inversion (Theorem \ref{f.th:ogf}.5) then gives
\[
m [x^m] A^{<-1>}(x) = [x^{m-1}] \left(\frac{1}{1-x^{k-1}}\right)^m = [x^{m-1}] \sum_{n\geq 0} {m + n - 1 \choose n} x^{(k-1)n}.
\]
It follows that the a plane $k$-ary tree must have $m=(k-1)n+1$ leaves for some integer $n$, and the number of such trees is the \textbf{$k$-Catalan number}
\[
C_n^k = \frac{1}{(k-1)n+1} {kn \choose n}.
\]
This is an alternative way to compute the ordinary Catalan numbers $C_n = C_n^2$.

The $k$-Catalan number $C_n^k$ also has many different interpretations \cite{f.HeubachLiMansour}; we mention two more. It counts the subdivisions of an $(n(k-1)+2)$-gon $P$ into (necessarily $n$) $(k+1)$-gons using diagonals of $P$, and the paths from $(0,0)$ to $(n, (k-1)n)$ with steps $(0,1)$ and $(1,0)$ that never rise above the line $y=(k-1)x$.
\end{enumerate}

\bigskip
\noindent \textbf{\textsf{Other applications.}} Let us now discuss a few other interesting applications which illustrate the power of Theorem \ref{f.th:ogf}.

\begin{enumerate}

\setcounter{enumi}{15}

\item (Motzkin paths) The \textbf{Motzkin number} $M_n$ is the number of paths from $(0,0)$ to $(n,0)$ using the steps $(1,1)$, $(1,-1)$, and $(1,0)$ which never go below the $x$-axis. Imitating our argument for Dyck paths, we obtain a formula for the generating function:
\[
M(x) = \frac{1}{1-(x + xM(x)x)} \qquad \Longrightarrow \qquad  M(x) = \frac{1-x - \sqrt{1-2x-3x^2}}{2x^2}. 
\]
The quadratic equation $x^2M^2 + (x-1)M + 1 = 0$ gives rise to the quadratic recurrence $M_n = M_{n-1} + \sum_{i} M_iM_{n-2-i}$. 
We will see in Section \ref{f.sec:algebraic} that the fact that $M(x)$ satisfies a polynomial equation leads to a more efficient recurrence:
\[
(n+2)M_n = (2n+1)M_{n-1} + (3n-3)M_{n-2}. %, \qquad M_n = \frac1{n+1} \sum_{k=0}^{n} {n+1 \choose k+1} {n - k \choose k}.
\]

\item (Schr\"oder paths)  
The (large) \textbf{Schr\"oder number }$r_n$ is the number of paths from $(0,0)$ to $(2n,0)$ using steps $NE=(1,1), SE=(1, -1),$ and $E=(2,0)$ which stays above the $x$ axis. Their generating function satisfies $R(x) = 1/(1-x-xR(x))$, and therefore
\[
R(x) = \frac{1-x-\sqrt{1-6x+x^2}}{2x}.
\]

\end{enumerate}

Let us see some additional applications of Theorem \ref{f.th:ogf}.3 to count combinatorial objects which factor uniquely as an ordered ``product" of ``irreducible" objects. 

\begin{enumerate}

\setcounter{enumi}{17}

\item (Domino tilings of rectangles) In Section \ref{f.sec:answer} we let $a_n$ be the number of domino tilings of a $2 \times n$ rectangle. Such a tiling is uniquely a sequence of blocks, where each block is either a vertical domino (of width $1$) or two horizontal dominoes (of width $2$). This truly explains the formula:
\[
A(x) = \frac{1}{1-(x+x^2)}.
\]
Similarly, if $a_{m,n}$ is the number of domino tilings of a $2 \times n$ rectangle using $v$ vertical tiles, we immediately obtain:
\[
\sum_{m, n \geq 0} a_{m,n} v^m x^n = \frac{1}{1-(vx+x^2)}.
\]

\end{enumerate}

\noindent Sometimes the enumeration of irreducible structures is not immediate, but still tractable.

\begin{enumerate}
\setcounter{enumi}{18}
\item(Monomer-dimer tilings of rectangles) Let $T(2,n)$ be the number of tilings of a $2 \times n$ rectangles with dominoes and unit squares. Say a tiling is irreducible if it does not contain an internal vertical line from top to bottom. Then $\textsf{Tilings} \cong \textsf{Seq}(\textsf{IrredTilings})$. It now takes some thought to recognize the irreducible tilings: 

\begin{figure}[ht]
 \begin{center}
  \includegraphics[scale=.8]{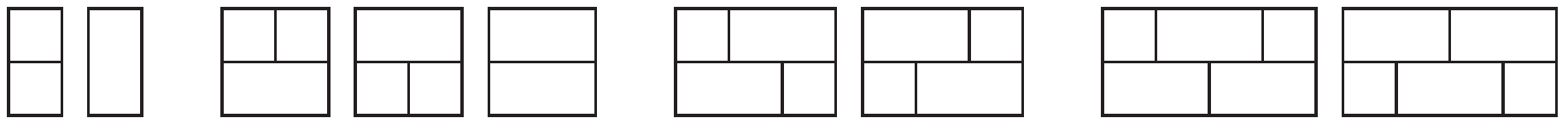}
  \caption{ \label{f.fig:irredtilings}
The irreducible tilings of $2 \times n$ rectangles into dominoes and unit squares.}
 \end{center}
\end{figure}

There are $3$ irreducible tilings of length $2$, and $2$ of every other length greater than or equal to $1$. Therefore
\[
\sum_{n \geq 0} T(2,n) x^n = \frac{1}{1-(2x+3x^2+2x^3+2x^4+\cdots)} = \frac{1-x}{1-3x-x^2+x^3}
\] 
We will see in Theorem \ref{f.th:rational}.2 that this gives $T(2,n) \sim c \cdot \alpha^n$ where $\alpha \approx 3.214\ldots$ is the inverse of the smallest positive root of the denominator.
\end{enumerate}

\noindent
Sometimes the enumeration of \textbf{all} objects is easier than the enumeration of the irreducible ones. In that case we can use Theorem \ref{f.th:ogf}.3 in the opposite direction.

\begin{enumerate}
\setcounter{enumi}{19}
\item (Irreducible permutations) 
A permutation $\pi$ of $[n]$ is \textbf{irreducible} if it does not factor as a permutation of $\{1, \ldots, m\}$ and a permutation of $\{m+1, \ldots, n\}$ for $1 \leq m < n$; that is, if $\pi([m]) \neq [m]$ for all $1 \leq m < n$. Clearly every permutation factors uniquely into irreducibles, so 
\[
\sum_{n \geq 0} n! x^n = \frac1{1-\textsf{IrredPerm}(x)}.
\]
This gives the series for $\textsf{IrredPerm}$.
%\comment{does the asymptotic enumeration follow? L. Comtet, Sur les coefficients de l'inverse de la serie formelle Sum $n! t^n$, Comptes Rend. Acad. Sci. Paris, A 275 (1972), 569-572.}
\end{enumerate}

\noindent
There are many interesting situations where it is possible, but not at all trivial, to decompose the objects that interest us into simpler structures. To a combinatorialist this is good news --  the techniques of this section are useful tools, but are not enough; there is no shortage of interesting work to do. Here is a great example.

\begin{enumerate}
\setcounter{enumi}{20}
\item (Domino towers) \cite{f.GouyouViennot, f.BetremaPenaud, f.Zeilberger} A \textbf{domino tower} is a stack of horizontal $2 \times 1$ bricks in a brickwork pattern, so that no brick is directly above another brick, such that the bricks on the bottom level are contiguous, and  every higher brick is (half) supported on at least one brick in the row below it. Let the size of a domino tower be the number of bricks. 

\begin{figure}[ht]
 \begin{center}
  \includegraphics[scale=.7]{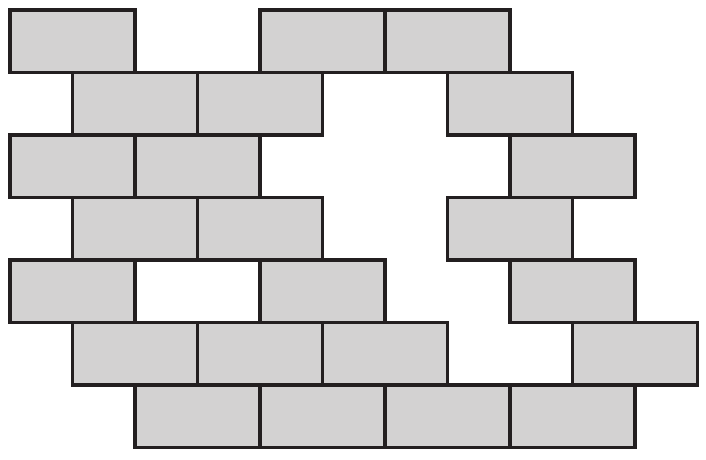}
  \caption{ \label{f.fig:xavier}
A domino tower of $19$ bricks.}
 \end{center}
\end{figure}

Remarkably, there are $3^{n-1}$ domino towers consisting of $n$ bricks. Equally remarkably, no simple bijection is known. The nicest argument that we know is as follows.

\begin{figure}[ht]
 \begin{center}
  \includegraphics[scale=.7]{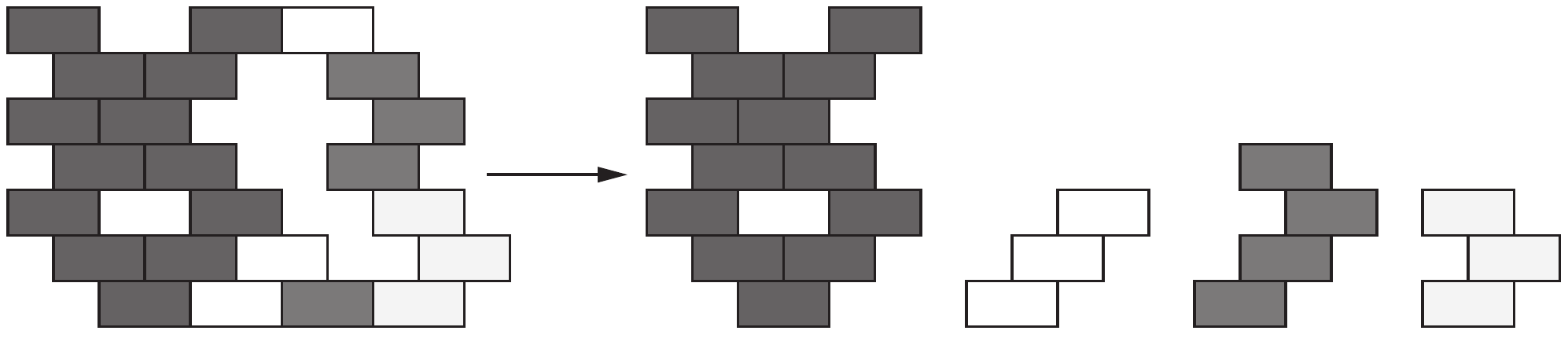}
  \caption{ \label{f.fig:xavier}
The decomposition of a domino tower into a pyramid and three half-pyramids.}
 \end{center}
\end{figure}

We decompose a domino tower $x$ into smaller pieces, as illustrated in Figure \ref{f.fig:xavier}. Each new piece is obtained by pushing up the leftmost remaining brick in the bottom row, dragging with it all the bricks encountered along the way. The first piece $p$ will be a \textbf{pyramid}, which we define to be a domino tower with only one brick in the bottom row. All subsequent pieces $h_1, \ldots, h_k$ are \textbf{half-pyramids}, which are pyramids containing no bricks to the left of the bottom brick. This decomposition is reversible. To recover $x$, we drop $h_k, h_{k-1}, \ldots, h_1, p$ from the top in that order; each piece is dropped in its correct horizontal position, and some of its bricks may get stuck on the previous pieces. This shows that the corresponding combinatorial classes satisfy $X \cong P \times \textsf{Seq}(H)$. 

\begin{figure}[ht]
 \begin{center}
  \includegraphics[scale=.7]{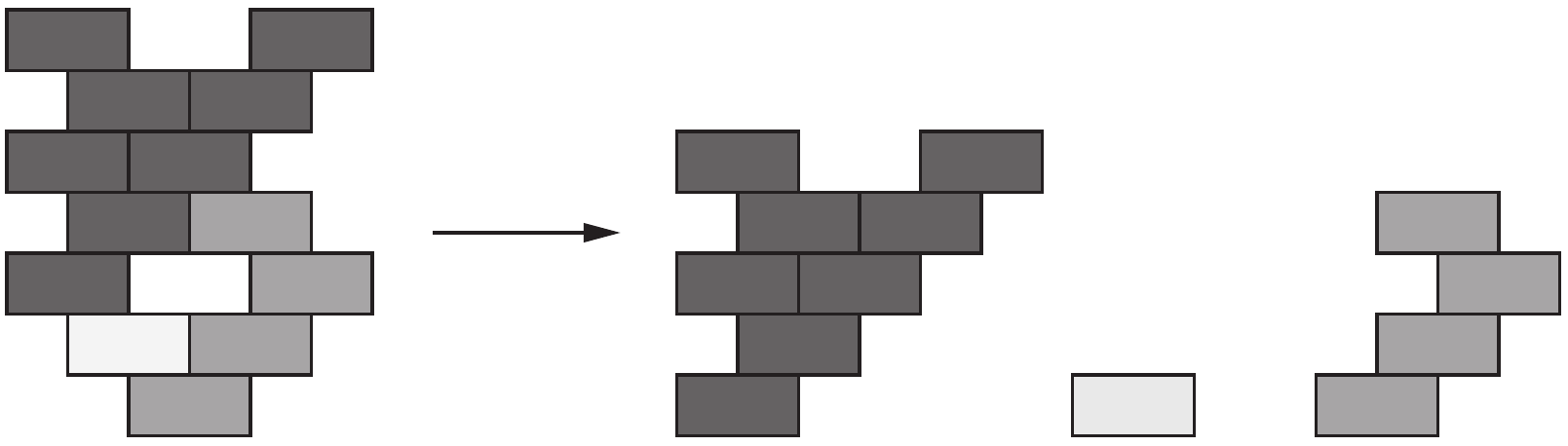}
  \caption{ \label{f.fig:pyramid}
A pyramid and its decomposition into half-pyramids.}
 \end{center}
\end{figure}

Similarly, we may decompose a pyramid $p$ into half-pyramids, as shown in Figure \ref{f.fig:pyramid}. Each new half-pyramid is obtained by pushing up the leftmost remaining brick (which is not necessarily in the bottom row), dragging with it all the bricks that it encounters along the way. This shows that $P \cong \textsf{Seq}_{\geq 1}(H) := H + (H \times H) + \cdots  $.

\begin{figure}[ht]
 \begin{center}
  \includegraphics[scale=.7]{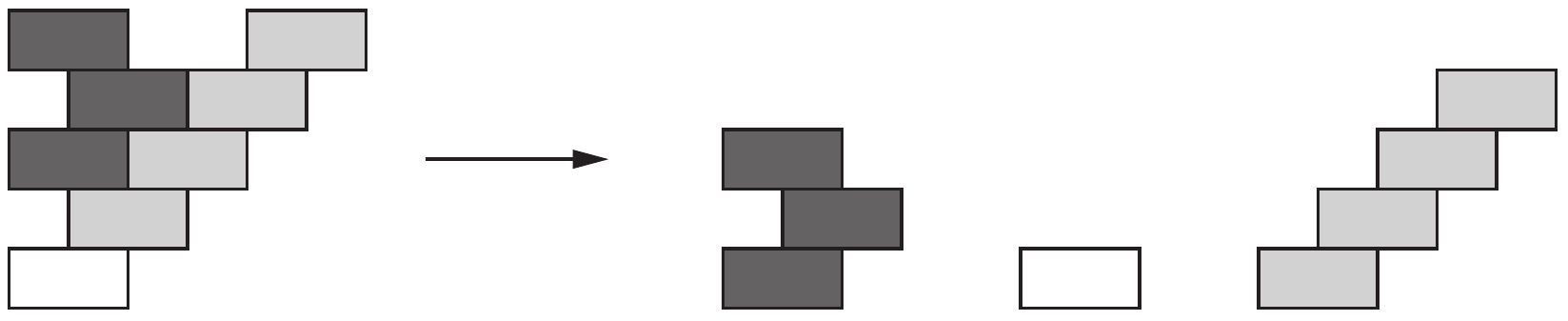}
  \caption{ \label{f.fig:halfpyramid}
A (non-half-pyramid) pyramid and its decomposition into two half-pyramids and a bottom brick.}
 \end{center}
\end{figure}

%Now consider a pyramid $p$. If it is not a half-pyramid, consider the lowest brick to the left of the bottom brick, and push it up, dragging with it all the bricks it encounters along the way. We are pushing up a pyramid $p'$, and what remains is a half-pyramid $h$. It is clear how to recover $p$ from $p'$ and $h$. Therefore $P \cong H + P \times H$.

Finally consider a half-pyramid $h$; there are two cases. If there are other bricks on the same horizontal position as the bottom brick, consider the lowest such brick, and push it up, dragging with it all the bricks it encounters along the way, obtaining a half-pyramid $h_1$. Now remove the bottom brick; what remains is a half-pyramid $h_2$. This is shown in Figure \ref{f.fig:halfpyramid}.
 As before, we can recover $h$ from $h_1$ and $h_2$. On the other hand, if there are no bricks above the bottom brick, removing the bottom brick leaves either a half-pyramid or the empty set. Therefore $H \cong (H \times \bullet \times H) + (\bullet \times H) + \bullet$. 

The above relations correspond to the following identities for the corresponding generating functions:
\[
X = \frac{P}{1-H}, \qquad P = \frac{H}{1-H}, \qquad H = xH^2 + xH + x
\]
\noindent Surprisingly cleanly, we obtain $X(x) = {x}/(1-3x) = \sum_{n \geq 1} 3^{n-1}x^n$. This proves that there are $3^{n-1}$ domino towers of size $n$. 

Although we do not need this here, it is worth noting that half-pyramids are enumerated by Motzkin numbers; their generating functions are related by $H(x)=xM(x)$.
\end{enumerate}

\subsection{\textsf{Exponential generating functions}}\label{f.sec:egfs}

Ordinary generating functions are usually not well suited for counting combinatorial objects with a labelled ground set. In such situations, exponential generating functions are a more effective tool.

Consider a family ${\mathcal{A}}={\mathcal{A}}_0 \sqcup {\mathcal{A}}_1 \sqcup A_2 \sqcup \cdots$ of labelled combinatorial structures, where ${\mathcal{A}}_n$ consists of the structures that we can place on the ground set $[n] = \{1, \ldots, n\}$ (or, equivalently, on any other labeled ground set of size $n$). 
If $a \in {\mathcal{A}}_n$  we let $|a|=n$ be the \textbf{size} of $a$. 
We also let $a_n$ be the number of elements of size $n$. 
The \textbf{exponential generating function} of ${\mathcal{A}}$ is
\[
A(x) = \sum_{a \in {\mathcal{A}}} \frac{x^{|a|}}{|a|!} = a_0\frac{x^0}{0!} + a_1\frac{x^1}{1!} + a_2\frac{x^2}{2!} + a_3\frac{x^3}{3!} + \cdots
\]

We may again assign a weight ${\mathrm{wt}}(a)$ to each object $a$, usually a monomial in variables $x_1, \ldots, x_n$, and consider the \textbf{weighted exponential generating function} of ${\mathcal{A}}$ to be the formal power series
\[
A_{{\mathrm{wt}}}(x_1, \ldots, x_n, x) =  \sum_{a \in {\mathcal{A}}} {\mathrm{wt}}(a) \frac{x^{|a|}}{|a|!} 
\]
Examples of combinatorial structures (with their respective size functions) are
permutations (number of elements), graphs (number of vertices), or set partitions (size of the set). We may weight these objects by $t^k$ where $k$ is, respectively, the number or cycles, the number of edges, or the number of parts.

\subsubsection{\textsf{Operations on labeled structures and their exponential generating functions}}\label{f.sec:operationsegfs}

Again, there are some simple operations on \emph{labelled} combinatorial structures,  which correspond to simple algebraic operations on the exponential generating functions. Starting with a few simple structures, these operations are sufficient to generate many interesting combinatorial structures. This will allow us to compute the exponential generating functions for those structures.

\begin{theorem}\label{f.th:egf}
Let ${\mathcal{A}}$ and ${\mathcal{B}}$ be labeled combinatorial structures. 
\begin{enumerate}
\item (${\mathcal{C}} = {\mathcal{A}} + {\mathcal{B}}$: Disjoint union)
If a ${\mathcal{C}}$-structure on a finite set $S$ is obtained by choosing an ${\mathcal{A}}$-structure on $S$ \textbf{or} a ${\mathcal{B}}$-structure on $S$, then 
\[
C(x) = A(x) + B(x)
\]

\item (${\mathcal{C}} = {\mathcal{A}} * {\mathcal{B}}$: Labelled Product) 
If a ${\mathcal{C}}$-structure on a finite set $S$ is obtained by partitioning $S$ into disjoint sets $S_1$ and $S_2$ and putting an ${\mathcal{A}}$-structure on $S_1$ \textbf{and} a ${\mathcal{B}}$-structure on $S_2$, then
\[
C(x) = A(x) B(x)
\]
This result also holds for weighted structures if the weight of a ${\mathcal{C}}$-structure is the product of the weights of the respective ${\mathcal{A}}$ and ${\mathcal{B}}$-structures.

\item (${\mathcal{C}} = \textsf{Seq}_*({\mathcal{B}})$: Labeled Sequence)
If a ${\mathcal{C}}$-structure on a finite set $S$ is obtained by choosing an \textbf{ordered} partition of $S$ into a sequence of blocks and putting a ${\mathcal{B}}$-structure on each block, then 
\[
C(x) = \frac{1}{1-B(x)}
\]
This result also holds for weighted structures if the weight of a ${\mathcal{C}}$-structure is the product of the weights of the respective ${\mathcal{B}}$ structures.

\item (${\mathcal{C}} = \textsf{Set}({\mathcal{B}})$: Set) \textbf{Exponential Formula.}
If a ${\mathcal{C}}$-structure on a finite set $S$ is obtained by choosing an \textbf{unordered} partition of $S$ into a set of blocks and putting a ${\mathcal{B}}$-structure on each block, then 
\[
C(x) = e^{B(x)}
\]
This result also holds for weighted structures if the weight of a ${\mathcal{C}}$-structure is the product of the weights of the respective ${\mathcal{B}}$-structures. 

In particular, if $c_k(n)$ is the number of ${\mathcal{C}}$-structures of an $n$-set which decompose into $k$ ``components" (${\mathcal{B}}$-structures), we have
\[
\sum_{n, k, \geq 0} c_k(n) \frac{x^n}{n!} y^k = e^{yB(x)} = C(x)^y
\]

\item (${\mathcal{C}} = {\mathcal{A}} \circ {\mathcal{B}}$: Composition) \textbf{Compositional Formula.}
If a ${\mathcal{C}}$-structure on a finite set $S$ is obtained by choosing an \emph{unordered} partition of  $S$ into a set of blocks, putting a ${\mathcal{B}}$-structure on each block, and putting an ${\mathcal{A}}$-structure on the set of blocks, then
\[
C(x) = A(B(x))
\]
This result also holds for weighted structures if the weight of a ${\mathcal{C}}$-structure is the product of the weights of the ${\mathcal{A}}$ structure on its set of blocks and the weights of the ${\mathcal{B}}$ structures on the individual blocks.

%\comment{\item (${\mathcal{C}} = {\mathcal{A}}^{-1}$) \textbf{Lagrange inversion formula}. Do I do this? In terms of Aigner's approach?}

\end{enumerate}
\end{theorem}

Theorem \ref{f.th:egf}.4 is a ``labeled" analog of Theorem \ref{f.th:ogf}.3; it is 
 useful when we are counting labeled combinatorial objects which ``decompose" uniquely as a set of ``indecomposable" objects. 
It tells us that we can count all objects if and only if we can count the indecomposable ones, or vice versa. Amazingly, we also obtain for free the finer enumeration of the objects by their number of components.

\begin{proof}
1. is clear. The identity in 2. is equivalent to $c_n = \sum_k {n \choose k} a_kb_{n-k}$, which corresponds to the given combinatorial description. Iterating 2., we see that the exponential generating functions for $k$-sequences of ${\mathcal{B}}$-structures is $B(x)^k$, and hence the one for $k$-sets of ${\mathcal{B}}$-structures is $B(x)^k/k!$. This readily implies 3, 4, and 5. The weighted statements follow similarly.
\end{proof}

The following statements are perhaps less fundamental, but also useful.  

\begin{theorem}\label{f.th:egf}
Let ${\mathcal{A}}$ be a labeled combinatorial structure. 
\begin{enumerate}
\item (${\mathcal{C}} = {\mathcal{A}}_{+}$: Shifting)
If a ${\mathcal{C}}$-structure on $S$ is obtained by adding a new element $t$ to $S$ and choosing an ${\mathcal{A}}$-structure on $S \cup \{t\}$, then
\[
C(x) = A'(x).
\]

\item (${\mathcal{C}} = {\mathcal{A}}_{\bullet}$: Rooting)
If a ${\mathcal{C}}$-structure on $S$ is a \emph{rooted ${\mathcal{A}}$-structure}, obtained by choosing an ${\mathcal{A}}$-structure on $S$ and an element of $S$ called the \textbf{root}, then
\[
C(x) = xA(x).
\]

\item (Sieving by parity of size) 
If the ${\mathcal{C}}$-structures are precisely the ${\mathcal{A}}$-structures of even size, 
\[
C(x) = \frac{A(x) + A(-x)}{2}. 
\]

\item (Sieving by parity of components) 
Suppose ${\mathcal{A}}$-structures decompose uniquely into components, so ${\mathcal{A}} = \textsf{Set}({\mathcal{B}})$ for some ${\mathcal{B}}$. If the ${\mathcal{C}}$-structures are the ${\mathcal{A}}$-structures having only components of even size, 
\[
C(x) = \sqrt{A(x)A(-x)}.
\]

\item (Sieving by parity of number of components) 
Suppose ${\mathcal{A}}$-structures decompose uniquely into components, so ${\mathcal{A}} = \textsf{Set}({\mathcal{C}})$ for some ${\mathcal{C}}$. 
If the ${\mathcal{C}}$-structures are precisely the ${\mathcal{A}}$-structures having an even number of components, 
\[
C(x) = \frac12\left(A(x) + \frac1{A(x)}\right).
\]
\end{enumerate}
Similar sieving formulas hold modulo $k$ for any $k \in {\mathbb{N}}$.
\end{theorem}

\begin{proof}
We have $c_n=a_{n+1}$ in 1., $c_n = na_n$ in 2., and $c_n =\frac12 (a_n + (-1)^na_n)$ in 3.; the generating function formulas follow. Combining 3. with the Exponential Formula we obtain 4. and 5.

Similarly we see that the generating function for ${\mathcal{A}}$-structures whose size is a multiple of $k$ is $\frac1k\left(A(x) + A(\omega x) + \cdots + A(\omega^{k-1} x)\right)$ where $\omega$ is a primitive $k$th root of unity. If we wish to count elements of size $i \textrm{ mod } k$, we use 1. to shift this generating function $i$ times.
\end{proof}

\subsubsection{\textsf{Examples}}\label{f.sec:egfexamples}

\bigskip
\noindent \textbf{\textsf{Classical applications.}} Once again, these simple ideas give very easy solutions to many classical enumeration problems.

\begin{enumerate}

\item (Trivial classes) Again we consider the trivial classes $\circ$ with only one element of size $0$, and $\bullet$ with only one element of size $1$. Their exponential generating functions are $1$ and $x$, respectively.

\item (Sets) A slightly less trivial class of $\textsf{Set} = \textsf{Set}(\bullet)$ contains one set of each size.  We also let \textsf{Set}$_{\geq 1}$ denote the class of non-empty sets, with generating function $e^x - 1$. The exponential generating functions are
\[
\textsf{Set}(x) = e^x, \qquad \textsf{Set}_{\geq 1}(x) = e^x-1.
\]

\item (Set Partitions) 
In Section \ref{f.sec:ogfexamples} we found the ordinary generating function for Stirling numbers $S(n,k)$ for a given $k$; but in fact it is easier to use exponential generating functions. Simply notice that $\textsf{SetPartition} \cong \textsf{Set}(\textsf{Set}_{\geq 1})$, and the Weighted Exponential Formula then gives 
\[
\textsf{SetPartition}(x,y) = \sum_{n, k \geq 0} S(n,k) \frac{x^n}{n!} y^k = e^{y(e^x-1)}.
\]

\item (Permutations) Let $\textsf{Perm}_n$ consist of the $n!$  permutations of $[n]$. A permutation is a labeled sequence of singletons, so $\textsf{Perm} = \textsf{Seq}_*(\bullet)$, and the generating function for permutations is 
\[
\textsf{Perm}(x) = \sum_{n \geq 0} n! \frac{x^n}{n!} = \frac{1}{1-x}.
\]

\item (Cycles) Let $\textsf{Cycle}_n$ consist of the cyclic orders of $[n]$. These are the ways of arranging $1, \ldots, n$ around a circle, where two orders are the same if they differ by a rotation of the circle. There is an $n$-to-$1$ mapping from permutations to cyclic orders obtained by wrapping a permutation around a circle, so 
\[
\textsf{Cycle}(x)=\sum_n (n-1)! x^n/n! = -\log (1-x).
\] 

There is a more indirect argument which will be useful to us later. Recall that a permutation $\pi$  can be written uniquely as a (commutative) product of disjoint cycles of the form $(i, \pi(i), \pi^2(i), \ldots, \pi^{k-1}(i))$ where $k$ is the smallest index such that $\pi^k(i)=i$. For instance, the permutation $835629741$ can be written in cycle notation as $(18469)(235)(7)$. Then $\textsf{Perm} = \textsf{Set}(\textsf{Cycle})$ so $1/(1-x) = e^{\textsf{Cycle}(x)}$.

\item (Permutations by number of cycles) 
The (signless) \textbf{Stirling number of the first kind} $c(n,k)$ is the number of permutations of $n$ having $k$ cycles. The Weighted Exponential Formula gives 
\[
\sum_{n, k \geq 0} c(n,k) \frac{x^n}{n!} y^k = e^{y \textsf{Cycle}(x)} =   \left(\frac{1}{1-x}\right)^{y} = \sum_{n \geq 0} y(y+1) \cdots (y+n-1) \frac{x^n}{n!}.
\]
It follows that the Stirling numbers of the first kind $c(n,k)$ are the coefficients of the polynomial $y(y+1) \cdots (y+n-1)$. 

\end{enumerate}

\bigskip

\noindent
\textbf{\textsf{Other applications}} The applications of these techniques are countless; let us consider a few more applications, old and recent.

\begin{enumerate}
\setcounter{enumi}{6}
\item (Permutations by cycle type) 
The \textbf{type} of a permutation $\pi \in S_n$ is $\textrm{type}(w) = (c_1, \ldots, c_n)$ where $c_i$ is the number of cycles of length $i$. For indeterminates $\mathbf{t} = (t_1, \ldots, t_n)$, let $\mathbf{t}^{\textrm{type}(w)} = t_1^{c_1} \cdots t_n^{c_n}$. The \textbf{cycle indicator} of the symmetric group $S_n$ is $Z_n = \frac{1}{n!} \sum_{w \in S_n} \mathbf{t}^{\textrm{type}(w)}$. The Weighted Exponential Formula immediately gives
\[
%\sum_{n \geq 0} Z_n x^n = e^{t_1x + t_2\frac{x^2}{2} + t_3 \frac{x^3}{3} + \cdots}
\sum_{n \geq 0} Z_n x^n = e^{t_1x \, + \, t_2{x^2}/{2} \, + \, t_3 {x^3}/{3} + \cdots}
\]
Let us discuss two special cases of interest.

\item (Derangements) A \textbf{derangement} of $[n]$ is a permutation such that $\pi(i) \neq i$ for all $i \in [n]$. Equivalently, a derangement is a permutation with no cycles of length $1$. It follows that $\textsf{Derangement} = \textsf{Set}(\textsf{Cycle}_{\geq 2})$, so  the number $d_n$ of derangements of $[n]$ is given by
\[
 \textsf{Derangement}(x) = \sum_{n \geq 0} d_n \frac{x^n}{n!} = e^{-\log(1-x)-x} = e^{-x} + x e^{-x} + x^2 e^{-x} + \cdots.
\]
which leads to the explicit formula:
\[
d_n = n! \left(1 - \frac1{1!} + \frac1{2!} - \frac1{3!} + \cdots \pm \frac1{n!}\right) \sim \frac{n!}{e}.
\]

\item (Involutions) An \textbf{involution} of $[n]$ is a permutation $w$ such that $w^2$ is the identity. Equivalently, an involution is a permutation with cycles of length $1$ and $2$, so the number $i_n$ of involutions of $[n]$ is given by
\[
\textsf{Inv}(x) = \sum_{n \geq 0} i_n \frac{x^n}{n!} =  e^{x+\frac{x^2}2}
\]
Note that $\textsf{Inv}'(x) = (x+1)\textsf{Inv}(x)$, which gives  $i_n = i_{n-1} + (n-1)i_{n-2}$. In Section \ref{f.sec:algebraic} we will explain the more general theory of D-finite power series, which turns differential equations for power series into recurrences for the corresponding sequences.

\item (Trees) A \textbf{tree} is a connected graph with no cycles. Consider a ``birooted" tree $(T,a,b)$ on $[n]$ with two (possibly equal) root vertices $a$ and $b$. Regard the unique path $a=v_0, v_1, \ldots, v_k = b$ as a ``spine" for $T$; the rest of the tree consists of rooted trees hanging from the $v_i$s; direct their edges towards the spine. Now regard $v_1 \ldots v_k$ as a permutation in one-line notation, and rewrite it in cycle notation, while continuing to hang the rooted trees from the respective $v_i$s.
This transforms $(T,a,b)$ into a directed graph consisting of a disjoint collection of cycles with trees directed towards them. Every vertex has outdegree $1$, so this defines a function $f:[n] \rightarrow [n]$. A moment's thought will convince us that this is a bijection. Therefore there are $n^n$ birooted trees on $[n]$, and hence there are $n^{n-2}$ trees on $[n]$.

\begin{figure}[h]
\begin{center}
\includegraphics[scale=.7]{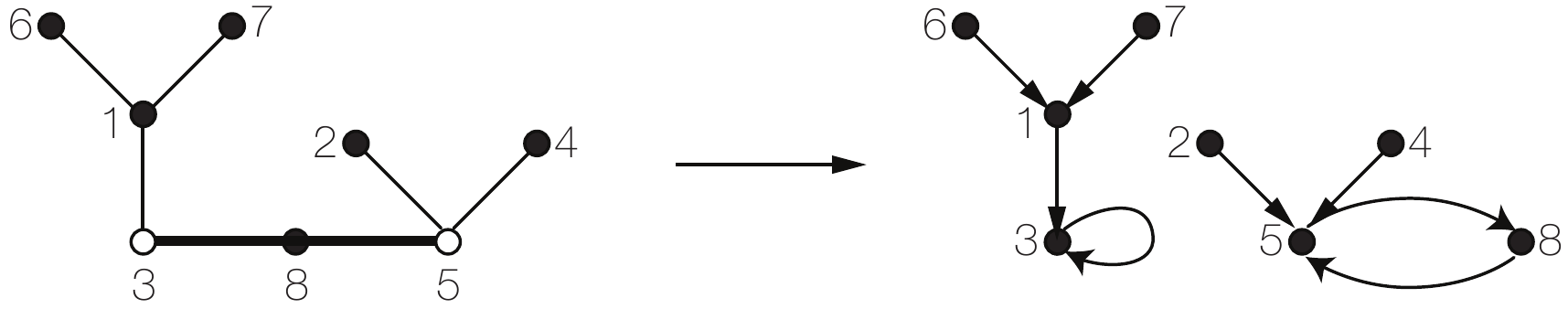} 
\end{center}
\caption{\label{f.fig:birooted}A tree on $[8]$ birooted at $a=3$ and $b=5$, and the corresponding function $f:[8] \rightarrow [8]$.}
\end{figure}

\item (Trees, revisited.) Let us count trees in a different way. Let a \textbf{rooted tree} be a tree with a special vertex called the root, and a \textbf{planted forest} be a graph with no cycles where each connected component has a root. Let $t_n, r_n, f_n$ and $T(x), R(x), F(x)$ be the sequences and exponential generating functions enumerating trees, rooted trees, and planted forests, respectively.

\begin{figure}[h]
\begin{center}
\includegraphics[scale=1.1]{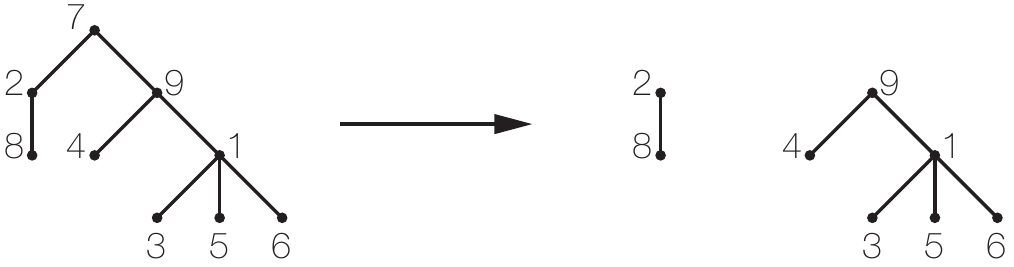} 
\end{center}
\caption{\label{f.fig:foresttotrees} A rooted tree seen as a root attached to the roots of a planted forest.}
\end{figure}

Planted forests are vertex-disjoint unions of rooted trees, so $F(x) = e^{R(x)}$.  Also, as illustrated in Figure \ref{f.fig:foresttotrees}, 
a rooted tree $T$ consists of a root attached to the roots of a planted forest, so $R(x) = xF(x)$. 
It follows that $x=R(x) e^{-R(x)}$, so 
\[
R(x) = (xe^{-x})^{<-1>}.
\]
Lagrange inversion (Theorem \ref{f.th:ogf}.5) gives $n \cdot \frac{r_n}{n!} = [x^{n-1}]e^{nx} = \frac{n^{n-1}}{(n-1)!}$, so 
\[
r_n=n^{n-1}, \qquad  f_n=(n+1)^{n-1}, \qquad t_n = n^{n-2}.
\]
We state a finer enumeration; see \cite[Theorem 5.3.4]{f.EC2} for a proof. The \textbf{degree sequence} of a rooted forest on $[n]$ is $(\deg 1, \ldots, \deg n)$ where $\deg i$ is the number of children of $i$. For example the degree sequence of the rooted tree in Figure \ref{f.fig:foresttotrees} is $(3, 1, 0, 0, 0, 0, 2, 0, 2)$. Then the number of planted forests with a given degree sequence $(d_1, \ldots, d_n)$ and (necessarily) $k=n - (d_1 + \cdots + d_n)$ components is
\[
{n-1 \choose k-1}{n-k \choose d_1, \ldots, d_n}.
\]

The number of forests on $[n]$ is given by a more complicated alternating sum; see \cite{f.Takacs}.
% Takacs. On the number of distinct forests.

\end{enumerate}

\begin{enumerate}
\setcounter{enumi}{11}
\item (Permutations, revisited) Here is an unnecessarily complicated way of proving there are $n!$ permutations of $[n]$. A permutation $\pi$ of $[n+1]$ decomposes uniquely as a concatenation $\pi = L(n+1)R$ for permutations $L$ and $R$ of two complementary subsets of $[n]$. Therefore $\textsf{Shift}(\textsf{Perm}) = (\textsf{Perm}) * (\textsf{Perm})$, and the generating function $P(x)$ for permutations satisfies $P'(x) = P(x)^2$ with $P(0)=1$. Solving this differential equation gives $P(x) = \frac{1}{1-x} = \sum_{n \geq 0} n! \frac{x^n}{n!}$

\item (Alternating permutations) The previous argument was gratuitous for permutations, but it will now help us to enumerate the class $\textsf{Alt}$ of \textbf{alternating} permutations $w$, which satisfy $w_1 < w_2 > w_3 < w_4 > \cdots$. The \textbf{Euler numbers} are $E_n = |\textsf{Alt}_n|$; let $E(x)$ be their exponential generating function. 
We will need the class RevAlt of permutations $w$ with $w_1 > w_2 < w_3 > w_4 < \cdots$. The map $w = w_1\ldots w_n \mapsto w' = (n+1-w_1)\ldots (n+1-w_n)$ on permutations of $[n]$ shows that $\textsf{Alt} \cong \textsf{RevAlt}$.

Now consider alternating permutations $L$ and $R$ of two complementary subsets of $[n]$. For $n \geq 1$, exactly one of the permutations $L(n+1)R$ and $L'(n+1)R$ is alternating or reverse alternating, and every such permutation arises uniquely in that way.  For $n = 0$ both are alternating. Therefore $\textsf{Shift}(\textsf{Alt} + \textsf{RevAlt}) = (\textsf{Alt} * \textsf{Alt}) + \circ$, so $2E'(x) = E(x)^2+1$ with $E(0) = 1$. Solving this differential equation we get
\[
E(x) = \sum_{n \geq 0} E_n \frac{x^n}{n!} = \sec x + \tan x
\]
Therefore $\sec x$ and $\tan x$ enumerate the alternating permutations of even and odd length, respectively. The Euler numbers are also called secant and tangent numbers for this reason. This surprising connection allows us to give combinatorial interpretations of various trigonometric identities, such as $1 + \tan^2 x = \sec^2 x$.

%\comment{
%\item \comment{Move this elsewhere} (Parking functions) A \textbf{parking function} of length $n$ is a sequence $(a_1, \ldots, a_n)$ of integers between $1$ and $n$ containing at least $i$ integers less than or equal to $i$ for $1 \leq i \leq n$. Equivalently, if we arrange the sequence in increasing order $b_1 \leq \cdots \leq b_n$, we should have $b_i \leq i$. Parking functions get their name from the following interpretation.
%
%Consider a one way street with parking spots labeled $1, \ldots, n$ in that order. Suppose that $n$ cars $C_1, \ldots, C_n$ drive into the street in that order. Each car $C_i$ has a preferred parking spot $a_i$. When it enters the street, $C_i$ tries to park in the parking spot $a_i$; if it is taken, then it will park on the next available spot, if there are any. In the end, all the cars will be able to park if and only if $(a_1, \ldots, a_n)$ is a parking function.
%}

\item (Graphs)
Let $g(v)$ and $g_{\textrm{conn}}(v)$ be the number of simple\footnote{containing no multiple edges or loops} graphs and connected graphs on $[v]$, respectively. The Exponential Formula tells us that their exponential generating functions are related by $G(x) = e^{G_\textrm{conn}(x)}$. In this case it is hard to count the connected graphs directly, but it is easy to count all graphs: to choose a graph we just have to decide whether each edge is present or not, so $g(v) = 2^{v \choose 2}$. This gives us 
\[
\sum_{v \geq 0} g_{\textrm{conn}}(v) \frac{x^v}{v!}  = \log \left(\sum_{v \geq 0} 2^{v \choose 2} \frac{x^v}{v!} \right)
\]

We may easily adjust this computation to account for edges and components. There are ${v(v-1)/2 \choose e}$ graphs on $[v]$ with $e$ edges; say $g(v,c,e)$ of them have $c$ components, and give them weight $y^cz^e$. Then
\[
\sum_{v, c, e \geq 0} g(v,c,e) \frac{x^v}{v!} y^c z^e = \left( \sum_{v,e \geq 0} {{v \choose 2} \choose e} \frac{x^v}{v!} z^e \right)^y = F(x, 1+z)^y
\]
where
\[
F(\alpha, \beta) = \sum_{n \geq 0} \frac{\alpha^n \, \beta^{n \choose 2}}{n!}.
\]
is the \textbf{deformed exponential function} of \cite{f.Sokal}.

%Let $g(n)$ and $g_{\textrm{conn}}(n)$ be the number of simple\footnote{containing no multiple edges or loops} graphs and connected graphs on $[n]$, respectively. The Exponential Formula tells us that their exponential generating functions are related by $G(x) = e^{G_\textrm{conn}(x)}$. In this case it is hard to count the connected graphs directly, but it is easy to count all graphs: to choose a graph we just have to decide whether each edge is present or not, so $g(n) = 2^{n \choose 2}$. This gives us 
%\[
%\sum_{n \geq 0} g_{\textrm{conn}}(n) \frac{x^n}{n!}  = \log \left(\sum_{n \geq 0} 2^{n \choose 2} \frac{x^n}{n!} \right), \qquad  \qquad
%\sum_{n, k \geq 0} g_k(n) \frac{x^n}{n!} y^k = \left(\sum_{n \geq 0} 2^{n \choose 2} \frac{x^n}{n!} \right)^y.
%\]
%where $g_k(n)$ is the number of graphs on $[n]$ with $k$ components. 
%
%We may further weight our graphs by the number of edges. There are ${n(n-1)/2 \choose m}$ graphs on $[n]$ with $m$ edges; suppose $g_k(m,n)$ of them have $k$ components. Then
%\[
%\sum_{n, m,  k \geq 0} g_k(m,n) \frac{x^n}{n!} y^k z^m = \sum_{n, m \geq 0} \left({n(n-1)/2 \choose m} \frac{x^n}{n!} z^m\right)^y = F(x, 1+z)^y
%\]
%where
%\[
%F(\alpha, \beta) = \sum_{n \geq 0} \frac{\alpha^n \, \beta^{n \choose 2}}{n!}.
%\]
%is the \emph{deformed exponential function} of \cite{f.Sokal} $F(\alpha, \beta)$, an evaluation of the three variable Rogers-Ramanujan function.
%

\item (Signed Graphs)
A simple \textbf{signed graph} $G$ is a set of vertices, with at most one ``positive" edge and one ``negative" edge connecting each pair of vertices. We say $G$ is \emph{connected} if and only if its underlying graph $\overline{G}$ (ignoring signs) is connected. A \emph{cycle} in $G$ corresponds to a cycle of $\overline{G}$; we call it \emph{balanced} if it contains an even number of negative edges, and \emph{unbalanced} otherwise. We say that $G$ is \emph{balanced} if all its cycles are balanced. Let $s(v, c_+,c_-, e)$ be the number of signed graphs with $v$ vertices, $e$ edges, $c_+$ balanced components, and $c_-$ unbalanced components; we will need the generating function
\[
S(x, y_+,y_-,z) = \sum_{G \textrm{ signed graph }} s(v, c_+,c_-, e)\, \frac{x^{v}}{v!} y_+^{c_+} \,y_-^{c_-}\, z^{e}
\]
in order to carry out a computation in Section \ref{f.sec:Tuttegeneralizations}; we follow \cite{f.ArdilaCastilloHenley}.

Let $S(x, y_+,y_-,z)$, $B(x, y_+,z) $, $C_+(x, z) $, and $C_-(x,z)$ be the generating functions for signed, balanced, connected balanced, and connected unbalanced graphs, respectively. The Weighted Exponential Formula gives:
\[
B =e^{y_+C_+}, \qquad S = e^{y_+C_+ + y_-C_-}
\]
so if we can compute $C_+$ and $C_-$ we will obtain $B$ and $S$. In turn, these equations give
\[
C_+(x,z) = \frac12 \log B(x, 2,z), \qquad C_+(x,z) + C_-(x,z) = \log S(x,1,1,z).
\]
and we now compute the right hand side of these two equations. (In the first equation, we set $t_+ = 2$ because, surprisingly, $B(x,2,z)$ is easier to compute than $B(x,1,z)$.)
One is easy:
\[
S(x,1,1,z) = \sum_{e, v \geq 0} {v(v-1) \choose e} \frac{x^v}{v!} z^e = F(x, (1+z)^2).
\]

\begin{figure}[h]
\begin{center}
\includegraphics[scale=1]{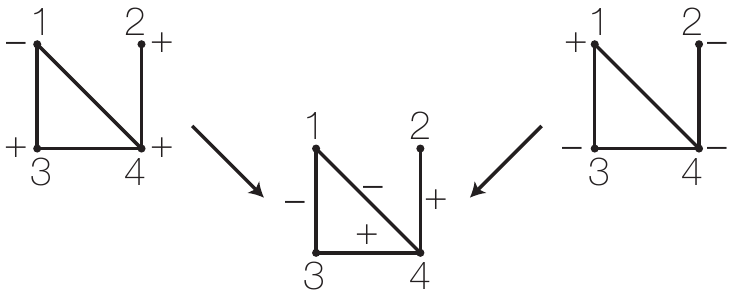} 
\end{center}
\caption{\label{f.fig:markedtosigned}The two marked graphs that give rise to one balanced signed graph.}
\end{figure}

For the other one, we count balanced signed graphs by relating them with \textbf{marked graphs}, which are simple graphs with a sign $+$ or $-$ on each \textbf{vertex}. \cite{f.HararyKabell} A marked graph $M$ gives rise to a balanced signed graph $G$ by assigning to each edge the product of its vertex labels. Furthermore, if $G$ has $c$ components, then it arises from precisely $2^{c}$ different marked graphs, obtained from $M$ by choosing some connected components and changing their signs. This correspondence is illustrated in Figure \ref{f.fig:markedtosigned}.
It follows that $B(x, 2y, z) = \sum_{B \textrm{ balanced}} 2^c \,  b(v, c, e) \frac{x^{v}}{v!} y^{c} z^{e} =  \sum_{M \textrm{ marked}} m(v, c, e) \frac{x^{v}}{v!} y^{c} z^{e} $ is the generating function for marked graphs, and hence $B(x, 2,z)$ may be computed easily:
\[
B(x, 2,z) = \sum_{e,v} {{v \choose 2} \choose e} 2^v \frac{x^v}{v!} z^e = F(2x,1+z)
\]
Putting these equations together yields 
\[
S(x,y_+,y_-,z) = F(2x, 1+z)^{(y_+ - y_-)/2} F(x, (1+z)^2)^{y_-}. 
\]

\end{enumerate}

%\comment{
%\subsection{\textsf{Ordinary or exponential?}}\label{f.sec:basic}
%
%
%Ordinary generating functions are best suited to study combinatorial objects on an unlabeled ground set. More vaguely speaking, many of the structures above have an underlying linear (left-to-right or top-to-bottom) structure that is crucial to makes this line of attack effective. 
%When the ground set comes with a natural set of labels, it is generally better to use exponential generating functions; see \cite{f.sec:egfs}. For instance, \comment{give an example where we compute both an ogf and an egf.}
%
%
%
%Sometimes use other things. See binomial posets.
%}
%
%

\subsection{\textsf{Nice families of generating functions}}\label{f.sec:nicegfs}

In this section we discuss three nice properties that a generating function can have: being rational, algebraic, or D-finite. Each one of these properties gives rise to useful properties for the corresponding sequence of coefficients.

\subsubsection{\textsf{Rational generating functions}}\label{f.sec:rational}

Many sequences in combinatorics and other fields satisfy three equivalent properties:  they satisfy a recursive formula with constant coefficients, they are given by an explicit formula in terms of polynomials and exponentials, and their generating functions are rational. We understand these sequences very well. The following theorem tells us how to translate any one of these formulas into the others.

\begin{theorem}\label{f.th:rational} \cite[Theorem 4.1.1]{f.EC1}
Let $a_0, a_1, a_2, \ldots$ be a sequence of complex numbers and let $A(x) = \sum_{n \geq 0} a_n x^n$ be its ordinary generating function. Let $q(x) = 1 + c_1 x + \cdots + c_d x^d = (1-r_1x)^{d_1} \cdots (1-r_kx)^{d_k}$ be a complex polynomial of degree $d$. The following are equivalent:
\begin{enumerate}
\item
The sequence satisfies the linear recurrence with constant coefficients
\[
a_n + c_1a_{n-1} + \cdots + c_da_{n-d} = 0 \qquad (n \geq d).
\]
\item
There exist polynomials $f_1(x), \ldots, f_k(x)$ with $\deg f_i(x) < d_i$ for $1 \leq i \leq n$ such that
\[
a_n = f_1(n) \,r_1^n + \cdots + f_k(n) \, r_k^n.
\]
\item
There exists a polynomial $p(x)$ with $\deg p(x) < d$ such that $A(x) = p(x)/q(x)$.
\end{enumerate}
\end{theorem}

Notice that Theorem \ref{f.th:rational}.2 gives us the asymptotic growth of $a_n$ immediately.
Let us provide more explicit recipes.

\medskip

\noindent 
$(1 \Rightarrow 2)$ Extract the inverses $r_i$ of the roots of $q(x) = 1 + c_1 x + \cdots + c_d x^d$ and their multiplicities $d_i$. The $d_1+\cdots+d_k=d$ coefficients of the $f_i$s are the unknowns in the system of $d$ linear equations $a_n = f_1(n) \,r_1^n + \cdots + f_k(n)r_k^n$  $(n=0, 1, \ldots, d-1)$, which has a unique solution.

\medskip
\noindent 
$(1 \Rightarrow 3)$ Read off $q(x) = 1 + c_1 x + \cdots + c_d x^d$ from the recurrence; the coefficients of $p(x)$ are $[x^k]p(x) = a_k + c_1a_{k-1} + \cdots + c_da_{k-d}$ for $0 \leq k < d$, where $a_i=0$ for $i < 0$.

\medskip
\noindent 
$(2 \Rightarrow 1)$ Compute the $c_i$s using $q(x) = \prod_{i}^k (1-r_ix)^{\deg f_i+1}$.

\medskip
\noindent 
$(2 \Rightarrow 3)$ Let $q(x) = \prod_i (1-r_ix)^{\deg f_i+1}$, and compute the first $k$ terms of $p(x)=A(x)q(x)$; the others are $0$.

\medskip
\noindent 
$(3 \Rightarrow 1)$ Extract the $c_i$s from the denominator $q(x)$.

\medskip
\noindent 
$(3 \Rightarrow 2)$ Compute the partial fraction decomposition $p(x)/q(x) = \sum_{i=1}^k p_i(x)/(1-r_ix)^{d_i}$ where $\deg p_i(x) < d_i$ and use $(1-r_ix)^{-d_i} = \sum_n {d_i+n-1 \choose d_i-1} r_i^n x^n$ to extract $a_n = [x^n]p(x)/q(x)$.

\bigskip

\noindent
\textbf{\textsf{Characterizing polynomials.}} 
As a special case of Theorem \ref{f.th:rational}, we obtain a useful characterization of  sequences given by a polynomial. The \textbf{difference operator} $\Delta$ acts on sequences, sending the sequence $\{a_n \, : \, n \in {\mathbb{N}}\}$ to the sequence $\{\Delta a_n \, : \, n \in {\mathbb{N}}\}$ where $\Delta a_n = a_{n+1}-a_n$.

\begin{theorem}\label{f.th:polynomial} \cite[Theorem 4.1.1]{f.EC1}
Let $a_0, a_1, a_2, \ldots$ be a sequence of complex numbers and let $A(x) = \sum_{n \geq 0} a_n x^n$ be its ordinary generating function. Let $d$ be a positive integer. The following are equivalent:
\begin{enumerate}
\item
We have $\Delta^{d+1} a_n = 0$ for all $n \in {\mathbb{N}}$.
\item
There exists a polynomial $f(x)$ with $\deg f \leq d$ such that $a_n = f(n)$ for all $n \in {\mathbb{N}}$.
\item
There exists a polynomial $p(x)$ with $\deg p(x) \leq d$ such that $A(x) = p(x)/(1-x)^{d+1}$.
\end{enumerate}
\end{theorem}

We have already seen some combinatorial polynomials and generating functions whose denominator is a power of $1-x$; we will see many more examples in the following sections.
%\comment{Link to algebra, polytopes, etc here.}

\subsubsection{\textsf{Algebraic and D-finite generating functions}}\label{f.sec:algebraic}

When the series $A(x) = \sum_n a_nx^n$ we are studying is not rational, the next natural question to ask is whether $A(x)$ is algebraic. If it is, then just as in the rational case, the sequence $a_n$ still satisfies a linear recurrence, although now the coefficients are polynomial in $n$. This general phenomenon is best explained by introducing the wider family of ``D-finite" (also known as ``differentially finite" or ``holonomic") power series. Let us discuss a quick example before we proceed to the general theory. 

We saw that the generating function for the Motzkin numbers satisfies the quadratic equation 

\begin{equation}\label{f.eq:Motzkin}
x^2M^2 + (x-1)M + 1 = 0
\end{equation}
which gives rise to the quadratic recurrence $M_n = M_{n-1} + \sum_{i} M_iM_{n-2-i}$ with $M_0=1$. This is not a bad recurrence, but we can find a better one.
Differentiating (\ref{f.eq:Motzkin}) we can express $M'$ in terms of $M$. Our likely first attempt leads us to $M'=-(2xM^2+M)/(2x^2M+x-1)$, which is not terribly enlightening. However, using (\ref{f.eq:Motzkin}) and a bit of purposeful algebraic manipulation, we can rewrite this as a linear equation with polynomial coefficients:
\[
(x-2x^2-3x^3)M' + (2-3x-3x^2)M - 2 = 0.
\]
Extracting the coefficient of $x^n$ we obtain the much more efficient recurrence relation
\[
(n+2)M_n - (2n+1)M_{n-1} - (3n-3)M_{n-2} = 0. \qquad {(n \geq 2)}
\]
We now explain the theoretical framework behind this example.

\bigskip

\noindent
\textbf{\textsf{Rational, algebraic, and D-finite series}}. Consider a formal power series $A(x)$ over the complex numbers. We make the following definitions.

\bigskip

\begin{center}
\begin{tabular}{| l | p{12cm} |}
\hline
$A(x)$ is \textbf{rational} & There exist polynomials $p(x)$ and $q(x) \neq 0$ such that \begin{center} $q(x)A(x)=p(x)$. \end{center}\\
\hline
$A(x)$ is \textbf{algebraic} & There exist polynomials $p_0(x), \ldots, p_d(x)$ such that  \begin{center} $
p_0(x) + p_1(x)A(x) + p_2(x) A(x)^2 + \cdots + p_d(x) A(x)^d = 0$. \end{center}\\
%\displaystyle \sum_{k=0}^d P_i(x) A(x)^i = 0$. \end{center}\\
\hline
$A(x)$ is \textbf{D-finite} & There exist polynomials $q_0(x), \ldots, q_d(x), q(x)$ such that  \begin{center} $
q_0(x)A(x) + q_1(x) A'(x) + q_2(x)A''(x) +  \cdots + q_d(x) A^{(d)}(x) = q(x)$. \end{center}\\
%\displaystyle \sum_{k=0}^d P_i(x) A^{(i)}(x) = 0$. \end{center}\\
\hline
\end{tabular}
\end{center}
\bigskip

%\begin{center}
%\begin{tabular}{| l | p{12cm} |}
%\hline
%$A(x)$ is \textbf{rational} & There exist polynomials $p(x)$ and $q(x) \neq 0$ such that \begin{center} $qA=p$. \end{center}\\
%\hline
%$A(x)$ is \textbf{algebraic} & There exist polynomials $p_0(x), \ldots, p_d(x)$ such that  \begin{center} $
%p_0+ p_1A + p_2 A^2 + \cdots + p_d A^d = 0$. \end{center}\\
%%\displaystyle \sum_{k=0}^d P_i(x) A(x)^i = 0$. \end{center}\\
%\hline
%$A(x)$ is \textbf{D-finite} & There exist polynomials $q_0(x), \ldots, q_d(x), q(x)$ such that  \begin{center} $
%q_0A + q_1 A'(x) + q_2A'' +  \cdots + q_d A^{(d)} = q$. \end{center}\\
%%\displaystyle \sum_{k=0}^d P_i(x) A^{(i)}(x) = 0$. \end{center}\\
%\hline
%\end{tabular}
%\end{center}
%\bigskip

Now consider the corresponding sequence $a_0, a_1, a_2 \ldots$ and make the following definitions.
\medskip

\begin{center}
\begin{tabular}{| p{5cm} | p{10.5cm} |}
\hline
$\{a_0, a_1, \ldots\}$ is 
 \textbf{c-recursive} & There are constants $c_0, \ldots, c_d \in {\mathbb{C}}$ such that 
\begin{center}
$c_0 a_n + c_1 a_{n-1} + \cdots + c_d a_{n-d} = 0 \qquad (n \geq d)$
\end{center} \\
\hline
$\{a_0, a_1, \ldots\}$ is \textbf{P-recursive} & There are complex  polynomials $c_0(x), \ldots, c_d(x)$ such that 
\begin{center}
$c_0(n) a_n + c_1(n) a_{n-1} + \cdots + c_d(n) a_{n-d} = 0 \qquad (n \geq d)$
\end{center} \\
\hline
\end{tabular}
\end{center}

\bigskip

%The converse is not true; these sequences correspond to the wider family of  ``D-finite" series. 
These families contain most (but certainly not all) series and sequences that we encounter in combinatorics. They are related as follows.

\begin{theorem}
Let $A(x) = a_0 + a_1x + a_2x^2 + \cdots$ be a formal power series. The following implications hold.

\begin{center}
\begin{tabular}{|ccccc|}
\hline
$A(x)$ is rational & $\Longrightarrow$ & $A(x)$ is algebraic & $\Longrightarrow$ & $A(x)$ is D-finite  \\
& & & &\\
$\Updownarrow$ & & & & $\Updownarrow$ \\
& & & &\\
$\{a_0, a_1, \ldots\}$ is c-recursive  & &  $\Longrightarrow$ & & $\{a_0, a_1, \ldots\}$ is P-recursive \\
%\\
%$\{a_n \, : \, n \in {\mathbb{N}}\}$ satisfies a  & &  & & $\{a_n \, : \, n \in {\mathbb{N}}\}$ satisfies a  \\
% linear recurrence with & & $\Longrightarrow$ & & linear recurrence with \\
%  constant coefficients & &  & &   polynomial coefficients  \\
\hline
 \end{tabular}
\end{center}

\end{theorem}

\begin{proof}
We already discussed the correspondence between rational series and c-recursive functions, and rational series are trivially algebraic. Let us prove the remaining  statements.
%
%\begin{proposition}
%An algebraic power series is $D$-finite.
%\end{proposition}
%
%\begin{proof}

\medskip
\noindent \emph{(Algebraic $\Rightarrow$ D-finite)} 
Suppose $A(x)$ satisfies an algebraic equation of degree $d$. Then $A$ is algebraic over the field ${\mathbb{C}}(x)$, and the field extension ${\mathbb{C}}(x, A)$ is a vector space over ${\mathbb{C}}(x)$ having dimension at most $d$.

Taking the derivative of the polynomial equation satisfied by $A$, we get an expression for $A'$ as a rational function of $A$ and $x$. Taking derivatives repeatedly, we find that all derivatives of $A$ are in ${\mathbb{C}}(x,A)$. It follows that $1, A, A', A'', \ldots, A^{(d)}$ are linearly dependent over ${\mathbb{C}}(x)$, and a linear relation between them is a certificate for the D-finiteness of $A$.

%
%\begin{proposition}
%A sequence $\{a_n \, : \, n \in {\mathbb{N}}\}$ is P-recursive if and only if its generating function $A(x) = \sum_n a_n x^n$ is D-finite.
%\end{proposition}
%\begin{proof}
\medskip
\noindent
\emph{(P-recursive $\Leftrightarrow$ D-finite):}  
If $q_0(x)A(x) + q_1(x) A'(x) + \cdots + q_d(x) A^{(d)}(x) = q(x)$, comparing the coefficients of $x^n$ gives a $P$-recursion for the $a_i$s. In the other direction, given a P-recursion for the $a_i$s of the form $c_0(n)a_n + \cdots + c_d(n)a_{n-d}=0$, it is easy to obtain the corresponding differential equation after writing $c_i(x)$ in terms of the basis $\{(x+i)_k \, : \, k \in {\mathbb{N}}\}$ of ${\mathbb{C}}[x]$, where $(y)_k = y(y-1)\cdots(y-k+1)$.
\end{proof}

The converses are not true. For instance, $\sqrt{1+x}$ is algebraic but not rational, and $e^x$ and $\log(1-x)$ are D-finite but not algebraic.

\begin{corollary}
The ordinary generating function $\sum_n a_n x^n$ is D-finite if and only if the exponential generating function $\sum_n a_n \frac{x^n}{n!}$ is D-finite.
\end{corollary}

\begin{proof}
This follows from the observation that $\{a_n \, : \, n \in {\mathbb{N}}\}$ is P-recursive if and only if $\{a_n/n! \, : \, n \in {\mathbb{N}}\}$ is P-recursive
\end{proof}

\noindent
\textbf{\textsf{A few examples}}. Before we discuss general tools, we collect some examples. We will prove all of the following statements later in this section.

The power series for subsets, Fibonacci numbers, and Stirling numbers are rational:
\[
\sum_{n \geq 0}2^nx^n = \frac{1}{1-2x}, \quad
\sum_{n \geq 0} F_nx^n = \frac{x}{1-x-x^2}, \quad \sum_{n \geq k}S(n,k)x^n = \frac{x}{1-x} \cdot \frac{x}{1-2x} \cdots \frac{x}{1-kx}. 
\]

The ``diagonal binomial", $k$-Catalan, and Motzkin series are algebraic but not rational:
\[
\sum_{n \geq 0} {2n \choose n} x^n = \frac1{\sqrt{1-4x}},
\quad
\sum_{n \geq 0} \frac1{(k-1)n+1}{kn \choose n}x^n,
\quad
\sum_{n \geq 0} M_nx^ n = \frac{1-x-\sqrt{1-2x-3x^2}}{2x^2}.
\]

The following series are D-finite but not algebraic:
\[
e^x,  \quad \log(1+x), \quad \sin x, \quad \cos x, \quad \arctan x, \quad 
\sum_{n \geq 0} {2n \choose n}^2 x^n, \quad
\sum_{n \geq 0} {3n \choose n,n,n} x^n
\]

The following series are not D-finite:
\[
\sqrt{1+\log(1+x^2)}, \quad \sec x, \quad  \tan x, \quad  \sum_{n \geq 0}p(n)x^n = \prod_{k \geq 0}\frac1{1-x^k}.
\]

\noindent
\textbf{\textsf{Recognizing algebraic and D-finite series}}. It is not always obvious whether a given power series is algebraic or D-finite, but there are some tools available. Fortunately, algebraic functions behave well under a few operations, and D-finite functions behave even better. This explains why these families contain most examples arising in combinatorics.

The following table summarizes the properties of formal power series that are preserved under various key operations. For example, the fifth entry on the bottom row says that if $A(x)$ and $B(x)$ are D-finite, then the composition $A(B(x))$ is not necessarily D-finite. 

\begin{center}
\begin{tabular}{|l|c|c|c|c|c|c|c|c|c|}
\hline
%& $cA(x)$ & $A(x)+B(x)$ & $A(x)B(x)$ & $1/A(x)$ & $A(B(x))$ & $A(x)\star B(x)$ & $A'(x)$ & $\int A(x)$ \\
& $cA$ & $A+B$ & $AB$ & $1/A$ & $A \circ B$ & $A \star B$ & $A'$ & $\int A$ & $A^{\left<-1\right>}$\\
\hline
rational & Y & Y& Y& Y& Y& Y & Y & N & N \\
algebraic  & Y & Y& Y& Y& Y & N & Y & N & Y \\  
D-finite  & Y & Y& Y& N& N& Y & Y & Y & N \\
\hline
\end{tabular}
\end{center}
\bigskip

Here $A \star B(x) := \sum_{n \geq 0} a_nb_nx^n$ denotes the \textbf{Hadamard product} of $A(x)$ and $B(x)$, $\int A(x):= \sum_{n \geq 1} \frac{a_{n-1}}n x^n % \frac{x^n}{n}
$ is the \textbf{formal integral} of $A(x)$, and $A^{\left<-1\right>}(x)$ is the \textbf{compositional inverse}  of $A(x)$. %, which will be discussed in more detail in Section \ref{f.sec:???}.

In the fourth column we are assuming that $A(0) \neq 0$ so that $1/A(x)$ is well-defined, in the fifth column we are assuming that $B(0) = 0$ so that $A(B(x))$ is well-defined, and in the last column we are assuming that $A(0)=0$ and $A'(0) \neq 0$, so that $A^{\left< -1 \right>}(x)$ is well-defined. 

For proofs of the ``Yes" entries, see \cite{f.StanleyD-finite}, \cite{f.EC2}, \cite{f.Flajolet}. For the ``No" entries, we momentarily assume the statements of the previous subsection. Then  we have the following counterexamples:

\medskip

\noindent $\bullet$ $\cos x$ is D-finite but $1/\cos x = \sec x$ is not. 

\noindent $\bullet$ $\sqrt{1+x}$ and $\log(1+x^2)$ are D-finite but their composition $\sqrt{1+\log(1+x^2)}$ is not. 

\noindent $\bullet$ $A(x) = \sum_{n \geq 0} {2n \choose n}x^n$ is algebraic but $A \star A(x) = \sum_{n \geq 0} {2n \choose n}^2 x^n$ is not.

\noindent $\bullet$ $1/(1+x)$ is rational and algebraic but its integral $\log(1+x)$ is neither.

\noindent $\bullet$ $x+x^2$ is rational but its compositional inverse $(-1+\sqrt{1+4x})/2$ is not. 

\noindent $\bullet$ $\arctan x$ is D-finite but its compositional inverse $\tan x$ is not.

\medskip

Some of these negative results have weaker positive counterparts:

\medskip

\noindent
$\bullet$ If $A(x)$ is algebraic and $B(x)$ is rational, then $A(x) \star B(x)$ is algebraic. 

\noindent
$\bullet$ If $A(x)$ is D-finite and $A(0) \neq 0$, $1/A(x)$ is D-finite if and only if $A'(x)/A(x)$ is algebraic. 

\noindent
$\bullet$ If $A(x)$ is D-finite and $B(x)$ is algebraic with $B(0)=0$, then $A(B(x))$ is D-finite. 

\medskip

\noindent See \cite[Proposition 6.1.11]{f.EC2},
%\cite{f.Jungen},
 \cite{f.HarrisSibuya}, and \cite[Theorem 6.4.10]{f.EC2} for the respective proofs.

\bigskip
The following result is also useful:

\begin{theorem} \cite[Section 6.3]{f.EC2}
Consider a multivariate formal power series $F(x_1, \ldots, x_d)$ which is rational in $x_1, \ldots, x_d$ and its \textbf{diagonal}:
\[
F(x_1, \ldots, x_d) = \sum_{n_1, \ldots, n_d \geq 0} a_{n_1,\ldots, n_d} x_1^{n_1} \cdots x_d^{n_d}, \qquad 
\textrm{diag } F(x) = \sum_{n \geq 0} a_{n, \ldots, n} x^n.
\]
\begin{enumerate}
\item
If $d=2$, then $\textrm{diag } F(x)$ is algebraic. 
\item
If $d>2$, then $\textrm{diag } F(x)$ is D-finite but not necessarily algebraic. 
\end{enumerate}
\end{theorem}

%\cite{f.Jungen} Sur les s?eries de Taylor n?ayant que des singularit?es alg?ebrico-logarithmiques sur leur cercle de convergence,

%@.A. Harris Jr. and Y. Sibuya. Advances in Math 58 (1985)

Now we are ready to prove our positive claims about the series at the beginning of this section.
The first three expressions are visibly rational, and the diagonal binomial and Motzkin series are visibly algebraic. We proved that the $k$-Catalan series is algebraic in Example 15 of Section \ref{f.sec:ogfexamples}. The functions $e^x, \, \log(1+x),\, \sin x,\, \cos x,\arctan x$ respectively satisfy the differential equations $y'=y, \, (1+x)y'=1, \, y''=-y,\,  y''=-y,\,  (1+x^2)y'=1$. The series $\sum_{n \geq 0} {2n \choose n}^2 x^n$ is the Hadamard product of $(1-4x)^{-1/2}$ with itself, and hence D-finite. Finally, $\sum_{n \geq 0} {3n \choose n,n,n} x^n$ is the diagonal of the rational function $\frac{1}{1-x-y-z} = \sum_{a,b,c \geq 0} {a+b+c \choose a,b,c} x^a y^b z^c$, and hence D-finite. 

Proving the negative claims requires more effort and, often, a bit of analytic machinery. We briefly outline some key results.

\bigskip

\noindent
\textbf{\textsf{Recognizing series that are not algebraic}}.
There are a few methods available to prove that a series is \textbf{not} algebraic. The simplest algebraic and analytic criteria are the following.

\begin{theorem}[Eisenstein's Theorem]\label{f.th:eisenstein} \cite{f.PolyaSzego2}
If a series $A(x) = \sum_{n \geq 0} a_nx^n$ with rational coefficients is algebraic, then there exists a positive integer $m$ such that $a_nm^n$ is an integer for all $n>0$. 
\end{theorem}

This shows that $e^x, \log(1+x), \sin x, \cos x,$  and $\arctan x$ are not algebraic.

\begin{theorem}\cite{f.Jungen}
If the coefficients of an algebraic power series $A(x) = \sum_{n \geq 0} a_nx^n$ satisfy  $a_n \sim c n^r \alpha^n$ for nonzero $c, \alpha \in {\mathbb{C}}$ and $r<0$, then $r$ cannot be a negative integer.
\end{theorem}

Stirling's approximation $n! \sim \sqrt{2\pi n} \left(n/e\right)^n$ gives ${2n \choose n}^2 \sim c \cdot 16^n/ n$ and ${3n \choose n,n,n} \sim c \cdot 27^n/n$, so the corresponding series are not algebraic.

%Using Stirling's approximation $n! \sim \sqrt{2\pi n} \left(n/e\right)^n$ we get the estimates ${2n \choose n}^2 \sim c \cdot 16^n/ n$ and ${3n \choose n,n,n} \sim c \cdot 27^n/n$. This shows that $\sum_{n \geq 0}{2n \choose n}^2$ and $\sum_{n \geq 0}{3n \choose n,n,n}$ are not algebraic.

Another useful analytic criterion is that an algebraic series $A(x)$ must have a Newton-Puiseux expansion at any of its singularities. See \cite[Theorem VII.7]{f.Flajolet} and \cite{f.FlajoletGerholdSalvy} for details.

\bigskip

\noindent
\textbf{\textsf{Recognizing series that are not D-finite}}. The most effective methods to show that a function is \textbf{not} D-finite are analytic. 

\begin{theorem}\cite[Theorem 9.1]{f.Henrici}
Suppose that $A(x)$ is analytic at $x=0$, and it is D-finite, satisfying the equation $q_0(x)A(x) + q_1(x) A'(x) + \cdots + q_d(x) A^{(d)}(x) = q(x)$ with $q_d(x) \neq 0$. Then $A(x)$ can be extended to an analytic function in any simply connected region of the complex plane not containing the (finitely many) zeroes of $q_d(x)$. 
\end{theorem}

%Peter Henrici, Applied and computational complex analysis, vol. 2, John Wiley, New York,1974.

Since $\sec x$ and $\tan x$ have a pole at every odd multiple of $\pi$, they are not D-finite. Similarly, $\sum_{n} p(n)x^n = \prod_{k=1}^\infty \frac1{1-x^k}$ is not D-finite because it has the circle $|x|=1$ as a natural boundary of analyticity.

There are other powerful analytic criteria to prove a series is not D-finite. See \cite[Theorem VII.7]{f.FlajoletGerholdSalvy} for details and further examples.

Sometimes it is possible to give \emph{ad hoc} proofs that series are D-finite. For instance, consider $y=\sqrt{1+\log(1+x^2)}$. By induction, for any $k \in {\mathbb{N}}$ there exist polynomials $r_1(x), \ldots, r_k(x)$ such that $y^{(k)} = r_1/y+ r_2/y^3 + \cdots + r_k/y^{2k-1}$. An equation of the form $\sum_{i=0}^d q_i(x)y^{(i)} = q(x)$ would then give rise to a polynomial equation satisfied by $y$. This would also make $y^2-1 = \log(1+x^2)$ algebraic; but this contradicts Theorem \ref{f.th:eisenstein}.

\newpage

\section{\textsf{Linear algebra methods}}\label{f.sec:linalg}

There are several important theorems in enumerative combinatorics which express a combinatorial quantity in terms of a determinant. Of course, evaluating a determinant is not always straightforward, but there is a wide array of tools at our disposal.

The goal of Section \ref{f.sec:dets1} is to reduce many combinatorial problems to ``just computing a determinant"; examples include walks in a graph, spanning trees, Eulerian cycles, matchings, and routings. In particular, we discuss the Transfer Matrix Method, which allows us to encode many combinatorial objects as walks in graphs, so that these linear algebraic tools apply. These problems lead us to many beautiful, mysterious, and highly non-trivial determinantal evaluations. We will postpone the proofs of the evaluations until Section \ref{f.sec:dets2}, which is an exposition of some of the main techniques in the subtle science of computing combinatorial determinants.

\subsection{\textsf{Determinants in combinatorics}} \label{f.sec:dets1}

\subsubsection{\textsf{Preliminaries: Graph matrices}}

An \textbf{undirected graph}, or simply a \textbf{graph} $G=(V,E)$ consists of a set $V$ of vertices and a set $E$ of edges $\{u,v\}$ where $u, v \in V$ and $u \neq v$. In an undirected graph, we write $uv$ for the edge $\{u,v\}$. The \textbf{degree} of a vertex is the number of edges incident to it. 
A \textbf{walk} is a set of edges of the form $v_1v_2, v_2v_3, \ldots, v_{k-1}v_k$. This walk is \textbf{closed} if $v_k = v_1$.

A \textbf{directed graph} or \textbf{digraph} $G=(V,E)$ consists of a set $V$ of vertices and a set $E$ of oriented edges $(u,v)$ where $u, v \in V$ and $u \neq v$. In an undirected graph, we write $uv$ for the directed edge $(u,v)$. The \textbf{outdegree} (resp. \textbf{indegree}) of a vertex is the number of edges coming out of it (resp. coming into it). A \textbf{walk} is a set of directed edges of the form $v_1v_2, v_2v_3, \ldots, v_{k-1}v_k$. This walk is \textbf{closed} if $v_k = v_1$.

We will see in this section that many graph theory problems can be solved using tools from linear algebra. There are several matrices associated to graphs which play a crucial role; we review them here.

\bigskip
\noindent \textsf{\textbf{Directed graphs.}}
Let $G=(V,E)$ be a directed graph. 

\begin{itemize}

\item
The  \textbf{adjacency matrix} $A=A(G)$ is the $V \times V$ matrix whose entries are
\[
a_{uv} = \textrm{number of edges from $u$ to $v$}.
\]

\item
The \textbf{incidence matrix} $M=M(G)$ is the $V \times E$ matrix with  
\[
m_{ve} =
\begin{cases}
1 & \textrm{ if $v$ is the final vertex of edge $e$}, \\
-1 & \textrm{ if $v$ is the initial vertex of edge $e$}, \\
0 & \textrm{ otherwise.}
\end{cases}
\]

\item
The  \textbf{directed Laplacian matrix} $\overrightarrow{L} = \overrightarrow{L}(G)$ is the $V \times V$ matrix whose entries are
\[
\overrightarrow{l}_{uv} =
\begin{cases}
-(\textrm{number of edges from $u$ to $v$)} & \textrm{if $u \neq v$}, \\
\textrm{outdeg}(u) & \textrm{if $u = v$}.
\end{cases}
\]
\end{itemize}

%\begin{figure}[ht]
% \begin{center}
%  \includegraphics[scale=.8]{./Pictures/digraph}
%  \caption{ \label{f.fig:four matrices} 
%Four matrices associated to a directed graph.}
% \end{center}
%\end{figure}
%

\bigskip
\noindent \textsf{\textbf{Undirected graphs.}}
Let $G=(V,E)$ be an undirected graph. 

\begin{itemize}
\item
The (undirected) \textbf{adjacency matrix} $A=A(G)$ is the $V \times V$ matrix whose entries are
\[
a_{uv} = \textrm{number of edges connecting $u$ and $v$}. 
\]
This is the directed adjacency matrix of the directed graph on $V$ containing edges $u \rightarrow v$ and $v \rightarrow u$ for every edge $uv$ of $G$.

\item
The (undirected) \textbf{Laplacian matrix} $L=L(G)$ is the $V \times V$ matrix with entries 
\[
l_{uv} =
\begin{cases}
-(\textrm{number of edges connecting $u$ and $v$)} & \textrm{if $u \neq v$} \\
\deg u & \textrm{if $u = v$}
\end{cases}
\]
If $M$ is the incidence matrix of any orientation of the edges of $G$, then $L = MM^T$.

\end{itemize}

\subsubsection{\textsf{Counting walks: the Transfer Matrix Method}}\label{f.sec:transfermatrix}

Counting walks in a graph is a fundamental problem, which (often in disguise) includes many important enumerative problems. The Transfer Matrix Method addresses this problem by expressing the number of walks in a graph $G$ in terms of its adjacency matrix $A(G)$, and then uses linear algebra to count those walks.

\bigskip
\noindent \textsf{\textbf{Directed or undirected graphs.}} The Transfer Matrix Method is based on the following  simple, powerful observation, which applies to directed and undirected graphs:

\begin{theorem}\label{f.th:transfer}
Let $G=(V,E)$ be a graph and let $A=A(G)$ be the $V \times V$ adjacency matrix of $G$, where $a_{uv}$ is the number of edges from $u$ to $v$. Then
\[
(A^n)_{uv} = \textrm{ number of walks of length $n$ in $G$ from $u$ to $v$}.
\]
\end{theorem}

\begin{proof} Observe that 
\[
(A^n)_{uv} = 
\sum_{w_1, \ldots, w_{n-1} \in V}
a_{uw_1}a_{w_1w_2} \cdots a_{w_{n-1}v}
\]
and there are $a_{uw_1}a_{w_1w_2} \cdots a_{w_{n-1}v}$ walks visiting vertices $u, w_1, \ldots, w_{n-1}, v$ in that order. 
\end{proof}

\begin{corollary}
The generating function $\sum_{n \geq 0} (A^n)_{uv} x^n$ for walks of length $n$ from $u$ to $v$ in $G$ is a rational function.
\end{corollary}

\begin{proof}
Using Cramer's formula, we have
\[
\sum_{n \geq 0} (A^n)_{uv} x^n = ((I-xA)^{-1})_{uv} = (-1)^{u+v} \frac{\det(I-xA \, : v,u)}{\det(I-xA)} 
\]
where $(M:v,u)$ denotes the cofactor of $M$ obtained by removing row $v$ and column $u$.
\end{proof}

\begin{corollary}\label{f.cor:closed}
If $C_G(n)$ is the number of closed walks of length $n$ in $G$, then
\[
C_G(n) = \lambda_1^n + \cdots + \lambda_k^n, \qquad 
\sum_{n \geq 1} C_G(n) x^n = \frac{-x  \, Q'(x)}{Q(x)}
\]
where $\lambda_1, \ldots, \lambda_k$ are the eigenvalues of the adjacency matrix $A$ and $Q(x) = \det(I-xA)$.
\end{corollary}

\begin{proof}
Theorem \ref{f.th:transfer} implies that $C_G(n) = {\mathrm{tr}}(A^n) = \lambda_1^n + \cdots + \lambda_k^n$. The second equation then follows from $Q(x) = (1-\lambda_1x)\cdots (1-\lambda_k x)$.
\end{proof}

In view of Theorem \ref{f.th:transfer}, we want to be able to compute powers of the adjacency matrix $A$. As we learn in linear algebra, this is very easy to do if we are able to diagonalize $A$. This is not always possible, but we can do it when $A$ is undirected.

%\comment{Example of ice configurations?}

\bigskip
\noindent \textsf{\textbf{Undirected graphs.}} When our graph $G$ is undirected, the adjacency matrix $A(G)$ is symmetric, and hence diagonalizable. %This advantageous property has nice algebraic and combinatorial consequences. 

\begin{theorem}
Let $G=(V,E)$ be an \textbf{undirected} graph and let $\lambda_1, \ldots, \lambda_k$ be the eigenvalues of the adjacency matrix $A=A(G)$. Then for any vertices $u$ and $v$ there exist constants $c_1, \ldots, c_k$ such that 
\[
\textrm{ number of walks of length $n$ from $u$ to $v$} = c_1 \lambda_1^n + \cdots + c_k \lambda_k^n.
\]
\end{theorem}

\begin{proof}
The key fact is that a real symmetric $k \times k$ matrix $A$ has $k$ real orthonormal eigenvectors $q_1, \ldots, q_k$ with real eigenvalues $\lambda_1, \ldots, \lambda_k$. Equivalently, the $k \times k$ matrix $Q$ with columns $q_1, \ldots, q_k$ is orthogonal (so $Q^T = Q^{-1}$) and diagonalizes $A$:
\[
Q^{-1}AQ = D = \textrm{diag}(\lambda_1, \ldots, \lambda_k)
\]
where $D=\textrm{diag}(\lambda_1, \ldots, \lambda_k)$ is the diagonal matrix with diagonal entries $\lambda_1, \ldots, \lambda_k$. 
The result then follows from $A^n = QD^nQ^{-1} = Q \, \textrm{diag}(\lambda_1^n, \ldots, \lambda_k^n) \, Q^{T}$, with $c_t = q_{it}q_{jt}$. 
\end{proof}

\bigskip
\noindent \textsf{\textbf{Applications.}} Many families of combinatorial objects can be enumerated by first recasting the objects as walks in a ``transfer graph", and then applying the transfer matrix method. We illustrate this technique with a few examples.

\bigskip

\begin{enumerate}
\item (Colored necklaces)
%As a warmup example, l
Let $f(n,k)$ be the number of ways of coloring the beads of a necklace of length $n$ with $k$ colors so that no two adjacent beads have the same color. (Different rotations and reflections of a coloring are considered different.) There are several ways to compute this number, but a very efficient one is to notice that such a coloring is a graph walk in disguise. If we label the beads $1, \ldots, n$ in clockwise order and let $a_i$ be the color of the $i$th bead, then the coloring corresponds to the closed walk $a_1, a_2, \ldots, a_n, a_1$ in the complete graph $K_n$. The adjacency graph of $K_n$ is $A=J-I$ where $J$ is the matrix all of whose entries equal $1$, and $I$ is the identity. Since $J$ has rank $1$, it has $n-1$ eigenvalues equal to $0$. Since the trace is $n$, the last eigenvalue is $n$. It follows that the eigenvalues of $A=J-I$ are $-1, -1, \ldots, -1, n-1$. 
%In Section \ref{f.sec:dets2} we will see that the eigenvalues of $A$ are $n-1, -1, \ldots, -1$, so 
Then Corollary \ref{f.cor:closed} tells us that
\[
f(n,k) = (n-1)^k + (n-1)(-1)^k.
\]
It is possible to give a bijective proof of this formula, but this algebraic proof is  simpler.

\item (Words with forbidden subwords, 1.)
Let $h_n$ be the number of words of length $n$ in the alphabet $\{a,b\}$ which do not contain $aa$ as a consecutive subword. This is the same as a walk of length $n-1$ in the transfer graph with vertices $a$ and $b$ and edges $a \rightarrow b$, $b \rightarrow a$ and $b \rightarrow b$. The absence of the edge $a \rightarrow a$ guarantees that these walks produce only the valid words we wish to count. 
The adjacency matrix and its powers are 
\[
A=\begin{pmatrix} 0 & 1 \\ 1 & 1\end{pmatrix}, \qquad  A^{n} = \begin{pmatrix} F_{n-1} & F_{n} \\ F_{n} & F_{n+1}\end{pmatrix},
\]
where the Fibonacci numbers $F_0, F_1, \ldots$ are defined recursively by $F_0=0, F_1=1$, and $F_k = F_{k-1} + F_{k-2}$ for $k \geq 2$. 

Since $h_n$ is the sum of the entries of $A^{n-1}$, we get that $h_n = F_{n+2}$, and $g_n \sim c \cdot \alpha^n$ where $\alpha =\frac12(1+\sqrt{5}) \approx 1.6179\ldots$ is the golden ratio.
% we see that
%$
%f_n = \frac1{\sqrt5}\left(\left(\frac{1+\sqrt5}2\right)^{n+2} - \left(\frac{1-\sqrt5}2\right)^{n+2}\right),
%$
%the $(n+2)$nd Fibonacci number. 
Of course there are easier proofs of this fact, but this approach works for any problem of enumerating words in a given alphabet with given forbidden consecutive subwords. Let us study a slightly more intricate example, which should make it clear how to proceed in general.

\bigskip

\item (Words with forbidden subwords, 2.)
Let $g_n$ be the number of cyclic words of length $n$ in the alphabet $\{a,b\}$ which do not contain $aa$ or $abba$ as a consecutive subword. We wish to model these words as walks in a directed graph. At first this may seem impossible because, as we construct the word sequentially, the validity of a new letter depends on more than just the previous letter. However, a simple trick resolves this difficulty: we can introduce more memory into the vertices of the transfer graph. In this case, since the validity of a new letter depends on the previous three letters, we let the vertices of the transfer graph be $aba, abb, bab, bba, bbb$ (the allowable ``windows" of length $3$) and put an edge $wxy \rightarrow xyz$ in the graph if the window $wxy$ is allowed to precede the window $xyz$; that is, if $wxyz$ is an allowed subword. The result is the graph of Figure \ref{f.fig:transfergraph}, whose adjacency matrix $A$ satisfies 
\[
\det(I-xA) = 
\det 
\begin{pmatrix}
1 & 0 & \,-x & 0 & 0 \\
0 & 1 & 0 & 0 & -x \\
\,-x & \,-x & 1 & 0 & 0 \\
0 & 0 & \,-x & 1 & 0 \\
0 & 0 & 0 & \,-x & 1-x
\end{pmatrix}
=
-x^4+x^3-x^2-x+1.
\]

\begin{figure}[ht]
 \begin{center}
  \includegraphics[scale=.7]{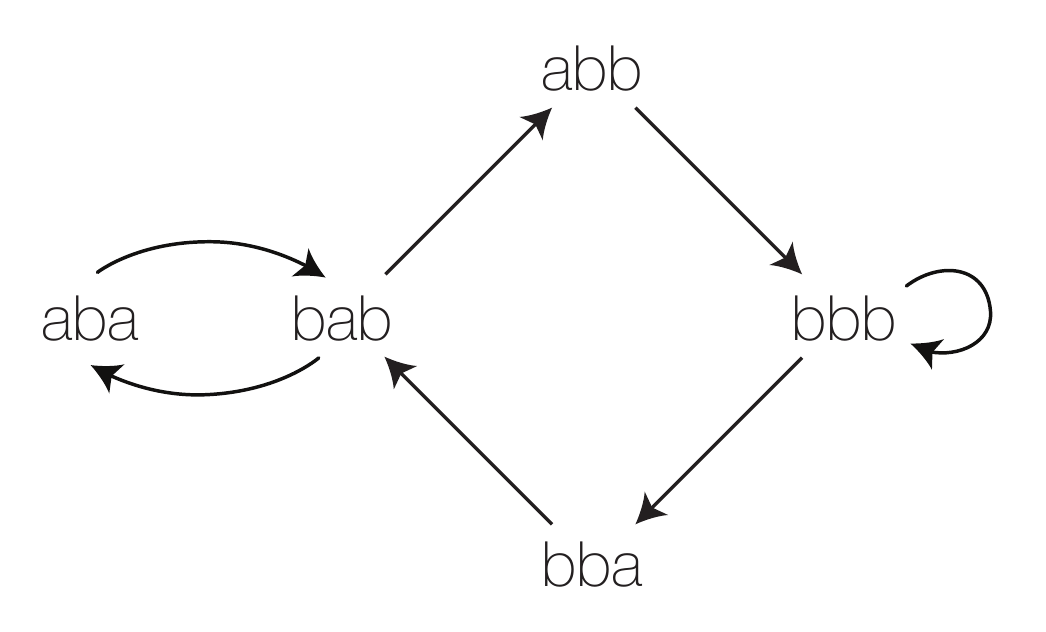}
  \caption{ \label{f.fig:transfergraph} 
The transfer graph for words on the alphabet $\{a,b\}$ avoiding $aa$ and $abba$ as consecutive subwords.}
 \end{center}
\end{figure}

The valid cyclic words of length $n$ correspond to the closed walks of length $n$ in the transfer graph, so Corollary \ref{f.cor:closed} tells us that the generating function for $g_n$ is
\[
\sum_{n \geq 0} g_nx^n =  \frac{x+2x^2-3x^3+4x^4}{1-x-x^2+x^3-x^4} = x+3x^2+x^3+7x^4+6x^5+15x^6+15x^7+31x^8+37x^9+\cdots.
\]
Theorem \ref{f.th:rational}.2 then tells us that $g_n \approx c \cdot \alpha^n$ where $\alpha \approx 1.5129$ is the inverse of the smallest positive root of $1-x-x^2+x^3-x^4=0$.
The values of $g_1, g_2, g_3$ may surprise us. Note that the generating function does something counterintuitive: it does not count the words $a$ (because $aa$ is forbidden), $aba$ (because $aa$ is forbidden), or $abb$ (because $abba$ is forbidden). 

\end{enumerate}

This example serves as a word of caution: when we use the transfer-matrix method to enumerate ``cyclic" objects using Corollary \ref{f.cor:closed}, the initial values of the generating function may not be the ones we expect. In a particular problem of interest, it will be straightforward to adjust those values accordingly.

To illustrate the wide applicability of this method, we conclude this section with a problem where the transfer graph is less apparent.

\begin{enumerate}
\item[4.] (Monomer-dimer problem) 
An important open problem in statistical mechanics is the \textbf{monomer-dimer problem} of  computing the number of tilings $T(m,n)$ of an $m \times n$ rectangle into dominoes ($2 \times 1$ rectangles) and unit squares. Equivalently, $T(m,n)$ is the number of partial matchings of an $m \times n$ grid, where each node is matched to at most one of its neighbors. 

There is experimental evidence, but no proof, that $T(n,n) \sim c \cdot \alpha^{n^2}$ where $\alpha \approx 1.9402...$ is a constant  for which no exact expression is known. The transfer-matrix method is able to solve this problem for any fixed value of $m$, proving that the generating function $\sum_{n \geq 0} T(m,n)x^n$ is rational. We carry this out for $m=3$.

Let $t(n)$ be the number of tilings of a $3 \times n$ rectangle into dominoes and unit squares. As with words, we can build our tilings sequentially from left to right by covering the first column, then the second column, and so on. The tiles that we can place on a new column depend only on the tiles from the previous column that are sticking out, and this can be modeled by a transfer graph.

\begin{figure}[ht]
 \begin{center}
  \includegraphics[scale=.7]{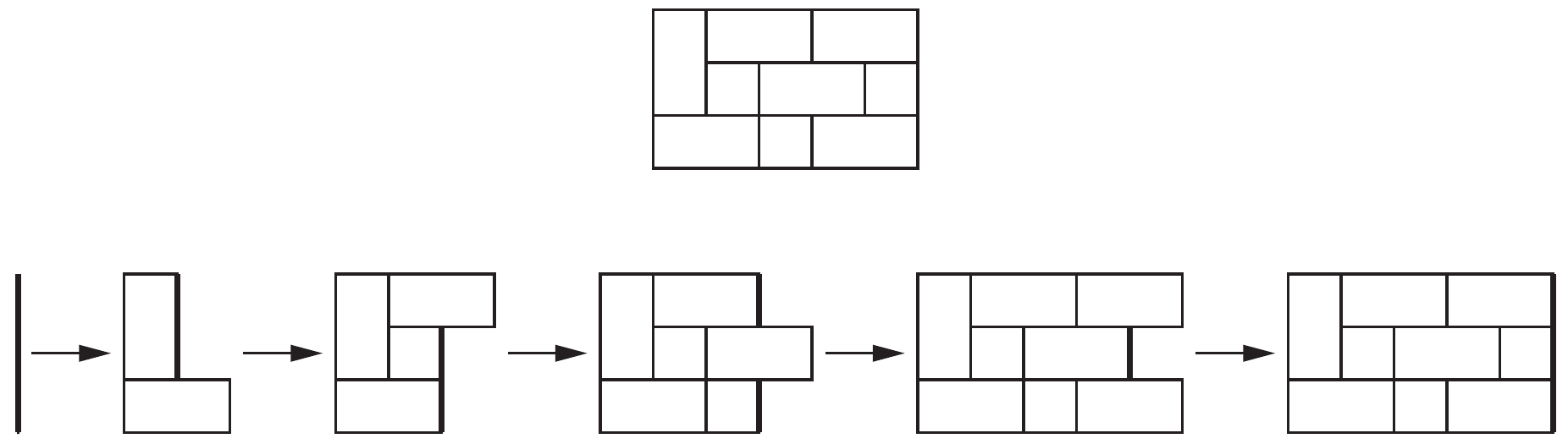}
  \caption{ \label{f.fig:dimermonomer} 
A tiling of a $3 \times 5$ rectangle into dominoes and unit squares.
}
 \end{center}
\end{figure}

More specifically, let $T$ be a tiling of a $3 \times n$ rectangle. We define $n+1$ triples $v_0, \ldots, v_n$ which record how $T$ interacts with the $n+1$ vertical grid lines of the rectangle. The $i$th grid line consists of three unit segments, and each coordinates of $v_i$ is $0$ or $1$ depending on whether these three segments are edges of the tiling or not. For example, Figure \ref{f.fig:dimermonomer} corresponds to the triples $111, 110, 011, 101, 010, 111$.

The choice of $v_i$ is restricted only by $v_{i-1}$. The only restriction is that $v_{i-1}$ and $v_i$ cannot both have a $0$ in the same position, because this would force us to put two overlapping horizontal dominoes in $T$. These compatibility conditions are recorded in the transfer graph of Figure \ref{f.fig:transfer2}.
When $v_{i-1}=v_i=111$, there are 3 ways of covering column $i$. If $v_{i-1}$ and $v_i$ share two $1$s in consecutive positions, there are 2 ways. In all other cases, there is a unique way. It follows that the tilings of a $3 \times n$ rectangle are in bijection with the walks of length $n$ from $111$ to $111$ in the transfer graph.

\begin{figure}[ht]
 \begin{center}
  \includegraphics[scale=.7]{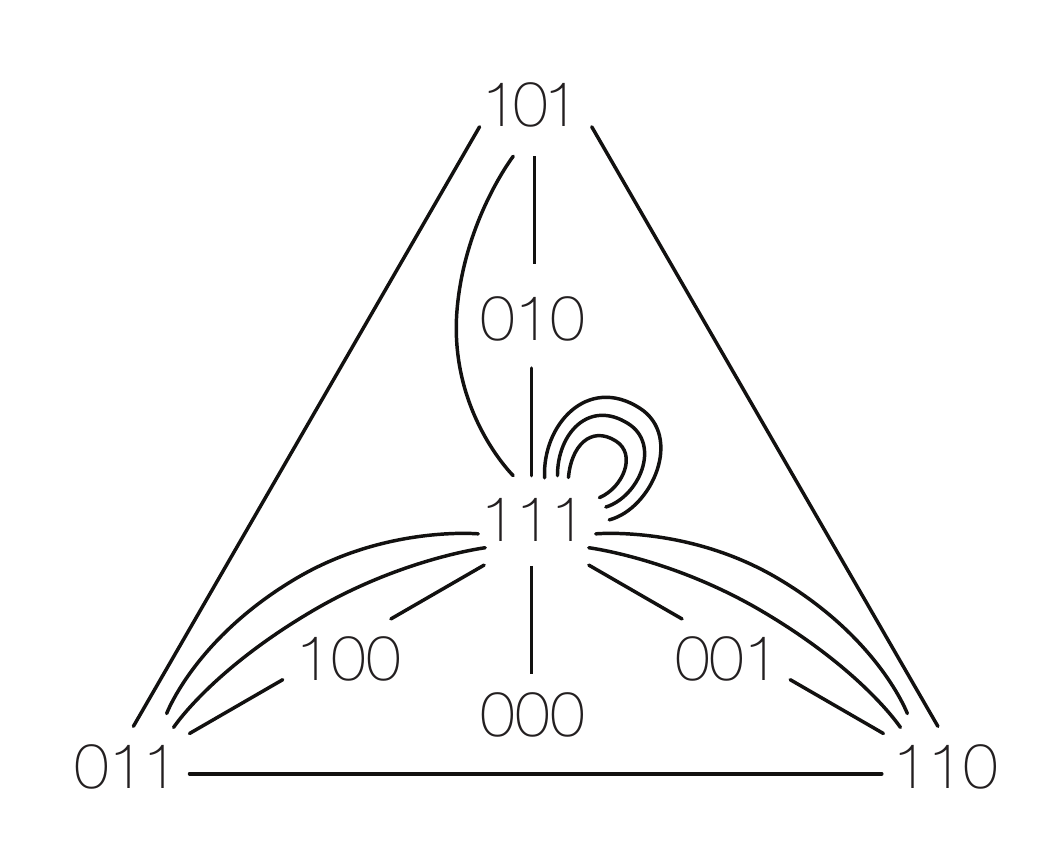}
  \caption{ \label{f.fig:transfer2} 
The transfer graph for tilings of $3 \times n$ rectangles into dominoes and unit squares.
}
 \end{center}
\end{figure}

Since the adjacency matrix is 

\[
A=
\bordermatrix{
       & 000 & 001 & 110 & 010 & 101 & 001 & 110 & 111 \cr
000 & 0 & 0 & 0 & 0 & 0 & 0 & 0 & 1 \cr
001 & 0 & 0 & 1 & 0 & 0 & 0 & 0 & 1 \cr
110 & 0 & 1 & 0 & 0  & 1 & 0 & 1 & 2 \cr
010 & 0 & 0 & 0 & 0 & 1 & 0 & 0 & 1 \cr
101 & 0 & 0 & 1 & 1 & 0 & 0 & 1 & 1 \cr
001 & 0 & 0 & 0 & 0 & 0 & 0 & 1 & 1 \cr
110 & 0 & 0 & 1 & 0 & 1 & 1 & 0 & 2 \cr
111 & 1 & 1 & 2 & 1 & 1 & 1 & 2 & 3
}
\]
Theorem \ref{f.th:transfer} tells us that
\begin{eqnarray*}
\sum_{n \geq 0}t(n)x^n &=& \frac{\det(I-xA : 111,111)}{\det(I-xA)} = \frac{(1+x-x^2)(1-2x-x^2)}{(1+x)(1-5x-9x^2+9x^3+x^4-x^5)} \\
&=& 1+3x+22x^2+131x^3+823x^4+5096x^5+31687x^6+196785x^7+%1222550x^8+
\cdots
\end{eqnarray*}
By Theorem \ref{f.th:rational}.2, $t_n \sim c \cdot \alpha^n$ where $\alpha \approx 6.21207\ldots$ is the inverse of the smallest positive root of the denominator $(1+x)(1-5x-9x^2+9x^3+x^4-x^5)$.

\end{enumerate}

\subsubsection{\textsf{Counting spanning trees: the Matrix-Tree theorem}}\label{f.sec:spanningtrees}

In this section we discuss two results: Kirkhoff's determinantal formula for the number of spanning trees of a graph, and Tutte's generalization to oriented spanning trees of directed graphs.

\bigskip
\noindent
\textsf{\textbf{Undirected Matrix-Tree theorem.}}
Let $G=(V,E)$ be a connected graph with no loops. 
A \textbf{spanning tree} $T$ of $G$ is a collection of edges such that for any two vertices $u$ and $v$, $T$ contains a unique path between $u$ and $v$. 
If $G$ has $n$ vertices, then:

$\bullet$ $T$ contains no cycles, 

$\bullet$ $T$ spans $G$; that is, there is a path from $u$ to $v$ in $T$ for any vertices $u \neq v$, and

$\bullet$ $T$ has $n-1$ edges. 

\noindent Furthermore, any two of these properties imply that $T$ is a spanning tree. Our goal in this section is to compute the number $c(G)$ of spanning trees of $G$. %The matrix-tree theorem reduces this problem to the computation of a determinant.

%\begin{itemize}
%\item $T$ contains no cycles
%\item $T$ spans $G$; that is, there is a path from $u$ to $v$ in $T$ for any vertices $u \neq v$
%\item $T$ has $n-1$ edges.
%\end{itemize}
%connects all vertices without forming any cycles. If $G$ has $n$ vertices, every spanning tree has $n-1$ edges. We are interested in computing the number $c(G)$ of spanning trees of $G$.

Orient the edges of $G$ arbitrarily. Recall that the incidence matrix $M$ of $G$ is the $V \times E$ matrix whose $e$th column is $\mathbf{e}_v-\mathbf{e}_u$ if $e=u \rightarrow v$, where $\mathbf{e}_i$ is the $i$th basis vector. The Laplacian $L=MM^T$ has entries 
\[
l_{uv} =
\begin{cases}
-(\textrm{number of edges connecting $u$ and $v$)} & \textrm{if $u \neq v$} \\
\deg u & \textrm{if $u = v$}
\end{cases}
\]
Note that $L(G)$ is singular because all its row sums are $0$. A \textbf{principal cofactor}  $L_v(G)$ is obtained from $L(G)$ by removing the $v$th row and $v$th column for some vertex $v$.

%
%\begin{figure}[ht]
% \begin{center}
%  \includegraphics[scale=.3]{./Pictures/vectorconfig2}
%  \caption{ \label{f.fig:Laplacian} 
%The Laplacian of a directed graph.}
% \end{center}
%\end{figure}

\begin{theorem}\label{f.th:Kirkhoff}(Kirkhoff's Matrix-Tree Theorem) The number $c(G)$ of spanning trees of a connected graph $G$ is
\[
c(G) = \det L_v(G) = \frac1n  \lambda_1 \cdots \lambda_{n-1}.
\]
where $L_v(G)$ is any principal cofactor of the Laplacian $L(G)$, and $\lambda_1, \ldots, \lambda_{n-1}, \lambda_n=0$ are the eigenvalues of $L(G)$.
\end{theorem}

\begin{proof}
We use the Binet-Cauchy formula, which states that if $A$ and $B$ are $m \times n$ and $n \times m$ matrices, respectively, with $m<n$, then
\[
\det AB = \sum_{S \subseteq [n] \, : \, |S| = m} \det A[S] \det B[S]
\]
where $A[S]$ (resp. $B[S]$) is the $n \times n$ matrix obtained by considering only the columns of $A$ (resp. the rows of $B$) indexed by $S$. 

We also use the following observation: if $M_v$ is the ``reduced" adjacency matrix $M$ with the $v$th row removed, and $S$ is a set of $n-1$ edges of $E$, then
\[
\det M_v[S] = 
\begin{cases}
\pm 1 & \textrm{if $S$ is a spanning tree}, \\
0 & \textrm{otherwise}. 
\end{cases}
\]
This observation is easily proved: if $S$ is not a spanning tree, then it contains a cycle $C$, which gives a linear dependence among the columns indexed by the edges of $C$. Otherwise, if $S$ is a spanning tree, think of $v$ as its root, and ``prune" it by repeatedly removing a leaf $v_i \neq v$ and its only incident edge $e_i$ for $1 \leq i \leq n-1$. Then if we list the rows and columns of $M[S]$ in the orders $v_1, \ldots, v_{n-1}$ and $e_1, \ldots, e_{n-1}$, respectively, the matrix will be lower triangular with $1$s and $-1$s in the diagonal. %(In Figure \ref{f.fig:Laplacian}, if $S=\{1,3,4\}$, then we could prune the vertices $a,d,b$ and edges $1,3,4$ in that order.)

%, proceed by induction. Consider an edge $e=uv$ incident to $v$; then the only non-zero entry in the $e$th column is the one in row $u$. By Laplace expansion, $\det M_v[S] = \pm \det M_u[S-e]$ where $M_u$ is the reduced incidence matrix for the graph $G$ with vertex $v$ and all adjacent edges removed. The result follows since $S-e$ is a spanning tree of that smaller graph.

Combining these two equations, we obtain the first statement:
\[
\det L_v(G) = \sum_{S \subseteq [n] \, : \, |S| = m} \det M[S] \det M^T[S] = \sum_{S \subseteq [n] \, : \, |S| = m} \det M[S]^2 = c(G).
\]

To prove the second one, observe that the coefficient of $-x^1$ in the characteristic polynomial $\det (L-xI) = (\lambda_1-x)\cdots (\lambda_{n-1}-x)(0-x)$ is the sum of the $n$ principal cofactors, which are all equal to $c(G)$.
\end{proof}

The matrix-tree theorem is a very powerful tool for computing the number of spanning trees of a graph. Let us state a few examples.

The \textbf{complete graph} $K_n$ has $n$ vertices and an edge joining each pair of vertices. The \textbf{complete bipartite graph} $K_{m,n}$ has $m$ ``top" vertices and $n$ ``bottom" vertices, and $mn$ edges joining each top vertex to each bottom vertex. The \textbf{hyperoctahedral graph} $\Diamond_n$ has vertices $\{1, 1', 2, 2', \ldots, n, n'\}$ and its only missing edges are $ii'$ for $1 \leq i \leq n$. The \textbf{$n$-cube graph} $C_n$ has vertices $(\epsilon_1, \ldots, \epsilon_n)$ where $\epsilon_i \in \{0, 1\}$, and an edge connecting any two vertices that differ in exactly one coordinate. 
The \textbf{$n$-dimensional grid of size $m$}, denoted $mC_n$, has vertices $(\epsilon_1, \ldots, \epsilon_n)$ where $\epsilon_i \in \{1, \ldots, m\}$, and an edge connecting any two vertices that differ in exactly one coordinate $i$, where they differ by $1$. 

\begin{theorem}\label{f.th:spanningtrees} The number of spanning trees of some interesting graphs are as follows.
\begin{enumerate}
\item (Complete graph): $c(K_n) = n^{n-2}$

\item (Complete bipartite graph): $c(K_{m,n}) = m^{n-1}n^{m-1}$

\item (Hyperoctahedral graph): $c(\Diamond_n) = 2^{2n-2}(n-1)^nn^{n-2}$.

\item ($n$-cube): $c(C_n) = 2^{2^n-n-1} \prod_{k=1}^n k^{n \choose k}$

\item ($n$-dimensional grid of size $m$): $c(mC_n) = m^{m^n-n-1} \prod_{k=1}^n k^{{n \choose k}(m-1)^k}$
%
%\item \comment{More? If I have time/space, I could mine Cvetkovic et. al. \emph{Spectral Graph Theory} for more results. Matroids? Root polytopes?}
\end{enumerate}
\end{theorem}

We will see proofs of the first and third example in Section \ref{f.sec:dets2}. For the others, and many additional examples, see \cite{f.Cvetkovic}.
%\item \comment{More? If I have time/space, I could mine Cvetkovic et. al. \emph{Spectral Graph Theory} for more results.}

\bigskip
\noindent
\textsf{\textbf{Directed Matrix-Tree theorem.}}
Now let $G=(V,E)$ be a \textbf{directed} graph containing no loops. An \textbf{oriented spanning tree rooted at $v$} is a collection of edges $T$ such that for any vertex $u$ there is a unique path from $u$ to $v$. The underlying unoriented graph $\underline{T}$ is a spanning tree of the unoriented graph $\underline{G}$. 
Let $c(G,v)$ be the number of spanning trees rooted at $G$. 

Recall that the  \textbf{directed Laplacian matrix} $\overrightarrow{L}$ has entries
\[
\overrightarrow{l}_{uv} =
\begin{cases}
-(\textrm{number of edges from $u$ to $v$)} & \textrm{if $u \neq v$} \\
\textrm{outdeg } u & \textrm{if $u = v$}
\end{cases}
\]
Now the matrix $\overrightarrow{L}(G)$ is not necessarily symmetric, but it is still singular.  

%\begin{figure}[ht]
% \begin{center}
%  \includegraphics[scale=.3]{./Pictures/vectorconfig2}
%  \caption{ \label{f.fig:descents} 
%The directed Laplacian of a directed graph.}
% \end{center}
%\end{figure}

\begin{theorem} (Tutte's Directed Matrix-Tree Theorem) Let $G$ be a directed graph and $v$ be a vertex. The number $c(G,v)$ of oriented spanning trees rooted at $v$
\[
c(G,v) = \det \overrightarrow{L}_v(G) 
\]
where $\overrightarrow{L}_v(G)$ is obtained from $L(G)$ by removing the $v$th row and column. 
Furthermore, if $G$ is \textbf{balanced}, so $\textrm{indeg } v = \textrm{outdeg } v$ for all vertices $v$, then
\[
c(G,v) = \frac1n  \lambda_1 \cdots \lambda_{n-1}
\]
where $\lambda_1, \ldots, \lambda_{n-1}, \lambda_n=0$ are the eigenvalues of $L(G)$.
\end{theorem}

\begin{proof} Proceed by induction. 
Consider an edge $e$ starting at a vertex $w \neq v$. Let $G'=G-e$ be obtained from $G$ by removing $e$. If $e$ is the only edge starting at $w$,  every spanning tree  uses it and 
\[
c(G,v) = c(G',v) =\det  \overrightarrow{L}_v(G') = \det \overrightarrow{L}_v(G).
\]
%
%
%so $c(G,v) = c(G',v)$ where $G'$ is obtained from $G$ by removing $e$ and $w$, and one easily checks that $\overrightarrow{L}_v(G) = \overrightarrow{L}_v(G')$
Otherwise, obtain $G''$ from $G$ by removing all edges starting at $w$ other than $e$. There are $c(G',v)$ oriented spanning trees rooted at $v$ which do not contain $e$, and $c(G'',v)$ which do contain $e$, so % we have
\[
c(G,v) = c(G',v) + c(G'',v) = \det \overrightarrow{L}_v(G') + \det \overrightarrow{L}_v(G'') = \det \overrightarrow{L}_v(G)
\]
where the last equality holds since determinants are multilinear. 

We postpone the proof of the second statement until Corollary \ref{f.cor:trees}. %, where it will be an immediate consequence of Theorem \ref{f.th:BEST}.
\end{proof}

\subsubsection{\textsf{Counting Eulerian cycles: the BEST theorem}} One of the earliest combinatorial questions is the problem of the Seven Bridges of K\"onigsberg. In the early 1700s, the Prussian city of K\"onigsberg was separated by the Pregel river into four regions, connected to each other by seven bridges. In the map of Figure \ref{f.fig:Konigsberg} we have labeled the regions $N,S,E$ and $I$ (north, south, east, and island);  there are two bridges between $N$ and $I$, two between $S$ and $I$, and three bridges connecting $E$ to each of $N, I$, and $S$. The problem was to find a walk through the city that crossed each bridge exactly once. Euler proved in 1735 that it is impossible to find such a walk; this is considered to be the first paper in graph theory.

\begin{figure}[ht]
 \begin{center}
  \includegraphics[scale=.333]{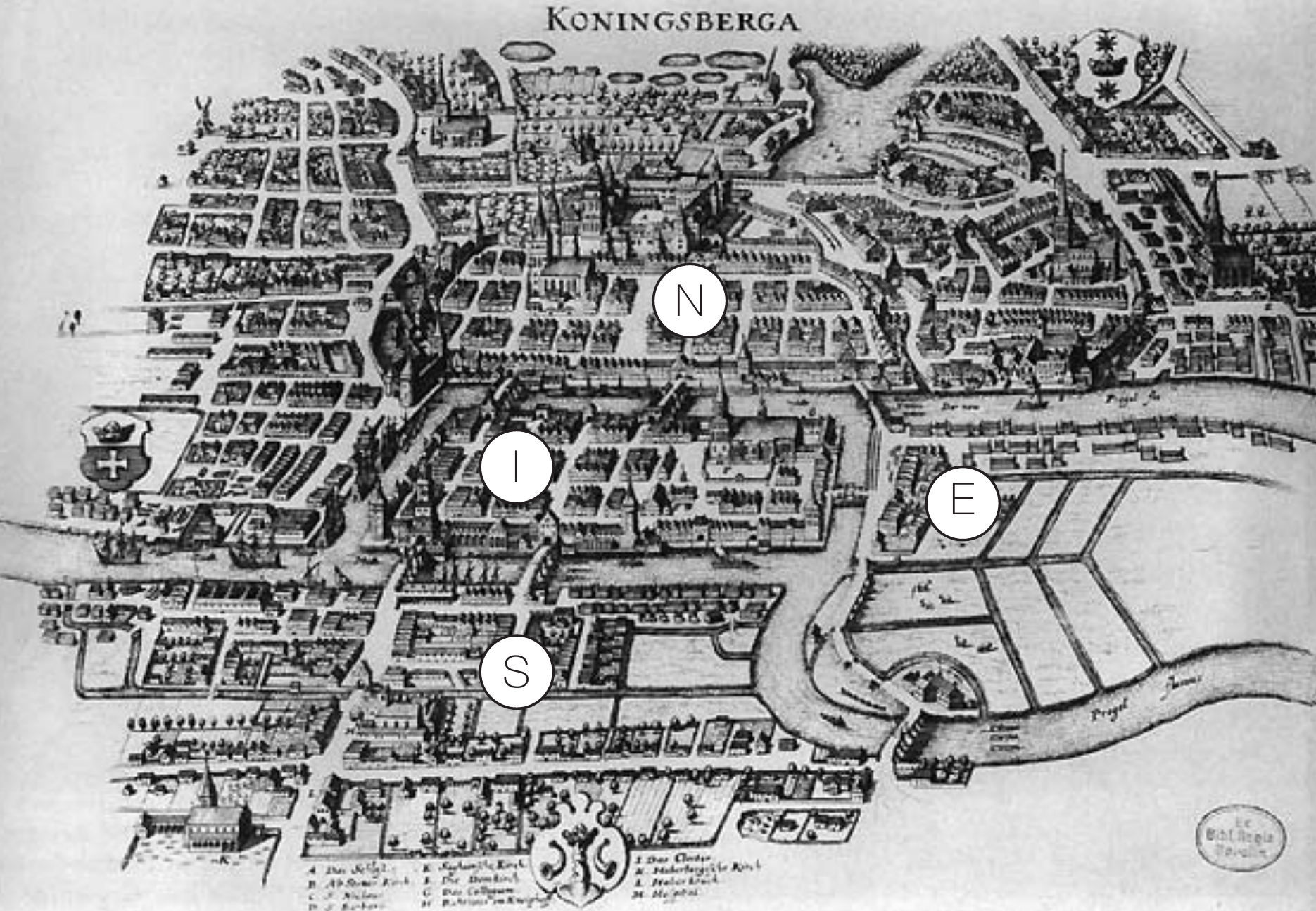}
  \caption{ \label{f.fig:Konigsberg} 
The seven bridges of K\"onigsberg. (Public domain map by Merian-Erben, 1652.)}
 \end{center}
\end{figure}

Euler's argument is simple, and relies on the fact that every region of K\"onigsberg is adjacent to an odd number of bridges. Suppose there existed such a walk, starting at region $A$ and ending at region $B$.  Now consider a region $C$ other than $A$ and $B$. Then our path would enter and leave $C$ the same number of times; but then it would not use all the bridges adjacent to $C$, because there is an odd number of such bridges.

In modern terminology, each region of the city is represented by a vertex, and each bridge is represented by an edge connecting two vertices. We will be more interested in the directed case, where every edge has an assigned direction. An \textbf{Eulerian path} is a path in the graph which visits every edge exactly once. If the path starts and ends at the same vertex, then it is called a \textbf{Eulerian cycle}. We say $G$ is an \textbf{Eulerian graph} if it has an Eulerian cycle.

\begin{theorem}
%A connected graph has an Eulerian cycle if and only if every vertex has even degree. It has an Eulerian path if and only if it has at most two edges of odd degree. 
A directed graph is Eulerian if and only if it is connected and every vertex $v$ satisfies $\textrm{indeg}(v) = \textrm{outdeg}(v)$.
\end{theorem}

\begin{proof}
%We focus on the last statement; the others are similar. 
If a graph has an Eulerian cycle $C$, then $C$ enters and leaves each vertex $v$ the same number of times. Therefore $\textrm{indeg}(v) = \textrm{outdeg}(v)$. 

To prove the converse, let us start by arbitrarily ``walking around $G$ until we get stuck". More specifically, we start at any vertex $v_0$, and at each step, we exit the current vertex by walking along any outgoing edge which we have not used yet. If there is no available outgoing edge, we stop.
%, and at each new vertex, 
%we use any available edge to walk out of it; we do this until we enter a vertex where every outgoing edge has already been used. 

Whenever we enter a vertex $v \neq v_0$, we will also be able to exit it since $\textrm{indeg}(v) = \textrm{outdeg}(v)$; so the walk can only get stuck at $v_0$. Hence the resulting walk $C$ is a cycle. If $C$ uses all edges of the graph, we are done. If not, then since $G$ is connected we can find a vertex $v'$ of $C$ with an unused outgoing edge, and we use this edge to start walking around the graph $G-C$ until we get stuck, necessarily at $v'$. The result will be a cycle $C'$. Starting at $v'$ we can traverse $C$ and then $C'$, thus obtaining 
a cycle $C \cup C'$ which is longer than $C$. Repeating this procedure, we will eventually construct an Eulerian cycle.
\end{proof}

There is a remarkable formula for the number of Eulerian cycles, due to de \textbf{B}ruijn, van Ardenne-\textbf{E}hrenfest, \textbf{S}mith, and \textbf{T}utte:  

\begin{theorem}\label{f.th:BEST}(BEST Theorem) If $G$ is an Eulerian directed graph, then the number of Eulerian cycles of $G$ is
\[
c(G,v) \cdot \prod_{w \in V} (\textrm{outdeg}(w)-1)!
\]
for any vertex $v$, where $c(G,v)$ is the number of oriented spanning trees rooted at $v$.
\end{theorem}

\begin{proof}
We fix an edge $e$ starting at $v$, and let each Eulerian cycle start at $e$. For each vertex $w$ let $E_w$ be the set of outgoing edges from $w$.

Consider an Eulerian cycle $C$. For each vertex $w \neq v$, let $e_w$ be the last outgoing edge from $w$ that $C$ visits, and let $\pi_w$ (resp. $\pi_v$) be the ordered set $E_w - e_w$ (resp. $E_v - e$) of the other outgoing edges from $w$ (resp. $v$), listed in the order that $C$ traverses them. 
It is easy to see that $T=\{e_w \, : \, w \neq v\}$ is an oriented spanning tree rooted at $v$. 

Conversely, an oriented tree $T$ and permutations $\{\pi_w \, : \, w \in V\}$ serve as directions to tour $G$. We start with edge $e$. Each time we arrive at vertex $w$, we exit it by using the first unused edge according to $\pi_w$. If we have used all the edges $E_w - e_w$ of $\pi_w$, then we use $e_w \in T$. It is not hard to check that this is a bijection. This completes the proof.
\end{proof}

\begin{corollary}\label{f.cor:trees}
In an Eulerian directed graph, the number of oriented spanning trees rooted at $v$ is the same for all vertices $v$; it equals
\[
c(G, v) = \frac1n  \lambda_1 \cdots \lambda_{n-1}
\]
where $\lambda_1, \ldots, \lambda_{n-1}, \lambda_n=0$ are the eigenvalues of $\overrightarrow{L}(G)$.
\end{corollary}

\begin{proof}
The BEST theorem implies that $c(G,v)$ is independent of $v$, and then the argument in the proof of Theorem \ref{f.th:Kirkhoff} applies to give the desired formula.
\end{proof}

The BEST theorem can be used beautifully to enumerate a very classical, and highly nontrivial, family of objects.
A $k$-ary \textbf{de Bruijn sequence} of order $n$ is a cyclic word $W$ of length $k^n$ in the alphabet $\{1,\ldots,k\}$ such that the $k^n$ consecutive subwords of $W$ of length $n$ are the $k^n$ distinct words of length $n$. For example, the $2$-ary deBruijn sequences of order $3$ are $11121222$ and $22212111$; these ``memory wheels" were described in Sanskrit poetry several centuries ago \cite{f.Kak}. 
Their existence and enumeration was proved by Flye Saint-Marie in 1894 for $k=2$ and by van Aardenne-Ehrenfest and de Bruijn in 1951 in general.

\begin{theorem}\cite{f.Flye, f.deBruijn} The number of $k$-ary de Bruijn sequences of order $n$ is $(k!)^{k^{n-1}}/k^n$.
\end{theorem}

\begin{proof} Consider the \textbf{de Bruijn graph} whose vertices are the $k^{n-1}$ sequences of length $n-1$ in the alphabet $\{1, \ldots, k\}$, and where there is an edge from $a_1a_2 \ldots a_{n-1}$ to the word $a_2a_3\ldots a_n$ for all $a_1, \ldots, a_n$. It is natural to label this edge $a_1a_2 \ldots a_n$. It then becomes apparent that $k$-ary de Bruijn sequence are in bijection with the Eulerian cycles of the de Bruijn graph. Since $\textrm{indeg}(v) = \textrm{outdeg}(v) = k$ for all vertices $v$, this graph is indeed Eulerian, and we proceed to count its Eulerian cycles.

Notice that for any vertices $u$ and $v$ there is a unique path of length $n$ from $u$ to $v$. Therefore the $k^{n-1} \times k^{n-1}$ adjacency matrix $A$ satisfies $A^n = J$, where $J$ is the matrix whose entries are all equal to $1$. We already saw that the eigenvalues of $J$ are $0, \ldots, 0, k^n$. Since the trace of $A$ is $k$, the eigenvalues of $A$ must be $0, \ldots, 0, k$. Therefore the Laplacian $L=kI-A$ has eigenvalues $k, \ldots, k, 0$. It follows from Corollary \ref{f.cor:trees} that the de Bruijn graph has 
$
c(G,v) = \frac1{k^{n-1}} k^{(k^{n-1}-1)} = k^{k^{n-1}-n}
$
oriented spanning trees rooted at any vertex $v$, and
\[
c(G,v)  \cdot \prod_{w \in V} (\textrm{outdeg}(w)-1)! = k^{k^{n-1}-n} \cdot (k-1)!^{k^{n-1}}
\]
Eulerian cycles, as desired.
\end{proof}

%\comment{Finite field method? See Fomin's notes.}

\subsubsection{\textsf{Counting perfect matchings: the Pfaffian method}} \label{f.sec:Pfaffian}

A \textbf{perfect matching} of a graph $G=(V,E)$ is a set $M$ of edges such that every vertex of $G$ is on exactly one edge from $M$. We are interested in computing the number $m(G)$ of perfect matchings of a graph $G$. We cannot expect to be able to do this in general; in fact, even for bipartite graphs $G$, the problem of computing $m(G)$ is \#P-complete. % \cite{f.Vallant}. 
However, for many graphs of interest, including all planar graphs, there is a beautiful technique which produces a determinantal formula for $m(G)$.
%reduces this problem to the computation of a %Pfaffian, which in turn relies on computing a related determinant.

\bigskip
\noindent \textsf{\textbf{Determinants and Pfaffians.}} Let $A$ be a \textbf{skew-symmetric} matrix of size $2m \times 2m$, so $A^T = -A$. The Pfaffian is a polynomial encoding the matchings of the complete graph $K_{2m}$.
A perfect matching $M$ of the complete graph $K_{2m}$ is a partition $M$ of $[2m]$ into disjoint pairs $\{i_1, j_1\}, \ldots, \{i_m, j_m\}$, where $i_k<j_k$ for $1 \leq k \leq m$. Draw the points $1, \ldots, 2m$ in order on a line and connect each $i_k$ to $j_k$ by a semicircle above the line. Let $\textrm{cr}(M)$ be the number of crossings in this drawing, and let $\textrm{sign}(M) = (-1)^{\textrm{cr}(M)}$. 
%\textbf{skew-symmetric} matrix, so $A^T = -A$. 
Let $a_M = a_{i_1j_1} \cdots a_{i_mj_m}$. 
%\[
%\textrm{Pf} = \sum_\pi (-1)^{\textrm{cr}(\pi)} \prod_{\{i,j\} \in \pi} a_{ij}.
%\]
The \textbf{Pfaffian} of $A$ is
\[
\textrm{Pf}(A) = \sum_{M} \textrm{sign}(M) a_M
\]
summing over all perfect matchings $M$ of the complete graph $K_{2m}$.

%Let us begin by reviewing the relevant concepts in linear algebra. The \textbf{determinant} of a matrix $A = (a_{ij})_{1 \leq i, j \leq n}$ is
%\[
%\det A = \sum_{\pi \in S_n} \textrm{sign}(\pi) a_\pi
%\]
%summing over all permutations $\pi \in S_n$, where  $a_\pi = a_{1\pi(1)} a_{2\pi(2)} \cdots a_{n \pi(n)}$. 
%If $A$ is a \textbf{skew-symmetric} matrix, so $A^T = -A$, then
%\begin{equation}\label{f.eq:detPf}
%\det(A) = \textrm{Pf}(A)^2
%\end{equation}
%for a polynomial $\textrm{Pf}(A)$ in the $a_{ij}$s, called the Pfaffian of $A$. \cite{f.Cayley} If $n$ is odd, it is easy to see that $\det A = \textrm{Pf}(A)=0$. Let us describe the Pfaffian for $n = 2m$.

\begin{theorem}\label{f.th:Pfaffian}
If $A$ is a skew-symmetric matrix, so $A^T = -A$, then
\begin{equation}\label{f.eq:detPf}
\det(A) = \mathrm{Pf}(A)^2.
\end{equation}
\end{theorem}

\begin{proof}[Sketch of Proof.]
The first step is to show that the skew symmetry of $A$ causes many cancellations in the determinant, and
\[
\det A = \sum_{\pi \in ECS_n} \textrm{sign}(\pi) a_\pi
\]
where $ECS_n \subset S_n$ is the set of permutations of $[n]$ having only  cycles of even length. Then, to prove that this equals $(\sum_{M} \textrm{sign}(M) a_M)^2$, we need a bijection between ordered pairs $(M_1, M_2)$ of matchings and permutations $\pi$ in $ECS_n$ such that $a_{M_1}a_{M_2} = a_\pi$ and $\textrm{sign}(M_1)\textrm{sign}(M_2) = \textrm{sign}(\pi)$. We now describe such a bijection.

\begin{figure}[ht]
 \begin{center}
  \includegraphics[scale=.8]{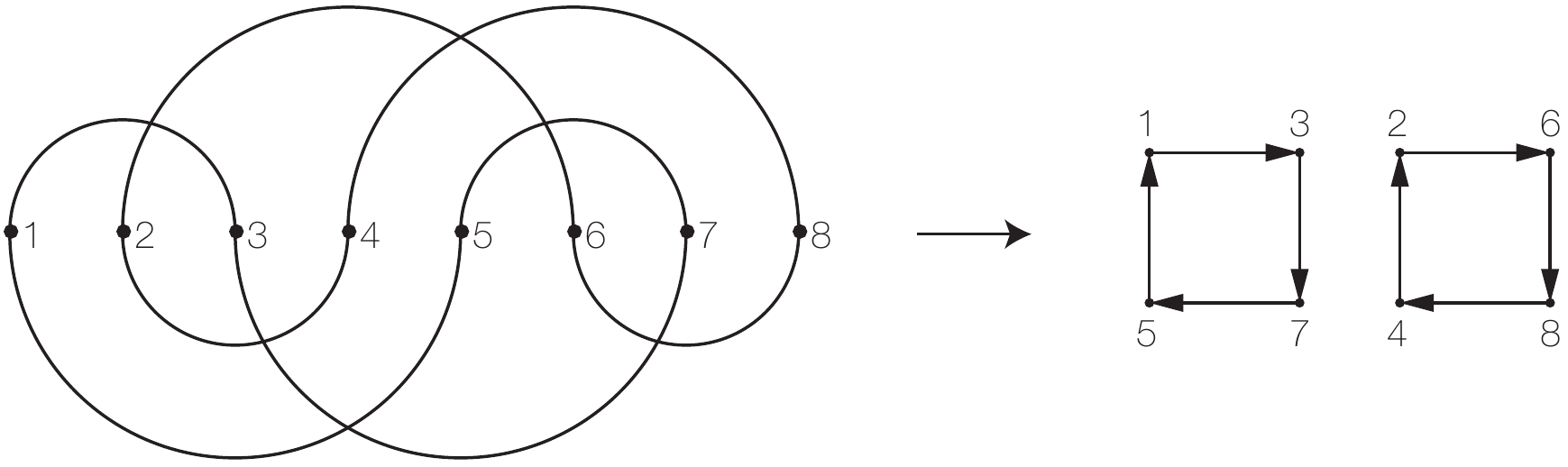}
  \caption{ \label{f.fig:matchings} 
The matchings $\{1,3\},\{2,6\}, \{4,8\}, \{5,7\}$ and $\{1,5\}, \{2,4\}, \{3,7\}, \{6,8\}$ give the permutation $(1375)(2684)$ with
$
%\textrm{sign}(M_1) a_{M_1} \textrm{sign}(M_2) a_{M_2} = 
(-a_{13}a_{26}a_{48}a_{57})(-a_{15}a_{24}a_{37}a_{68}) = 
a_{13}a_{37}a_{75}a_{51}a_{26}a_{68}a_{84}a_{42} 
%= \textrm{sign}(\pi) a_\pi
.$}
 \end{center}
\end{figure}

Draw the matchings $M_1$ and $M_2$ above and below the points $1, \ldots, n$ on a line, respectively. Let $\pi$ be the permutation given by the cycles of the resulting graph, where each cycle is oriented following the direction of $M_1$ at its smallest element. This is illustrated in Figure \ref{f.fig:matchings}. It is clear that $a_{M_1}a_{M_2} = a_\pi$, while some care is required to show that $\textrm{sign}(M_1)\textrm{sign}(M_2) = \textrm{sign}(\pi)$. For details, see for example \cite{f.Aigner}.
\end{proof}

\bigskip
\noindent
\textsf{\textbf{Counting perfect matchings via Pfaffians.}} Suppose we wish to compute the number $m(G)$ of perfect matchings of a graph $G=(V,E)$ with no loops. After choosing an orientation of the edges, we define the $V \times V$ \textbf{signed adjacency matrix} $S(G)$ whose entries are
\[
s_{ij} =
\begin{cases}
1 & \textrm{ if $i \rightarrow j$ is an edge of $G$} \\
-1 & \textrm{ if $j \rightarrow i$ is an edge of $G$} \\
0 & \textrm{ otherwise.}
\end{cases}
\]
Then $s_M = s_{i_1j_1} \cdots s_{i_mj_m}$ is nonzero if and only if $\{i_1, j_1\}, \ldots, \{i_m,j_m\}$ is a perfect matching of $G$. 

We say that our edge orientation is \textbf{Pfaffian} if all the perfect matchings of $G$ have the same sign.  At the moment there is no efficient test to determine whether a graph admits a Pfaffian orientation. There is a simple combinatorial restatement: an orientation is Pfaffian if and only if every even cycle $C$ such that $G \backslash V(C)$ has a perfect matching has an odd number of edges in each direction.

Fortunately, we have the following result of Kasteleyn \cite{f.Kasteleyn, f.LovaszPlummer}: 
\[
\textrm{Every planar graph has a Pfaffian orientation.}
\]
%It is easy to construct such an orientation inductively; %, by making sure that all internal cycles have an odd number of edges oriented clockwise. 
%%For details, 
%see \cite{f.Thomas} and the references therein. 

This is very desirable, because Theorem \ref{f.th:Pfaffian} implies:
\[
\textrm{For a Pfaffian orientation of $G$,} \,\,  m(G) = \sqrt{\det S(G)}.
\]

Therefore the number of matchings of a planar graph is reduced to the evaluation of a combinatorial determinant. We will see in Section \ref{f.sec:dets2} that there are many techniques at our disposal to carry out this evaluation.

\bigskip

Let us illustrate this method with an important example, due to Kasteleyn \cite{f.Kasteleyn} and Temperley--Fisher \cite{f.TemperleyFisher}.

\begin{theorem}\label{f.th:Kasteleyn}
The number $m(R_{a,b})$ of matchings of the $a \times b$ rectangular grid $R_{a,b}$ (where we assume $b$ is even) is
\[
m(R_{a,b}) = 4^{\lfloor a/2 \rfloor (b/2)}
%{\lfloor \frac{a}2 \rfloor \frac{b}2} 
\prod_{j=1}^{\lfloor a/2 \rfloor} \prod_{k=1}^{b/2} \left(\cos^2 \frac{\pi j}{a+1} + \cos^2 \frac{\pi k}{b+1}\right) \sim c \cdot e^{\frac{G}{\pi}ab} \sim c \cdot 1.3385^{ab},
\]
where $G=1 - \frac19 + \frac1{25} - \frac1{49} + \cdots$ is Catalan's constant. 
\end{theorem}

Clearly this is also the number of domino tilings of an $a \times b$ rectangle.

\begin{proof}[Sketch of Proof.]
Orient all columns of $R_{a,b}$ going up, and let the rows alternate between going right or left, assigning the same direction to all edges of the same row. The resulting orientation is Pfaffian because every square has an odd number of edges in each direction. The adjacency matrix $S$ satisfies 
$m(R_{a,b}) = \sqrt{\det S}$. To compute this determinant, it is slightly easier\footnote{In fact, this is the matrix that Kasteleyn uses in his computation.} to consider the following $mn \times mn$ matrix $B$:
\[
b_{ij} =
\begin{cases}
1 & \textrm{ if $i$ and $j$ are horizontal neighbors} \\
i & \textrm{ if $i$ and $j$ are vertical neighbors} \\
0 & \textrm{ otherwise.}
\end{cases}
\]
We can obtain $B$ from the $S$ by scaling the rows and columns by suitable powers of $i$, so we still have $m(R_{a,b}) = \sqrt{|\det B|}$. We will prove the product formula for this determinant in Section \ref{f.sec:dets2}.
%\[
%m(R_{a,b}) = 4^{\lfloor a/2 \rfloor (b/2)}
%%{\lfloor \frac{a}2 \rfloor \frac{b}2} 
%\prod_{j=1}^{\lfloor a/2 \rfloor} \prod_{k=1}^{b/2} \left(\cos^2 \frac{\pi j}{a+1} + \cos^2 \frac{\pi k}{b+1}\right).
%\]

We then use this product formula to give an asymptotic formula for $m(R_{a,b})$. Note that $\log m(R_{a,b})/ab$ may be regarded as a Riemann sum; as $m,n \rightarrow \infty$ it converges to 
\[
c = \frac1{\pi^2} \int_0^{\pi/2} \int_0^{\pi/2} \log (4\cos^2 x + 4 \cos^2 y) \, dx \, dy = \frac{G}{\pi}
\]
where $G$ is Catalan's constant. Therefore 
\[
m(R_{a,b}) \approx e^{\frac{G}{\pi}ab} \approx 1.3385^{ab}.
\]
Loosely speaking, this means that in a matching of the rectangular grid there are about $1.3385$ degrees of freedom per vertex.
\end{proof}

Obviously, this beautiful formula is not an efficient method of computing the exact value of $m(R_{a,b})$ for particular values of $a$ and $b$; it is not even clear why it gives an integer! There are alternative determinantal formulas for this quantity that are more tractable; see for example \cite[Section 10.1]{f.Aigner}. 

%Naturally, in enumerative combinatorics we are mostly interested in Pfaffians and/or determinants with elegant combinatorial formulas. Pfaffians and determinants are quite different in nature; computing the permanent of a matrix is \#P-complete problem \cite{f.Vallant}, while there are several algorithms for computing determinants in polynomial time, and numerous techniques for computing ``combinatorial" determinants of interest. However, for a relatively small but important class of matrices, $\textrm{Pf } M$ can be computed in terms of $\det M$. This will allow us to compute $m(G)$ for many graphs.
%
%It is clear from the definitions that the number of matchings of $G$ is 
%\[
%m(G) = \textrm{Pf }\underline{A}
%\]
%where $\underline{A}$ is  the $V \times V$ (unsigned) adjacency matrix, whose entries are
%\[
%\underline{a}_{ij} =
%\begin{cases}
%1 & \textrm{ if $ij$ is an edge of $G$} \\
%0 & \textrm{ otherwise.}
%\end{cases}
%\]

\subsubsection{\textsf{Counting routings: the Lindstr\"om--Gessel--Viennot lemma}}\label{f.sec:GesselViennot}

Let $G$ be a directed graph with no directed cycles, which has a weight $\textrm{wt}(e)$ on each edge $e$. We are most often interested in the unweighted case, where all weights are $1$. Let $S=\{s_1, \ldots, s_n\}$ and $T=\{t_1, \ldots, t_n\}$ be two (not necessarily disjoint) sets of vertices, which we call sources and sinks, respectively. A \textbf{routing} from $S$ to $T$ is a set of paths $P_1, \ldots, P_n$ from the $n$ sources $s_1, \ldots, s_n$ to the $n$ sinks $t_1, \ldots, t_n$ such that no two paths share a vertex. Let $\pi$ be the permutaiton of $[n]$ such that $P_i$ starts at source $s_i$ and ends at sink $t_{\pi(i)}$, and 
% where $\pi$ is a permutation of $[n]$. 
define $\textrm{sign}(R) = \textrm{sign}(\pi)$.

Let the weight of a path or a routing be the product of the weights of the edges it contains.
Consider the $n \times n$ \textbf{path matrix} $Q$ whose $(i,j)$ entry is 
\[
q_{ij} = \sum_{P \textrm{ path from $s_i$ to $t_j$}} \textrm{wt}(P).
\]

%\footnote{A source/sink in a graph is a vertex with no incoming/outgoing edges, respectively. A routing from $S$ to $T$ cannot use incoming edges to $S$ or outgoing edges from $T$, so we may assume that these are indeed vertices and sinks.}

\begin{theorem}\label{f.th:Lindstrom}(Lindstr\"om--Gessel--Viennot Lemma) 
Let $G$ be a directed acyclic graph with edge weights, and let $S=\{s_1, \ldots, s_n\}$ and $T=\{t_1, \ldots, t_n\}$ be sets of vertices in $G$. Then the determinant of the  $n \times n$ path matrix $Q$ is
\[
\det Q = \sum_{R \textrm{ routing  from $S$ to $T$}} \mathrm{sign}(R) \mathrm{wt}(R).
\]
In particular, if all edge weights are $1$ and if every routing takes $s_i$ to $t_i$ for all $i$, then 
\[
\det Q = \textrm{ number of routings from $S$ to $T$}.
\]
\end{theorem}

\begin{proof}
We have $\det A =  \sum_P \textrm{sign}(P) \textrm{wt}(P)$ summing over \textbf{all} path systems $P=\{P_1, \ldots, P_n\}$ from $S$ to $T$; we need to cancel out the path systems  which are not routings. For each such $P$, consider  the lexicographically first pair of paths $P_i$ and $P_j$ which intersect, and let $v$ be their first vertex of intersection. Now exchange the subpath of $P_i$ from $s_i$ to $v$ and the subpath of $P_j$ from $s_j$ to $v$, to obtain new paths $P_i'$ and $P_j'$. Replacing $\{P_i, P_j\}$ with $\{P_i', P_j'\}$, we obtain a new path system $\varphi(P)$ from $S$ to $T$. Notice that $\varphi(\varphi(P)) = P$, and $\textrm{sign}(\varphi(P))\textrm{wt}(\varphi(P)) +\textrm{sign}(P) \textrm{wt}(P) = 0$; so for all non-routings $P$, the path systems $P$ and $\varphi(P)$ cancel each other out.
\end{proof}

This theorem was also anticipated by Karlin and McGregor \cite{f.KarlinMcGregor} in the context of birth-and-death Markov processes.

%\comment{
%\bigskip
%\noindent \textsf{\textbf{Routings and tilings via determinants.}} 
%
%Lindstr\"om discovered Theorem \ref{f.th:Lindstrom} in the context of matroid theory. An acyclic directed graph $G=(V,E)$ and a set $S$ of $n$ vertices give rise to a matroid, whose \textbf{bases} are the sets $T$ of $n$ vertices such tath there exists a routing from $S$ to $T$ in $G$. Lindstr\"om's observation is there exist vectors $\{a_v \, : \, v \in V\} \in {\mathbb{R}}^S$ such that a set of vertices $T$ is a basis if and only if $\{a_t \, : \, t \in T\}$ is a basis for ${\mathbb{R}}^S$. Assign generic weights to the edges, and define the $S \times V$ path matrix where $a_{sv}$ is the sum of the weights of the paths from source $s \in S$ to $v \in V$. Then, by the Linstro\"om-Gessel-Viennot lemma, the columns of this matrix are the desired vectors.
%}

\bigskip
\noindent \textsf{\textbf{Determinants via routings.}}
The Lindstr\"om-Gessel-Viennot Lemma is also a useful combinatorial tool for computing determinants of interest, usually by enumerating routings in a lattice. We illustrate this with several examples.

\begin{enumerate}

\item (Binomial determinants)
Consider the \textbf{binomial determinant}
\[
{a_1, \ldots, a_n \choose b_1, \ldots, b_n} = 
\det \left[ {a_i \choose b_j} \right]_{1 \leq i, j \leq n}
\]
where $0 \leq a_1 < \cdots < a_n$ and $0 \leq b_1 < \cdots < b_n$ are integers. These determinants arise as coefficients of the Chern class of the tensor product of two vector bundles. \cite{f.Lascoux} This algebro-geometric interpretation implies these numbers are positive integers; as combinatorialists, we would like to know what they count.

A \textbf{SE path} is a lattice path in the square lattice ${\mathbb{N}}^2$ consisting of unit steps south and east.  Consider the sets of points $A=\{A_1, \ldots, A_n\}$ and $B=\{B_1, \ldots, B_n\}$ where $A_i = (0, a_i)$ and $B_i = (b_i,b_i)$ for $1 \leq i \leq n$. Since there are ${a_i \choose b_j}$ SE paths from $A_i$ to $B_j$, and since every SE routing from $A$ to $B$ takes $A_i$ to $B_i$ for all $i$, we have
\[
%\det \left( {a_i \choose b_j} \right)_{1 \leq i, j \leq n} = 
{a_1, \ldots, a_n \choose b_1, \ldots, b_n} = 
\textrm{ number of SE routings from $A$ to $B$}.
\]
This is the setting in which Gessel and Viennot discovered Theorem \ref{f.th:Lindstrom}; they also evaluated these determinants in several special cases. \cite{f.GesselViennot}
We  discuss a particularly interesting  case.

\item (Counting permutations by descent set)
The \textbf{descent set} of a permutation $\pi$ is the set of indices $i$ such that $\pi_i > \pi_{i+1}$. We now prove that
\[
{c_1, \ldots, c_k, n \choose 0, c_1, \ldots, c_k} = \textrm{ number of permutations of $[n]$ with descent set } \{c_1, \ldots, c_k\},
\]
for any $0< c_1 < \cdots < c_k <n$. It is useful to define $c_0=0, c_{k+1}=n$.

Encode such a permutation $\pi$ by a routing as follows. 
%Write $\pi$ as a word. Below each number, write the amount of numbers to the left of it and smaller than it in $\pi$. It is easy to recover $\pi$ from $f(\pi)$ and, furthermore, the descents of $\pi$ are the positions where $f(\pi)$ does not decrease.
% 
% 
%For each $i$ let $f_i$ be the number of indices $j<i$ such that $\pi_j < \pi_i$. It is easy to recover $\pi$ from the word $f(\pi) = f_1\ldots f_n$. Note that the descents of 
%
%Now split $f(\pi)$ into its increasing consecutive subsequences $f^1, \ldots, f^k$. Note that 
%
%Furthermore, $d$ is a descent of $\pi$ if and only if $f_d \geq f_{d+1}$. 
%
For each $i$ let $f_i$ be the number of indices $j \leq i$ such that $\pi_j \leq \pi_i$. %It is easy to recover $\pi$ from the word $f=f_1\ldots f_n$. 
Note that the descents $c_1, \ldots, c_k$ of $\pi$ are the positions where $f$ does not increase. Splitting $f$ at these positions, we are left with $k+1$ increasing subwords $f^1, \ldots, f^{k+1}$. Now, for $1 \leq i \leq n+1$ let $P_i$ be the NW path from $(c_{i-1}, c_{i-1})$ to $(0, c_i)$ taking steps north precisely at the steps listed in $f^i$. These paths give one of the routings enumerated by the binomial determinant in question, and this is a bijection. \cite{f.GesselViennot}

\begin{figure}[ht]
 \begin{center}
  \includegraphics[scale=.8]{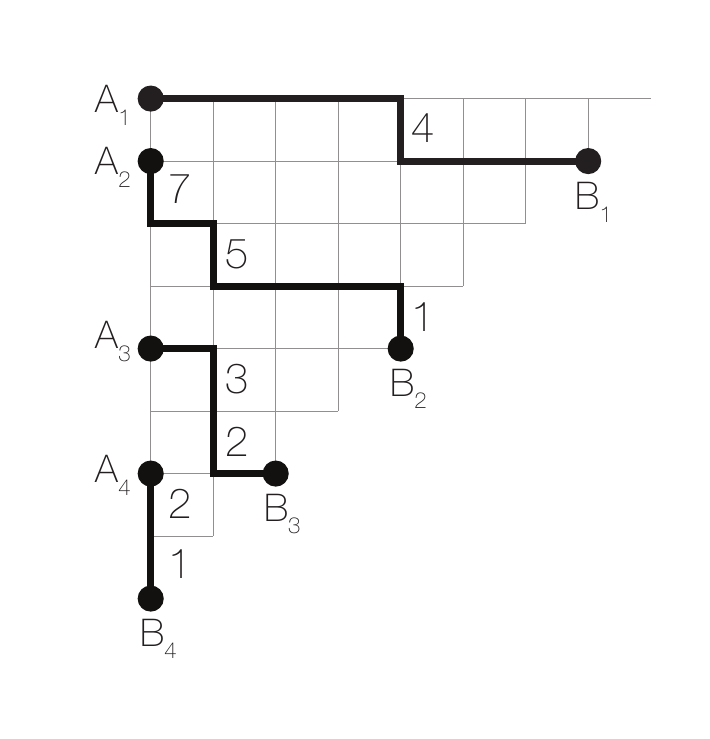}
  \caption{ \label{f.fig:descents} 
The routing corresponding to $\pi=28351674$ and $f(\pi) = 12.23.157.4$.}
 \end{center}
\end{figure}

\item (Rhombus tilings and plane partitions) Let $R_n$ be the number of tilings of a regular hexagon of side length $n$ using unit rhombi with angles $60^\circ$ and $120^\circ$. Their enumeration is due to MacMahon \cite{f.MacMahon}. There are several equivalent combinatorial models for this problem, illustrated in Figure \ref{f.fig:rhombi}, which we now discuss.

\begin{figure}[ht]
 \begin{center}
  \includegraphics[scale=.5]{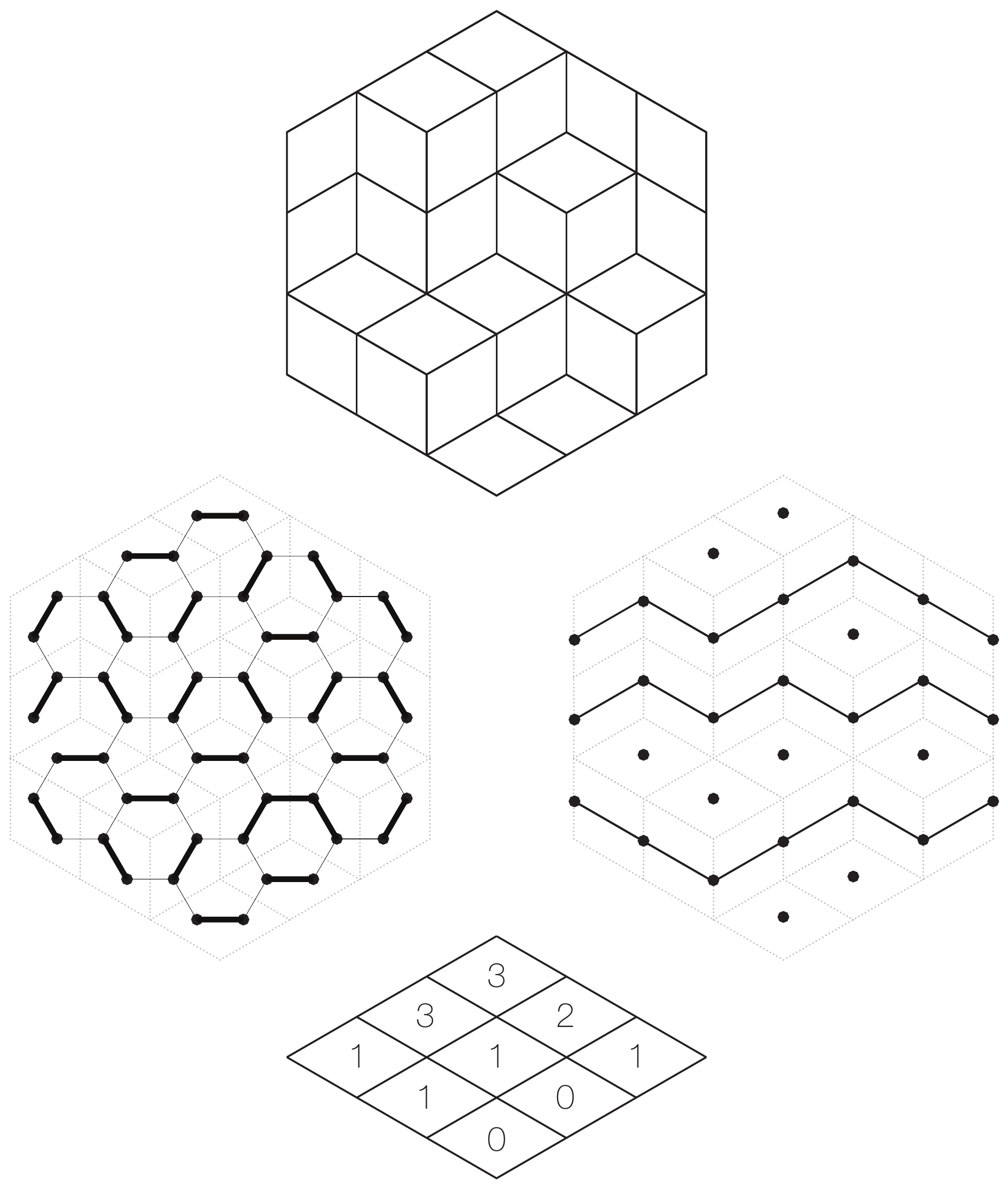}
  \caption{ \label{f.fig:rhombi} 
Four models for the rhombus tilings of a hexagon.}
 \end{center}
\end{figure}

Firstly, it is almost inevitable to view these tilings as three-dimensional pictures. This shows that $R_n$ is also the number of ways of stacking unit cubes into the corner of a cubical box of side length $n$. Incidentally, this three-dimensional view makes it apparent that there are exactly $n^2$ rhombi of each one of the 3 possible orientations.

Secondly, we may consider the triangular grid inside our hexagon, and place a dot on the center of each triangle. These dots form a hexagonal grid, where two dots are neighbors if they are at distance $1$ from each other. 
Finally, join two neighboring dots when the corresponding triangles are covered by a tile. The result is a perfect matching of the hexagonal grid.

Next, on each one of the $n^2$ squares of the floor of the box, write down the number of cubes above it. The result is a \textbf{plane partition}: an array of non-negative integers (finitely many of which are non-zero) which is weakly decreasing in each row and column. We conclude that $R_n$ is also the number of plane partitions whose non-zero entries are at most $n$, and fit inside an $n \times n$ square.

Finally, given such a rhombus tiling, construct $n$ paths as follows. Each path starts at the center of one of the vertical edges on the western border of the hexagon, and successively crosses each tile splitting it into equal halves.
It eventually comes out at the southeast side of the diamond, at the same height where it started (as is apparent from the 3-D picture). The final result is a routing from the $n$ sources $S_1, \ldots, S_n$ on the left to the sinks $T_1, \ldots, T_n$ on the right in the ``rhombus" graph shown below. It is clear how to recover the tiling from the routing. Since there are ${2n \choose n+i-j}$ paths from $S_i$ to $T_j$, the Lindstr\"om--Gessel-Viennot lemma tells us that $R_n$ is given by the determinant
\[
R_n = \det \left[ {2n \choose n+i-j} \right]_{1 \leq i, j \leq n} = \prod_{i,j,k=1}^n \frac{i+j+k-1}{i+j+k-2}.
\]
We will prove this product formula in Section \ref{f.sec:dets2}.

%\comment{There are many remarkable results about the enumeration of plane partitions with symmetry; see \cite{f.bakersdozen}.}

\item (Catalan determinants, multitriangulations, and Pfaffian rings) 
The \textbf{Hankel matrices} of a sequence $A=(a_0, a_1, a_2, \ldots)$ are
\[
H_n(A) = \begin{pmatrix}
a_0 & a_1 & \cdots & a_n \\
a_1 & a_2 & \cdots & a_{n+1} \\
\vdots & \vdots & \ddots & \vdots \\
a_n & a_{n+1} & \cdots & a_{2n} 
\end{pmatrix}, 
\qquad
H'_n(A) = \begin{pmatrix}
a_1 & a_2 & \cdots & a_{n+1} \\
a_2 & a_3 & \cdots & a_{n+2} \\
\vdots & \vdots & \ddots & \vdots \\
a_{n+1} & a_{n+2} & \cdots & a_{2n+1} 
\end{pmatrix}.
\]
Note that if we know the Hankel determinants $\det H_n(A)$ and $\det H'_n(A)$ and they are  nonzero for all $n$, then we can use them as a recurrence relation to recover each $a_k$ from $a_0, \ldots, a_{k-1}$.

There is a natural interpretation of the Hankel matrices of the Catalan sequence $C=(C_0, C_1, C_2, \ldots)$. Consider the ``diagonal" grid on the upper half plane with steps $(1,1)$ and $(1,-1)$. Let $A_i=(-2i,0)$ and $B_i = (2i,0)$. Then there are $C_{i+j}$ paths from $A_i$ to $B_j$, and there is clearly a unique routing from $(A_0, \ldots, A_n)$ to $(B_0, \ldots, B_n)$. This proves that $\det H_n(C) = 1$, and an analogous argument proves that $\det H'_n(C) = 1$. Therefore
\[
\det H_n(A) = \det H'_n(A) = 1 \textrm{ for all } n \geq 0 \quad  \Longleftrightarrow \quad
A \textrm{ is the Catalan sequence.}
\]

\begin{figure}[ht]
 \begin{center}
  \includegraphics[width=1.5in]{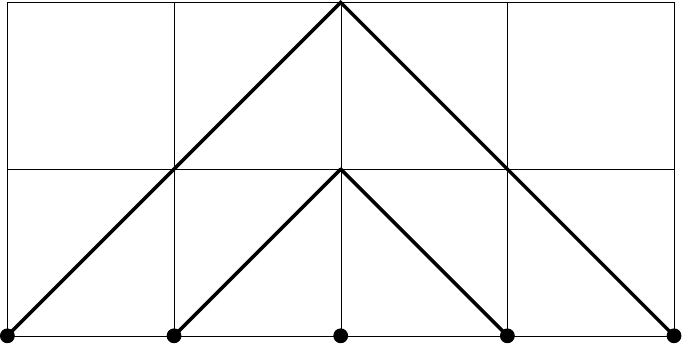}
  \caption{ \label{f.fig:descents} 
Routing interpretation of the Hankel determinant $H_n(C)$.
}
 \end{center}
\end{figure}

The Hankel determinants of the shifted Catalan sequences also arise naturally in several contexts; they are given by:
%\begin{equation}\label{f.eq:detCatalan}
%\begin{pmatrix}
%C_{n-2k} & C_{n-2k+1} & \cdots & C_{n-k} \\
%C_{n-2k+1} & C_{n-2k+2} & \cdots & C_{n-k+1} \\
%\vdots & \vdots & \ddots & \vdots \\
%C_{n-k} & C_{n-k+1} & \cdots & C_{n} 
%\end{pmatrix} = 
%\prod_{i+j \leq n-2k-1} \frac{i+j+2k}{i+j}.
%\end{equation}
\begin{equation}\label{f.eq:detCatalan}
\det
\begin{pmatrix}
C_{n-2k} & C_{n-2k+1} & \cdots & C_{n-k-1} \\
C_{n-2k+1} & C_{n-2k+2} & \cdots & C_{n-k} \\
\vdots & \vdots & \ddots & \vdots \\
C_{n-k-1} & C_{n-k} & \cdots & C_{n-2} 
\end{pmatrix} = 
\prod_{i+j \leq n-2k-1} \frac{i+j+2(k-1)}{i+j}.
\end{equation}
There are several ways of proving (\ref{f.eq:detCatalan}); for instance, it is a consequence of \cite[Theorem 26]{f.Krattenthaler}. We describe three appearances of this determinant.

(a) A \textbf{$k$-fan of Dyck paths of length $2n$} is a collection of $k$ Dyck paths from $(-n,0)$ to $(n,0)$ which do not cross (although they necessarily share some edges). Shifting the $(i+1)$st path $i$ units up and adding $i$ upsteps at the beginning and $i$ downsteps at the end, we obtain a routing of $k$ Dyck paths starting at the points $A=\{-(n+k-1), \ldots, -(n+1), -n\}$ and ending at the points $B=\{n, n+1, \ldots, n+k-1\}$ on the $x$ axis. It follows that the number of $k$-fans of Dyck paths of length $2(n-2k)$ is given by (\ref{f.eq:detCatalan}).

\begin{figure}[ht]
 \begin{center}
  \includegraphics[scale=.8]{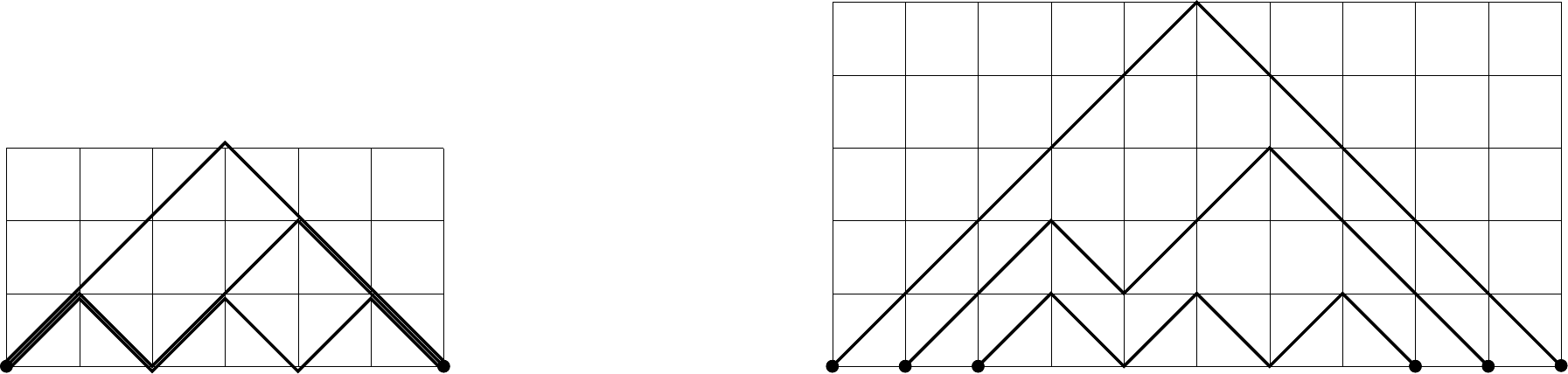}
  \caption{ \label{f.fig:descents} 
A $k$-fan of Dyck paths.
}
 \end{center}
\end{figure}

% \textbf{Dyck routings} of $k+1$ non-intersecting Dyck paths from $\{(-2n,0), \ldots, (-2(n-2k),0)\}$ to $\{(2(n-2k),0), \ldots, (2n,0)\}$. 

(b) There is also an extension of the classical one-to-one correspondence between Dyck paths and triangulations of a polygon. Define a \textbf{$k$-crossing} in an $n$-gon to be a set of $k$ diagonals which cross pairwise. A \textbf{$k$-triangulation} is a maximal set of diagonals with no $(k+1)$-crossings. The main enumerative result, due to Jonsson \cite{f.Jonsson}, is that the number of $k$-triangulations of an $n$-gon is also given by  (\ref{f.eq:detCatalan}). A subtle bijection with fans of Dyck paths is given in \cite{f.SerranoStump}. 

Several properties of triangulations extend non-trivially to this context. For example, every $k$-triangulation has exactly $k(2n-2k-1)$ diagonals. 
\cite{f.Nakamigawa, f.DressKoolenMoulton}. The $k$-triangulations are naturally the facets of a simplicial complex called the \textbf{multiassociahedron}, which is topologically a sphere \cite{f.Jonssonunpublished}; it is not currently  known whether it is a polytope. There is a further generalization in the context of Coxeter groups, with very interesting connections to cluster algebras \cite{f.CeballosLabbeStump}.

(c) These determinants also arise naturally in the commutative algebraic properties of Pfaffians, defined earlier in this section. Let $A$ be a skew-symmetric $n \times n$ matrix whose entries above the diagonal are indeterminates $\{a_{ij} \, : \, 1 \leq i<j \leq n\}$ over a field ${\mathbbm{k}}$.
%Recall that a matrix $A$ is \textbf{skew-symmetric} if $A^T = -A$. 
%It is known that 
%\[
%\det(A) = \textrm{Pf}(A)^2
%\]
%where the \textbf{Pfaffian} $\textrm{Pf}(A)$ is a polynomial in the $a_{ij}$s, which is $0$ if $n$ is odd. To describe the Pfaffian when $n = 2m$, consider the \textbf{perfect matchings} $\pi$ of %the complete graph $K_{2m}$ is a partition 
%$[2m]$ into disjoint pairs $\{i_1, j_1\}, \ldots, \{i_m, j_m\}$. We draw the points $1, \ldots, 2m$ in order on a line and connect each $i_k$ to $j_k$ by a semicircle above these points. The crossing number  $\textrm{cr}(\pi)$ is the number of crossings in this drawing. For $n=2m$, the Pfaffian is a signed generating function for perfect matchings:
%\[
%\textrm{Pf} = \sum_\pi (-1)^{\textrm{cr}(\pi)} \prod_{\{i,j\} \in \pi} a_{ij}.
%\]
%Now c
Consider the \textbf{Pfaffian ideal} $I_k(A)$  generated by the ${n \choose 2k}$  Pfaffian minors of $A$ of size $2k\times 2k$, and the \textbf{Pfaffian ring} $R_k(A) = {\mathbbm{k}}[a_{ij}]/I_k(A)$. Then the multiplicity of the Pfaffian ring $R_k(X)$ is also given by (\ref{f.eq:detCatalan}). \cite{f.HerzogTrung, f.GhorpadeKrattenthaler}

%\comment{What about determinantal ideals?}

\item (Schr\"oder determinants and Aztec diamonds.)
Recall from Section \ref{f.sec:ogfexamples} that a \textbf{Schr\"oder path} of length $n$ is a path from $(0,0)$ to $(2n,0)$ using steps $NE=(1,1), SE=(1, -1),$ and $E=(2,0)$ which stays above the $x$ axis. %The generating function enumerating Schr\"oder paths is 
%(large) Schr\"oder number $r_n$ is the number of such paths. Let $R(x) = \sum_{n \geq 0} r_nx^n$. From the techniques of Section \ref{f.sec:GF} we see that $R(x) = 1/(1-x-xR(x))$, and therefore
%\[
%R(x) = \frac{1-x-\sqrt{1-6x+x^2}}{2x}
%\]
The Hankel determinant $\det H_n(R)$ counts the routings of Schr\"oder paths from the points $A=\{0, -2, \ldots, -(2n)\}$ to the points $B=\{0, 2,\ldots, 2n\}$ on the $x$-axis. 

These Hankel determinants have a natural interpretation in terms of tilings. Consider the \textbf{Aztec diamond}\footnote{This shape is called the \textbf{Aztec diamond} because it is reminiscent of designs of several Native American groups. Perhaps the closest similarity is with Mayan pyramids, such as the Temple of Kukulc\'an in Chich\'en Itz\'a; the name \textbf{Mayan diamond} may have been more appropriate.} $AD_n$ consisting of $2n$ rows centered horizontally, consisting successively of $2, 4, \ldots, 2n, 2n, \ldots, 4, 2$ squares. We are interested in counting the tilings of the Aztec diamond into dominoes.

\begin{figure}[ht]
 \begin{center}
  \includegraphics[scale=.45]{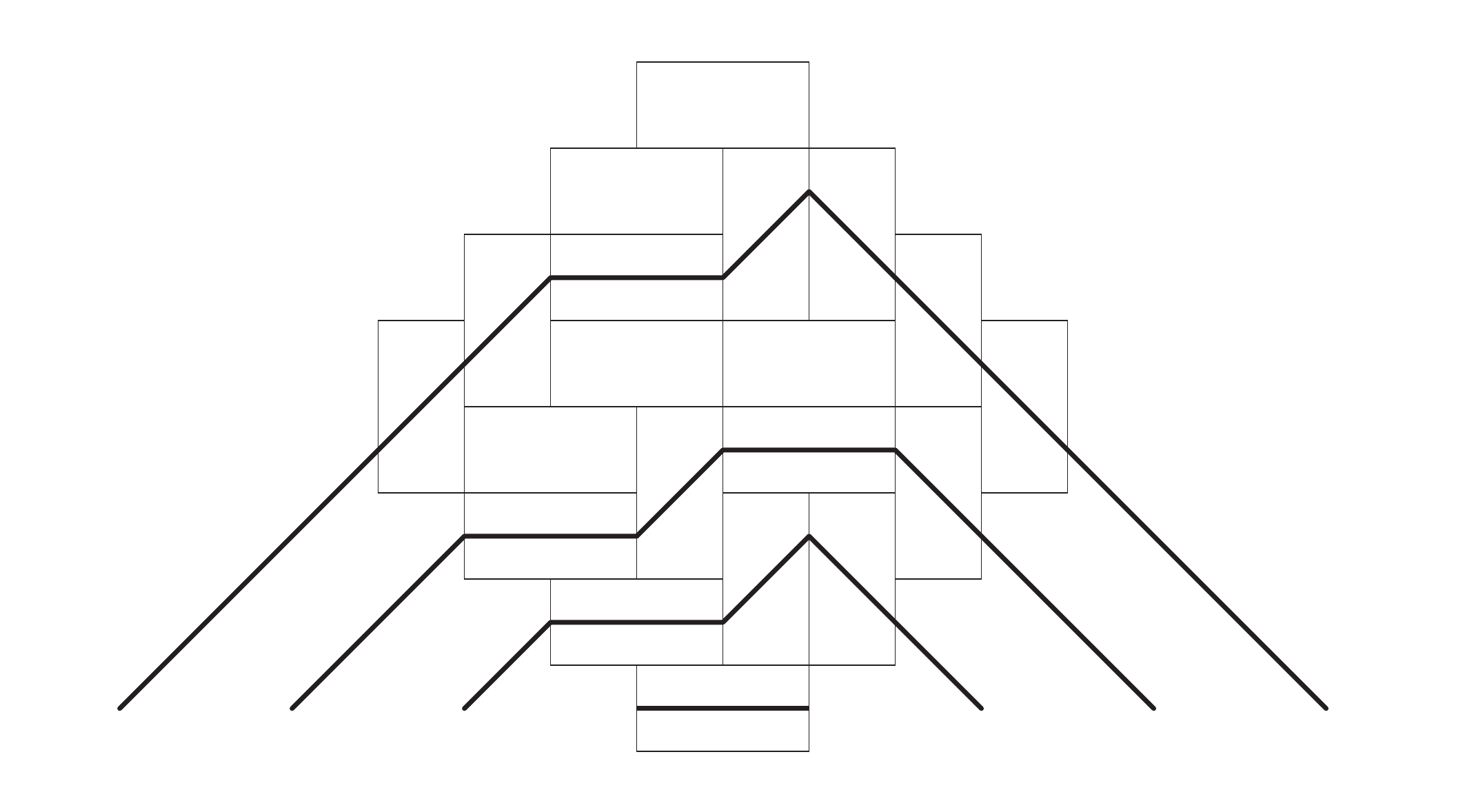}
  \caption{ \label{f.fig:descents} 
A tiling of the Aztec diamond and the corresponding routing.
}
 \end{center}
\end{figure}

Given a domino tiling of $AD_n$, construct $n$ paths as follows. Each path starts at the center of one of the vertical unit edges on the southwest border of the diamond, and successively crosses each tile that it encounters following a straight line through the center of the tile. It eventually comes out at the southeast side of the diamond, at the same height where it started. If we add $i$ initial NE steps $NE$ steps and $i$ final $SE$ steps to the $(i+1)$st path for each $i$, the result will be a routing of Schr\"oder paths from $A=\{-(2n), \ldots, -2, 0\}$ to $B=\{0, 2,\ldots, 2n\}$. In fact this correspondence is a bijection \cite{f.EuFu}.

We will prove in Section \ref{f.sec:dets2} that 
\[
\begin{cases}
&\det H_n(A) = 2^{n(n-1)/2} \\
&\det H'_n(A) = 2^{n(n+1)/2} 
\end{cases}
\textrm{ for all } n \geq 0 \quad  \Longleftrightarrow \quad
A \textrm{ is the Schr\"oder sequence.}
\]
It will follow that 
\[
\textrm{number of domino tilings of the Aztec diamond } AD_n = 2^{n(n+1)/2}.
\]

This elegant result is originally due to Elkies, Kuperberg, Larsen, and Propp. For several other proofs, see \cite{f.EKLP1, f.EKLP2}.
 
%
%
%The $(i+1)$st path in a Schr\"oder routing must start with $i$ upsteps and ends with $i$ downsteps; let us remove these. Now place a $2 \times 1$ domino on top of each step, so that their centers coincide; place the domino horizontally if the step is horizontal, and vertically otherwise. 
%
%Consider the ``diagonal" grid on the upper half plane with steps $(1,1)$ and $(1,-1)$. Let $A_i=(-2i,0)$ and $B_i = (2i,0)$. Then there are $C_{i+j}$ paths from $A_i$ to $B_j$, and there is clearly a unique routing from $(A_0, \ldots, A_n)$ to $(B_0, \ldots, B_n)$. This proves that $\det H_n(C) = 1$, 

\end{enumerate}

\subsection{\textsf{Computing determinants}}\label{f.sec:dets2}

In light of Section \ref{f.sec:dets1}, it is no surprise that combinatorialists have become talented at computing determinants.
%The previous section naturally
%%In the previous section we saw one natural occurrence of determinants in combinatorics. This 
%raises the question: if we encounter a determinant in a problem of interest, how should we try to evaluate it? 
Fortunately, this is a very classical topic with connections to many branches of mathematics and physics, and by now there are numerous general techniques and guiding examples available to us. Krattenthaler's surveys \cite{f.Krattenthaler} and \cite{f.Krattenthaler2} are excellent references which have clearly influenced the exposition in this section. We now highlight some of the key tools and examples.

\subsubsection{\textsf{{Is it known?}}}

Of course, when we wish to evaluate a new determinant, one first step is to check whether it is a special case of some known determinantal evaluation. Starting with classical evaluations such as the Vandermonde determinant
\begin{equation}\label{f.eq:Vandermonde}
\det (x_i^{j-1}) _{1 \leq i, j \leq n} = \prod_{1 \leq i < j \leq n} (x_j-x_i),
\end{equation}
%and the Cauchy double alternant
%\begin{equation}\label{f.eq:doublealternant}
%\det_{1 \leq i, j \leq n} \frac1{x_i+y_j} = 
%\frac{\displaystyle\prod_{1 \leq i < j \leq n} (x_j-x_i)(y_j-y_i)}{\displaystyle\prod_{1 \leq i , j \leq n} (x_i+y_j)},
%\end{equation}
there is now a wide collection of powerful results at our disposal.  A particularly useful one \cite[Lemma 3]{f.Krattenthaler} states that for any $x_1, \ldots, x_n, a_2, \ldots, a_n, b_2, \ldots, b_n$ we have:
\begin{equation}\label{f.eq:generaldet}
\det 
\left[(x_i+b_2)\cdots (x_i+b_j)(x_i+a_{j+1}) \cdots (x_i+a_n)\right]_{1 \leq i, j \leq n}
= \prod_{1 \leq i < j \leq n} (x_j-x_i) \prod_{1 \leq i < j \leq n} (b_i-a_j).
 \end{equation}
For instance, as pointed out in \cite{f.GhorpadeKrattenthaler} and \cite{f.Krattenthaler}, the Catalan determinant (\ref{f.eq:detCatalan}) is a special case of this formula. Recognizing it as such is not immediate, but the product formula for Catalan numbers gives an indication of why this is feasible.

In fact, here is a counterintuitive principle: often the easiest way to prove a determinantal identity is to generalize it. It is very useful to introduce as many parameters as possible into a determinant, while making sure that the more general determinant still evaluates nicely. We will see this principle in action several times in what follows.

%\comment{Other Weyl groups?}

\subsubsection{\textsf{{Row and column operations}}}
A second step is to check whether the standard methods of computing determinants are useful: Laplace expansion by minors, or performing row and column operations until we get a matrix whose determinant we can compute easily. For example, recall  the determinant $L_0(K_n)$ of the $(n-1) \times (n-1)$ reduced Laplacian of the complete graph $K_n$, discussed in Section \ref{f.sec:spanningtrees}. We can compute it by first adding all rows to the first row, and then adding the first row to all rows:
\[
\det L_0(K_n) = 
\begin{vmatrix}
n-1 & -1 & \cdots & -1 \\
-1 & n-1 & \cdots & -1 \\
\vdots &\vdots & \ddots & \vdots \\
-1 & -1 & \cdots & n-1 
\end{vmatrix}
=
\begin{vmatrix}
1 & 1 & \cdots & 1 \\
-1 & n-1 & \cdots & -1 \\
\vdots &\vdots & \ddots & \vdots \\
-1 & -1 & \cdots & n-1 
\end{vmatrix}
=
\begin{vmatrix}
1 & 1 & \cdots & 1 \\
0 & n & \cdots & 0 \\
\vdots &\vdots & \ddots & \vdots \\
0 & 0 & \cdots & n
\end{vmatrix}
= n^{n-2},
\]
reproving Theorem \ref{f.th:spanningtrees}.1.

\subsubsection{\textsf{{Identifying linear factors}}}
 Many $n \times n$ determinants of interest have formulas of the form $\det M({\mathbf{x}}) = c L_1({\mathbf{x}}) \cdots L_n({\mathbf{x}})$ where $c$ is a constant and the $L_i({\mathbf{x}})$ are linear functions in the variables ${\mathbf{x}} = (x_1, \ldots, x_k)$. We may prove such a formula by first checking that each $L_i({\mathbf{x}})$ is indeed a factor of $M$, and then computing the constant $c$.

The best known application of this technique is the proof of the formula (\ref{f.eq:Vandermonde}) for Vandermonde's determinant $V(x_1, \ldots, x_n)$. If $x_i = x_j$ for $i \neq j$, then rows $i$ and $j$ are equal, and the determinant is $0$. It follows that $x_i-x_j$ must be a factor of the polynomial $\det V(x_1, \ldots, x_n)$. Since this polynomial is homogeneous of degree ${n \choose 2}$, it must equal a constant times $\prod_{i<j} (x_i-x_j)$. Comparing the coefficients of $x_1^0 x_2^1 \cdots x_n^{n-1}$ we see that the constant equals $1$.

A similar argument may be used to prove the more general formula (\ref{f.eq:generaldet}). %r or the Weyl denominator factorizations. \comment{blue}

To use this technique, it is sometimes necessary to introduce new variables into our determinant. For example, the formula $\det (i^{\, j-1})_{1 \leq i, j \leq n} = 1^{n-1}2^{n-2}\cdots (n-1)^1$ cannot immediately be treated with this technique. However, the factorization of the answer suggests that this may be a special case of a more general result where this method does apply; in this case, Vandermonde's determinant.

\subsubsection{\textsf{{Computing the eigenvalues}}}
 Sometimes we can compute explicitly the eigenvalues of our matrix, and multiply them to get the determinant. One common technique is to produce a complete set of  eigenvectors. 

\begin{enumerate}
\item (The Laplacian of the complete graph $K_n$)
Revisiting the example above, the Laplacian of the complete graph is $L(K_n) = nI-J$ where $I$ is the identity matrix and $J$ is the matrix all of whose entries equal $1$. We first find the eigenvalues of $J$: $0$ is an eigenvalue of multiplicity $n-1$, as evidenced by the linearly independent eigenvectors
$\mathbf{e}_1-\mathbf{e}_2, \ldots, \mathbf{e}_{n-1}-\mathbf{e}_{n}$. Since the sum of the eigenvalues is $\textrm{tr}(J) = n$, the last eigenvalue is $n$; an eigenvector is $\mathbf{e}_1 + \cdots + \mathbf{e}_{n-1}$. Now, if $v$ is an eigenvector for $J$ with eigenvalue $\lambda$, then it is an eigenvector for $nI-J$ with eigenvalue $n - \lambda$. Therefore the eigenvalues of $nI-J$ are $n, n, \ldots, n, 0$. Using Theorem \ref{f.th:Kirkhoff}, we have reproved yet again that $\det L_0(K_n) = \frac1n(n^{n-1}) =  n^{n-2}$.

\item (The Laplacian of the $n$-cube $C_n$)
A more interesting is  the reduced Laplacian $L_0(C_n)$ of the graph of the $n$-dimensional cube, from Theorem \ref{f.th:spanningtrees}.4. By producing explicit eigenvectors, one may prove that if the Laplacians $L(G)$ and $L(H)$ have eigenvalues $\{\lambda_i \, : \, 1 \leq i \leq a\}$ and $\{\mu_j \, : \, i \leq j \leq b\}$ then the Laplacian of the product graph $L(G \times H)$ has eigenvalues $\{\lambda_i + \mu_j \, : \, 1 \leq i \leq a, 1 \leq j \leq b\}$. Since $C_1$ has eigenvalues $0$ and $2$, this implies that $C_n = C_1 \times \cdots \times C_1$ has eigenvalues $0, 2, 4, \ldots, 2n$ with multiplicities ${n \choose 0}, {n \choose 1}, \ldots, {n \choose n}$, respectively. Therefore the number of spanning trees of the cube $C_n$ is
\[
\det L_0(C_n) = \frac1{2^n} 2^{n \choose 1} 4^{n \choose 2} \cdots (2n)^{n \choose n} = 
2^{2^n-n-1} 1^{n \choose 1} 2^{n \choose 2} \cdots n^{n \choose n}.
\]

\item (The perfect matchings of a rectangle)
An even more interesting example comes from the perfect matchings of the $a  \times b$ rectangle, which we discussed in Section \ref{f.sec:Pfaffian}. %Here the eigenvalues are $2\cos\frac{\pi k}{2n+1} + 2 i \cos \frac{\pi l}{2m+1}$ so
Let $V$ be the $4mn$-dimensional vector space of functions $f:[2m] \times [2n] \rightarrow {\mathbb{C}}$, and consider the linear transformation $L: V \rightarrow V$ given by
\[
(Lf)(x,y) = f(x-1,y) + f(x+1,y) + i f(x, y-1) + i f(x,y+1),
\]
where $f(x,y)=0$ when $x \in \{0, a+1\}$ or $y \in \{0, b+1\}$. The matrix of this linear transformation is precisely the one we are interested in. %It is useful to extend $f:[2m] \times [2n] \rightarrow {\mathbb{C}}$ to the unique function $\hat{f}: {\mathbb{Z}} \times {\mathbb{Z}} \rightarrow {\mathbb{C}}$ which agrees with $f$ on $f:[2m] \times [2n]$, is an odd function in $x$ and in $y$, and is periodic in $x$ and $y$ with periods $4m+2$ and $4n+2$, respectively.  
% such that $\hat{f}(-x,y) = \hat{f}(x,-y) = -\hat{f}(x,y)$, and $\hat{f}(x+4m+2,y) = \hat{f}(x,y+4n+2) = \hat{f}(x,y)$. 
A straightforward computation shows that the following are eigenfunctions and eigenvalues of $L$:
\[
g_{k,l}(x,y) = \sin\frac{k\pi x}{a+1} \sin \frac{l\pi y}{b+1}, \qquad \lambda_{k,l} = 2\cos \frac{k\pi}{a+1} + 2i \cos \frac{l\pi}{b+1}
\]
for $1 \leq k \leq a$ and $1 \leq l \leq b$. (Note that $g_{k,l}(x,y)=0$ for $x \in \{0, a\}$ or $y \in \{0, b\}$.) 
This is then the complete list of eigenvalues for $L$, so
\[
\det L = 2^{ab}\prod_{k=1}^{a} \prod_{l=1}^{b} \left( \cos \frac{k\pi}{a+1} + i \cos \frac{l\pi}{b+1}\right)
%^{\lfloor a/2 \rfloor (b/2)}
%{\lfloor \frac{a}2 \rfloor \frac{b}2} 
%\prod_{j=1}^{\lfloor a/2 \rfloor} \prod_{k=1}^{b/2} \left(\cos^2 \frac{\pi j}{a+1} + \cos^2 \frac{\pi k}{b+1}\right).
%
%
%= 2^{ab} \prod_{k=1}^{n} \prod_{l=1}^{m} \left( \cos^2 \frac{k\pi}{2n+1} + \cos^2 \frac{l\pi}{2m+1}\right)^2
\]
which is easily seen to equal the expression in Theorem \ref{f.th:Kasteleyn}.  
  \end{enumerate}

\subsubsection{\textsf{{LU factorizations}}}
 A classic result in linear algebra states that, under mild hypotheses, a square matrix $M$ has a unique factorization 
\[
M = LU
\]
where $L$ is a lower triangular matrix and $U$ is an upper triangular matrix with all diagonal entries equal to $1$. Computer algebra systems can compute the LU-factorization of a matrix, and if we can guess and prove such a factorization 
it will follow immediately that $\det M$ equals the product of the diagonal entries of $L$. 

An interesting application of this technique is the determinant
\begin{equation}\label{f.eq:gcd}
\det (\gcd(i,j))_{1 \leq i, j \leq n} = \prod_{i=1}^n \varphi(i), 
\end{equation}
where $\varphi(k) = \{i \in {\mathbb{N}} \, : \, (\gcd(i,k) = 1 \textrm{ and } 1 \leq i \leq k\}$ is Euler's totient function. This is a special case of a more general formula for semilattices which is easier to prove. For this brief computation, we assume familiarity with the M\"obius function $\mu$ and the zeta function $\zeta$ of a poset; these will be treated in detail in Section \ref{f.sec:incidence}.

Let $P$ be a finite meet semilattice and consider any function $F:P \times P \rightarrow {\mathbbm{k}}$.
 %be a function in the incidence algebra ${\mathcal{I}}nc(P)$. 
We will prove the \textbf{Lindstr\"om--Wilf determinantal formula}: 
\begin{equation}\label{f.LindstromWilf}
\det F(p \vee q,p)_{p, q \in P} = \prod_{p \in P} \left(\sum_{r \geq p} \mu(p,r) F(r,p) \right)
\end{equation}
%\[
%\det F(p \wedge q,p)_{p, q \in P} = \prod_{p \in P} \left(\sum_{r \leq p} F(r,p) \mu(r,p)\right)
%\]
%One way to discover this formula is to compute the LU factorization of the $P \times P$ matrix $F_{pq} = F(p \wedge q,p)_{p, q \in P}$ in a few examples. 
%%To do this, w
%We need to decide what order to list the rows and columns in; \emph{i.e.}, a linear order for $P$. We choose any order where
% %The most natural choice is a linear extension, so 
%$p$ precedes $q$ whenever $p \leq q$. Then it is not difficult to guess 
Computing some examples will suggest that the LU factorization of $F$ is $F=M Z$ where
\[
%F = MZ \qquad 
M_{pq} = \begin{cases}
 \sum_{r \geq q} \mu(q,r) F(r,p) & \textrm{ if } p \leq q, \\
0 & \textrm{ otherwise},
\end{cases}
\qquad
Z_{pq} = \begin{cases}
1 & \textrm{ if } p \geq q, \\
0 & \textrm{ otherwise}.
\end{cases}
\]
%
%\[
%F = H^TZ \qquad H(p,q) = \begin{cases}
% \sum_{r \leq p} \mu(r,p) F(r,q) & \textrm{ if } p \leq q, \\
%0 & \textrm{ otherwise},
%\end{cases}
%\qquad
%Z(p,q) = \begin{cases}
%1 & \textrm{ if } p \leq q, \\
%0 & \textrm{ otherwise}.
%\end{cases}
%\]
This guess is easy to prove, and it immediately implies 
(\ref{f.LindstromWilf}). In turn, applying the Lindstr\"om--Wilf to the poset of integers $\{1, \ldots, n\}$ ordered by reverse divisibility and the function $F(x,y)=x$, we obtain (\ref{f.eq:gcd}). 
 
Another interesting special case is the determinant:
\[
\det(x^{{\mathrm{rank}}(p \vee q)})_{p, q \in P} = \prod_{p \in P} \left( x^{{\mathrm{rank}}(p)} \,  \chi_{[p, {\widehat{1}}]}(1/x).
%\left(\frac1x\right)
\right)
\]
where $\chi_{[p, \hat{1}]}(x)$ is the characteristic polynomial of the interval $[x, \hat{1}]$. 
When $P$ is the partition lattice $\Pi_n$, this determinant arises in Tutte's work on the Birkhoff-Lewis equations \cite{f.TutteBL}.

\subsubsection{\textsf{{Hankel determinants and continued fractions}}}
For Hankel determinants, the following connection with continued fractions \cite{f.Wall} 
is extremely useful. If the expansion of the generating function for a sequence $f_0, f_1, \ldots$ as a \textbf{J-fraction} is 
\[
\sum_{n=0}^\infty f_n x^n = \frac{f_0}{1+a_0x - \frac{b_1x^2}{1+a_1x- \frac{b_2x^2}{1+a_2x - \cdots}}},
\]
then the Hankel determinants of $f_0, f_1, \ldots$ equal
\[
\det H_n(A) = f_0^n b_1^{n-1} b_2^{n-2} \cdots b_{n-2}^2 b_{n-1}
\]
For instance, using the generating function for the Schr\"order numbers $r_n$, it is easy to prove that 
\[ 
\sum_{n=0}^\infty r_{n}x^n = \frac{1}{1-2x - \frac{2x^2}{1-3x- \frac{2x^2}{1-3x - \cdots}}},
\qquad
\sum_{n=0}^\infty r_{n+1}x^n = 
%2+6x+22x^2+90x^3+394x^4+\cdots 
\frac{2}{1-3x - \frac{2x^2}{1-3x- \frac{2x^2}{1-3x - \cdots}}}.
\]
Therefore
\[
\det H_n(R) = 2^{n(n-1)/2}, \qquad \det H_n'(R) = 2^{n(n+1)/2},
\]
as stated in Example 5 of  Section \ref{f.sec:GesselViennot}.

By computer calculation, it is often easy to guess J-fractions experimentally. With a good guess in place, there is an established procedure for proving their correctness, rooted in the theory of orthogonal polynomials; see \cite[Section 2.7]{f.Krattenthaler}.

%
%\[
%\frac{1-x-\sqrt{1-6x+x^2}}{2x} = \sum_{n=0}^\infty r_nx^n = 
%%2+6x+22x^2+90x^3+394x^4+\cdots 
%\frac{2}{1-3x - \frac{2x^2}{1-3x- \frac{2x^2}{1-3x - \cdots}}}
%\]
%
%\[
%\frac{1-3x-\sqrt{1-6x+x^2}}{2x^2} = \sum_{n=0}^\infty r_{n+1}x^n = \frac{1}{1-2x - \frac{2x^2}{1-3x- \frac{2x^2}{1-4x - \cdots}}}
%\]

%\comment{Varchenko's determinant?}

\bigskip

\noindent \textsf{\textbf{Dodgson condensation.}} It is often repeated that Lewis Carroll, author of \emph{Alice in Wonderland}, was also an Anglican deacon and a mathematician, publishing under his real name, Rev. Charles L. Dodgson. His contributions to mathematics are discussed less often, and one of them is an elegant method for computing determinants.

To compute an $n \times n$ determinant $A$, we create a square pyramid of numbers, consisting of $n+1$ levels of size $n+1, n \ldots, 1$, respectively.
On the bottom level we place an $(n+1) \times (n+1)$ array of $1$s, and on the next level we place the $n \times n$ matrix $A$. Each subsequent floor is obtained from the previous two by the following rule: each new entry is given by $f=(ad-bc)/e$ where $f$ is directly above the entries $\begin{pmatrix} a & b \\ c & d \end{pmatrix}$ and two floors above the entry $e$.\footnote{Special care is required when $0$s appear in the interior of the pyramid.} The top entry of the pyramid is the determinant. For example, the computation

\[
\begin{pmatrix}
1 & 1 & 1 & 1 & 1 \\
1 & 1 & 1 & 1 & 1 \\
1 & 1 & 1 & 1 & 1 \\
1 & 1 & 1 & 1 & 1 \\
1 & 1 & 1 & 1 & 1
\end{pmatrix}
\rightarrow \begin{pmatrix}
2 & 7 & 5 & 4 \\
1 & 9 & 7 & 7 \\
2 & 3 & 2 & 1 \\
5 & 7 & 6 & 3 
\end{pmatrix}
\rightarrow
\begin{pmatrix}
11 & 4 & 7  \\
-15 & -3 & -7  \\
-1 & 4 & 0 
\end{pmatrix}
\rightarrow
\begin{pmatrix}
3 &  -1  \\
-21 & 28  
\end{pmatrix}
\rightarrow
\begin{pmatrix}
21  
\end{pmatrix}
\]
shows that the determinant of the $4 \times 4$ determinant is 21.

\medskip

Dodgson's condensation method relies on the following fact, due to Jacobi. If $A$ is an $n \times n$ matrix and  $A_{i_1, \ldots, i_k\,;\, j_1, \ldots, j_k}$ denotes the matrix $A$ with rows $i_1, \ldots, i_k$ and columns $j_1, \ldots, j_k$ removed, then
\begin{equation}\label{f.eq:Dodgson}
\det A \cdot \det A_{1,n\, ;\, 1,n} = \det A_{1\,;\,1} \cdot \det A_{n\,;\,n} - \det A_{1\,;\,n} \cdot \det A_{n\,;\,1}. 
\end{equation}
This proves that the numbers appearing in the pyramid are precisely the determinants of the ``contiguous" submatrices of $A$, consisting of consecutive rows and columns.

If we have a guess for the determinant of $A$, as well as the determinants of its contiguous submatrices, Dodgson condensation is an extremely efficient method to prove it. All we need to do is to verify that our guess satisfies (\ref{f.eq:Dodgson}).

\bigskip

To see how this works in an example, let us use Dodgson condensation to prove the formula in Section \ref{f.sec:GesselViennot} for $R_n = \det  {2n \choose n+i-j}_{1 \leq i, j \leq n}$, the number of stacks of unit cubes in the corner of an $n \times n \times n$ box. The first step is to guess the determinant of the matrix in question, as well as all its contiguous submatrices; they are all of the form $R(a,b,c) = \det {a+b \choose a+i-j}_{1 \leq i, j \leq c}$, where $a+b=2n$. This more general determinant is equally interesting combinatorially: it counts the stacks of unit cubes in the corner of an $a \times b \times c$ box. By computer experimentation, it is not too difficult to arrive at the following guess:
\[
R(a,b,c) = \det \left[ {a+b \choose a+i-j} \right]_{1 \leq i, j \leq c} = \prod_{i=1}^a \prod_{j=1}^b \prod_{k=1}^c \frac{i+j+k-1}{i+j+k-2}.
\]
Proving this formula by Dodgson condensation is then straightforward; we just need to check that our conjectural product formula holds for $c=0,1$ and that it satisfies (\ref{f.eq:Dodgson}); that is,
\[
R(a,b,c+1) R(a,b,c-1) = R(a,b,c)^2 - R(a+1,b-1,c) R(a-1,b+1,c).
\]

For more applications of Dodgson condensation, see for example \cite{f.AmdeberhanZeilberger}.
%{Cite Amdeberhan - Zeilberger: Dets through the looking glass}

\bigskip

There is a wonderful connection between Dodgson condensation, Aztec diamonds, and \textbf{alternating sign matrices}, which we now describe. Let us construct a square pyramid of numbers where levels $n+2$ and $n+1$ are given by two matrices $\mathbf{y}=(y_{ij})_{1 \leq i, j \leq n+2}$ and $\mathbf{x} = (x_{ij})_{1 \leq i, j \leq n+1}$, respectively, and levels $n-1, \ldots, 2, 1$ are computed in terms of the lower rows using Dodgson's recurrence $f=(ad-bc)/e$. Let $f_n(\mathbf{x}, \mathbf{y})$ be the entry at the top of the pyramid.

Remarkably, all the entries of the resulting pyramid will be Laurent monomials in the $x_{ij}$s and $y_{ij}$s; that is, their denominators are always monomials. This is obvious for the first few levels, but it becomes more and more surprising as we divide by more and more intricate expressions. 

The combinatorial explanation for this fact is that each entry in the $(n-k)$th level of the pyramid encodes the domino tilings of a Aztec diamond $AD_{k}$. For instance, if $n=2$, the entry at the top of the pyramid is
%\begin{eqnarray*}
%&&
%\frac{x_{11}x_{22}x_{33}}{y_{22}y_{33} }
%\, - \, \frac{x_{11}x_{22}x_{23}x_{32}}{x_{22}y_{22}y_{33}}
%\,-\, \frac{x_{12}x_{21}x_{22}x_{33}}{x_{22}y_{22}y_{33}}
%\,+\, \frac{x_{12}x_{21}x_{23}x_{32}}{x_{22}y_{22}y_{33}} \\
%&-& \frac{x_{12}x_{21}x_{23}x_{32}}{x_{22}y_{23}y_{32}}
%\,+\, \frac{x_{12}x_{22}x_{23}x_{31}}{x_{22}y_{23}y_{32}}
%\,+\, \frac{x_{13}x_{21}x_{22}x_{32}}{x_{22}y_{23}y_{32}}
%\,- \,\frac{x_{13}x_{31}x_{32}}{y_{23}y_{32}}
%\end{eqnarray*}

%\[
%f_2(\mathbf{x}, \mathbf{y}) = \frac{x_{11}x_{22}x_{33}}{y_{22}y_{33} }
%\, - \, \frac{x_{11}x_{22}x_{23}x_{32}}{x_{22}y_{22}y_{33}}
%\,-\, \frac{x_{12}x_{21}x_{22}x_{33}}{x_{22}y_{22}y_{33}}
%\,+\, \frac{x_{12}x_{21}x_{23}x_{32}}{x_{22}y_{22}y_{33}} 
%\]
%\[
%\qquad \qquad
%- \, \frac{x_{12}x_{21}x_{23}x_{32}}{x_{22}y_{23}y_{32}}
%\,+\, \frac{x_{12}x_{22}x_{23}x_{31}}{x_{22}y_{23}y_{32}}
%\,+\, \frac{x_{13}x_{21}x_{22}x_{32}}{x_{22}y_{23}y_{32}}
%\,- \,\frac{x_{13}x_{31}x_{32}}{y_{23}y_{32}}
%\]

\[
f_2(\mathbf{x}, \mathbf{y}) = \frac{x_{11}x_{22}x_{33}}{y_{22}y_{33} }
\, - \, \frac{x_{11}x_{23}x_{32}}{y_{22}y_{33}}
\,-\, \frac{x_{12}x_{21}x_{33}}{y_{22}y_{33}}
\,+\, \frac{x_{12}x_{21}x_{23}x_{32}}{x_{22}y_{22}y_{33}} 
\]
\[
\qquad \qquad
- \, \frac{x_{12}x_{21}x_{23}x_{32}}{x_{22}y_{23}y_{32}}
\,+\, \frac{x_{12}x_{23}x_{31}}{y_{23}y_{32}}
\,+\, \frac{x_{13}x_{21}x_{32}}{y_{23}y_{32}}
\,- \,\frac{x_{13}x_{31}x_{32}}{y_{23}y_{32}}.
\]

There is a simple bijection between the 8 terms of $f_2$ and the 8 domino tilings of $AD_2$, illustrated in Figure \ref{f.fig:MayanDodgson}. Regard a tiling of $AD_2$ as a graph with vertices on the underlying lattice, and add a vertical edge above and below the tiling, and a horizontal edge to the left and to the right of it. Now rotate the tiling $45^\circ$ clockwise. Record the degree of each vertex, ignoring the outside corners on the boundary of the diamond, and subtract 3 from each vertex. This leaves us with an $n \times n$ grid of integers within an $(n+1) \times (n+1)$ grid of integers. Assign to it the monomial whose $x$ exponents are given by the outer grid and whose $y$ exponents are given by the inner grid. For example, the tiling in Figure \ref{f.fig:MayanDodgson} corresponds to the monomial 
$({x_{12}x_{23}x_{31}})/({y_{23}y_{32}})$.
%$(x_{12}x_{21}x_{23}x_{32})/(x_{22}y_{23}y_{32})$. 
%$\frac{x_{12}x_{21}x_{23}x_{32}}{x_{22}y_{23}y_{32}}$. 

\begin{figure}[ht]
 \begin{center}
  \includegraphics[scale=.9]{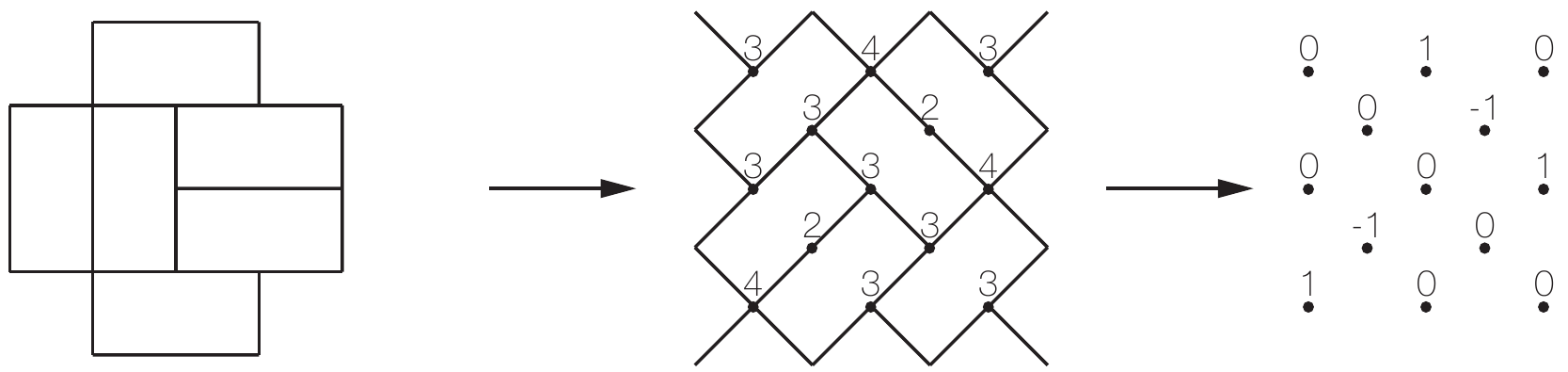}
  \caption{ \label{f.fig:MayanDodgson} 
A domino tiling of $AD_2$ and the corresponding monomial in $f_2(\mathbf{x}, \mathbf{y})$.}
 \end{center}
\end{figure}

In general, this gives a bijection between the terms of $f_n(\mathbf{x}, \mathbf{y})$ and the domino tilings of the Aztec diamond $AD_n$. One may also check that there are no cancellations, so Dodgson condensation tell us that the number $m_n$ of terms in $f_n$ satisfies $m_{n-1}m_{n+1} = 2m_n^2$. This gives an alternative proof that the Aztec diamond $AD_n$ has $2^{n(n+1)/2}$ domino tilings.

\bigskip

We may also consider the patterns formed by the $x_{ij}$s by themselves (or of the $y_{ij}$s by themselves). In each individual monomial of $f_n(\mathbf{x}, \mathbf{y})$, the exponents of the $x_{ij}$s form an $n \times n$ \textbf{alternating sign matrix} (ASM): a matrix of $1$s, $0$s, and $-1$s such that the nonzero entries in any row or column alternate $1, -1, \ldots, -1, 1$. Similarly, the negatives of the exponents of the $y_{ij}$s form an ASM of size $n-1$. 

Alternating sign matrices are fascinating objects in their own right, with connections to representation theory, statistical mechanics, and other fields. 
The number of alternating sign matrices of size $n$ is
\[
\frac{1!\, 4!\, 7! \cdots (3n-2)!}{n!\, (n+1)! \, (n+2)! \cdots (2n-1)!}.
\]
For details on the history and solution of this difficult enumeration problem see \cite{f.BressoudPropp, f.RobbinsASM, f.ZeilbergerASM}.

%D.P.Robbins,Thestoryof1,2,7,42,429,7436,...,Math.Intelligencer13no.2(1991),12?19.
%[30] D. Zeilberger, Proof of the alternating sign matrix conjecture, Electronic J. Comb. 3 (1996), R13;

\newpage

\section{\textsf{Posets}}\label{f.sec:posets}

This section is devoted to the enumerative aspects of the theory of partially ordered sets (posets). Section \ref{f.sec:posetdefs}  introduces key definitions and examples.  Section \ref{f.sec:lattices} discusses some families of lattices that are of special importance. In Section \ref{f.sec:orderpoly} we count chains and linear extensions of posets. 

The remaining sections are centered around the  \emph{M\"obius Inversion Formula}, which is perhaps the most useful enumerative tool in the theory of posets. This formula helps us count sets which have an underlying poset structure; it applies to many combinatorial settings of interest. 

In Section \ref{f.sec:inclusionexclusion} we discuss the Inclusion-Exclusion Principle, a special case of great importance. In Section \ref{f.sec:Mobius} we introduce M\"obius functions and the M\"obius Inversion Formula. In particular, we catalog the M\"obius functions of many important posets. The incidence algebra, a nice algebraic framework for understanding and working with the M\"obius function, is discussed in Section \ref{f.sec:incidence}. In Section \ref{f.sec:computingMobius} we discuss methods for computing M\"obius functions of posets, and sketch proofs for the posets of Section \ref{f.sec:Mobius}. Finally, in Section \ref{f.sec:Eulerian}, we discuss Eulerian posets and the enumeration of their flags, which gives rise to the ${\mathbf{a}}{\mathbf{b}}$-index and ${\mathbf{c}}{\mathbf{d}}$-index.

\subsection{\textsf{Basic definitions and examples}}\label{f.sec:posetdefs}

A \textbf{partially ordered set} or \textbf{poset} $(P, \leq)$ is a set $P$ together with a binary relation $\leq$, called a \emph{partial order}, such that 
\begin{itemize}
\item
For all $p \in P$, we have $p \leq p$.
\item
For all $p,q \in P$, if $p \leq q$ and $q \leq p$ then $p=q$.
\item
For all $p,q,r \in P$, if $p \leq q$ and $q \leq r$ then $p \leq r$.
\end{itemize}
We say that $p<q$ if $p \leq q$ and $p \neq q$. We say that $p$ and $q$ are \textbf{comparable} if $p<q$ or $p>q$, and they are \textbf{incomparable} otherwise. We say that $q$ \textbf{covers} $p$  if $q>p$ and there is no $r \in P$ such that $q>r>p$. When $q$ covers $p$ we write $q \gtrdot p$.

\begin{example}\label{f.ex:posets}
Many sets in combinatorics come with a natural partial order, and often the resulting poset structure is very useful for enumerative purposes. Some of the most important examples are the following:
\begin{enumerate}
\item \emph{(Chain)}
The poset $\mathbf{n} = \{1,2,\ldots, n\}$ with the usual total order. ($n \geq 1$)
\item \emph{(Boolean lattice)}
The poset $2^{A}$ of subsets of a set $A$, where $S \leq T$ if $S \subseteq T$. 
\item \emph{(Divisor lattice)}
The poset $D_n$ of divisors of $n$, where $c \leq d$ if $c$ divides $d$. ($n \geq 1$)
\item \emph{(Young's lattice)}
The poset $Y$ of integer partitions, where $\lambda \leq \mu$ if $\lambda_i \leq \mu_i$ for all $i$. 
\item \emph{(Partition lattice)}
The poset $\Pi_n$ of set partitions of $[n]$, where $\pi \leq \rho$ if $\pi$ \emph{refines} $\rho$; that is, if every block of $\rho$ is a union of blocks of $\pi$. ($n \geq 1$)
\item \emph{(Non-crossing partition lattice)}
The subposet $NC_n$ of $\Pi_n$ consisting of the non-crossing set partitions of $[n]$, where there are no elements  $a<b<c<d$ such that $a,c$ are together in one block and $b,d$ are together in a different block. ($n \geq 1$)
\item \emph{(Bruhat order on permutations)}
The poset $S_n$ of permutations of $[n]$, where $\pi$ covers $\rho$ if $\pi$ is obtained from $\rho$ by choosing two adjacent numbers $\rho_i = a < b = \rho_{i+1}$ in $\rho$  and exchanging their positions.  ($n \geq 1$)
\item \emph{(Subspace lattice)}
The poset $L({\mathbb{F}}_q^n)$ of subspaces of a finite dimensional vector space ${\mathbb{F}}_q^n$, where $U \leq V$ if $U$ is a subspace of $V$.  ($n \geq 1$, $q$ a prime power)
\item \emph{(Distributive lattice)}
The poset $J(P)$ of order ideals of a poset $P$ (subsets $I \subseteq P$ such that $j \in P$ and $i < j$ imply $i \in P$) ordered by containment.
\item \emph{(Face poset of a polytope)}
The poset $F(P)$ of faces of a polytope $P$, ordered by inclusion.
\item \emph{(Face poset of a subdivision of a polytope)}
The poset $\widehat{{\mathcal{T}}}$ of faces of a subdivision ${\mathcal{T}}$ of a polytope $P$ ordered by inclusion, with an additional maximum element.
\item \emph{(Subgroup lattice of a group)}
The poset $L(G)$ of subgroups of a group $G$, ordered by containment.
%\comment{
%\item Tamari?}
\end{enumerate}
\end{example}

The \textbf{Hasse diagram} of a finite poset $P$ is obtained by drawing a dot for each element of $P$ and an edge going down from $p$ to $q$ if $p$ covers $q$. Figure \ref{f.fig:posets} below shows the Hasse diagrams of some of the posets above. In particular, the Hasse diagram of $2^{[n]}$ is the 1-skeleton of the $n$-dimensional cube. %Figure \ref{f.fig:posets} shows some examples. 

\begin{figure}[ht]
 \begin{center}
  \includegraphics[scale=1.4]{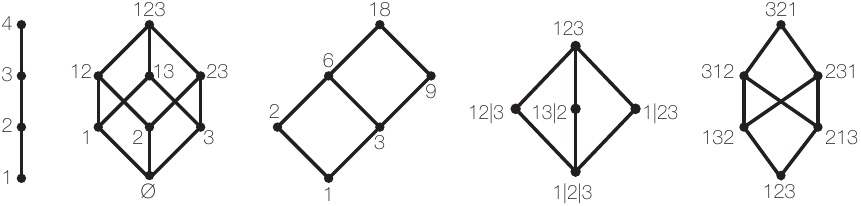}
  \caption{ \label{f.fig:posets} 
The Hasse diagrams of the chain $\mathbf{4}$, Boolean lattice $2^{[3]}$, divisor lattice $D_{18}$, partition lattice $\Pi_3$, and Bruhat order $S_3$.}
 \end{center}
\end{figure}

A subset $Q$ of $P$ is a \textbf{chain} if every pair of elements is comparable, and it is an \textbf{antichain} if every pair of elements is incomparable. The length of a chain $C$ is $|C|-1$. If there is a \textbf{rank function} $r: P \rightarrow {\mathbb{N}}$ such that $r(x) = 0$ for any minimal element $x$ and $r(y) = r(x)+1$ whenever $y \gtrdot x$, then  $P$ is called \textbf{graded} or \textbf{ranked}. The largest rank is called the \textbf{rank} or \textbf{height} of $P$. 
The \textbf{rank-generating function} of a finite graded poset is
\[
R(P ; x) = \sum_{p \in P} x^{r(p)}
\]
All the posets of Example \ref{f.ex:posets} are graded except for subgroup lattices.
%; we have $r_{\textbf{n}}(m) = m-1$,\, $r_{2^{A}}(S) = |S|$,\, $r_{D_n}(d)$ equals the number of (possibly repeated) prime factors of $d$,\, $r_Y(\lambda) = |\lambda|$,\,  $r_{\Pi_n}(\pi) = n - |\pi|$, \,$r_{S_n}(\pi) = \textrm{inv}(\pi)$, and $r_{L(W)}(V) = \dim V$. From these descriptions, in view of the results in Section \ref{f.sec:basic}, it is easy to write down the rank generating functions of these posets. 

A poset $P$ induces a poset structure on any subset $Q \subseteq P$; a special case of interest is the \textbf{interval} $[p,q] = \{r \in P \, : \, p \leq r \leq q\}$. We call a poset \textbf{locally finite} if all its intervals are finite.
 
Given posets $P$ and $Q$ on disjoint sets, the \textbf{direct sum} $P+Q$ is the poset on $P \cup Q$ inheriting the order relations from $P$ and $Q$, and containing no additional order relations between elements of $P$ and $Q$. The \textbf{direct product} $P \times Q$ is the poset on $P \times Q$ where $(p,q) \leq (p', q')$ if $p \leq p'$ and $q \leq q'$.

We have already seen examples of product posets. The Boolean lattice is $2^{A} \cong \mathbf{2} \times \cdots \times \mathbf{2}$. Also, if $n = p_1^{t_1} \cdots p_k^{t_k}$ is the prime factorization of $n$, then $D_n \cong \mathbf{(t_1 + 1)} \times \cdots \times \mathbf{(t_k+1)}$.

\subsection{\textsf{Lattices}}\label{f.sec:lattices}

A poset is a \textbf{lattice} if every two elements $p$ and $q$ have a least upper bound $p \vee q$ and a greatest lower bound $p \wedge q$, called their \textbf{meet} and  \textbf{join}, respectively. We will see this additional algebraic structure can be quite beneficial for enumerative purposes. 

\begin{example}
All the posets in Example \ref{f.ex:posets} are lattices, except for the Bruhat order. In most cases, the meet and join have easy descriptions. In $\mathbf{n}$, the meet and join are the minimum and maximum, respectively. In $2^{A}$ they are the intersection and union. In $D_n$ they are the greatest common divisor and least common multiple. In $Y$ they are the componentwise minimum and maximum. In $\Pi_n$ and in $NC_n$ the meet of two partitions $\pi$ and $\rho$ is the collection of intersections of a block of $\pi$ and a block of $\rho$.
 In $L({\mathbb{F}}_q^n)$ the meet and join are the intersection and the span. In $J(P)$ they are the intersection and the union. In $F(P)$ the meet is the intersection. In $L(G)$ the meet is the intersection.
\end{example}

Any lattice must have a unique minimum element ${\widehat{0}}$ and maximum element ${\widehat{1}}$. An element covering ${\widehat{0}}$ is called an \textbf{atom}; an element covered by ${\widehat{1}}$ is called a \textbf{coatom}.
To prove that a finite poset $P$ is a lattice, it is sufficient to check that it has a ${\widehat{1}}$ and that any $x,y \in P$ have a meet; then the join of $x$ and $y$ will be the (necessarily non-empty) meet of their common upper bounds. Similarly, it suffices to check that $P$ has a ${\widehat{0}}$ and that any $x,y \in P$ have a join.

\bigskip

\noindent
\textsf{\textbf{Distributive Lattices.}}  A lattice $L$ is \textbf{distributive} if the join and meet operations satisfy the distributive properties:
\begin{equation}\label{f.e:distrib}
x \vee (y \wedge z) = (x \vee y) \wedge (x \vee z), \qquad 
x \wedge (y \vee z) = (x \wedge y) \vee (x \wedge z) 
\end{equation}
for all $x,y,z \in L$. To prove that $L$ is distributive, it is sufficient to check that to verify that one of the equations in (\ref{f.e:distrib}) holds for all $x,y,z \in L$.

\begin{example}\label{f.ex:distrib}
There are several distributive lattices in Example \ref{f.ex:posets}: the chains $\mathbf{n}$, the Boolean lattices $2^{A}$, the divisor lattices $D_n$, and Young's lattice $Y$. This follows from the fact that the pairs of operations $(\min, \max)$, $(\gcd, {\mathrm{lcm}})$ and $(\cap, \cup)$ satisfy the distributive laws. The others are not necessarily distributive; for example, $\Pi_3$ and $S_3$.
\end{example}

The most important -- and in fact, the only -- source of finite distributive lattices is the construction of Example \ref{f.ex:posets}.9: Given a poset $P$, a \textbf{downset} or \textbf{order ideal} $I$ is a subset of $P$ such that if $i \in I$ and $j < i$ then $j \in I$. A \textbf{principal} order ideal is one of the form $P_{\leq p} = \{q \in P \, : \, q \leq p\}$. Let $J(P)$ be the \textbf{poset of order ideals} of $P$, ordered by inclusion. 

\begin{theorem}\label{f.th:Birkhoff} \emph{(Fundamental Theorem for Finite Distributive Lattices.)}
A poset $L$ is a distributive lattice if and only if there exists a poset $P$ such that $L \cong J(P)$.
\end{theorem}

\begin{proof}[Sketch of Proof.] 
Since the collection of order ideals of a poset $P$ is closed under union and intersection, $J(P)$ is a sublattice of $2^P$. The distributivity of $2^P$ then implies that $J(P)$ is a distributive lattice.

For the converse, let $L$ be a distributive lattice, and let $P$ be the set of join-irreducible elements of $L$;  that is, the elements $p > {\widehat{0}}$ which cannot be written as $p = q \vee r$ for $q,r < p$.  These are precisely the elements of $L$ that cover exactly one element. The set $P$ inherits a partial order from $L$, and this is the poset such that $L \cong J(P)$. The isomorphism is given by
\begin{eqnarray*}
\phi: J(P) & \longrightarrow & L \\
I & \longmapsto & \bigvee_{p \in I} p
\end{eqnarray*}
and the inverse map is given by $\phi^{-1}(l) =  \{p \in P \, : \, p \leq l\}$.
\end{proof}

\begin{figure}[ht]
 \begin{center}
  \includegraphics[scale=.6]{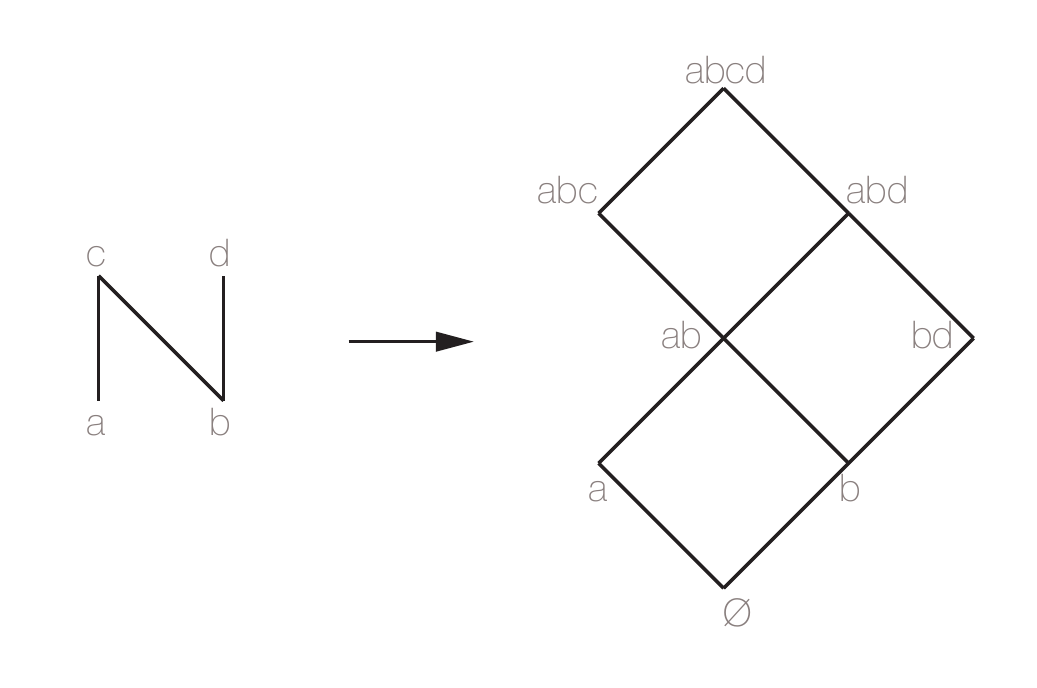}
  \caption{ \label{f.fig:distrib} 
A poset and the corresponding distributive lattice.}
 \end{center}
\end{figure}

Theorem \ref{f.th:Birkhoff} extends to some infinite posets with minor modifications. Let $J_f(P)$ be the set of finite order ideals of a poset $P$. %Say that a distributive lattice $L$ is \textbf{locally finite} if it has a ${\widehat{0}}$ and every interval in $L$ is finite. 
Then the map $P \mapsto J_f(P)$ is a bijection between the posets whose principal order ideals are finite and the locally finite distributive lattices with ${\widehat{0}}$.

\begin{example}
The posets $P$ of join-irreducibles of the distributive lattices $L \cong J(P)$ of Example \ref{f.ex:distrib} are as follows. For $L=\mathbf{n}$, $P=\mathbf{n-1}$ is a chain. For $L = 2^{A}$, $P=\mathbf{1} + \cdots + \mathbf{1}$ is an antichain. For $L=D_n$, where $n = p_1^{t_1} \cdots p_k^{t_k}$, $P=\mathbf{t_1}+ \cdots + \mathbf{t_k}$ is the disjoint sum of $k$ chains. For $L=Y$, $P={\mathbb{N}} \times {\mathbb{N}}$ is a ``quadrant".
\end{example}

Theorem \ref{f.th:Birkhoff} explains the abundance of cubes in the Hasse diagram of a distributive lattice $L$. For any element $l \in L$ covered by $n$ elements $l_1, \ldots, l_n$ of $L$, the joins of the $2^n$ subsets of $\{l_1, \ldots, l_n\}$ are distinct, and form a copy of the Boolean lattice $2^{[n]}$ inside $L$. The dual result holds as well.

The \textbf{width} of a poset $P$ is the size of the largest antichain of $P$. \textbf{Dilworth's Theorem} \cite{f.Dilworth} states that this is the smallest integer $w$ such that $P$ can be written as a disjoint union of $w$ chains.

\begin{theorem}
The distributive lattice $J(P)$ can be embedded as an induced subposet of the poset ${\mathbb{N}}^w$, where $w$ is the width of $P$.
\end{theorem}

\begin{proof}
Decompose $P$ as the disjoint union of $w$ chains $C_1, \ldots, C_c$. The map
\begin{eqnarray*}
\phi: J(P) & \longrightarrow & {\mathbb{N}}^w \\
I & \longmapsto & (|I \cap C_1|, \ldots, |I \cap C_w|)
\end{eqnarray*}
gives the desired inclusion.
\end{proof}

\bigskip

\noindent
\textsf{\textbf{Geometric Lattices.}}  Now we introduce another family of lattices of great importance in combinatorics. 
We say that a lattice $L$ is:

 $\bullet$ \textbf{semimodular} if the following two equivalent conditions hold:

\qquad o $L$ is graded and $r(p) + r(q) \geq r(p \wedge q) + r(p \vee q)$ for all $p, q \in L$.

\qquad o  If $p$ and $q$ both cover $p \wedge q$, then $p \vee q$ covers both $p$ and $q$.

$\bullet$ \textbf{atomic} if every element is a join of atoms. 

$\bullet$ \textbf{geometric} if it is semimodular and atomic.

\begin{example}\label{f.ex:geometric}
In Figure \ref{f.fig:posets}, the posets $2^{[3]}$ and $\Pi_3$ are geometric lattices, while the posets $\mathbf{4}$, $D_{18}$, and $S_3$ are not.
\end{example}

Not surprisingly, the prototypical example of a geometric lattice comes from a natural geometric construction, illustrated in Figure \ref{f.fig:points}. 
%There are three equivalent models.
%\noindent $\bullet$ 
Let $A=\{v_1, \ldots, v_n\}$ be a set of vectors in a vector space $V$. A \textbf{flat} is an subspace of $V$ generated by a subset of $A$. We identify a flat with the set of $v_i$s that it contains. 
Let $L_A$ be the set of flats of $A$, ordered by inclusion. %(In the third model, $F \leq G$ if $F \subseteq G$ as subsets of $A$, or if $F \supseteq G$ geometrically.) 
Then $L_A$ is a geometric lattice.

\begin{figure}[ht]
 \begin{center}
  \includegraphics[scale=.8]{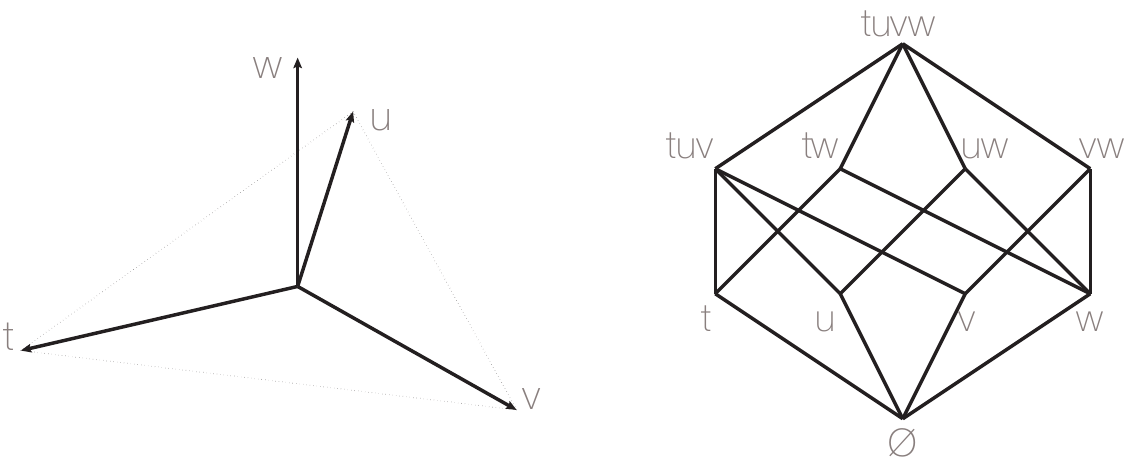}
  \caption{ \label{f.fig:points} 
A vector configuration ${\mathbf{t}},{\mathbf{u}},{\mathbf{v}},{\mathbf{w}}$ (where ${\mathbf{t}},{\mathbf{u}},{\mathbf{v}}$ are coplanar in $\mathbb{R}^3$) and the corresponding geometric lattice.}
 \end{center}
\end{figure}

Geometric lattices are essentially equivalent to \emph{matroids}, which are the subject of Section \ref{f.sec:matroids}.

%
%\noindent $\bullet$ Let $A=\{v_1, \ldots, v_n\}$ be a set of vectors in a vector space $V$. A {\mathbf{k}}extbf{flat} is a subspace of $V$ spanned by a subset of $A$. We identify a flat with the set of $v_i$s that it contains. 
%
%\noindent $\bullet$ Let $A=\{H_1, \ldots, H_n\}$ be a set of hyperplanes (codimension 1 subspaces) in a vector space $V$. A \textbf{flat} is a subspace obtained by intersecting some of the hyperplanes in $A$. We identify a flat with the set of $H_i$s that contain it. 

%\comment{
%Almost every geometric lattice that arises ``in nature" comes from this geometric construction.
%Geometric lattices are essentially equivalent to \textbf{matroids}, which we will study in depth in Section \ref{f.sec:matroids}.
%}

\bigskip

\noindent
\textsf{\textbf{Supersolvable Lattices.}} A lattice $L$ is \textbf{supersolvable} if there exists a maximal chain, called an M-chain, such that the sublattice generated by $C$ and any other chain of $L$ is distributive. \cite{f.Stanleysupersolvable}

Again unsurprisingly, an important example comes from supersolvable groups, but there are  other interesting examples. Here is a list of supersolvable lattices, and an M-chain in each case.

\begin{enumerate}

\item Distributive lattices: every maximal chain is an M-chain. 

\item Partition lattice $\Pi_n$:  $1|2|\cdots|n < 12|3|\cdots |n < 123|4|\cdots|n < \cdots < 123\cdots n$ 

\item Noncrossing partition lattice $NC_n$: the same chain as above.

\item Lattice of subspaces $L({\mathbb{F}}_q^n)$ of the vector space ${\mathbb{F}}_q^n$ over a finite field ${\mathbb{F}}_q$: every maximal chain is an M-chain.

\item Subgroup lattices of finite supersolvable groups $G$: an M-chain is given by any normal series ${1} = H_0 \triangleleft H_1 \triangleleft \cdots \triangleleft H_k=G$ where each $H_i$ is normal and $H_i/H_{i-1}$ is cyclic of prime order.

\end{enumerate}
%
%\noindent 1. Distributive lattices. Every chain is an M-chain. 
%
%\noindent 2. Partition lattice $\Pi_n$:  $1|2|\cdots|n < 12|3|\cdots |n < 123|4|\cdots|n < \cdots < 123\cdots n$ 
%
%\noindent 3. Noncrossing partition lattice $NC_n$: the same chain as above.
%
%\noindent 4. The subgroup lattice of a finite supersolvable group $G$. An M-chain is given by any normal series ${1} = H_0 \triangleleft H_1 \triangleleft \cdots \triangleleft H_k=G$ where each $H_i$ is normal and $H_i/H_{i-1}$ is cyclic of prime order.

Fortunately, there is a simple criterion to verify semimodularity. 
An \textbf{$R$-labeling} of a poset $P$ is a labeling of the edges of the Hasse diagram of $P$ with integers such that for any $s \leq t$ there exists a unique maximal chain $C$ from $s$ to $t$.

\begin{theorem} \cite{f.Mcnamara} A finite graded lattice of rank $n$ is supersolvable if and only if it has an $R$-labeling for which the labels on every maximal chain are a permutation of $\{1, \ldots, n\}$.
\end{theorem}

% EL-labelings, supersolvability and 0-Hecke algebra actions on posets

\subsection{\textsf{Zeta polynomials and order polynomials}}\label{f.sec:orderpoly}

The \textbf{zeta polynomial} of a finite poset $P$ counts the multichains of various lengths in $P$. 
A \textbf{multichain} of length $k$ in $P$ is a sequence of possibly repeated elements $t_0, t_1, \ldots, t_k \in P$ such that $t_0 \leq t_1 \leq \cdots \leq t_k$. Let
\begin{equation}\label{f.e:zeta}
Z_P(k) = \textrm{number of multichains of length $k-2$ in $P$} \qquad 
(k \geq 2). 
\end{equation}
There is a unique polynomial $Z_P(k)$ satisfying (\ref{f.e:zeta}) for all integers $k \geq 2$; it is given by
\begin{equation}\label{f.e:Z}
Z_P(k) = \sum_{i \geq 2} b_i {k-2 \choose i-2},
\end{equation}
where $b_i$ is the number of chains of length $i-2$ in $P$. This polynomial is called the \textbf{zeta polynomial} of $P$.

\begin{example}\label{f.ex:zeta} The following posets have particularly nice zeta polynomials:
\begin{enumerate}
\item $P=\mathbf{n}$:
\[
Z(k) = {n+k-2 \choose n-1}
\]
\item $P=B_n$:
\[
Z(k) = k^n
\]
\item $P=NC_n$: (Kreweras, \cite{f.Kreweras})
\[
Z(k) = \frac1n {kn \choose n-1}
\]
\end{enumerate}
\end{example}

The \textbf{order polynomial} of $P$ counts the order-preserving labelings of $P$; it is defined by 
\[
\Omega_P(k) = \textrm{number of maps $f: P \rightarrow [k]$ such that $p<q$ implies $f(p) \leq f(q)$} 
\]
for $k \in {\mathbb{N}}$. The next proposition shows that, once again, there is a unique polynomial taking these values at the natural numbers.
\begin{proposition}\label{f.p:OZ}
For any poset $P$, $\Omega_P(k) = Z_{J(P)}(k)$.
\end{proposition}
\begin{proof}
An order-preserving map $f: P \rightarrow [k]$ gives rise to a sequence of order ideals $f^{-1}(\{1\}) \subseteq f^{-1}(\{1,2\}) \subseteq \cdots f^{-1}(\{1, \ldots, k\})$, which is a multichain in $J(P)$. Conversely, every sequence arises uniquely in this way.
\end{proof}

A \textbf{linear extension} of $P$ is an order-preserving labeling of the elements of $P$ with the labels $1, \ldots, n = |P|$ which extends the order of $P$; that is, a bijection $f: P \rightarrow [n]$ such that $p<q$ implies $f(p)<f(q)$. Let
\[
e(P) = \textrm{ number of linear extensions of } P.
\]
It follows from Proposition \ref{f.p:OZ} and (\ref{f.e:Z}) that the order polynomial $\Omega_P$ has degree $|P|$, and leading coefficient $e(P)/|P|!$.

The following is a method for computing $e(P)$ recursively. 

\begin{proposition}\label{f.prop:linext}
Define $e: J(P) \rightarrow {\mathbb{N}}$ recursively by
\[
e(I) = \begin{cases}
1 & \textrm{ if } I = {\widehat{0}}, \\
\displaystyle \sum_{J \lessdot I} e(J) & \textrm{ otherwise}.
\end{cases}
\]
Then $e({\widehat{1}})$ is the number $e(P)$ of linear extensions of $P$.
\end{proposition}

\begin{proof}
Let $p_1, \ldots, p_k$ be the maximal elements of $P$. In a linear extension of $P$,  one of $p_1, \ldots, p_k$ has to be labelled $n$, and therefore
\[
e(P) = e(P \backslash \{p_1\}) + \cdots + e(P \backslash \{p_k\}).
\]
This is equivalent to the desired recurrence.
\end{proof}

It is useful to keep in mind that $J(P)$ is a subposet of ${\mathbb{N}}^w$ for $w=w(P)$. The recurrence of Proposition \ref{f.prop:linext} generalizes Pascal's triangle, which corresponds to the case $P = {\mathbb{N}} + {\mathbb{N}}$. When we apply it to $P={\mathbb{N}} + \cdots + {\mathbb{N}}$, we get the recursive formula for multinomial coefficients.

\begin{example} 
In some special cases, the problem of enumerating linear extensions is of fundamental importance.
\begin{itemize} 
\item
$e(\mathbf{n}_1 + \cdots + \mathbf{n}_k) = {n_1 + \cdots + n_k \choose n_1, \ldots, n_k}$
\item
$e(\mathbf{2} \times \mathbf{n}) = C_n = \frac1{n+1}{2n \choose n}$.
\item
Let $T$ be a \emph{tree poset} of $n$ elements, such that the Hasse diagram is a tree rooted at ${\widehat{0}}$. For each vertex $v$ let $t_v = |T_{\geq v}| = |\{w \in T \, : \, w \geq v\}|$. Then
\[
e(T) = \frac{n!}{\prod_{v \in T} t_v}.
\]
\item
Let $\lambda$ be a Ferrers diagram of $n$ cells, partially ordered by decreeing that each cell is covered by the cell directly below and the cell directly to the right, if they are in $\lambda$. The \textbf{hook} $H_c$ of a cell $c$ consists of cells on the same row and to the right of $c$, those on the same column and below $c$, and $c$ itself. Let $h_c = |H_c|$.  
Then
\[
e(\lambda) = \frac{n!}{\prod_{c \in D} h_c}.
\]
This is the dimension of the irreducible representation of the symmetric group $S_n$ corresponding to $\lambda$. \cite{f.Sagan}
\end{itemize}
\end{example}

%\comment{Include connection with P-partitions. Also include shifted Ferrers, Selberg poset, Selberg book, Proctor's D-complete posets, zig zag poset, root poset. Also, what is the maximum $e(P)$? The poset $P_w$ of a Rothe diagram?}

%
%\subsection{\textsf{Sieves}}
%
%Sieve methods count the elements of a set by starting with a very rough estimate, and  obtaining successively better approximations by adding or subtracting to that estimate. 

\subsection{\textsf{The Inclusion-Exclusion Formula}}\label{f.sec:inclusionexclusion}

Our next goal is to discuss one of the most useful enumerative tools for posets: 
M\"obius functions and the M\"obius inversion theorem. Before we do that, we devote this section to a special case that preceded and motivated them: the inclusion-exclusion formula.
%
%The inclusion-exclusion formula is a simple but fundamental tool which allows us to count the size of a union or an intersection of finite sets. In practice, we usually have a set $S$ of objects, and $n$ properties that each object in $S$ may or may not satisfy. We are interested in counting the elements that satisfy at least one of the given properties; or equivalently, count the elements that satisfy none of the given properties. If we let $A_i$ be the set of elements in $X$ having the $i$th property, then the answer is given by the following formula:

\begin{theorem} (Inclusion-Exclusion Principle)
For any finite sets $A_1, \ldots, A_n \subseteq X$, we have
\begin{enumerate}
\item
$\displaystyle |A_1 \cup \cdots \cup A_n| = \sum_i |A_i| - \sum_{i < j} |A_i \cap A_j| +  \sum_{i<j<k} |A_i \cap A_j \cap A_k|  - \cdots \pm |A_1 \cap \cdots \cap A_n|$.
\item
$\displaystyle |\overline{A_1} \cap \cdots \cap \overline{A_n}| = |X| -  \sum_i |A_i| + \sum_{i < j} |A_i \cap A_j| - \cdots \pm |A_1 \cap \cdots \cap A_n|.$
\end{enumerate}

%\begin{eqnarray*}
%|A_1 \cup \cdots \cup A_n| &=& \sum_i |A_i| - \sum_{i < j} |A_i \cap A_j| +  \sum_{i<j<k} |A_i \cap A_j \cap A_k|  - \cdots \pm |A_1 \cap \cdots \cap A_n|.\\
%|\overline{A_1} \cap \cdots \cap \overline{A_n}| &=& |X| -  \sum_i |A_i| + \sum_{i < j} |A_i \cap A_j| - \cdots \pm |A_1 \cap \cdots \cap A_n|.
%\end{eqnarray*}
\end{theorem}

\begin{proof} It suffices to prove one of these two equivalent equations. To prove the first one, consider an element $x$ appearing in $k \geq 1$ of the given sets. The number of times that $x$ is counted in the right hand side is $k - {k \choose 2} + \cdots \pm {k \choose k} = 1$. 
\end{proof}

%We often think of the inclusion-exclusion principle as a sequence of successive estimates for $N=|A_1 \cup \cdots \cup A_n|$. A first estimate $N_1$ for $N$ is the sum of the sizes of the $A_i$s. This is an overestimate; a second estimate $N_2$ is obtained by subtracting the sizes of the double intersections $A_i \cap A_j$, which are counted more than once in the first estimate. Now this is an underestimate, and a third estimate $N_3$ is obtained by adding back the sizes of the triple intersections $A_i \cap A_j \cap A_k|$, which are undercounted in $N_2$. We repeat this process successively until we obtain the correct answer.

We now present a slightly more general formulation. 
\begin{theorem} \label{f.th:IE} (Inclusion-Exclusion Principle)
Let $A$ be a set and consider two functions $f_=, f_\geq: 2^A \longrightarrow {\mathbbm{k}}$ be functions from $2^A$ to a field ${\mathbbm{k}}$. Then
\begin{enumerate}
\item
$\displaystyle f_\geq(S) = \sum_{T \supseteq S} f_=(T) \,\, \textrm{ for all } S \subseteq A 
%\,\, \textrm{ if and only if }  \,\, 
\quad	\Longleftrightarrow \quad
f_=(S) = \sum_{T \supseteq S} (-1)^{|T-S|} f_\geq(T) \,\, \textrm{ for all  } S \subseteq A$.
\item
$\displaystyle f_\leq(S) = \sum_{T \subseteq S} f_=(T) \,\, \textrm{ for all  } S \subseteq A  
%\,\, \textrm{ if and only if }  \,\, 
\quad	\Longleftrightarrow \quad
f_=(S) = \sum_{T \subseteq S} (-1)^{|S-T|} f_\leq(T) \,\, \textrm{ for all  } S \subseteq A$.
\end{enumerate}

%\begin{eqnarray*}
%f_\geq(S) = \sum_{T \supseteq S} f_=(T) \,\, \textrm{ for  } S \subseteq A \quad &\textrm{ if and only if }& \quad  \textsf{\textbf{
%f_=(S) = \sum_{T \supseteq S} (-1)^{|T-S|} f_\geq(T) \,\, \textrm{ for  } S \subseteq A, \\
%f_\leq(S) = \sum_{T \subseteq S} f_=(T) \,\, \textrm{ for  } S \subseteq A \quad &\textrm{ if and only if }& \quad 
%f_=(S) = \sum_{T \subseteq S} (-1)^{|S-T|} f_\leq(T) \,\, \textrm{ for  } S \subseteq A.
%\end{eqnarray*}
\end{theorem}

The most common interpretation is the following. Suppose we have a set $U$ of objects and a set $A$ of properties that each object in $U$ may or may not satisfy. If we know, for each $S \subseteq A$, the number $f_\geq(S)$ of elements having \emph{at least} the properties in $S$ (or the number $f_\leq(S)$ of elements having \emph{at most} the properties in $S$), then we obtain, for each $S \subseteq A$, the number $f_=(S)$ of elements having \emph{exactly} the properties in $S$. We are often interested in the number $f_=(\emptyset)$ or $f_=(A)$ of elements satisfying none or all of the given properties.

\medskip

Theorem \ref{f.th:IE} has a simple linear algebraic interpretation. Consider the two $2^A \times 2^A$ matrices $C,D$ whose non-zero entries are $C_{S,T} = 1$ for $S \subseteq T$, and $D_{S,T} = (-1)^{|T-S|}$ for $S \subseteq T$. Then the inclusion-exclusion formula is equivalent to the assertion that $C$ and $D$ are inverse matrices. This can be proved directly, but we prefer to deduce it as a special case of the M\"obius inversion formula (Theorem \ref{f.th:Mobiusinv}). We now present two applications.

\bigskip
\noindent \textsf{\textbf{Derangements}}. One of the classic applications of inclusion-exclusion principle is the enumeration of the \textbf{derangements} of $[n]$, discussed in Example 8 of Section \ref{f.sec:operationsegfs}. These are the permutations $\pi \in S_n$ such that $\pi(i) \neq i$ for all $i$. Let $A=\{A_1, \ldots, A_n\}$ where $A_i$ is the property that $\pi(i) = i$. Then $f_\geq(T) = (n-|T|)!$, so the number $D_n$ of derangements of $[n]$ is 
\[
D_n = f_=(\emptyset) = \sum_T (-1)^{|T|} f_\geq(T) =  \sum_{k=0}^n {n \choose k} (-1)^k (n-k)! = n! \left(\frac1{0!} - \frac1{1!} +  \frac1{2!} - \cdots \pm  \frac1{n!}\right).
\]
It follows that $D_n$ is the integer closest to $n!/e$.

\bigskip
\noindent \textsf{\textbf{Discrete derivatives}}. Consider the ${\mathbbm{k}}$-vector space $\Gamma$ of functions $f: {\mathbb{Z}} \rightarrow {\mathbbm{k}}$. The \textbf{discrete derivative} of $f$ is the function $\Delta f$ given by $\Delta f (n) = f(n+1) - f(n)$. We now wish to show that, just as with ordinary derivatives,
\[
\Delta^{d+1} f = 0 \textrm{ if and only if $f$ is a polynomial of degree at most $d$}.
\] 
Recall that this was part of Theorem \ref{f.th:polynomial}.
Regarding $\Delta$ as a linear operator on $\Gamma$, we have $\Delta = E - 1$ where $Ef(n) = f(n+1)$ and $1$ is the identity. Then $\Delta^k = (E-1)^k = \sum_{i=0}^k {k \choose i} E^i (-1)^{k-i}$, so the $k$-th discrete derivative is
$
\Delta^k f(n) = \sum_{i=0}^k (-1)^{k-i}{k \choose i} f(n+i).
$

The functions $f_\leq(S) = f(n+|S|)$ and $f_=(S)=\Delta^kf(|S|)$ satisfy Theorem \ref{f.th:IE}.2, so we have $f(n+k) = \sum_{i=0}^k {k \choose i} \Delta^i f(n)$. (This is equivalent to $E^k = (\Delta + 1)^k$.)
%We can invert this formula using inclusion-exclusion or observing that $E^k = (\Delta+1)^k$, to obtain $f(n+k) = \sum_{i=0}^k {k \choose i} \Delta^i f(n)$. 
If $\Delta^{d+1} f = 0$, this gives 
$
f(k) = \sum_{i=0}^{d} {k \choose i} \Delta^i f(0)
$
which is a polynomial in $k$ of degree at most $d$. The converse follows from the observation that $\Delta$ lowers the degree of a polynomial by $1$.

%\bigskip
%\noindent \textsf{\textbf{An algebraic version}}. Let 
%\[
%0 \rightarrow V_n \rightarrow V_{n-1} \rightarrow \cdots \rightarrow V_0 \rightarrow W \rightarrow 0
%\]
%be an exact sequence of vector spaces over some field; that is, 
%Then
%\[
%\dim W = \dim V_0 - \dim V_1 + \dim V_2 - \cdots \pm \dim V_n
%\]
%
%\comment{Does this actually have to do with inclusion-exclusion?}

\subsection{\textsf{M\"obius functions and M\"obius inversion}}\label{f.sec:Mobius}

\subsubsection{{\textsf{The M\"obius function}}}
Given a locally finite poset $P$, let ${\mathrm{Int}}(P) = \{[x,y] \, : \, x, y \in P, \, x \leq y\}$ be the set of intervals of $P$. The (two-variable) \textbf{M\"obius function} of a poset $P$ is the function $\mu: {\mathrm{Int}}(P) \rightarrow {\mathbb{Z}}$ defined by
\begin{equation}\label{f.e:Mobius1}
\sum_{p \leq r \leq q} \mu(p,r) = \begin{cases}
1 & \textrm{ if } p=q, \\
0 & \textrm{ otherwise}.
\end{cases}
\end{equation}
Here we are denoting $\mu(p,q) = \mu([p,q])$. We will later see that the M\"obius function can be defined equivalently by the equations:
\begin{equation}\label{f.e:Mobius2}
\sum_{p \leq r \leq q} \mu(r,q) = \begin{cases}
1 & \textrm{ if } p=q, \\
0 & \textrm{ otherwise}.
\end{cases}
\end{equation}

When $P$ has a minimum element ${\widehat{0}}$, the (one-variable) \textbf{M\"obius function} $\mu: P \rightarrow {\mathbb{Z}}$ is $\mu(x) = \mu({\widehat{0}}, x)$. If $P$ also has a ${\widehat{1}}$,  the \textbf{M\"obius number} of $P$ is $\mu(P) = \mu({\widehat{0}},{\widehat{1}})$. %Notice that $\mu_P(p,q) = \mu([p,q])$.

Computing the M\"obius function is a very important problem, because the M\"obius function is the poset analog of a derivative; and as such, it is a fundamental invariant of a poset. This problem often leads to very interesting enumerative combinatorics, as can be gleaned from the following gallery of M\"obius functions.

\begin{theorem}\label{f.th:Mobiusformulas}
The M\"obius functions of some key posets are as follows. 
\begin{enumerate}

\item (Chain) $P = \mathbf{n}$: %One verifies directly that 
\[
\mu_{\mathbf{n}}(i,j) = \begin{cases}
1 & \textrm{if } j=i\\
-1 & \textrm{if } j=i+1\\
0 & \textrm{otherwise.}
\end{cases}
\]

\item (Boolean lattice) $P=2^A$: %Since $2^A \cong {\mathbf{2}} \times \cdots \times {\mathbf{2}}$, we get
\[
\mu_{2^A}(S,T) = (-1)^{T-S}.
\]

\item (Divisor lattice) $P=D_n$: %Here $D_n \cong \mathbf{(t_1 + 1)} \times \cdots \times \mathbf{(t_k+1)}$ where $n = p_1^{t_1} \cdots p_k^{t_k}$ is the prime factorization of $n$, so w
We have $\mu_{D_n}(k,l) = \mu(l/k)$ where
\[
\mu(m) = \begin{cases}
(-1)^t & \textrm{if $m$ is a product of $t$ distinct primes, and}\\
0 & \textrm{otherwise}.
\end{cases}
\]
is the classical \textbf{M\"obius function} from number theory.

\item (Young's lattice) $P=Y$: 
\[
\mu(\lambda, \mu) = \begin{cases}
(-1)^{|\mu - \lambda|} & \textrm{if $\mu - \lambda$ has no two adjacent squares, and}\\
0 & \textrm{otherwise}.
\end{cases}
\]

\item (Partition lattice) $P=\Pi_n$: The M\"obius number of $\Pi_n$ is
\[
\mu(\Pi_n) = (-1)^{n-1}(n-1)!,
\]
from which a (less elegant) formula for the complete M\"obius function can be derived.

\item (Non-crossing partition lattice) $P=NC_n$: The M\"obius number of $NC_n$ is
\[
\mu(NC_n) = (-1)^{n-1}C_{n-1},
\]
where $C_{n-1}$ is the $(n-1)$-st Catalan number. This gives a (less elegant) formula for the complete M\"obius function.

\item (Bruhat order) $P=S_n$:
\[
\mu(u, v) = (-1)^{\ell(v)-\ell(u)},
\]
where the \textbf{length} $\ell(w)$ of a permutation $w \in S_n$ is the number of \textbf{inversions} $(i,j)$ where $1 \leq i < j \leq n$ and $w_i>w_j$. (There is a generalization of this result to the Bruhat order on any Coxeter group $W$, or even on a parabolic subgroup $W^J$; see Section \ref{f.sec:computingMobius}.)

\item (Subspace lattice) $P=L({\mathbb{F}}_q^n)$: %For the subspace lattice of a finite vector space ${\mathbb{F}}_q^n$, 
\[
\mu (U,V) = (-1)^d q^{d \choose 2},
\]
where $d = \dim V - \dim U$.

\item (Distributive lattice) $L=J(P)$:
\[
\mu(I, J) = \begin{cases}
(-1)^{|J-I|} & \textrm{if $J-I$ is an antichain in $P$, and}\\
0 & \textrm{otherwise}.
\end{cases}
\]

\item (Face poset of a polytope) $L=F(P)$:
\[
\mu(F,G) = (-1)^{\dim G - \dim F}.
\]

\item (Face poset of a subdivision ${\mathcal{T}}$ of a polytope $P$) $L=\widehat{T}$:
\[
\mu(F, G) = \begin{cases}
(-1)^{\dim G - \dim F} & \textrm{if } G < {\widehat{1}} \\
(-1)^{\dim P - \dim F+1} & \textrm{if } G = {\widehat{1}} \textrm{ and $F$ is not on the boundary of $P$} \\
0 & \textrm{if } G = {\widehat{1}} \textrm{ and $F$ is on the boundary of $P$} 
\end{cases}
\]

\item (Subgroup lattice of a finite $p$-group) If $|G| = p^n$ for $p$ prime, $n \in {\mathbb{N}}$, then in $L=L(G)$:
\[
\mu(A,B) = \begin{cases}
(-1)^k p^{k \choose 2} & \textrm{if $A$ is a normal subgroup of $B$ and $B/A \cong {\mathbb{Z}}_p^k$, and}\\
0 & \textrm{otherwise}.
\end{cases}
\]

%\comment{See EC2, p. 501 and EC2, Problem 5.29.}
\end{enumerate}
\end{theorem}

Some of the formulas above follow easily from the definitions, while others require more sophisticated methods. In the following sections we will develop some of the basic theory of M\"obius functions and discuss the most common methods for computing them. Along the way, we will sketch proofs of all the formulas above.

It is worth remarking that a version of 11. holds more generally for the face poset of any finite regular cell complexes $\Gamma$ such that the underlying space $|\Gamma|$ is a manifold with or without boundary; see \cite[Prop. 3.8.9]{f.EC1} for details.

%
%\begin{example} Let us discuss two motivating examples.
%\begin{enumerate}
%\item
%If $P=2^{S}$ then $\mu_{2^A}(T) = (-1)^{|T|}$ for $T \subseteq S$.
%\item
%If $P = D_n$ is the poset of divisors of $n$, then
%\[
%\mu_{D_n}(m) = \begin{cases}
%0 & \textrm{if } p^2 | m \textrm{ for some prime } p, \\
%(-1)^k & \textrm{if } m=p_1 \cdots p_k \textrm{ for distinct primes } p_1, \ldots, p_k
%\end{cases}
%\]
%is the classical M\"obius function from number theory.
%\end{enumerate}
%One can prove these formulas by induction, but we will soon give more conceptual proofs.
%\end{example}
%

%
%
%\begin{example} 
%If $P=2^{S}$ then $[T,U] \cong 2^{U-T}$, and hence $\mu_P(T,U) = (-1)^{|U|-|T|}$ for $T \subseteq U$.
%If $P=D_n$ then $[l,m] \cong D_{m/l}$, so $\mu_{D_n}(l,m) = \mu(m/l)$ for $ l | m$.\end{example}
%
%Two comments are in order. Firstly, notice that by holding $p$ fixed and letting $q$ vary, equation (\ref{f.e:Mobius2}) may be understood as a recursive formula for $\mu(p, *)$,  analogous to (\ref{f.e:Mobius1}). In fact we have $\mu_P(p,r) = \mu_{P_{\geq p}}(r)$, where $P_{\geq p}$ is the subposet of $P$ of elements greater than or equal to $p$ in $P$. Secondly, the M\"obius function can be defined equivalently by the equations:
%
%\begin{equation}\label{f.e:Mobius2}
%\sum_{p \leq r \leq q} \mu(r,q) = \begin{cases}
%1 & \textrm{ if } p=q, \\
%0 & \textrm{ otherwise}.
%\end{cases}
%\end{equation}
%
%
%

\subsubsection{{\textsf{M\"obius inversion}}}
In enumerative combinatorics, there are many situations where have a set $U$ of objects, and a natural way of assigning to each object $u$ of $U$ an element $f(u)$ of a poset $P$. We are interested in counting the objects in $U$ that map to a particular element $p \in P$. Often we find that it is much easier to count the objects in $s$ that map to an element \textbf{less than or equal to}\footnote{or greater than or equal to} $p$ in $P$. 
The following theorem tells us that this easier enumeration is sufficient for our purposes, as long as we can compute the M\"obius function of $P$.

\begin{theorem}\label{f.th:Mobiusinv} \emph{(M\"obius Inversion formula)}
Let $P$ be a poset and let $f,g: P \rightarrow {\mathbbm{k}}$ be functions from $P$ to a field ${\mathbbm{k}}$. Then
\begin{eqnarray*}
1. \quad g(p) = \sum_{q \geq p} f(q) \,\, \textrm{ for all } p \in P 
\quad &\Longleftrightarrow& \quad
 \,\, \displaystyle f(p) = \sum_{q \geq p} \mu(p,q)g(q) \,\, \textrm{ for all } p \in P,  \textrm{ and}\\
2. \quad g(p) = \sum_{q \leq p} f(q) \,\, \textrm{ for all } p \in P 
\quad &\Longleftrightarrow& \quad
\displaystyle f(p) = \sum_{q \leq p} \mu(q,p)g(q) \,\, \textrm{ for all } p \in P.
\end{eqnarray*}
%
%
%
%\[
%g(p) = \sum_{q \geq p} f(q) \,\, \textrm{ for all } p \in P
%\]
%if and only if 
%\[
%\displaystyle f(p) = \sum_{q \geq p} \mu(p,q)g(q) \,\, \textrm{ for all } p \in P.
%\]
%
%\begin{enumerate}
%\item For all $p \in P$, \, $\displaystyle g(p) = \sum_{q \geq p} f(q)$. 
%\item For all $p \in P$, \,  $ \displaystyle f(p) = \sum_{q \geq p} \mu(p,q)g(q)$.
%\end{enumerate}
\end{theorem}

%In the situation of the previous paragraph, $f(p)$ counts the objects mapping to $p \in P$, and $g(p)$ counts the objects mapping to elements greater than or equal to $p \in P$. The M\"obius inversion formula expresses $f$ in terms of $g$ and $\mu_P$.
%
%A slightly different formulation is sometimes more useful. 

%\begin{theorem}\label{f.th:Mobiusinv2} \emph{(M\"obius inversion formula, version 2)}
%Let $P$ be a poset and let $f,g: P \rightarrow {\mathbbm{k}}$ be functions from $P$ to a field ${\mathbbm{k}}$. Then
%\[
%g(p) = \sum_{q \leq p} f(q) \,\, \textrm{ for all } p \in P
%\]
%if and only if 
%\[
%\displaystyle f(p) = \sum_{q \leq p} \mu(q,p)g(q) \,\, \textrm{ for all } p \in P.
%\]
%%
%%\begin{enumerate}
%%\item For all $p \in P$, \, $\displaystyle g(p) = \sum_{q \geq p} f(q)$. 
%%\item For all $p \in P$, \,  $ \displaystyle f(p) = \sum_{q \geq p} \mu(p,q)g(q)$.
%%\end{enumerate}
%\end{theorem}

In his paper \cite{f.Rota1}, which pioneered the use of the M\"obius inversion formula as a tool for counting in combinatorics, Rota described this enumerative philosophy as follows:

\begin{quote}
It often happens that a set of objects to be counted possesses a natural ordering, in general only a partial order. It may be unnatural to fit the enumeration of such a set into a linear order such as the integers: instead, it turns out in a great many cases, that a more effective technique is to work with the natural order of the set. One is led in this way to set up a ``difference calculus" relative to an arbitrary partially ordered set. 
\end{quote}

Indeed, one may think of the M\"obius function as a poset-theoretic analog of the Fundamental Theorem of Calculus: $g$ is analogous to the integral of $f$, as it stores the cumulative value of this function. Like the Fundamental Theorem of Calculus, the M\"obius inversion formula tells us how to recover the function $f$ from its cumulative values.

It is not difficult to prove Theorem \ref{f.th:Mobiusinv} %and \ref{f.th:Mobiusinv2} 
directly, but we will soon discuss an algebraic framework that really explains it. In the meantime, we discuss two key applications. We will see several other applications later on.

\begin{example} M\"obius inversion is particularly important for chains, Boolean and divisor lattices.

 \begin{enumerate}
\item For $P={\mathbb{N}}$, Theorem \ref{f.th:Mobiusinv}.1 is a simple but important result for partial sums:
\[
g(n) = \sum_{i=0}^n f(i)  \,\, \textrm{ for all } n \in {\mathbb{N}} \quad \textrm{ if and only if } \quad 
f(n) = g(n)-g(n-1) \,\, \textrm{ for all  } n \in {\mathbb{N}}
\]
where $g(-1)=0$.

\item For $P=2^A$, Theorem \ref{f.th:Mobiusinv}.1 is the inclusion-exclusion formula:
\[
g(S) = \sum_{T \supseteq S} f(T) \,\, \textrm{ for all } S \subseteq A \quad \textrm{ if and only if } \quad 
f(S) = \sum_{T \supseteq S} (-1)^{|T-S|} g(T) \,\, \textrm{ for all } S \subseteq A.
\]
%If we have a set $U$ of objects and a set $A$ of properties, then we associate to each object the subset of properties that it satisfies. If we know, for each $S \subseteq A$, the number $g(S)$ of elements having \emph{at least} the properties in $S$, then we obtain, for each $S \subseteq A$, the number $f(S)$ of elements having \emph{exactly} the properties in $S$.

\item For $P=D_n$, Theorem \ref{f.th:Mobiusinv}.2 is  the M\"obius inversion formula from number theory:
\[
g(m) = \sum_{m|d|n} f(d) \,\, \textrm{ for all } m|n \quad \textrm{ if and only if } \quad 
f(m) = \sum_{m|d|n} \mu(d/m) g(d) \,\, \textrm{ for all  } m|n.
\]
%We have a set $U$ of numbers, and assign a divisor of $n$ to each of them.
A typical application is the computation of Euler's totient function
\[
\varphi(n) = |\{u \, : \, 1 \leq u \leq n, \,\, \gcd(u,n)=1\}|,
\]
or more generally $f(m) = |\{u \, : \, 1 \leq u \leq n, \,\, \gcd(u,n)=m\}|$ for $m|n$.
Here $U=[n]$, and we assign to each $u \in U$ the divisor $\gcd(n,u)$ of $n$. 
There are $g(m) = n/m$ multiples of $m$ in $U$, so
\[
\varphi(n) = f(1) = \sum_{d | n} \mu(d)\frac{n}{d} = n\left(1-\frac1{p_1}\right) \cdots \left(1-\frac1{p_k}\right)
\]
\end{enumerate}

\end{example}

We will be interested in applying the M\"obius inversion formula in many other contexts, and for that reason it is important that we gain a deeper understanding of M\"obius functions;
%, and develop techniques to compute them. T
this is one of the main goals of the following sections.

%\comment{\subsubsection{\textsf{\textbf{Applications of M\"obius Inversion}}}}

\subsubsection{\textsf{The incidence algebra}}\label{f.sec:incidence}

The M\"obius function has a very natural algebraic interpretation, which we discuss in this section. Given a field ${\mathbbm{k}}$, recall that a ${\mathbbm{k}}$-algebra $A$ is a vector space over ${\mathbbm{k}}$ equipped with a bilinear product. 

The \textbf{incidence algebra} $I(P)$ of a locally finite poset $P$ is the ${\mathbbm{k}}$-algebra of functions $f:{\mathrm{Int}}(P) \rightarrow {\mathbbm{k}}$ from the intervals of $P$ to ${\mathbbm{k}}$, equipped with the \textbf{convolution product} $f \cdot g$ given by
\[
f \cdot g\, (p,r) = \sum_{p \leq q \leq r} f(p,q) \, g(q,r) \qquad \textrm{for } p \leq r.
\]

Alternatively, let $\mathrm{Mat}(P)$ be the set of $P \times P$ matrices $A$ with entries in ${\mathbbm{k}}$ whose only nonzero entries $a_{pq} \neq 0$ occur in positions where $p \leq q$ in $P$. There is no canonical way of listing the rows and columns of  the matrix $P$ in a linear order. We normally list them in the order given by a linear extension of $P$, so that the matrices in $\mathrm{Mat}(P)$ will be upper triangular. Then $\mathrm{Mat}(P)$ is a ${\mathbbm{k}}$-algebra under matrix multiplication, and it is clear from the definitions that
\[
I(P) \cong \mathrm{Mat}(P).
\]

\noindent \textbf{\textsf{The unit and inverses.}}
The product in $I(P)$ is clearly associative, and has a \textbf{unit}
\[
{\mathbf{1}}(p,q) = \begin{cases}
1 & \textrm{ if } p=q, \\ 
0 & \textrm{ if } p<q. 
\end{cases}
\]
which is a (two-sided) multiplicative identity.
An element  $f \in I(P)$ has a (necessarily unique) left and right multiplicative inverse $f^{-1}$ if and only if $f(p,p) \neq 0$ for all $p \in P$.

\bigskip

\noindent \textbf{\textsf{The zeta function and counting chains.}}
An important element of $I(P)$ is the \textbf{zeta function}
\[
\zeta(p,q) = 1 \qquad \textrm{ for all } p \leq q.
\]
Notice that $\zeta^k(p,q)$ is the number of multichains of length $k$ from $p$ to $q$. 

If $P$ has a ${\widehat{0}}$ and ${\widehat{1}}$, then $\zeta_P^{\,k}({\widehat{0}},{\widehat{1}})$ counts all multichains of length $k-2$ in $P$, so the zeta polynomial and zeta function are related by 
\begin{equation}\label{f.e:zetas}
Z_P(k) = \zeta_P^{\,k}({\widehat{0}},{\widehat{1}}) \qquad \textrm{for } k \geq 2.
\end{equation}

Similarly, $(\zeta - {\mathbf{1}})^k(p,q)$ is the number of chains of length $k$ from $p$ to $q$.
The sum ${\mathbf{1}} + (\zeta - {\mathbf{1}}) + (\zeta - {\mathbf{1}})^2 + \cdots$ has finitely many non-zero terms, and it equals $({\mathbf{2}} - \zeta)^{-1}$, where ${\mathbf{2}} = 2 \cdot {\mathbf{1}}$; so
\[
({\mathbf{2}} - \zeta)^{-1}(p,q) = \textrm{ total number of chains in $P$ from $p$ to $q$}.
\]

\bigskip

\noindent \textbf{\textsf{The M\"obius function and M\"obius inversion.}} The (equivalent)  equations (\ref{f.e:Mobius1}) and (\ref{f.e:Mobius2}) defining the M\"obius function can be rewritten as $\mu \zeta = \textbf{1}$ and $\zeta \mu = \textbf{1}$, respectively. This explains why these two equations are equivalent: they say that 
\[
\mu = \zeta^{-1}.
\]

\begin{proof}[Proof of the M\"obius inversion formula.] Consider the left action of the incidence algebra $I(P)$ on the vector space ${\mathbbm{k}}^P$ of functions $f: P \rightarrow {\mathbbm{k}}$, given by
\[
a \cdot f \, (p) = \sum_{q \geq p} a(p,q) f(q)
\]
for $a \in I(P)$ and $f: P \rightarrow {\mathbbm{k}}$. The M\"obius inversion formula then states that $g = \zeta \cdot f $
if and only if $\mu \cdot g = f$; this follows immediately from $\mu = \zeta^{-1}$.
\end{proof}

%\comment{Counting maximal chains? In $NC_n$ there are $n^{n-2}$. (Edelman)}

\subsubsection{\textsf{Computing M\"obius functions}}\label{f.sec:computingMobius}

%To apply the M\"obius inversion Formula successfully, it is often necessary to compute the M\"obius functions of the posets involved. 
In this section, 
 we discuss some of the main tools for computing M\"obius functions. Along the way, we prove the formulas for the M\"obius functions of the posets of Theorem \ref{f.th:Mobiusformulas}.

\begin{enumerate}
\item (Chain $\mathbf{n}$) 
We can check the formula for $\mu_{\mathbf{n}}$ manually.
\end{enumerate}

\bigskip
\noindent \textsf{\textbf{M\"obius functions of products.}} A simple but important fact is that M\"obius functions behave well under poset multiplication:
\[
\mu_{P \times Q}((p,q), (p',q')) = \mu_P(p, p') \mu_Q(q, q')
\]
This is easily verified directly, and also follows from the fact that $I(P \times Q) \cong I(P) \otimes_{\mathbbm{k}} I(Q)$. 

\begin{enumerate}
\item[2.] (Boolean lattice $2^{A}$)
Since $2^A \cong {\mathbf{2}} \times \cdots \times {\mathbf{2}}$, the M\"obius function of ${\mathbf{2}}$ tells us that the M\"obius function of $2^A$ is $\mu(S,T) = (-1)^{|T-S|}$. 

\item[3.] (Divisor lattice $D_n$) 
Since $D_n \cong \mathbf{(t_1 + 1)} \times \cdots \times \mathbf{(t_k+1)}$ for $n = p_1^{t_1} \cdots p_k^{t_k}$, we get the formula for $\mu_{D_n}$ from the formula for the M\"obius function of a chain.

\end{enumerate}

\bigskip

\noindent \textsf{\textbf{M\"obius functions through M\"obius inversion.}} So far we have thought of the M\"obius function as a tool to apply M\"obius inversion. Somewhat  counterintuitively, it is possible to use M\"obius inversion in the other direction, as a tool to compute M\"obius functions. We carry out this approach to compute the M\"obius function of the partition lattice $\Pi_n$.

\begin{enumerate}
\item[5.] (Partition lattice $\Pi_n$) 
Let $\Pi_A$ denote the lattice of set partitions of a set $A$ ordered by refinement. Every interval $[\pi, \rho]$ in $\Pi_A$ is a product of partition lattices, as illustrated by the following example: in $\Pi_9$ we have $[18|2|37|4|569, 13478|2569] \cong \Pi_{\{18, 37, 4\}} \times \Pi_{\{2, 569\}} \cong \Pi_3 \times \Pi_2$. Since the M\"obius function is multiplicative, to compute the M\"obius function of the partition lattices it suffices to show that $\mu_{\Pi_n}({\widehat{0}}, {\widehat{1}}) = (-1)^{n-1}(n-1)!$.

Let $W$ be the set of words $w=w_1\ldots w_n$ of length $n$ in the alphabet $\{0, 1, \ldots, q-1\}$. 
Classify the words according to the equalities among their coordinates; namely, to each word $w$, associate the partition of $[n]$ where $i$ and $j$ are in the same block when $w_i=w_j$. 

Let $f(\pi)$ be the number of words whose partition is $\pi$, and let $g(\pi)$ be the number of words whose partition is a coarsening of $\pi$. This is a situation where both $f(\pi)$ and $g(\pi)$ are easily computed: if $\pi$ has $b$ blocks then we have
\[
f(\pi) = q(q-1)\cdots (q-b+1), \qquad g(\pi) = q^b.
\]
Since $g(\pi) = \displaystyle \sum_{\rho \geq \pi} f(\rho)$ we get $f(\pi) = \displaystyle  \sum_{\rho \geq \pi} \mu(\pi, \rho) g(\rho)$. For $\pi={\widehat{0}}$ this says that
\[
q(q-1) \cdots (q-n+1)  =  \sum_{\rho \in \Pi_n} \mu({\widehat{0}}, \rho) q^{|\rho|}.
\]
Equating the coefficients of $q^1$ we get the desired result.

\end{enumerate}

\bigskip
\noindent \textsf{\textbf{M\"obius functions through closures.}} A function $\overline{\,\cdot\,}: P \rightarrow P$ is a \textbf{closure operator} if 
\begin{itemize}
\item $p \leq \overline{p}$ for all $p \in P$,
\item $p \leq q$ implies $\overline{p} \leq \overline{q}$ for all $p,q \in P$, and
\item $\overline{\overline{p}} = \overline{p}$ for all $p \in P$.
\end{itemize}
An element of $p$ is \textbf{closed} if $\overline{p} = p$; let $\textrm{Cl}(P)$ be the subposet of closed elements of $P$.

\begin{proposition}\label{f.prop:closure}
If $\overline{\,\cdot\,}: P \rightarrow P$ is a closure operator, then for any $p \leq q$ in $P$,
\[
\sum_{r\, : \,  \overline{r} = q} \mu(p,r) = 
\begin{cases}
\mu_{\textrm{Cl(P)}}(p,q) & \textrm{ if $p$ and $q$ are closed} \\
0 & \textrm{ otherwise}
\end{cases}
\]
\end{proposition}

\begin{proof}
We have an inclusion of incidence algebras $I(Cl(P)) \rightarrow I(P)$ given by $f \mapsto \overline{f}$ where
% the only possibly non-zero values of $\overline{f}$ are $\overline{f}(p,q) = f(p,q)$ when $p$ and $q$ are closed.
\[
\overline{f}(p,q) = \begin{cases}
f(p,q) & \textrm{if $p$ and $q$ are closed}\\
0 & \textrm{otherwise}
\end{cases}
\]
Let $\overline{1}, \overline{\zeta},$ and $\overline{\mu}$ be the image of the unit, zeta, and M\"obius functions of $Cl(P)$ in $I(P)$. Also consider the ``closure function" $c \in I(P)$ whose non-zero values are $c(p, \overline{p}) = 1$ for all $p$. Note that $c \, \overline{\zeta} = \zeta \, \overline{1}$ because when $q$ is closed,  $\overline{p} \leq q$ if and only if $p \leq q$. Now we compute in $I(P)$:
\[
\mu \, c \, \overline{1} = \mu c \, \overline{\zeta} \, \overline{\mu} = \mu \, \zeta \, \overline{1} \, \overline{\mu} = 1 \, \overline{\mu} = \overline{\mu}.
\]
This is equivalent to the desired equality.
\end{proof}

\bigskip
\noindent \textsf{\textbf{M\"obius functions of lattices.}} When $L$ is a lattice, there are two other methods for computing M\"obius functions. The first one gives an alternative to the defining recursion for $\mu(p,q)$. This new recurrence is usually much shorter, at the expense of requiring some understanding of the join operation.

\begin{proposition} (Weisner's Theorem) 
For any $p < a \leq q$ in a lattice $L$ we have
\[
\sum_{p \leq r \leq q \, : \, r \vee a = q} \mu(p,r) = 0
\]
\end{proposition}

Weisner's Theorem follows from Proposition \ref{f.prop:closure} because $r \mapsto r \vee a$ is a closure operation, whose closed sets are the elements greater than or equal to $a$. There is also a dual version, obtained by applying the result to the reverse lattice $L^{\textrm{op}}$ obtained from $L$ by reversing all order relations.

Let us apply Weisner's Theorem to three examples.
%We can use it to compute the M\"obius function of $L({\mathbb{F}}_q^n)$, the subspace lattice of a finite vector space ${\mathbb{F}}_q^n$, as follows. 

\begin{enumerate}
\item[8.] (Subspace lattice $L({\mathbb{F}}_q^n)$) 
Since an interval of height $r$ in $L({\mathbb{F}}_q^n)$ is isomorphic to $L({\mathbb{F}}_q^r)$, it suffices to compute $\mu(L({\mathbb{F}}_q^n)) =: \mu_n$.
 %\mu_{\Pi_n}({\widehat{0}}, {\widehat{1}})$. 
Let $p={\widehat{0}}$, $q={\widehat{1}}$, and let $a$ be any line. The only subspaces $r \neq {\widehat{1}}$ with $r \vee a = {\widehat{1}}$ are the $q^{n-1}$ hyperplanes not containing $a$. 
Each one of these hyperplanes satisfies $[{\widehat{0}},r] \cong L({\mathbb{F}}_q^{n-1})$, so Weisner's Theorem gives $\mu_n +q^{n-1}\mu_{n-1} = 0$, from which $\mu_n = (-1)^nq^{n(n-1)/2}$.

\item[5.] (Partition lattice $\Pi_n$, revisited)
A similar (and easier) argument may be used to compute $\mu_n = \mu({\Pi_n})$, though it is now easier to use the dual to Weisner's Theorem. Let $p={\widehat{0}}$, $q={\widehat{1}}$, and let $a$ be the coatom $12\ldots{n-1}|n$. The only partitions $r \neq {\widehat{0}}$ with $r \wedge a = {\widehat{0}}$ are those with only one non-singleton block of the form $\{i,n\}$ for $i \neq n$. Each such partition $\pi$ has $[\pi, {\widehat{1}}] \cong \Pi_{n-1}$. The dual to Weisner's Theorem tells us that $\mu_n + (n-1)\mu_{n-1}=0$, from which $\mu_n = (-1)^n(n-1)!$.

\item[11.] (Subgroup lattice $L(G)$ of a finite $p$-group $G$)
A similar argument may be used for the subgroup lattice $L(G)$ of a $p$-group $G$. However, now one needs to invoke some facts about $p$-groups; see \cite{f.Weisner}. 
\end{enumerate}

%
%We can use Weisner's Theorem to compute the M\"obius function of $L({\mathbb{F}}_q^n)$, the subspace lattice of a finite vector space ${\mathbb{F}}_q^n$. Since an interval of height $r$ in$L({\mathbb{F}}_q^n)$ is isomorphic to $L({\mathbb{F}}_q^r)$, it suffices to compute $\mu(L({\mathbb{F}}_q^n)) =: \mu_n$.
% %\mu_{\Pi_n}({\widehat{0}}, {\widehat{1}})$. 
%Let $p={\widehat{0}}$, $q={\widehat{1}}$, and $a$ be the hyperplane $x_n=0$. The only subspaces $r \neq {\widehat{0}}$ with $r \wedge a = {\widehat{0}}$ are the $q^{n-1}$ lines going through each of the points $(x_1, \ldots, x_{n-1},1)$ for $x_i \in {\mathbb{F}}_q$. Each one of these lines satisfies $[r,{\widehat{1}}] \cong L({\mathbb{F}}_q^{n-1})$, so Weisner's Theorem gives $\mu_n +q^{n-1}\mu_{n-1} = 0$, from which $\mu_n = (-1)^nq^{n \choose 2}$. A similar argument works for the M\"obius function of the partition lattice.

%We can use Weisner's Theorem to compute $\mu(\Pi_n) = \mu_{\Pi_n}({\widehat{0}}, {\widehat{1}})$. Let $p={\widehat{0}}$, $q={\widehat{1}}$, and $a=12\cdots(n-1) | n$. The only partitions $r \neq {\widehat{0}}$ with $r \wedge a = {\widehat{0}}$ are $r=1|2|\cdots|\widehat{i}|\cdots|(n-1)|in$ for $1 \leq i \leq n-1$, and they all satisfy $[r,{\widehat{1}}] \cong \Pi_{n-1}$. Therefore $\mu(\Pi_n) +(n-1)\mu(\Pi_{n-1}) = 0$, from which $\mu(\Pi_n) = (-1)^{n-1}(n-1)!$.

\begin{theorem}\label{f.th:crosscut} (Crosscut Theorem) 
Let $L$ be a lattice and let $X$ be the set of atoms of $L$. Then
\[
\mu({\widehat{0}},{\widehat{1}}) = \sum_{k} (-1)^k N_k
\]
where $N_k$ is the number of $k$-subsets of $X$ whose join is ${\widehat{1}}$.
\end{theorem}

We will sketch a proof at the end of this section. Meanwhile, we point out a simple corollary of the Crosscut Theorem: 

\begin{equation}\label{f.eq:joinnot1}
\textrm{If the join of the atoms is not ${\widehat{1}}$ then $\mu({\widehat{0}}, {\widehat{1}})=0$.}
\end{equation}

\begin{enumerate}
\item[9.] (Distributive lattice $L=J(P)$) 
%This result suffices to compute the M\"obius function of a distributive lattice $L=J(P)$. 
If $J-I$ is an antichain of $P$, then $[I,J]$ is a Boolean lattice in $L$ and $\mu_L(I,J) = (-1)^{J-I}$. Otherwise,
% if $J-I$ is not an antichain, then 
the join of the atoms of $[I,J]$ is $I \cup \min(J-I) \neq J$, and hence $\mu_L(I,J) = 0$.

\item[4.] (Young's lattice $Y$) We obtain this M\"obius function for free since $Y$ is distributive.
\end{enumerate}

Naturally, there are dual formulations to the previous two propositions, obtained by reversing the order of $L$. Also, there are many different versions of the crosscut theorem; see for example \cite{f.Rota1}.

%
%\bigskip
%
%
%\comment{Galois connections?}

\bigskip
\noindent \textsf{\textbf{M\"obius functions through multichains.}} If we know the zeta polynomial of a poset $P$ with ${\widehat{0}}$ and ${\widehat{1}}$, we can obtain its M\"obius number $\mu(P)$ immediately.

\begin{proposition} If $P$ is a poset with ${\widehat{0}}$ and ${\widehat{1}}$,
\begin{equation} \label{f.e:Z(-1)}
Z_P(-1) = \mu_P({\widehat{0}},{\widehat{1}}).
\end{equation}
\end{proposition}

\begin{proof}
We saw in (\ref{f.e:zetas}) that $Z_P(k) = \zeta_P^k({\widehat{0}},{\widehat{1}})$ for all integers $k \geq 2$. It would be irresponsible to just set $k=-1$, but it is very tempting, since $\zeta^{-1} = \mu$.

In the spirit of \textbf{combinatorial reciprocity}, this is an instance where such irresponsible behavior pays off, with a bit of extra care. We know that $Z_P(k)$ is polynomial for $k \in {\mathbb{Z}}$, and we leave it as an exercise to show that $\zeta^k({\widehat{0}},{\widehat{1}})$ is also polynomial for $k \in {\mathbb{Z}}$. Since these two polynomials agree on infinitely many values, they also agree for $k=-1$, and the result follows.
%\comment{The polynomiality of $\zeta^k$ can be proved through an ad hoc argument, but we will derive it as a consequence of more general facts about Hopf algebras.}
\end{proof}

In light of (\ref{f.e:Z(-1)}), the zeta polynomial of $P$ will give us the M\"obius number $\mu(P)$ automatically. This is advantageous because sometimes the zeta polynomial is easier to compute than the M\"obius function, as it is the answer to an explicit enumerative question.

\begin{enumerate}
\item[6.] (Non-crossing partition lattice $NC_n$) From the zeta polynomial of $NC_n$ in Example \ref{f.ex:zeta} we immediately obtain that $\mu(NC_n) = (-1)^{n-1}C_{n-1}$ where $C_{n-1}$ is the $(n-1)$st Catalan number. This gives a formula for the full M\"obius function, since every interval in $NC_n$ is a product of smaller non-crossing partition lattices.
\end{enumerate}

%As another application of (\ref{f.e:Z(-1)}), we sketch a simple proof of the Crosscut Theorem of the previous section. Let $A$ be the set of atoms of $P$. For each subset $S \subseteq A$ let $s = \bigvee S$ and let $P_S = P_{\geq s} \cup {\widehat{0}}$. The inclusion-exclusion formula gives that $\sum_{S \subseteq A} (-1)^{|S|} Z_{P_S}(n) = 0$. Plugging in $n=-1$, the only surviving terms are $\mu({\widehat{0}}, {\widehat{1}})$ and the terms where $s={\widehat{1}}$. (But to know this, I need a different proof of (\ref{f.eq:joinnot1}).)
%

\bigskip
\noindent \textsf{\textbf{M\"obius functions through topology.}} Equation (\ref{f.e:Z(-1)}) has a topological interpretation which is an extremely powerful method for computing M\"obius functions. 

\begin{proposition} (Phillip Hall's Theorem) \label{f.p:PhillipHall} Let $P$ be a finite poset with a ${\widehat{0}}$ and ${\widehat{1}}$, and let $c_i$ be the number of chains ${\widehat{0}} = p_0 < p_1 < \cdots < p_i = {\widehat{1}}$ of length $i$ from ${\widehat{0}}$ to ${\widehat{1}}$ in $P$. Then
\[
\mu_P({\widehat{0}}, {\widehat{1}}) = c_0 - c_1 + c_2 - \cdots 
\]
\end{proposition}
This formula is equivalent to (\ref{f.e:Z(-1)}) in light of (\ref{f.e:Z}) and the relations $b_i = c_i + 2c_{i-1} + c_{i-2}$; it may also be proved directly in the incidence algebra of $P$. 

Let us now interpret this result topologically. The \textbf{order complex} $\Delta(P)$ of a poset $P$ is the simplicial complex whose vertices are the elements of $P$, and whose faces are the chains of $P$. 

\begin{theorem}\label{f.th:muEuler} Let $P$ be a finite poset with a ${\widehat{0}}$ and ${\widehat{1}}$, and let $\overline{P} = P - \{{\widehat{0}}, {\widehat{1}}\}$. Then
\[
\mu_P({\widehat{0}}, {\widehat{1}}) = {\widetilde{\chi}}(\Delta(\overline{P}))
\]
is the reduced Euler characteristic of $\Delta(\overline{P})$.
\end{theorem}

\begin{proof} In light of Proposition \ref{f.p:PhillipHall}, this follows immediately from the combinatorial formula
\[
{\widetilde{\chi}}(\Delta) = \sum_{k=-1}^d (-1)^k f_k
\]
for the Euler characteristic of $|\Delta|$, where $f_k$ is the number of $k$-dimensional faces of $\Delta$.
\end{proof}

Let us use this topological description to sketch proofs of the remaining M\"obius functions of Theorem \ref{f.th:Mobiusformulas}.

\begin{enumerate}
\item[10.] (Face lattice of a polytope $L(P)$)
The \textbf{barycentric subdivision} $\textrm{sd}(P)$ is a simplicial complex with a vertex $v(F)$ at the barycenter of each proper face $F$. It has a simplex connecting vertices $v(F_1), \ldots, v(F_k)$ whenever $F_1 \subset \cdots \subset F_k$. As abstract simplicial complexes, $\textrm{sd}(P)$ equals $\Delta(\overline{F(P)})$. Geometrically, $\textrm{sd}(P)$ is a subdivision of the boundary of the polytope $P$, and hence is homeomorphic to the sphere $\mathbb{S}^{\dim P -1}$. Therefore 
\[
\mu({\widehat{0}}, {\widehat{1}}) = {\widetilde{\chi}}(\Delta(\overline{P})) = {\widetilde{\chi}}(\mathbb{S}^{\dim P - 1}) = (-1)^{\dim P}.
\]
We will see in Section \ref{f.sec:polytopes} that every interval of $F(P)$ is itself the face poset of a polytope. It then follows that $\mu(F,G) = (-1)^{\dim G - \dim F}.$

\item [11.] (Face lattice of a subdivision ${\mathcal{T}}$ of a polytope $P$) A similar argument holds, though the details are slightly more subtle; see \cite[Prop. 3.8.9]{f.EC1}.

\item[7.] (Bruhat order in permutations $S_n$)
There are several known proofs of the fact that $\mu(u,v) = (-1)^{l(u)-l(v)}$, none of which is easy. The first proof was an ad hoc combinatorial argument due to Verma \cite{f.Verma}. Later Kazhdan and Lusztig \cite{f.KazhdanLusztig} and Stembridge \cite{f.Stembridge} gave algebraic proofs using Kazhdan-Lusztig polynomials and Hecke algebras, respectively. Bj\"orner and Wachs \cite{f.BjornerWachsBruhat} gave a topological proof, based on Theorem \ref{f.th:muEuler} and further tools from topological combinatorics.
As we remarked earlier, there are similar formulas for the M\"obius function of an arbitrary parabolic quotient of an arbitrary Coxeter group; see \cite{f.BjornerBrenti} or \cite{f.Stembridge}.
\end{enumerate}

Theorem \ref{f.th:muEuler} tells us that in order to compute M\"obius functions of posets of interest, it can be very useful to understand the topology of their underlying order complexes. Conversely, combinatorial facts about M\"obius functions often lead to the discovery of topological properties of these complexes. This is the motivation for the very rich study of \textbf{poset topology}. We refer the reader to the survey \cite{f.Wachs} for further information on this topic.

%\comment{
%
%\subsection{\textsf{Poset topology}}\label{f.sec:posettopology}
%
%A \textbf{simplicial complex} on a vertex set $V$ is a set $\Delta$ of subsets of $V$, called \textbf{faces}, which contains all singletons, and is such that any subset of a face of $\Delta$ is a face of $\Delta$. The \textbf{reduced Euler characteristic} is 
%\[
%{\widetilde{\chi}}(\Delta) = \sum_{k=-1}^d (-1)^k \textrm{rank } \widetilde{H}_k(|\Delta|; {\mathbb{Z}})
%\]
%where $\widetilde{H}_k(|\Delta|; {\mathbb{Z}})$ is the $k$-th reduced homology group of $|\Delta|$. The Euler characteristic it is a topological invariant, which only depends on the underlying space $|\Delta|$. In fact, it is determined entirely by the homotopy type of $|\Delta|$. However, it can also be computed combinatorially thanks to the formula
%\[
%{\widetilde{\chi}}(\Delta) = \sum_{k=-1}^d (-1)^k f_k
%\]
%where $f_k$ is the number of faces of dimension $k$.
%
%Shellability, discrete Morse theory, Cohen-Macaulay complexes, Gorenstein* posets
%
%Multitriangulations?
%
%\subsubsection{\textsf{Eulerian posets}}
%
%cd-index
%}
%

\subsection{{\textsf{Eulerian posets, flag $f$-vectors, and flag $h$-vectors}}} \label{f.sec:Eulerian}

Let $P$ be a graded poset of height $r$. The \textbf{flag $f$-vector} $(f_S \, : \, S \subseteq [0,r])$ and the \textbf{flag $h$-vector} $(h_S \, : \, S \subseteq [0,r])$ are defined by 
\[
f_S = \textrm{number of chains } p_1 < p_2 < \cdots < p_k \textrm{ with } \{r(p_1), \ldots r(p_k)\} = S, \quad \textrm{and}
\]
%for each subset $S \subseteq [0,n]$. The \textbf{flag $h$-vector} $(h_S \, : \, S \subseteq [0,n])$ is defined by 
\[
h_S = \sum_{T \subseteq S} (-1)^{|S|-|T|}f_T, \qquad  \qquad 
f_S = \sum_{T \subseteq S} h_T
\]
for $S \subseteq [0,r]$.
%
%\[
%h_S(P) = \sum_{T \subseteq S} (-1)^{|S| - |T|} f_T(P)
%\]
%We call $(f_S \, : \, S \subseteq [d])$ and $(h_S \, : \, S \subseteq [d])$ the \textbf{flag $f$-vector} and \textbf{flag $h$-vector} of $P$, respectively.

%For each $S$ let $f_S$ be the number of $S$-flags, and let the \textbf{flag $f$-vector} of $P$ be $(f_S\, : \, S \subseteq [d])$. Let the \textbf{flag $h$-vector} of $P$ be $(h_S \, : \, S \subseteq [d])$ where
%\[
%h_S = \sum_{T \subseteq S} (-1)^{|S|-|T|}f_T, \qquad 
%f_S = \sum_{T \subseteq S} h_T
%\]
%for $S \subseteq [d]$. 

\begin{example}\label{f.ex:prism}
Let $H$ be a hexagonal prism of Figure \ref{f.fig:hexagonalprism}, and let $P=L(H) - \{\widehat{0}, \widehat{1}\}$ be its face lattice, with the top and bottom element (the empty face and the full face $H$) removed. It is not so enlightening to draw the poset, but we can still compute its flag $f$ and $h$-vectors.
For example, $f_{\{0,2\}}=36$ counts the pairs $(v,f)$ of a vertex $v$ contained in a 2-face $f$. We get:
%The flag $f$-vector and flag $h$-vector are:
\begin{center}
\begin{tabular}{|c|cccccccc|}
\hline
 & $\emptyset$ & 0 & 1 & 2 & 01 & 02 & 12 & 012 \\
 \hline
$f_S$ &  1 & 12 & 18 & 8 & 36 & 36 & 36 & 72 \\
$h_S$ & 1 & 11 & 17 & 7 & 7 & 17 & 11 & 1 \\
\hline
\end{tabular}
\end{center}
\end{example}

\begin{figure}[ht]
 \begin{center}
  \includegraphics[height=3cm]{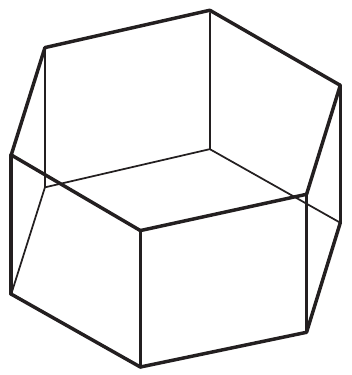}
  \caption{ \label{f.fig:hexagonalprism}
A hexagonal prism.}
  \end{center}
\end{figure}

This example suggests that there may be additional structure in the flag $h$-vector; for instance in this example we have $h_S = h_{[0,d]-S}$ for all $S$. This equation is not true in general, but it does hold for an important class of posets, which we now discuss.

\bigskip

\noindent
\textsf{\textbf{Eulerian posets.}} Say a graded poset $P$ is \textbf{Eulerian} if $\mu(x,y) = (-1)^{r(y)-r(x)}$ for all $x \leq y$. 

In Theorem \ref{f.th:Mobiusformulas} we saw three important families of Eulerian posets: Boolean lattices, the Bruhat order, and face posets of polytopes. In particular, the poset $P$ of Example \ref{f.ex:prism} is Eulerian.

For an Eulerian poset, there are many linear relations among the $f_S$s, which are  easier to describe in terms of the $h_S$s; for example, $h_S = h_{[0,d]-S}$ for all $S$. To describe them all, we further encode the flag $f$-vector in a polynomial in non-commuting variables ${\mathbf{a}}$ and ${\mathbf{b}}$ called the \textbf{${\mathbf{a}}{\mathbf{b}}$-index}, which is defined to be
\[
\Phi_P[{\mathbf{a}},{\mathbf{b}}] = \sum_{S \subseteq [d]} h_Su_S
\]
where $u_S=u_1\ldots u_{d}$ and $u_i={\mathbf{a}}$ if $i \notin S$ and $u_i={\mathbf{b}}$ if $i \in S$. This is an element of the ring $\mathbb{Z}\langle{\mathbf{a}},{\mathbf{b}}\rangle$ of integer polynomials in the non-commutative variables ${\mathbf{a}}$ and ${\mathbf{b}}$.

\begin{theorem} \label{f.th:cd-index} \cite{f.BayerBillera, f.BayerKlapper, f.Stanley.cdindex}
The ${\mathbf{a}}{\mathbf{b}}$-index of an Eulerian poset $P$ can be expressed uniquely as a polynomial in
\[
{\mathbf{c}}={\mathbf{a}}+{\mathbf{b}}, \qquad {\mathbf{d}}={\mathbf{a}}{\mathbf{b}}+{\mathbf{b}}{\mathbf{a}}
\]
called the ${\mathbf{c}}{\mathbf{d}}$-index $\Psi_P({\mathbf{c}},{\mathbf{d}})$ of $P$. Furthermore, if $P$ is the face poset of a polytope, then the coefficients of the ${\mathbf{c}}{\mathbf{d}}$-index are non-negative.
\end{theorem}

\begin{example} The ${\mathbf{a}}{\mathbf{b}}$-index of the hexagonal prism in Example \ref{f.ex:prism}
is
\begin{eqnarray*}
\Phi({\mathbf{a}},{\mathbf{b}}) &=& {\mathbf{a}}{\mathbf{a}}{\mathbf{a}}+11{\mathbf{b}}{\mathbf{a}}{\mathbf{a}}+17{\mathbf{a}}{\mathbf{b}}{\mathbf{a}}+7{\mathbf{a}}{\mathbf{a}}{\mathbf{b}}+7{\mathbf{b}}{\mathbf{b}}{\mathbf{a}}+17{\mathbf{b}}{\mathbf{a}}{\mathbf{b}}+11{\mathbf{a}}{\mathbf{b}}{\mathbf{b}}+{\mathbf{b}}{\mathbf{b}}{\mathbf{b}} \\
&=& ({\mathbf{a}}+{\mathbf{b}})^3 + 6({\mathbf{a}}+{\mathbf{b}})({\mathbf{a}}{\mathbf{b}}+{\mathbf{b}}{\mathbf{a}}) + 10({\mathbf{a}}{\mathbf{b}}+{\mathbf{b}}{\mathbf{a}})({\mathbf{a}}+{\mathbf{b}})
\end{eqnarray*}
so the ${\mathbf{c}}{\mathbf{d}}$-index is
\[
\Psi({\mathbf{c}},{\mathbf{d}}) = {\mathbf{c}}^3+6{\mathbf{c}}{\mathbf{d}}+10{\mathbf{d}}{\mathbf{c}}.
\]
In general, the 8 entries of the flag $f$-vector of a $3$-polytope (or any other Eulerian poset of rank $3$) are determined completely by only 3 numbers, namely, the coefficients of ${\mathbf{c}}^3, {\mathbf{c}}{\mathbf{d}},$ and ${\mathbf{d}}{\mathbf{c}}$ in the ${\mathbf{c}}{\mathbf{d}}$-index.
\end{example}

The ${\mathbf{c}}{\mathbf{d}}$-index encodes optimally the linear relations among the entries of flag $f$-vectors of polytopes. Since the number of monomials in ${\mathbf{c}}$ and ${\mathbf{d}}$ of ${\mathbf{a}}{\mathbf{b}}$-degree $d$ is the Fibonacci number $F_{d+1}$, this is the smallest number of entries in $(f_S)$ from which we can recover the whole flag $f$-vector -- this is much smaller than $2^d$.

\begin{corollary}\cite{f.BayerBillera, f.BayerKlapper}
The subspace of ${\mathbb{R}}^{2^d}$ spanned by the flag $f$-vectors of Eulerian posets of rank $d$ (or by the flag $f$-vectors of $d$-polytopes) has dimension equal to the Fibonacci number $F_{d+1}$.
\end{corollary}
%
%\noindent \textbf{Remark.} More generally, say a graded poset $Q$ is \textbf{Eulerian} if $\mu(p,q) = (-1)^{r(q)-r(p)}$ for all $p \leq q$. The face poset of a polytope is Eulerian by Theorem \ref{f.th:Mobiusformulas}. For $S \subseteq \{1, 2, \ldots, r(Q)-1\}$ let $f_S$ be the number of chains having elements in the ranks specified by $S$. The flag $f$-vector of \textbf{any} Eulerian poset can still be expressed in the same way as a $\c{\mathbf{d}}$-index.
%%Define the flag $h$-vector and ${\mathbf{a}}\b$-index as above. Then the ${\mathbf{a}}\b$-index can still be expressed as a $\c{\mathbf{d}}$-index. 
%However, the coefficients no longer have to be non-negative.

\newpage
\begin{Large}
\noindent {\textbf{\textsf{PART 2. DISCRETE GEOMETRIC METHODS}}}
\end{Large}

\bigskip

Part 2 is devoted to discrete geometry, which studies the connections between combinatorics and the geometry of subspaces (points, lines, planes, $\ldots$, hyperplanes) in a Euclidean space. 
Configurations in the plane and 3-space  have received great attention, and feature many interesting results and open questions; however, very few of them involve exact enumeration in a meaningful way. Instead, we will focus on the discrete geometry of higher dimensions, where:
%
%\begin{itemize}
%\item Studying ``general" geometric configurations leads to interesting enumerative questions.
%\item we have enough room to construct ``special" geometric configurations that model various combinatorial structures of interest. 
%\end{itemize}

\smallskip
\noindent $\bullet$ studying ``general" geometric configurations leads to interesting enumerative questions, and

\smallskip
\noindent $\bullet$ we have enough room to construct ``special" geometric configurations that model various combinatorial structures of interest. 

\bigskip

This second part is divided into three sections on closely interrelated topics. In Section \ref{f.sec:polytopes} we discuss polytopes, which are the higher dimensional generalization of polygons. Section \ref{f.sec:hyparrs} discusses 
arrangements of hyperplanes in a vector space.
Finally Section \ref{f.sec:matroids} is devoted to matroids, which are  combinatorial objects that simultaneously abstract arrangements of vectors, graphs, and matching problems, among others.

\section{\textsf{Polytopes}}\label{f.sec:polytopes}

The theory of polytopes is a vast area of study, with deep connections to pure (algebraic geometry, commutative algebra, representation theory) and applied mathematics (optimization). Again, we focus on aspects related to enumeration. For a general introduction to polytopes, see \cite{f.Grunbaum, f.Ziegler}.

After discussing the basic definitions and facts in Section \ref{f.sec:polytopefacts} and some important examples in Section \ref{f.sec:polytopeexamples}, we turn to enumerative questions. Section \ref{f.sec:countingfaces} is devoted to the enumeration of faces of various dimensions. Section \ref{f.sec:Ehrhart} is on \emph{Ehrhart theory}, which measures polytopes by counting the lattice points that they (and their dilations) contain.

\subsection{\textsf{Basic definitions and constructions}}\label{f.sec:polytopefacts}

Recall that a set $S$ in Euclidean space ${\mathbb{R}}^d$ is \textbf{convex} if for every pair of points $u, v$ in $S$, the line segment $uv$ is in $S$. The \textbf{convex hull} ${\mathrm{conv }}(S)$ of a set $S \subseteq {\mathbb{R}}^d$ is the minimal convex set containing $S$. If $S = \{{\mathbf{v}}_1, \ldots, {\mathbf{v}}_n\}$ then
\begin{eqnarray*}
{\mathrm{conv }} \{{\mathbf{v}}_1, \ldots, {\mathbf{v}}_n\} &=& \left\{\lambda_1 {\mathbf{v}}_1 + \cdots + \lambda_n {\mathbf{v}}_n \, : \, \lambda_1, \ldots, \lambda_n \geq 0,  \lambda_1 + \cdots + \lambda_n = 1\right\}\\
&=& \textrm{ intersection of all convex sets containing } {\mathbf{v}}_1, \ldots, {\mathbf{v}}_n. 
\end{eqnarray*}
We will only be interested in convex polytopes, and when we talk about polytopes, it will be assumed that they are convex.

A \textbf{hyperplane} $H$ in Euclidean space ${\mathbb{R}}^d$ is an affine subspace of dimension $d-1$; it is given by a  linear equation 
\[
H = \{{\mathbf{x}} \in {\mathbb{R}}^d \, : \, {\mathbf{a}} \cdot {\mathbf{x}} = b\} \qquad \textrm{ for some ${\mathbf{a}} \in {\mathbb{R}}^d - \{\mathbf{0}\}$ and $b \in {\mathbb{R}}$}.
\]
It separates ${\mathbb{R}}^d$ into the two halfspaces given by the inequalities ${\mathbf{a}} \cdot {\mathbf{x}} \leq b$ and ${\mathbf{a}} \cdot {\mathbf{x}} \geq b$, respectively.

There are two equivalent ways of defining convex polytopes: the \emph{V-description} gives a polytope in terms of its vertices, and the \emph{H-description} gives it in terms of its defining inequalities. 

\medskip

\noindent \textbf{\textsf{The V-description}}. A \textbf{convex polytope} $P$ is the convex hull of finitely many points ${\mathbf{v}}_1, \ldots, {\mathbf{v}}_n$ in a Euclidean space ${\mathbb{R}}^d$:
\[
P = {\mathrm{conv }} \{{\mathbf{v}}_1, \ldots, {\mathbf{v}}_n\} 
\]

\noindent \textbf{\textsf{The H-description}}. A \textbf{convex polytope} $P$ is a bounded intersection of finitely many halfspaces in a Euclidean space ${\mathbb{R}}^d$:
%subset of a Euclidean space ${\mathbb{R}}^d$ which may be described by a finite system of linear inequalities:
\begin{eqnarray*}
P & = & \left\{{\mathbf{x}} \in {\mathbb{R}}^d \, : \, {\mathbf{a}}_1 \cdot {\mathbf{x}} \leq b_1, \ldots, {\mathbf{a}}_m \cdot {\mathbf{x}} \leq b_m \right\} \\
&=& \left\{x \in {\mathbb{R}}^d \, : \, A {\mathbf{x}} \leq {\mathbf{b}}\right\}
\end{eqnarray*}
where ${\mathbf{a}}_1, \ldots, {\mathbf{a}}_m \in {\mathbb{R}}^d$ and $b_1, \ldots, b_m \in {\mathbb{R}}$. In the second expression we think of ${\mathbf{x}}$ as a column vector, $A$ is the $m \times d$ matrix with rows ${\mathbf{a}}_1, \ldots {\mathbf{a}}_m$ and ${\mathbf{b}}$ is the column vector with entries $b_1, \ldots, b_m$.

\begin{theorem}
A subset $P \subseteq {\mathbb{R}}^d$ is the convex hull of a finite set of points if and only if it is a bounded intersection of finitely many halfspaces. 
\end{theorem}

For theoretical and practical purposes, it is useful to have \textbf{both} the V-description and the H-description of a polytope. For example, it is clear from the H-description (but not at all from the V-description) that the intersection of two polytopes is a polytope, and it is clear from the V-description (but not at all from the H-description) that a projection of a polytope is a polytope. It is a nontrivial task to translate one description into the other; see \cite[Notes to Chapter 1]{f.Ziegler} for a discussion and references on this problem. Here are some simple examples.

\begin{example}\label{f.ex:polytopes}
The \textbf{standard simplex} $\Delta_{d-1}$, the \textbf{cube} $\square_d$, and the \textbf{crosspolytope} $\Diamond_d$ are:
\begin{eqnarray*}
1. \quad \Delta_{d-1} &=& {\mathrm{conv }}\{{\mathbf{e}}_1, \ldots, {\mathbf{e}}_d\}\\
&=& \{{\mathbf{x}} \in {\mathbb{R}}^d \, : \, x_1 + \cdots + x_d = 1 \textrm{ and } x_i \geq 0 \textrm{ for } i=1, \ldots, d\},
\\
2. \quad \quad \square_{d} &=& {\mathrm{conv }}\{
\pm {\mathbf{e}}_1 \pm \cdots \pm {\mathbf{e}}_d \textrm{ for any choice of signs}\} \\
%(\epsilon_1, \ldots, \epsilon_d) \, : \, \epsilon_1, \ldots \epsilon_d \in \{-1,1\} \} 
&=& \{{\mathbf{x}} \in {\mathbb{R}}^d \, : \, -1 \leq x_i \leq 1 \textrm{ for } i=1, \ldots, d\}, \qquad  \\
3. \quad \quad \Diamond_{d} &=& {\mathrm{conv }}\{-{\mathbf{e}}_1, {\mathbf{e}}_1, \ldots, -{\mathbf{e}}_d, {\mathbf{e}}_d\} \\
&=& \{{\mathbf{x}} \in {\mathbb{R}}^d \, : \,  \pm x_1 \pm \cdots \pm x_d \leq 1 \textrm{ for any choice of signs}%whenever } \epsilon_i \in \{-1,1\} \textrm{ for } i=1, \ldots, d
\},
\end{eqnarray*}
respectively, where ${\mathbf{e}}_1, \ldots, {\mathbf{e}}_d$ are the standard basis in ${\mathbb{R}}^d$. These polytopes are illustrated in Figure \ref{f.fig:simplexetc} for $d=3$.
\end{example}

\begin{figure}[ht]
 \begin{center}
  \includegraphics[scale=.8]{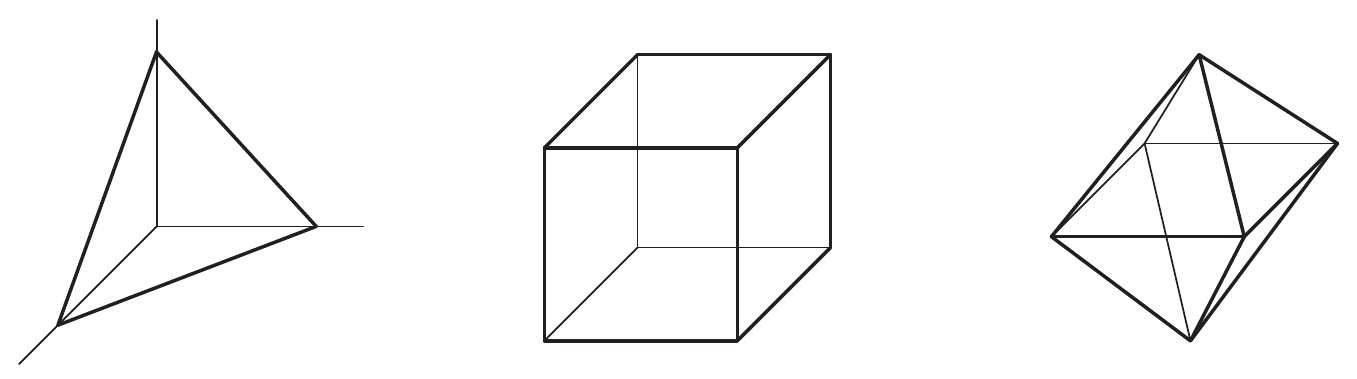}
  \caption{ \label{f.fig:simplexetc}
The triangle $\Delta_2$, the cube $\square_3$, and the octahedron $\Diamond_3$.}
  \end{center}
\end{figure}

The \textbf{dimension} of a polytope $P$ is the dimension of the affine subspace spanned by $P$:
\[
{\mathrm{aff }} (P) = \left\{\lambda_1 {\mathbf{v}}_1 + \cdots + \lambda_k {\mathbf{v}}_k \, : \, {\mathbf{v}}_1, \ldots, {\mathbf{v}}_k \in {\mathbb{R}}^d, \lambda_1, \ldots, \lambda_k \in {\mathbb{R}}, \lambda_1 + \cdots + \lambda_k = 0\right \}.
\]
For example, $\dim (\Delta_{d-1}) = d-1$ because $\Delta_{d-1}$ 
%even though $\Delta_{d-1}$ lives in ${\mathbb{R}}^d$, its dimension is $d-1$ because it 
lies on the hyperplane $x_1 + \cdots + x_d = 1$ in ${\mathbb{R}}^d$.

The \textbf{interior} $\mathrm{int}(P)$ of $P \subseteq {\mathbb{R}}^d$ is its topological interior. It is often more useful to consider the \textbf{relative interior} $\mathrm{relint}(P)$, which is its topological interior as a subset of the affine space ${\mathrm{aff }}(P)$.

\bigskip

\noindent \textbf{\textsf{Polar polytopes}}.
The similarity of our descriptions of $\square_d$ and $\Diamond_d$ is a manifestation of a  general notion of duality between V-descriptions and H-descriptions of polytopes.

Let $P$ be a polytope in ${\mathbb{R}}^d$ such that $\mathbf{0} \in {int}(P)$.\footnote{We can apply an affine transformation to any full-dimensional polytope to make its interior  contain the origin.}
The \textbf{polar polytope} $P^{\triangle}$ of  $P$ is 
\[
P^\triangle = \left\{{\mathbf{a}} \in {\mathbb{R}}^d \, : \, {\mathbf{a}} \cdot {\mathbf{x}} \leq 1 \textrm{ for all } {\mathbf{x}} \in P\right\}.
\]
If $P$ has vertices ${\mathbf{v}}_1, \ldots {\mathbf{v}}_n$ and facets ${\mathbf{a}}_1 \cdot {\mathbf{x}} \leq 1, \ldots, {\mathbf{a}}_m \cdot {\mathbf{x}} \leq 1$ (and hence contains $\mathbf{0}$), then the polar $P^\triangle$ has vertices ${\mathbf{a}}_1, \ldots {\mathbf{a}}_m$ and facets ${\mathbf{v}}_1 \cdot {\mathbf{x}} \leq 1, \ldots, {\mathbf{v}}_n \cdot {\mathbf{x}} \leq 1$. It follows that if $\mathbf{0} \in P$, then
\[
(P^\triangle)^\triangle = P.
\]

\medskip

\noindent \textbf{\textsf{Faces}}.
For each vector ${\mathbf{a}} \in {\mathbb{R}}^d$, let  
\[
P_{\mathbf{a}} = \left \{{\mathbf{x}} \in P \, : \, {\mathbf{a}} \cdot {\mathbf{x}}  \geq {\mathbf{a}} \cdot {\mathbf{y}} \textrm{ for all } {\mathbf{y}} \in P\right\}
\]
be the subset of $P$ where the linear function ${\mathbf{a}} \cdot {\mathbf{x}}$ is maximized. Such a set is called a \textbf{face} of $P$. %It is customary to allow 
We consider the empty set to be a face of $P$; faces other than $\emptyset$ and $P$ are called \textbf{proper faces}.
If $\dim(P) = d$, the faces of dimension $0, 1, d-2, d-1$ are called \textbf{vertices}, \textbf{edges}, \textbf{ridges}, and \textbf{facets}, respectively. We let $V(P)$ be the set of vertices of $P$. We collect a few basic facts  about faces:

\smallskip

 $\bullet$
A polytope is the convex hull of its vertices: $P = {\mathrm{conv }}(V(P))$.

 $\bullet$
A polytope is the intersection of the halfspaces determined by its facets.

 $\bullet$
The vertices of a face $F$ of $P$ are the vertices of $P$ contained in $F$: $V(F) = V(P) \cap F$.

 $\bullet$
 A face $F$ of $P$ equals the intersection of the facets of $P$ containing $F$. 

 $\bullet$
If $F$ is a face of $P$, then any face of $F$ is a face of $P$.

 $\bullet$
The intersection of two faces of $P$ is a face of $P$.

$\bullet$
A polytope is the disjoint union of the relative interiors of its faces: $\displaystyle P=\bigcup_{F \textrm{ face }} \mathrm{relint}(F)$.

\medskip

\noindent \textbf{\textsf{The face lattice}}.
The \textbf{face lattice} $L(P)$ is the poset of faces of $P$, ordered by containment. It is indeed a lattice with $F \wedge G = F \cap G$ and $F \vee G = {\mathrm{aff }}(F \cup G) \cap P$. It is graded with $\mathbf{rk}(F) = \dim F + 1$. We say that polytopes $P$ and $Q$ are \textbf{combinatorially isomorphic} if $L(P) \cong L(Q)$. We collect some basic properties of face lattices:

\smallskip

$\bullet$ 
For each face $F$ of $P$, the interval $[{\widehat{0}}, F]$ of $L(P)$ is isomorphic to the face lattice of $F$.

$\bullet$ 
For each face $F$ of $P$, the interval $[F, {\widehat{1}}]$ of $L(P)$ is isomorphic to the face lattice of a polytope, called the \textbf{face figure} $P/F$ of $F$.%\footnote{There is no canonical choice face figure is only determined  

$\bullet$
Every interval $[F,G]$ of $L(P)$ is isomorphic to the face lattice of some polytope.

$\bullet$
The face lattice $L(P^\triangle)$ of the polar polytope $P^\triangle$ is isomorphic to the opposite poset $L(P)^{\textrm{op}}$, obtained by reversing the order relations of $L(P)$.

\medskip

There are different constructions of the face figure $P/F$, giving rise to combinatorially isomorphic polytopes. We discuss one such construction. Let $F^\Diamond$ be the face of $P^\triangle$ corresponding to the face $F$ of $P$ under the isomorphism $L(P^\triangle) \cong L(P)^{\mathrm{op}}$. Then we can define the face figure to be the polar polytope of $F^\Diamond$; that is, $P/F = (F^\Diamond)^\triangle$. When $F={\mathbf{v}}$ is a vertex, there is a more direct construction:  let $P/{\mathbf{v}} = P \cap H$, where $H$ is a hyperplane separating $v$ from all other vertices of $P$. 

\smallskip

We say a $d$-polytope $P$ is \textbf{simplicial} if every face is a simplex. We say $P$ is \textbf{simple} if every vertex is on exactly $d$ facets (or, equivalently, on $d$ edges). Note that the convex hull of generically chosen points is a simplicial polytope. Similarly, a bounded intersection of generically chosen half-spaces is a simple polytope. Also note that $P$ is simplicial if and only if its polar $P^\triangle$ is simple.

\bigskip

\noindent \textbf{\textsf{Triangulations and subdivisions}}. In many contexts, it is useful to subdivide a polytope into simpler polytopes (most often simplices). It is most convenient to do it in such a way that the pieces of the subdivision meet face to face:

A \textbf{subdivision} of a polytope $P$ is a finite collection ${\mathcal{T}}$ of polytopes such that 

$\bullet$ $P$ is the union of the polytopes in ${\mathcal{T}}$,

$\bullet$ if $P \in {\mathcal{T}}$ then every face of $P$ is in ${\mathcal{T}}$, and

$\bullet$ if $P, Q \in {\mathcal{T}}$ then $P \cap Q \in {\mathcal{T}}$.

\noindent The elements of ${\mathcal{T}}$ are called the \textbf{faces} of ${\mathcal{T}}$; the full-dimensional ones are called \textbf{facets}. If all the faces are simplices, then ${\mathcal{T}}$ is called a \textbf{triangulation} of $P$.

In many situations, it is useful to assume that a subdivision does not introduce new vertices; that is, that the only points in ${\mathcal{T}}$ are the vertices of $P$. We will assume that throughout the rest of this section.

\begin{theorem}
Every convex polytope has a triangulation.
\end{theorem}

\begin{proof}[Sketch of Proof.]
Let $V$ be the set of vertices of our polytope $P \in {\mathbb{R}}^d$. For each \textbf{height function} $h:V \rightarrow {\mathbb{R}}$ consider the set of lifted points $P^h = \{(v, h(v)), v \in V\}$ in ${\mathbb{R}}^{d+1}$. Let $Q$ be the convex hull of $P^h$, and consider the set $\mathcal{F}$ of ``lower facets" of $Q$ that are visible from below, that is, the facets maximizing some linear function ${\mathbf{a}} \cdot {\mathbf{x}}$ with $a_{d+1} = -1$. For each such facet $F$, let $\pi(F)$ be its projection back down to ${\mathbb{R}}^d$. One may check that $\{\pi(F) \, : \, F \textrm{ is a lower facet of }Q\}$ is a subdivision of $P$. If the height function is chosen generically, this subdivision is actually a triangulation.
\end{proof}

\begin{figure}[ht]
 \begin{center}
  \includegraphics[scale=1.1]{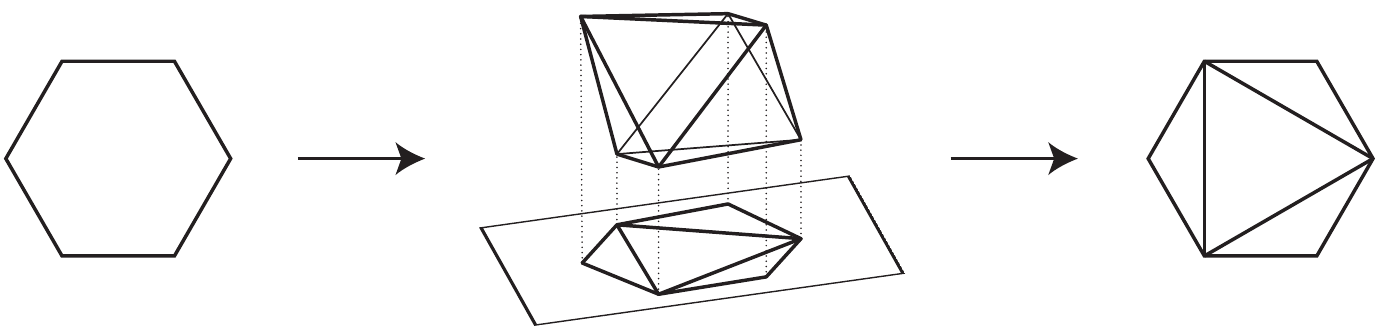}
  \caption{ \label{f.fig:regsubdiv}
A regular subdivision of a hexagon.}
  \end{center}
\end{figure}

The subdivisions that may be obtained from the lifting construction above are called \textbf{regular} or \textbf{coherent}. In general, not every subdivision is regular, and the question of distinguishing regular and non-regular triangulations is subtle. In the simplest case, when $P$ is a convex polygon, every subdivision is regular, and the number of subdivisions is a Catalan number, as we saw in Example 14(b). of Section \ref{f.sec:ogfexamples}. Few other exact enumerative results of this sort are known.

In any case, regular subdivisions are of great importance for several reasons. One reason is that they are easy to define and construct; for example, choosing the heights $h(v) = v_1^2 + \cdots v_d^2$ leads to the \textbf{Delaunay triangulation}, which is very easy to compute and has several desirable properties. Secondly, they have a very elegant structure. Every polytope $P$ gives rise to a \textbf{secondary polytope}, whose faces are in bijection with the regular subdivisions of $P$; and 
faces $F$ and $G$ satisfy that $F \subset G$ if and only if the subdivision of $F$ refines the subdivision of $G$. Thirdly, they are widely applicable, in particular, because they play a key role in the theory of Gr\"obner bases in commutative algebra. For our enumerative purposes the regularity question will not be too important. For readers interested in this and other aspects of triangulations, we recommend \cite{f.triangulations, f.GKZ, f.Sturmfelspolytopes}.

\subsection{\textsf{Examples}}\label{f.sec:polytopeexamples}

The following polytopes have particularly nice enumerative properties. We give references to some relevant results at the end of Section \ref{f.sec:Ehrhart}.

\begin{enumerate}

\item (Product of two simplices) $\Delta_{c-1} \times \Delta_{d-1} = {\mathrm{conv }}\{({\mathbf{e}}_i,{\mathbf{f}}_j)  \, : \, 1 \leq i \leq c, 1 \leq j \leq d\}$ in ${\mathbb{R}}^c \times {\mathbb{R}}^d$
where $\{{\mathbf{e}}_1, \ldots, {\mathbf{e}}_c\}$ and $\{{\mathbf{f}}_1, \ldots, {\mathbf{f}}_d\}$ are the standard bases for ${\mathbb{R}}^c$ and ${\mathbb{R}}^d$, respectively. 

%\item (Second hypersimplex) $\Delta(2, d) = {\mathrm{conv }}\{{\mathbf{e}}_i  + {\mathbf{e}}_j  \, : \, 1 \leq i < j \leq d\}$ in ${\mathbb{R}}^d$

\item (Hypersimplex) $\Delta(r, d) = {\mathrm{conv }}\{{\mathbf{e}}_{i_1} + \cdots + {\mathbf{e}}_{i_r}  \, : \, 1 \leq i_1 < \cdots < i_r \leq d\}$ in ${\mathbb{R}}^d$.

\item (Permutahedron) $\Pi_{d-1} = {\mathrm{conv }}\{(a_1, \ldots, a_d) \, : \, \{a_1, \ldots, a_d\} \textrm{ is a permutation of } [d]\}$.  

\item (Zonotopes) The zonotope of a vector configuration $A = \{{\mathbf{a}}_1, \ldots, {\mathbf{a}}_k\} \subset {\mathbb{R}}^d$ is the \emph{Minkowski sum} $Z(A) = [\mathbf{0},{\mathbf{a}}_1] + \cdots + [\mathbf{0},{\mathbf{a}}_k] :=  \{\lambda_1 {\mathbf{a}}_1 + \cdots + \lambda_k {\mathbf{a}}_k \, : \, 0 \leq \lambda_1, \ldots, \lambda_k \leq 1\}$. The permutahedron $\Pi_{d-1}$ is the zonotope of the \emph{root system} $\{{\mathbf{e}}_i  - {\mathbf{e}}_j  \, : \, 1 \leq i < j \leq d\}$.

\item (Cyclic polytope) $C_d(m) = {\mathrm{conv }}\{(1, t_i, t_i^2, \ldots, t_i^{d-1}) \, | \, 1 \leq i \leq m 
\}$ for $t_1 < \cdots < t_m$.

\item (Order polytope of a poset $P$) $\mathcal{O}(P) = \{{\mathbf{x}} \in {\mathbb{R}}^P \, : \, x_i \geq 0 \textrm{ for all $i$ and } x_i \leq x_j \textrm{ if $i<j$ in $P$}\}$.

\item (Chain polytope of a poset $P$) 
$\mathcal{C}(P) = \{{\mathbf{x}} \in {\mathbb{R}}^P \, : \, x_i \geq 0 \textrm{ for all  $i$ and } {x_{i_1} + \cdots + x_{i_k} \leq 1} \allowbreak \textrm{ for each chain  }  i_1< \cdots < i_k \textrm{ in } P\}$.

\item (Type A root polytope) $A_{d-1} = {\mathrm{conv }}\{{\mathbf{e}}_i  - {\mathbf{e}}_j  \, : \, 1 \leq i \neq j \leq d\}$ in ${\mathbb{R}}^d$.

\item (Type A positive root polytope) $A^+_{d-1} = {\mathrm{conv }}\{{\mathbf{e}}_i  - {\mathbf{e}}_j  \, : \, 1 \leq i < j \leq d\}$ in ${\mathbb{R}}^d$.

\item (Flow polytope) Given a directed graph $G=(V,E)$ and a vector ${\mathbf{b}} \in {\mathbb{R}}^V$, the \emph{flow polytope} is 
$F_G({\mathbf{b}}) = \{{\mathbf{f}} \in {\mathbb{R}}^E \, : \, f_e \geq 0 \textrm{ for all $e \in E$ and } \sum_{vw \in E} f_{vw} - \sum_{uv \in E} f_{uv} = b_v \textrm{ for all } v \in V\}$.
We think of $f_e$ as a \emph{flow} on edge $e$, so that the \emph{excess flow} or \emph{leak} at each vertex $v$ equals $b_v$.

\item (CRY polytope) $\textrm{CRY}_n = F_{K_{n+1}}(1, 0, \ldots, 0, -1)$ where $K_{n+1}$ is the complete graph on $[n+1]$ with edges directed $i \rightarrow j$ for $i<j$.

\item (Associahedron) $\textrm{Assoc}_{d-1}$ is a polytope whose faces are in bijection with the ways of subdividing a convex $(d+3)$-gon into polygons without introducing new vertices. Faces $F$ and $G$ satisfy that $F \subset G$ if and only if the subdivision of $F$ refines the subdivision of $G$. There are several different polytopal realizations of $\textrm{Assoc}_{d-1}$; see \cite{f.CeballosZiegler} for a survey. 

\item (Matroid polytope) If $M=\{v_1, \ldots, v_d\}$ is a set of vectors spanning a vector space ${{\mathbbm{k}}}^r$, the \emph{matroid (basis) polytope} is $P_M = {\mathrm{conv }}\{{\mathbf{e}}_{i_1} + \cdots + {\mathbf{e}}_{i_r}  \, : \, \{v_{i_1}, \ldots , v_{i_r}\} \textrm{ is a basis of } {{\mathbbm{k}}}^r \}$ in ${\mathbb{R}}^d$. When $M$ is generic we get $P_M = \Delta(r,d)$.
This construction is better understood in the context of matroids; see Section \ref{f.sec:matroids}.

\item (Generalized permutahedra / polymatroids) Many interesting polytopes are deformations of the permutahedron, obtained by moving the vertices of $\Pi_{d-1}$ while respecting all the edge directions. Examples include the polytopes $\Delta_{d-1}, \,
\Delta_{e-1} \times \Delta_{d-e-1}, \, \Delta(r,d), \, \Pi_{d-1}, \, A_{d-1}, \, P_M,$ and $\textrm{Assoc}_{d-1}$ above.
Such polytopes are called \emph{polymatroid base polytopes} or \emph{generalized permutahedra}; see \cite{f.Fujishige, f.Postnikovgenperm}.
\end{enumerate}

Many of these polytopes are related to the permutation group $S_d$ and the corresponding \emph{type A root system} ${\mathrm{conv }}\{{\mathbf{e}}_i  - {\mathbf{e}}_j  \, : \, 1 \leq i \neq j \leq d\}$. Many have generalizations to the wider context of finite Coxeter groups and root systems; see for example \cite{f.ArdilaBecketal, f.Coxetermatroids, f.FominReading, f.FultonHarris, f.MeszarosMorales} and the references therein. 

\medskip

The faces of most of these polytopes can be described combinatorially. This should help us appreciate these polytopes, because there are not many interesting families of polytopes for which we can do this. We describe the most interesting ones.

\begin{enumerate}
\item[3.] {(Permutahedron)} The H-description of $\Pi_{d-1}$ is 
\[
%\Pi_{d-1} = \{{\mathbf{x}} \in {\mathbb{R}}^d \, : \,
 x_1 + \cdots + x_d = d(d+1)/2 , \qquad  x_{i_1} + \cdots + x_{i_k} \geq k(k+1)/2 \,\, \textrm{ for  } \emptyset \subsetneq  \{i_1, \ldots,  i_k\} \subsetneq [d]
\]
There is a bijection $\mathcal{S} \leftrightarrow F_{\mathcal{S}}$ between the ordered set partitions of $[d]$ and the faces of the permutahedron $\Pi_{d-1}$. %If $\mathcal{S} = S_1 | \ldots | S_k$, the vertices of $F_\mathcal{S}$ correspond to the permutations $

\item[4.] {(Zonotope)} The face enumeration of $Z(A)$ is the subject of the upcoming Theorem \ref{f.th:cd-indexzonotope}.

\item[5.] {(Cyclic polytope)} Let $\mathbf{t}_i = (1, t_i, \ldots, t_i^{d-1})$. The facets of the cyclic polytope $C_d(m)$ are the simplices ${\mathrm{conv }}\{\mathbf{t}_{s}\, : \, s \in S\}$ for the $d$-subsets $S \subseteq [m]$ satisfying \textbf{Gale's evenness condition}: between any $i<j$ not in $S$ there is an even number of elements of $S$. There is a similar description for all faces of the cyclic polytope. In particular, the combinatorics of $C_d(m)$ is independent of $t_1, \ldots, t_m$. Two other remarkable facts are the following.

$\bullet$ Every subset of at most $d/2$ vertices forms a face of $C_d(m)$. 

$\bullet$ (McMullen's Upper Bound Theorem \cite{f.McMullenShephard}) Among all $d$-polytopes with $m$ vertices, $C_d(m)$ maximizes the number of faces of dimension $k$ for all $2 \leq k \leq d-1$.

The rich theory of \emph{positroids} can be seen as a generalization of the study of cyclic polytopes; see Section \ref{f.sec:matroids} and \cite{f.ArdilaRinconWilliams2, f.PostnikovTNN} for this connection.

\item[6.] {(Order polytope)} The vertices of $\mathcal{O}(P)$ are $\sum_{i \notin I} {\mathbf{e}}_i$ for the order ideals $I \subseteq P$.

\item[7.] {(Chain polytope)} The vertices of $\mathcal{C}(P)$ are $\sum_{i \in A} {\mathbf{e}}_i$ for the antichains $A \subseteq P$.

\item[13.] {(Matroid polytope)} The matroid polytope is cut out (non-minimally) by the inequalities
\[
\sum_{e \in E} x_e = r(E), \qquad \sum_{e \in S} x_e \leq r(S) \,\, \textrm{ for } E \subset S.
\]
The facets are characterized in \cite{f.FeichtnerSturmfels}.

\item[14.] (Generalized permutahedra) There are several interesting result on various classes of generalized permutahedra; see \cite{f.PostnikovReinerWilliams, f.Postnikovgenperm}.

\end{enumerate}

%\bigskip

\subsection{{\textsf{Counting faces}}} \label{f.sec:countingfaces}

The \textbf{$f$-vector} of a $d$-polytope $P$ is 
\[
f_P=(f_0, f_1, \ldots, f_{d-1}, f_d)
\]
where $f_i$ is the number of $(i-1)$-dimensional faces of $P$ for $0 \leq i \leq d$.\footnote{Note that we include the empty face and omit the full-dimensional face $P$ in this enumeration.} The problem of characterizing the $f$-vectors of various kinds of polytopes (or more general polyhedral complexes) is a central one in combinatorics. We offer a very brief discussion; for more detailed accounts, see \cite{f.BilleraBjorner, f.MuraiNevo, f.Ziegler}. A one sentence summary is this: we completely understand the $f$-vectors of simplicial (or equivalently, of simple) polytopes, but we are far from understanding the $f$-vectors of general polytopes. There are many interesting results and (mostly) open questions in between.
\medskip

\noindent \textbf{\textsf{$f$-vectors.}} The most important result about $f$-vectors of arbitrary polytopes is the following.

\begin{theorem} \cite{f.McMullenShephard} (McMullen's Upper Bound Theorem)
For any polytope $P$ of dimension $d$  with $m$ vertices, we have
\[
f_i(P) \leq f_i(C_d(m)) \qquad \textrm{ for }i=0,1, \ldots, d-1,
\]
where $C_d(m)$ is the cyclic polytope.
\end{theorem}

\noindent \textbf{\textsf{$h$-vectors.}} 
If $P$ is a simplicial $d$-polytope (or more generally, any simplicial complex of dimension $d-1$), we define
%it is often better to encode the $f$-vector in 
the \textbf{$h$-vector} $h_P=(h_0, h_1, \ldots, h_d)$  by the equivalent equations:
\[
h_0x^d+h_1x^{d-1}+ \cdots + h_dx^0 = f_0(x-1)^d + f_1(x-1)^{d-1}+\cdots + f_d(x-1)^0, \quad \textrm{or}
\]
\[
h_i=\sum_{j=0}^i (-1)^{i-j}{d-j \choose i-j} f_{j} \qquad \textrm{ and }  \qquad f_{i} = \sum_{j=0}^i {d-j \choose i-j} h_j \qquad \textrm{ for } 0 \leq i \leq d.
\]
The $h$-vector is a more economical way of storing the $f$-vector, due to the \emph{Dehn-Sommerville relations}: $h_i = h_{d-i}$ for $0 \leq i \leq d$.

%This set of equations is best possible in the sense that it generates \textbf{all} the linear relations among the $f$-vectors of simplicial polytopes. 
In fact, the \textbf{$g$-theorem} characterizes completely the $f$-vectors of simplicial polytopes \cite{f.BilleraLee, f.Stanley.g-theorem}! This spectacular result (which McMullen conjectured, and described as \emph{``even more intriguing, if rather less plausible"} \cite{f.McMullenShephard}) is one of the most important achievements of algebraic and geometric combinatorics to date. To state it, we need some definitions.

%%It is not difficult to show that for 
%For any positive integers $a$ and $i$, there is a unique representation $a={a_i \choose i} + {a_{i-1} \choose i-1} + \cdots + {a_j\choose j}$ for some $a_i>a_{i-1}>\cdots>a_j \geq j \geq 1$. We then define $a^{\langle i \rangle} ={a_i \choose i+1} + {a_{i-1} \choose i} + \cdots + {a_j\choose j+1}$. 

We say a sequence of nonnegative integers $(m_0, m_1, \ldots, m_d)$ is an \textbf{M-sequence} if there exists a set $S$ of monomials in $x_1, \ldots, x_n$, containing exactly $m_i$ monomials of degree $i$ for $i=0,1, \ldots, d$, such that $m' \in S$ and $m|m'$ implies $m \in S$. Macaulay gave a numerical characterization, as follows.

For any positive integers $a$ and $i$, there is a unique representation $a={a_i \choose i} + {a_{i-1} \choose i-1} + \cdots + {a_j\choose j}$ for some $a_i>a_{i-1}>\cdots>a_j \geq j \geq 1$. We then define $a^{\langle i \rangle} ={a_i \choose i+1} + {a_{i-1} \choose i} + \cdots + {a_j\choose j+1}$. Then $(m_0, m_1, \ldots, m_d)$ is an M-sequence if and only if $m_0=1$ and $m_{i+1} \leq m_i^{\langle i \rangle}$ for $i=1,2, \ldots, d-1$. 

\begin{theorem} \cite{f.BilleraLee, f.Stanley.g-theorem} (Billera-Lee-Stanley's $g$-theorem) A sequence $(h_0, \ldots, h_d)$ of positive integers is the $h$-vector of a simplicial $d$-polytope if and only if
\begin{enumerate}
\item
$h_i = h_{d-i}$ for $i=0,1,\ldots,\lfloor d/2 \rfloor$ (Dehn-Somerville equations), and
\item
$(g_0, g_1, \ldots, g_{\lfloor d/2 \rfloor})$ is an $M$-sequence, where $g_0=h_0$ and $g_i=h_i-h_{i-1}$ for $1 \leq i \leq \lfloor d/2 \rfloor$.
\end{enumerate}
\end{theorem}

For general polytopes, the situation is much less clear, even in dimension $4$, a case which has been studied extensively. \cite{f.Ziegler4polytopes} Stanley \cite{f.Stanley.generalizedh-vectors} defined a more subtle \textbf{toric $h$-vector}. It coincides with the $h$-vector when $P$ is simplicial, and it also satisfies the Dehn-Sommerville equations. His definition may be seen as a combinatorial formula for the dimensions of the intersection cohomology groups of the corresponding projective toric variety (if $P$ is a rational polytope). In a different direction, for ``cubical" polytopes $P$ whose proper faces are all cubes, Adin \cite{f.Adin} defined the \textbf{cubical $h$-vector}, which also satisfies the Dehn-Sommerville relations. Characterizing these toric and cubical $h$-vectors is an important open problem.

\bigskip

\noindent \textbf{\textsf{Flag $f$-vectors and $h$-vectors.}} The \emph{flag $f$-vector} $(f_D)_{D \subseteq [0,d-1]}$ of a $d$-polytope enumerates the flags $F_1 \subset \cdots \subset F_k$ of given dimensions $D=\{d_1<\cdots < d_k\}$ for all $D \subseteq [0,d-1]$. As we described in Section \ref{f.sec:Eulerian} (in the wider context of Eulerian posets), this information can be more economically stored in the ${\mathbf{c}}{\mathbf{d}}$-index of $P$. This encoding incorporates all linear relations among the flag $f$-vector.

As we saw in Theorem \ref{f.th:cd-index}, the ${\mathbf{c}}{\mathbf{d}}$-index of a polytope $P$ is non-negative; this is not true for general Eulerian posets. Another very interesting question, which is wide open, is to classify the ${\mathbf{c}}{\mathbf{d}}$-indices of polytopes.

\bigskip

\subsection{{\textsf{Counting lattice points: Ehrhart theory}}}\label{f.sec:Ehrhart} In this section we are interested in ``measuring" a polytope $P$ by counting the lattice points in its integer dilations $P, 2P, 3P, \ldots$. We limit our attention to \textbf{lattice polytopes}, whose vertices are lattice points, and to \textbf{rational polytopes}, whose vertices have rational coordinates; at the moment there is no good theory for general polytopes.

\begin{theorem}\label{f.th:Ehrhart} (Ehrhart's Theorem) Let $P \subset {\mathbb{R}}^d$ be a lattice polytope.
There are polynomials $L_P(x)$ and $L_{P^o}(x)$ of degree $\dim P$, called the \textbf{Ehrhart polynomial} and \textbf{interior Ehrhart polynomial} of $P$, such that: 
\begin{enumerate}
\item
The number of lattice points in the $n$th dilation of $P$ and its interior $P^o$ are
\[
L_P(n) = | nP \cap {\mathbb{Z}}^d| \qquad L_{P^o}(n) = | nP^o \cap {\mathbb{Z}}^d|   \qquad \textrm{ for all } n \in {\mathbb{N}}.
\]
\item 
The \textbf{Ehrhart reciprocity law} holds:
\[
(-1)^{\dim P} L_P(-x) = L_{P^o}(x).
\]

%There exists a polynomial $L_P(x)$ of degree $\dim P$, called the \textbf{Ehrhart polynomial} of $P$, such that 
%\begin{enumerate}
%\item
%the number of lattice points in the $n$th dilation of $P$ is
%\[
%L_P(n) = | nP \cap {\mathbb{Z}}^d| \qquad \textrm{ for all } n \in {\mathbb{N}}.
%\]
%\item
%the number of lattice points in the interior of the $n$th dilate of $P$ is 
%\[
%(-1)^{\dim P} L_P(-n) = |\mathrm{int}(nP) \cap {\mathbb{Z}}^d| \qquad \textrm{ for all } n \in {\mathbb{N}}.
%\]
\end{enumerate}
\end{theorem}

\noindent 
\textbf{Note.} If $P$ is a rational polytope, then $L_P(n)$ is instead given by a \textbf{quasipolynomial}; that is, there exist an integer $m$ and polynomials $L_1(x), \ldots, L_m(x)$ such that $|nP \cap {\mathbb{Z}}^d| = L_k(n)$ whenever $n \equiv k \, (\textrm{mod } m)$.

\begin{proof}[Sketch of Proof of Theorem \ref{f.th:Ehrhart}] By working in one dimension higher, we can consider the various dilations of $P$ all at once. We embed ${\mathbb{R}}^d$ into ${\mathbb{R}}^{d+1}$ by mapping ${\mathbf{v}}$ to $({\mathbf{v}}, 1)$, and consider the cone
\[
{\mathrm{cone }}(P) = \{\lambda_1 ({\mathbf{v}}_1, 1) + \cdots + \lambda_n ({\mathbf{v}}_n,1) \, : \, \lambda_1, \ldots, \lambda_n \geq 0\}
\]

\begin{figure}[ht]
 \begin{center}
  \includegraphics[scale=1]{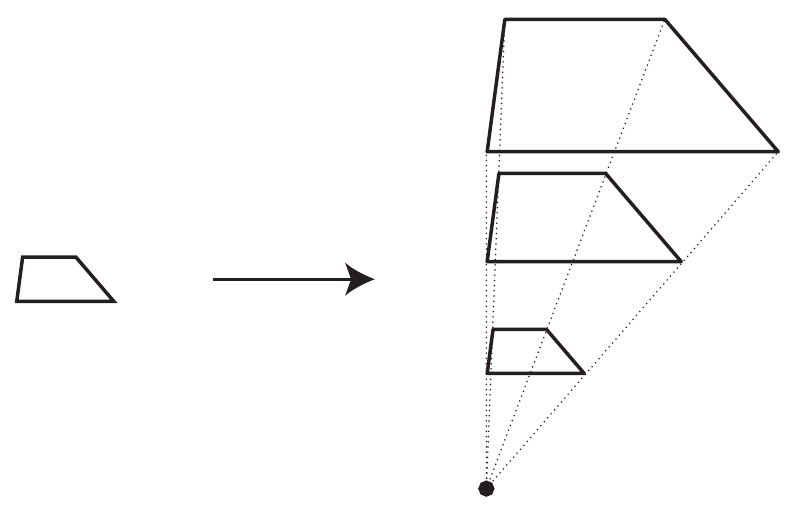}
  \caption{ \label{f.fig:Ehrhart}
The cone of a polytope and its slices at height $0,1,2,3$.}
  \end{center}
\end{figure}

\noindent 
Then for each $n \in {\mathbb{N}}$, the slice $x_{d+1}=n$ of ${\mathrm{cone }}(P)$ is a copy of the dilation $nP$, as shown in Figure \ref{f.fig:Ehrhart}. Two key ingredients of the proof will be the \textbf{lattice point enumerator} of ${\mathrm{cone }}(P)$ and the \textbf{Ehrhart series} of $P$:
\begin{equation}\label{f.eq:Ehr}
\sigma_{{\mathrm{cone }}(P)}({\mathbf{z}}) = \sum_{{\mathbf{p}} \in {\mathrm{cone }}(P)} {\mathbf{z}}^{\mathbf{p}}, \qquad \qquad {\mathrm{Ehr}}_P(z) = \sum_{n \geq 0} L_P(n) z^n = \sigma_{{\mathrm{cone }}(P)}(1, \ldots, 1, z),
\end{equation}
where ${\mathbf{z}}^{\mathbf{p}} = z_1^{p_1} \cdots z_{d+1}^{p_{d+1}}$.
%, and the \textbf{Ehrhart series}
%\begin{equation}\label{f.eq:Ehr}
%{\mathrm{Ehr}}_P(z) = \sum_{n \geq 0} L_P(n) z^n = \sigma_{{\mathrm{cone }}(P)}(1, \ldots, 1, z). 
%\end{equation}

We first prove the theorem for simplices, and then use triangulations to prove the general case.

\medskip

\noindent \emph{Step 1. (Simplices)} 
First we prove the result when $P$ is a simplex, which we may assume is full-dimensional. 
%%By Theorem ???, (1) is equivalent to proving that the \textbf{Ehrhart series}
%\[
%{\mathrm{Ehr}}_P(z) = \sum_{n \geq 0} L_P(n) z^n 
%\]
%is a rational function with denominator $(1-z)^{\dim P+1}$. Notice that
%\[
%{\mathrm{Ehr}}_P(z) = \sigma_{{\mathrm{cone }}(P)}(1, \ldots, 1, z). 
%\]
In this case ${\mathrm{cone }}(P)$ is an orthant generated by $d+1$ linearly independent vectors ${\mathbf{w}}_1, \ldots, {\mathbf{w}}_{d+1}$, where ${\mathbf{w}}_i = ({\mathbf{v}}_i, 1)$. 

\smallskip

1. Define the \textbf{lower} and \textbf{upper fundamental (half-open) parallelepipeds}
\begin{eqnarray*}
\Pi &=& \{ \lambda_1 {\mathbf{w}}_1 + \cdots + \lambda_{d+1} {\mathbf{w}}_{d+1} \, : \, 0 \leq \lambda_1, \ldots, \lambda_{d+1} < 1\} \\
\Pi^+ &=& \{ \lambda_1 {\mathbf{w}}_1 + \cdots + \lambda_{d+1} {\mathbf{w}}_{d+1} \, : \, 0 < \lambda_1, \ldots, \lambda_{d+1} \leq 1\}.
\end{eqnarray*}
Then ${\mathrm{cone }}(P)$ is tiled by the various non-negative ${\mathbf{w}}$-integer translates of $\Pi$, namely, the parallelepipeds %$\Pi_\k = 
$\Pi + k_1{\mathbf{w}}_1 + \cdots + k_{d+1}{\mathbf{w}}_{d+1}$ for ${\mathbf{k}}=(k_1, \ldots, k_{d+1}) \in {\mathbb{N}}^{d+1}$. Similarly, the interior $\textrm{int}({\mathrm{cone }}(P))$ is tiled by the various non-negative ${\mathbf{w}}$-integer translates of $\Pi^+$. 
Therefore
\begin{eqnarray*}
\sigma_{{\mathrm{cone }}(P)}({\mathbf{z}}) %&=& \sum_{{\mathbf{k}} \in {\mathbb{N}}^{d+1}} \sigma_{\Pi_{\mathbf{k}}}(\z) {\mathbf{z}}
&=& \sigma_\Pi({\mathbf{z}})\left(\frac{1}{1-{\mathbf{z}}^{{\mathbf{w}}_1}}\right) \cdots \left(\frac{1}{1-{\mathbf{z}}^{{\mathbf{w}}_{d+1}}}\right) \\
\sigma_{\textrm{int}({\mathrm{cone }}(P))}({\mathbf{z}}) %&=& \sum_{{\mathbf{k}} \in {\mathbb{N}}^{d+1}} \sigma_{\Pi^+_{\mathbf{k}}}({\mathbf{z}}) = 
&=& \sigma_{\Pi^+}({\mathbf{z}})\left(\frac{1}{1-{\mathbf{z}}^{{\mathbf{w}}_1}}\right) \cdots \left(\frac{1}{1-{\mathbf{z}}^{{\mathbf{w}}_{d+1}}}\right).
\end{eqnarray*}
Using (\ref{f.eq:Ehr}) we get
\[
{\mathrm{Ehr}}_P(z) = \frac{\sigma_\Pi(1, \ldots, 1, z)}{(1-z)^{d+1}}, \qquad {\mathrm{Ehr}}_{P^o}(z) = \frac{\sigma_{\Pi^+}(1, \ldots, 1, z)}{(1-z)^{d+1}},
\]
which are rational functions. 
Then, by Theorem \ref{f.th:polynomial}, $L_P(n)$ and $L_{P^o}(n)$ are polynomial functions of $n$ of degree $d$, as desired.

\smallskip

2. Now observe that $\sigma_{\Pi^+}({\mathbf{z}}) = {\mathbf{z}}^{{\mathbf{w}}_1 + \cdots + {\mathbf{w}}_{d+1}} \sigma_{\Pi}(1/z_1, \ldots, 1/z_{d+1})$, because $\Pi^+$ is the translation of $-\Pi$ by  the vector ${\mathbf{w}}_1 + \cdots + {\mathbf{w}}_{d+1}$. This gives
\[
{\mathrm{Ehr}}_{P^o}(z) + (-1)^d {\mathrm{Ehr}}_P(1/z) = 0.
\]
It remains to invoke the fact that if $f$ is a polynomial, then $F^+(z) = \sum_{n \geq 0} f(n)z^n$ (which is a rational function of $z$) and $F^-(z) = \sum_{n < 0} f(n)z^n$ (which is a rational function of $1/z$, and hence of $z$) satisfy $F^+(z)+F^-(z) = 0$ as rational functions. This implies that $(-1)^{\dim P} L_P(-n) = L_{P^o}(n)$ for every positive integer $n$, and hence these two polynomials are equal.

\medskip

\noindent \emph{Step 2. (The general case)} Now let $P$ be a general lattice polytope, and let ${\mathcal{T}}$ be a triangulation of $P$. Recall that ${\mathcal{T}}^o$ is the set of non-boundary faces of ${\mathcal{T}}$. 

\smallskip

1. We have 
\[
L_P(n) = \sum_{F \in {\mathcal{T}}} L_{F^o}(n), \qquad L_{P^o}(n) = \sum_{F \in {\mathcal{T}}^o} L_{F^o}(n) 
\]
which implies that $L_P$ and $L_{P^o}$ are polynomials of degree $\dim P$, by Step 1.

\smallskip

2. Let $\widehat{{\mathcal{T}}}$ be the face poset of ${\mathcal{T}}$, with an additional maximum element ${\widehat{1}}$. We have the M\"obius dual relations 
\[
L_F = \sum_{G \leq F} L_{G^o}, \qquad L_{F^o} = \sum_{G \leq F} \mu(G,F)L_{G} \qquad \textrm{ for all } F \in {\mathcal{T}}
\]
omitting the arguments of the polynomials in question. This gives us  
\[
L_P = \sum_{F \in {\mathcal{T}}} L_{F^o} = \sum_{F \in {\mathcal{T}}} \sum_{G \leq F} \mu(G,F)L_{G} =  \sum_{G \in {\mathcal{T}}} -\mu(G, {\widehat{1}}) L_{G} = \sum_{G \in {\mathcal{T}}^o} (-1)^{\dim P - \dim G} L_G.
\]
where in the last step we are invoking Theorem \ref{f.th:Mobiusformulas}.11. Now, using that 
the simplices $G \in {\mathcal{T}}^o$ satisfy Ehrhart reciprocity: we obtain
\[
L_P(-n) = \sum_{G \in {\mathcal{T}}^o} (-1)^{\dim P - \dim G} L_G(-n) = (-1)^{\dim P} \sum_{G \in {\mathcal{T}}^o} L_{G^o}(n) = (-1)^{\dim P} L_{P^o}(n),
\]
as desired.
\end{proof}

Next we observe that the volume of a lattice polytope can be recovered from its Ehrhart polynomial. If $\dim P < d$, there is a small subtlety: Let $V$ be the affine span of $P$, and $\Lambda = V \cup {\mathbb{Z}}^d$. Then we need to normalize the volumes along $V$, so that any ``unit" cube, generated by a ${\mathbb{Z}}$-basis of $\Lambda$, has normalized volume $1$. We let ${\mathrm{vol}}(P)$ denote the \textbf{normalized volume} or \textbf{lattice volume} of $P$. When $P$ is full-dimensional, this is the usual volume.

For example, the segment from $(1,1)$ to $(4,7)$ has normalized volume (lattice length) $3$, because in this case the lattice $\Lambda$  is generated by the primitive vector $(1,2)$, and $(4,7) - (1,1) = \textbf{3}(1,2)$.

\begin{proposition} If $L_P(t) = c_dt^d + \cdots + c_1t+c_0$ is the Ehrhart polynomial of a lattice polytope $P$, then $c_d = {\mathrm{vol}}(P)$ and $c_0=1$.
\end{proposition}

\begin{proof}[Sketch of Proof.] We obtain better and better approximations of the volume of $P$ by choosing finer and finer grids $(\frac1n {\mathbb{Z}})^d$, and placing a cube of volume $\frac1{n^d}$ centered at each lattice point in $P \cap (\frac1n {\mathbb{Z}})^d$. Therefore
\[
{\mathrm{vol}}(P) = \lim_{n \rightarrow \infty} \frac1{n^d} \left| P \cap \left(\frac1n {\mathbb{Z}}\right)^d \right|= 
 \lim_{n \rightarrow \infty} \frac1{n^d} \left| nP \cap {\mathbb{Z}}^d \right|= 
\lim_{n \rightarrow \infty} \frac{L_P(n)}{n^d} = c_d.
\]
Since the only lattice point in $0P$ is the origin, $c_0 = L_P(0) = 1$.
\end{proof}

A lattice simplex $S$ in $\mathbb{R}^d$ with vertices $\mathbf{v}_1, \ldots, \mathbf{v}_k$ is \textbf{unimodular} if the vectors $\mathbf{v}_2 - \mathbf{v}_1, \ldots, \mathbf{v}_d - \mathbf{v}_1$ are a lattice basis for the lattice $\textrm{aff } S \cap \mathbb{Z}^d$. A triangulation is \textbf{unimodular} if all its simplices are unimodular. The following result shows that unimodular triangulations are particularly useful for enumerative purposes.

\begin{proposition}\label{f.prop:h*} 1. For every lattice polytope $P \subset {\mathbb{R}}^d$ the \textbf{Ehrhart $h^*$-polynomial}  $h^*_P(x) = h^*_0 + h^*_1 x + \cdots + h_d^* x^d$, which is defined by 
\[
{\mathrm{Ehr}}_P(z) = \frac{h^*_P(z)}{(1-z)^{d+1}}
\]
has nonnegative coefficients: $h^*_k \geq 0$ for all $k$. \\
2. If $P$ has a unimodular triangulation ${\mathcal{T}}$, then 
\[
h^*_P(z) = h_{\mathcal{T}}(z),
\]
where the \textbf{$h$-polynomial}  of ${\mathcal{T}}$ is the generating function $h_{\mathcal{T}}(z) = h_0 + h_1z + \cdots + h_dz^d$ for the \emph{$h$-vector} $(h_0, \ldots, h_d)$ of ${\mathcal{T}}$, 
defined in terms of the $f$-vector $(f_0, \ldots, f_d)$ of ${\mathcal{T}}$ as in Section 1.6.3.
%as defined in Section \ref{f.sec:countingfaces}. 
\end{proposition}

\begin{proof}
1. is due to Stanley \cite{f.Stanleyh*}; for a short proof see \cite[Theorem 3.12]{f.BeckRobins}. For 2., we have
\[
{\mathrm{Ehr}}_P(z) = \sum_{F \in {\mathcal{T}}} {\mathrm{Ehr}}_{F^o}(z) = \sum_{F \in {\mathcal{T}}} \frac{\sigma_{\Pi^+_F(1, \ldots, 1, z)}}{(1-z)^{\dim F + 1}}
\]
where $\Pi_+(F)$ is the upper fundamental parallelepiped of $F$. Since $F$ is unimodular, the only lattice point in $\Pi^+_F$ is the sum of its generators, which is at height $\dim F + 1$, so
\[
{\mathrm{Ehr}}_P(z) = \sum_{F \in {\mathcal{T}}} \left(\frac{z}{1-z}\right)^{\dim F + 1} = \sum_{k=0}^{d+1}f_{k}\left(\frac{z}{1-z}\right)^k = \frac{\sum_{k=0}^{d+1} h_kz^k}{(1-z)^{d+1}}
\]
as desired.
\end{proof}

\noindent {\textbf{\textsf{Examples}}.}
The following polytopes have particularly nice Ehrhart polynomials and Ehrhart series. %We give references for the more involved results.

\begin{enumerate}
\item (Polygon) $P$ a lattice polygon in ${\mathbb{R}}^2$:
\[
L_P(n) = \left(I + \frac{B}2 - 1\right) n^2 + \frac{B}2 n + 1, \qquad {\mathrm{Ehr}}_P(z) = \frac{Iz^2 + (I + B - 3)z + 1}{(1-z)^3},
\]
where $P$ contains $I$ lattice points in its interior and $B$ lattice points on its boundary. This is equivalent to Pick's formula for the area of a lattice polygon.

\item (Simplex) $\Delta_{d-1}$:
\[
L_{\Delta_{d-1}}(n) = {n+d-1 \choose d-1}, \qquad {\mathrm{Ehr}}_{\Delta_{d-1}}(z) = \frac1{(1-z)^d}. 
\]

\item (Cube) $\square_d$:
\[
L_{\square_d}(n) = (n+1)^d, \qquad {\mathrm{Ehr}}_{\square_d}(z) = \frac{\sum_{k=0}^d A(d,k) z^k}{(1-z)^{d+1}},
\]
where the \textbf{Eulerian number} $A(d, k)$ is the number of permutations $\pi$ of $[d]$ with $k-1$ descents; that is, positions $i$ with $\pi(i) > \pi(i+1)$.

\item (Crosspolytope) $\Diamond_d$:
\[
L_{\Diamond_d}(n) = \sum_{k=0}^d 2^k{d \choose k}{n \choose k}, %{n + d - k \choose n-k, d-k, k}, 
\qquad {\mathrm{Ehr}}_{\Diamond_d}(z) = \frac{(1+z)^d}{(1-z)^{d+1}}.
\]

\item (Product of two simplices) $P=\Delta_{c-1} \times \Delta_{d-1}$:
\[
L_{P}(n) = {n+c-1 \choose c-1}{n+d-1 \choose d-1}, \qquad {\mathrm{Ehr}}_{P}(z) = \frac{\sum_{k=0}^{\min(c-1, d-1)} {c-1 \choose k} {d-1 \choose k} z^k}{(1-z)^{c+d-1}}. 
\]
Every triangulation of $ \Delta_{c-1} \times \Delta_{d-1}$ is unimodular, with $h$-vector given by the $h^*$-vector above. These triangulations are very interesting combinatorially, and play an important role in tropical geometry and other contexts; see \cite{f.ArdilaBilley, f.ArdilaDevelin, f.DevelinSturmfels, f.Santos} and the references therein.

%\item (Second hypersimplex) $\Delta(2, d)$:
%\[
%L_{\Delta_{2,d}}(n) = {2n+ d-1 \choose d-1} - d{n+d-2 \choose d-1}, \qquad 
%{\mathrm{Ehr}}_{\Delta_{2,d}}(n) = \frac{1+ \frac12 d(d-3)z + \sum_{i=2}^{\lfloor d/2 \rfloor} {d \choose 2i} z^i}{(1-z)^d}
%%\sum_{s=0}^{r-1} (-1)^s{d \choose s}{(r-s)n+d-1-s \choose d-1}
%\]

\item (Hypersimplex) $\Delta(r, d)$:
\[
L_{\Delta({r,d})}(n) = [z^{rn}]\left(\frac{1-z^{n+1}}{1-z}\right)^d, \qquad {\mathrm{vol}} (\Delta_{r,d}) = \frac{A(d-1, r)}{(d-1)!}
\]
where, again, $A(d-1, r)$ denotes the Eulerian numbers.
%number of permutations of $[d-1]$ with $k-1$ descents.
No simple formula is known for the Ehrhart series of $\Delta({r,d})$. There is a nice formula for the ``half-open hypersimplex"; see \cite{f.Lihypersimplex}. An elegant triangulation of $\Delta({r,d})$ was given (using four different descriptions) by Lam-Postnikov, Stanley, Sturmfels, and Ziegler; see \cite{f.LamPostnikov}. 

\item (Permutahedron) $P=\Pi_{d-1}$:
\[
L_{\Pi_{d-1}}(n) = \sum_{i=0}^{d-1} f_i n^i, \qquad {\mathrm{vol}}(\Pi_{d-1}) = d^{d-2}
\]
where $f_i$ is the number of forests on $[d]$ with $i$ vertices. \cite{f.Stanleyzonotope} 

\item (Zonotope) $P=Z(A)$:
\[
L_{Z(A)}(n) = n^rM\left(1+\frac1n, 1\right)
\]
where $r$ is the rank of $A$ and $M(x,y)$ is the \emph{arithmetic Tutte polynomial} of Section \ref{f.sec:arithmeticTutte}. \cite{f.Stanleyzonotope} This polynomial is difficult to compute in general; when $A$ is a root system, explicit formulas are given in \cite{f.ArdilaCastilloHenley}, based on the computation of Example 15 in Section \ref{f.sec:egfs}.

\item (Cyclic polytope) $P=C_d(t_1, \ldots, t_m)$:
\[
L_P(n) = \sum_{k=0}^d {\mathrm{vol}} \, C_k(t_1, \ldots, t_m) n^k.
\]
See \cite{f.Liucyclic}. The triangulations of the cyclic polytope are unusually well behaved; see \cite{f.triangulations, f.Rambau} and the references therein. In particular, when $m=d+4$, this is one of the few polytopes whose triangulations have been enumerated exactly (and non-trivially) \cite{f.AzaolaSantos}.

\item (Order polytope and chain polytope): $\mathcal{O}(P), \mathcal{C}(P)$:
\[
L_{\mathcal{O}(P)} = L_{\mathcal{C}(P)} = \Omega_P(n+1), \qquad {\mathrm{vol}}(\mathcal{O}(P)) = {\mathrm{vol}}(\mathcal{C}(P)) = e(P)/|P|!,
\] 
where $\Omega_P$ is the order polynomial of $P$ and $e(P)$ is the number of linear extensions, as discussed in Section \ref{f.sec:orderpoly}. \cite{f.Stanleyposetpolytopes} Remarkably, $\mathcal{O}(P)$ and $\mathcal{C}(P)$ have the same Ehrhart polynomial, even though they are not metrically, or even combinatorially equivalent in general. There is a nice characterization of the posets $P$ such that  $\mathcal{O}(P)$ and $\mathcal{C}(P)$ may be obtained from one another by a unimodular change of basis \cite{f.HibiLi}.

\item (Root polytope): $A_{d-1}$:
\[
{\mathrm{Ehr}}_{A_d}(z) = \frac{\sum_{k=0}^d {d \choose k}^2 z^k}{(1-x)^{d}}, \qquad {\mathrm{vol}}(A_d) = \frac{{2d \choose d}}{d!}.
\]
There are similar formulas for the other classical root polytopes $B_d, C_d, D_d$, as well as for the positive root polytopes. For example, ${\mathrm{vol}}(A^+_{d}) = C_d/d!$ where $C_d$ is the $d$th Catalan number. Explicit unimodular triangulations were constructed in \cite{f.ArdilaBecketal}.

\item (CRY polytope / Flow polytope) $\textrm{CRY}_n = F_{K_{n+1}}(1, 0, \ldots, 0, -1)$:
\[
{\mathrm{vol}} (\textrm{CRY}_n) = C_0C_1C_2 \cdots C_{n-2},
\]
a product of Catalan numbers. \cite{f.ZeilbergerCRY}. No combinatorial proof of this fact is known. Flow polytopes are of great importance due to their close connection with the Kostant partition function, which Gelfand described as \emph{``the transcendental element
which accounts for many of the subtleties of the Cartan-Weyl theory"} of representations  of semisimple Lie algebras. \cite{f.Retakh}. There are many other interesting combinatorial results; see for example \cite{f.BaldoniVergne, f.MeszarosMorales}

\item (Matroid polytopes) A combinatorial formula for the volume of a matroid polytope is given in \cite{f.ArdilaBenedettiDoker}.

\item (Generalized permutahedra / polymatroids)
There are many interesting results about the volumes and lattice points of various families of generalized permutahedra; see \cite{f.Postnikovgenperm}.

\item (Cayley polytope) $\mathbf{C}_n = \{{\mathbf{x}} \in {\mathbb{R}}^n \, : \, 1 \leq x_1 \leq 2 \textrm{ and }1 \leq x_i \leq 2x_{i-1} \textrm{ for } 2 \leq i \leq n\}$:
\[
{\mathrm{vol}} (\mathbf{C}_n) = \frac{c_{n+1}}{n!}
\]
where $c_n$ is the number of connected graphs on $[n]$. The related \emph{Tutte polytope} has volume given by an evaluation of the Tutte polynomial (see Section \ref{f.sec:Tutte}) of the complete graph. \cite{f.KonvalinkaPak}
\end{enumerate}

%\comment{Maybe talk about Brion? Constant term identities? What else?}

\newpage

\section{\textsf{Hyperplane arrangements}}\label{f.sec:hyparrs}

We now discuss arrangements of hyperplanes in a vector space. The questions that we ask depend on whether the underlying field is $\mathbb{R}, \mathbb{C}$, or a finite field $\mathbb{F}_q$; but in every case, the underlying combinatorics plays an important role. The presentation of this section is heavily influenced by \cite{f.Stanleyhyparr}. 
See \cite{f.OrlikTerao} for a great introduction to more algebraic and topological aspects of the theory of hyperplane arrangements.

After developing the basics in Section \ref{f.sec:hyparrbasics}, in Section \ref{f.sec:charpoly} we introduce the \emph{characteristic polynomial}, which plays a crucial role in this theory. In Section \ref{f.sec:charpolyproperties} we discuss some of its important properties, and in Section \ref{f.sec:computechar} we develop several techniques for computing it. We illustrate these techniques by computing the characteristic polynomials of many arrangements of interest. Finally, in Section \ref{f.sec:arrcdindex}, we give a remarkable formula for the ${\mathbf{c}}{\mathbf{d}}$-index of an arrangement, which enumerates the flags of faces of given dimensions.

\subsection{\textsf{Basic definitions}}\label{f.sec:hyparrbasics}

Let ${\mathbbm{k}}$ be a field and $V = {\mathbbm{k}}^d$. A \textbf{hyperplane arrangement} ${\mathcal{A}}=\{H_1, \ldots, H_n\}$ is a collection of affine hyperplanes in $V$, say,
\[
H_i = \{ x \in V \, : \, v_i \cdot x = b_i\}
\]
for nonzero normal vectors $v_1, \ldots, v_n \in V$ and constants $b_1, \ldots, b_n \in {\mathbbm{k}}$. We say ${\mathcal{A}}$ is \textbf{central} if all hyperplanes have a common point -- in the most natural examples, the origin is a common point. Figure \ref{f.fig:arr} shows a central arrangement of 4 hyperplanes in ${\mathbb{R}}^3$.

\begin{figure}[ht]
 \begin{center}
  \includegraphics[scale=.45]{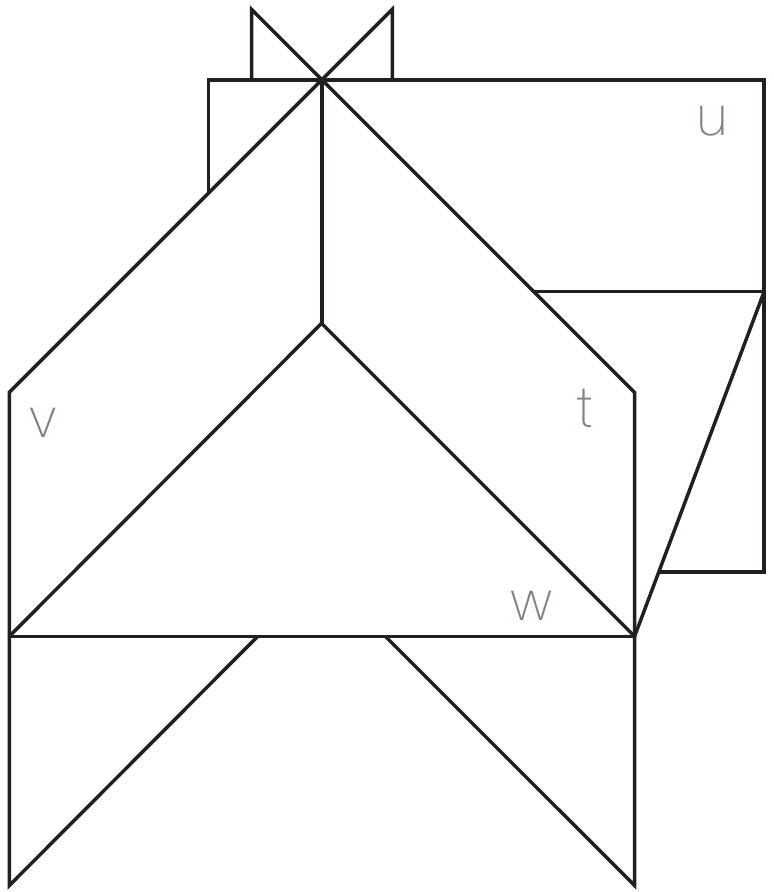}
  \caption{ \label{f.fig:arr}
A hyperplane arrangement.}
  \end{center}
\end{figure}

In some ways, central arrangements are slightly better behaved than affine arrangements. We can \textbf{centralize} an affine arrangement ${\mathcal{A}}$ in ${\mathbbm{k}}^n$ to obtain the \textbf{cone} of ${\mathcal{A}}$, an arrangement $c{\mathcal{A}}$ in ${\mathbbm{k}}^{n+1}$, by turning the hyperplane $a_1x_1 + \cdots + a_nx_n = a$ in ${\mathbbm{k}}^n$ into the hyperplane $a_1x_1 + \cdots + a_nx_n = ax_{n+1}$ in ${\mathbbm{k}}^{n+1}$, and adding the hyperplane $x_{n+1}=0$

Sometimes arrangements are ``too central", in the sense that their intersection is a subspace $L$ of positive dimension. In that case, there is little harm in intersecting our arrangement with the orthogonal complement $L^\perp$. We define the \textbf{essentialization} of ${\mathcal{A}}$ to be the arrangement ess$({\mathcal{A}}) = \{H \cap L^\perp  \, : \, H \in {\mathcal{A}}\}$ in $L^\perp$. The result is an \textbf{essential} arrangement, where the intersection of the hyperplanes is the origin. 
In most problems of interest, there is no important difference between ${\mathcal{A}}$ and ess$({\mathcal{A}})$.

%
%\comment{HERE}
%There is one technical subtlety that we must mention. 
%An arrangement ${\mathcal{A}}$ in $V$ is \textbf{essential} if $r({\mathcal{A}}) = {\mathbf{d}}im V$. When ${\mathcal{A}}$ is not essential, all hyperplanes $H$ contain a common lineality space $L$, that is, $H + L \subseteq H$. The essentialization of ${\mathcal{A}}$ is the arrangement ess$({\mathcal{A}}) = \{A / L \, : \, A \in H\}$ in $V/L$ obtained by modding out by this common lineality space. Clearly ${\mathcal{A}}$ and ess$({\mathcal{A}})$ have the same intersection poset. 

\bigskip

A key object is the \textbf{complement} 
\[
V({\mathcal{A}}) = V \, \backslash \, \left(\bigcup_{H \in {\mathcal{A}}} H\right).
\]
and we now introduce a polynomial which is a fundamental tool in the study of $V({\mathcal{A}})$.

\subsection{\textsf{The characteristic polynomial}} \label{f.sec:charpoly} There is a combinatorial polynomial which knows a tremendous amount about the complement $V({\mathcal{A}})$ of an arrangement ${\mathcal{A}}$. The kinds of questions that we ask about $V({\mathcal{A}})$ depend on the underlying field. 

\smallskip

$\bullet$
If ${\mathbbm{k}} = {\mathbb{R}}$ then every hyperplane $v_i \cdot x = b_i$ divides $V$ into two half-spaces, where $v_i \cdot x < b_i$ and $v_i \cdot x > b_i$, respectively. Therefore an arrangement ${\mathcal{A}}$ divides ${\mathbb{R}}^d$ into $a({\mathcal{A}})$ \textbf{regions}, which are the connected components of the complement ${\mathbb{R}}^d \backslash {\mathcal{A}}$. Let $b({\mathcal{A}})$ be the number of those regions which are bounded.\footnote{If ${\mathcal{A}}$ is not essential, we let $b({\mathcal{A}})$ be the number of \textbf{relatively bounded} regions. These are the regions which become bounded in the essentialization $\textrm{ess}({\mathcal{A}}).$}
 A central question about real hyperplane arrangements is to compute the numbers $a({\mathcal{A}})$ and $b({\mathcal{A}})$ of regions and bounded regions.

$\bullet$
If ${\mathbbm{k}} = {\mathbb{C}}$, then it is possible to walk around a hyperplane without crossing it; this produces a loop in the complement $V({\mathcal{A}})$. Therefore $V({\mathcal{A}})$ has nontrivial topology, and it is natural to ask for its Betti numbers.

$\bullet$
If ${\mathbbm{k}} = {\mathbb{F}}_q$ is a finite field, where $q$ is a prime power, then $V({\mathcal{A}})$ is a finite set, and the simplest question we can ask is how many points it contains.

Amazingly, the \textbf{characteristic polynomial} $\chi_{\mathcal{A}}(q)$ can answer these  questions immediately. Let us define it.

\bigskip

Define a \textbf{flat} of ${\mathcal{A}}$ to be an affine subspace obtained as an intersection of hyperplanes in ${\mathcal{A}}$. The \textbf{intersection poset} $L_{\mathcal{A}}$ is the set of flats partially ordered by reverse inclusion. If ${\mathcal{A}}$ is central, then $L_{\mathcal{A}}$ is a \textbf{geometric lattice}, as discussed in Section \ref{f.sec:lattices}. If ${\mathcal{A}}$ is not central, then $L_{\mathcal{A}}$ is only a \textbf{geometric meet semilattice} \cite{f.WachsWalker}. The \textbf{rank} $r=r({\mathcal{A}})$ of ${\mathcal{A}}$ is the height of $L_{\mathcal{A}}$.

\begin{figure}[ht]
 \begin{center}
  \includegraphics[scale=1.1]{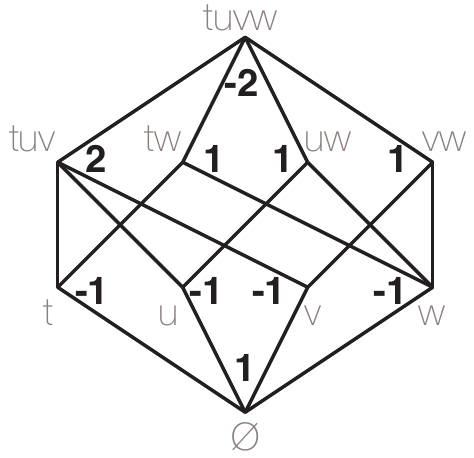}
  \caption{ \label{f.fig:intposet}
  The intersection poset of ${\mathcal{A}}$ and its M\"obius function.}
  \end{center}
\end{figure}

The \textbf{characteristic polynomial} of ${\mathcal{A}}$ is 
\[
\chi_A(q) = \sum_{F \in L_{\mathcal{A}}} \mu(\widehat{0}, F) q^{\dim F}.
\]
Figure \ref{f.fig:intposet} shows the intersection poset of the arrangement ${\mathcal{A}}$ in Figure \ref{f.fig:arr}; its characteristic polynomial $\chi_{\mathcal{A}}(q) = q^3 - 4q^2 + 5q - 2$ is easily computed by adding the M\"obius numbers on each level of $L_{\mathcal{A}}$.

%The reader may wish to compute some examples, such as the characteristic polynomial of an arrangement of lines in ${\mathbb{R}}^2$, or the arrangement of $n$ coordinate hyperplanes in ${\mathbbm{k}}^n$. 
We will see some general techniques to compute characteristic polynomials in Section \ref{f.sec:computechar}.

\begin{theorem}\label{f.th:charpoly} The characteristic polynomial $\chi_{\mathcal{A}}(x)$ contains the following information about the complement $V({\mathcal{A}})$ of a hyperplane arrangement ${\mathcal{A}}$.
\begin{enumerate}
\item $({\mathbbm{k}} = {\mathbb{R}})$
\cite{f.Zaslavsky} (Zaslavsky's Theorem)
Let ${\mathcal{A}}(A)$ be a \textbf{real} hyperplane arrangement in ${\mathbb{R}}^n$. The number of regions and relatively bounded regions of the complement $V({\mathcal{A}})$ are
\[
a({\mathcal{A}}) = (-1)^n\chi_{\mathcal{A}}(-1), \qquad b({\mathcal{A}}) = (-1)^{r({\mathcal{A}})}\chi_{\mathcal{A}}(1).
\]

\item $({\mathbbm{k}} = {\mathbb{C}})$
\cite{f.GoreskyMacPherson, f.OrlikSolomon}
Let ${\mathcal{A}}(A)$ be a \textbf{complex} hyperplane arrangement in ${\mathbb{C}}^n$. The complement $V({\mathcal{A}})$ has Poincar\'e polynomial
\[
\sum_{k \geq 0} \mathrm{rank } \, H^k(V(A), {\mathbb{Z}}) q^k = (-q)^n\chi_{\mathcal{A}}\left(\frac{-1}{q} \right).
\]

\item $({\mathbbm{k}} = {\mathbb{F}}_q)$
\cite{f.Athanasiadis, f.CrapoRota}
Let ${\mathcal{A}}(A)$ be a hyperplane arrangement in ${\mathbb{F}}_q^n$ where ${\mathbb{F}}_q$ is the finite field of $q$ elements for a prime power $q$. The complement $V({\mathcal{A}})$ has size
\[
| V(A) | = \chi_{\mathcal{A}}(q).
\]
\end{enumerate}
\end{theorem}

\begin{proof}
3. This is a typical enumerative problem where our set of objects (the points of ${\mathbb{F}}_q^n$) are stratified by a partial order, according to how special their position is with respect to ${\mathcal{A}}$. This is a natural setting to apply M\"obius inversion. 

For each flat $F$, let $f_=(F)$ be the number of points of ${\mathbb{F}}_q^n$ which are on $F$, and on no smaller flat. Since there are  $f_{\geq}(F) = q^{\dim F}$ points on $F$, we have
\[
q^{\dim F} = \sum_{G \geq F} f_=(G)
\]
which inverts to
\[
f_=(F) = \sum_{G \geq F} \mu(F,G) q^{\dim G}.
\]
Setting $F=\widehat{0}$ gives the desired result.

1. For each flat $F$ let ${\mathcal{A}}/F$ be the arrangement inside $F$ obtained by intersecting the hyperplanes ${\mathcal{A}}-F$ with $F$. The arrangement subdivides each flat $F$ into (relatively) open faces, namely, the $(\dim G)$-dimensional regions of ${\mathcal{A}}/G$ for each flat $G \geq F$. Since the Euler characteristic of $F \cong {\mathbb{R}}^{\dim F}$ is $(-1)^{\dim F}$, we get
\[
(-1)^{\dim F} = \sum_{G \geq F} (-1)^{\dim G} a({\mathcal{A}}/G),
\]
which inverts to 
\[
(-1)^{\dim F} a({\mathcal{A}}/F)= \sum_{G \geq F} \mu(F,G)(-1)^{\dim G}. 
\]
Setting $F=\widehat{0}$ gives the desired result. The same strategy works for $b({\mathcal{A}})$, using the result that the union of the bounded faces of ${\mathcal{A}}$ is contractible, and hence has Euler characteristic equal to $1$. \cite[Theorem 4.5.7(b)]{f.OMs}

2. is beyond the scope of this writeup; see \cite{f.GoreskyMacPherson, f.OrlikTerao}.
\end{proof}

\subsection{{\textsf{Properties of the characteristic polynomial}}}  \label{f.sec:charpolyproperties}

\noindent 
\textbf{\textsf{Whitney's formula and the Tutte polynomial.}} 
Since the intersection poset ${\mathcal{A}}$ is constructed from its atoms (the hyperplanes of ${\mathcal{A}}$), it is natural to invoke the Crosscut Theorem \ref{f.th:crosscut} and obtain the following result.

\begin{theorem} \label{f.th:Whitney} (Whitney's Theorem) The characteristic polynomial of an arrangement ${\mathcal{A}}$ is given by
\[
\chi_{{\mathcal{A}}}(q) = \sum_{\stackrel{{\mathcal{B}} \subseteq {\mathcal{A}}}{{\mathcal{B}} \textrm{ central}}} (-1)^{|{\mathcal{B}}|} \, q^{n - r({\mathcal{B}})} % = (-1)^r q^{n-r} T_{\mathcal{A}}(1-q,0).
\]
where $\displaystyle r({\mathcal{B}}) = n - \dim \cap_{H \in {\mathcal{B}}} H$.
\end{theorem}

The \textbf{Tutte polynomial} is another important polynomial associated to an arrangement:
\[
T_{\mathcal{A}}(x,y) = \sum_{\stackrel{{\mathcal{B}} \subseteq {\mathcal{A}}}{{\mathcal{B}} \textrm{ central}}} (x-1)^{r-r({\mathcal{B}})} \, (y-1)^{|{\mathcal{B}}|-r({\mathcal{B}})}.
\]
Whitney's Theorem can then be rephrased as:
\[
\chi_{{\mathcal{A}}}(q)  = (-1)^r q^{n-r} T_{\mathcal{A}}(1-q,0).
\]
%\comment{We will discuss the Tutte polynomial at length in Section \ref{f.sec:Tutte}. The results that we will present there will apply to the characteristic polynomial of an arrangement as well. (For affine arrangements some small adjustments are necessary; see \cite{f.yosemimatroids, f.yoTutte, f.WachsWalker}.)
%}
%
%\bigskip

\subsubsection{{\textsf{Deletion and contraction}}} A common technique for inductive arguments in hyperplane arrangement ${\mathcal{A}}$ is to choose a hyperplane $H$ and study how ${\mathcal{A}}$ behaves without $H$ (in the deletion ${\mathcal{A}} \backslash H$) and how $H$ interacts with the rest of ${\mathcal{A}}$ (in the contraction ${\mathcal{A}} / H$).

For a hyperplane $H$ of an arrangement ${\mathcal{A}}$ in $V$, the
 \textbf{deletion} 
 %and \textbf{contraction} of $H$ to be
%\begin{eqnarray*}
\[
{\mathcal{A}} \backslash H = \{A \in {\mathcal{A}} \, : \, A \neq H\} \\
\]
%{\mathcal{A}} / H &=& \{A \cap H \, : \, A \in {\mathcal{A}}, A \neq H\}
%\end{eqnarray*}
is the arrangement in $V$ consisting of the hyperplanes other than $H$, and the \textbf{contraction}
\[
{\mathcal{A}} / H = \{A \cap H \, : \, A \in {\mathcal{A}}, A \neq H\}
\]
is the arrangement in $H$ consisting of the intersections of the other hyperplanes with $H$. Figure \ref{f.fig:deletioncontraction} shows a hyperplane arrangement ${\mathcal{A}} = \{t,u,v,w\}$ in ${\mathbb{R}}^3$ and the deletion ${\mathcal{A}} \backslash w$ and contraction ${\mathcal{A}}/w$.

\begin{figure}[ht]
 \begin{center}
  \includegraphics[scale=.45]{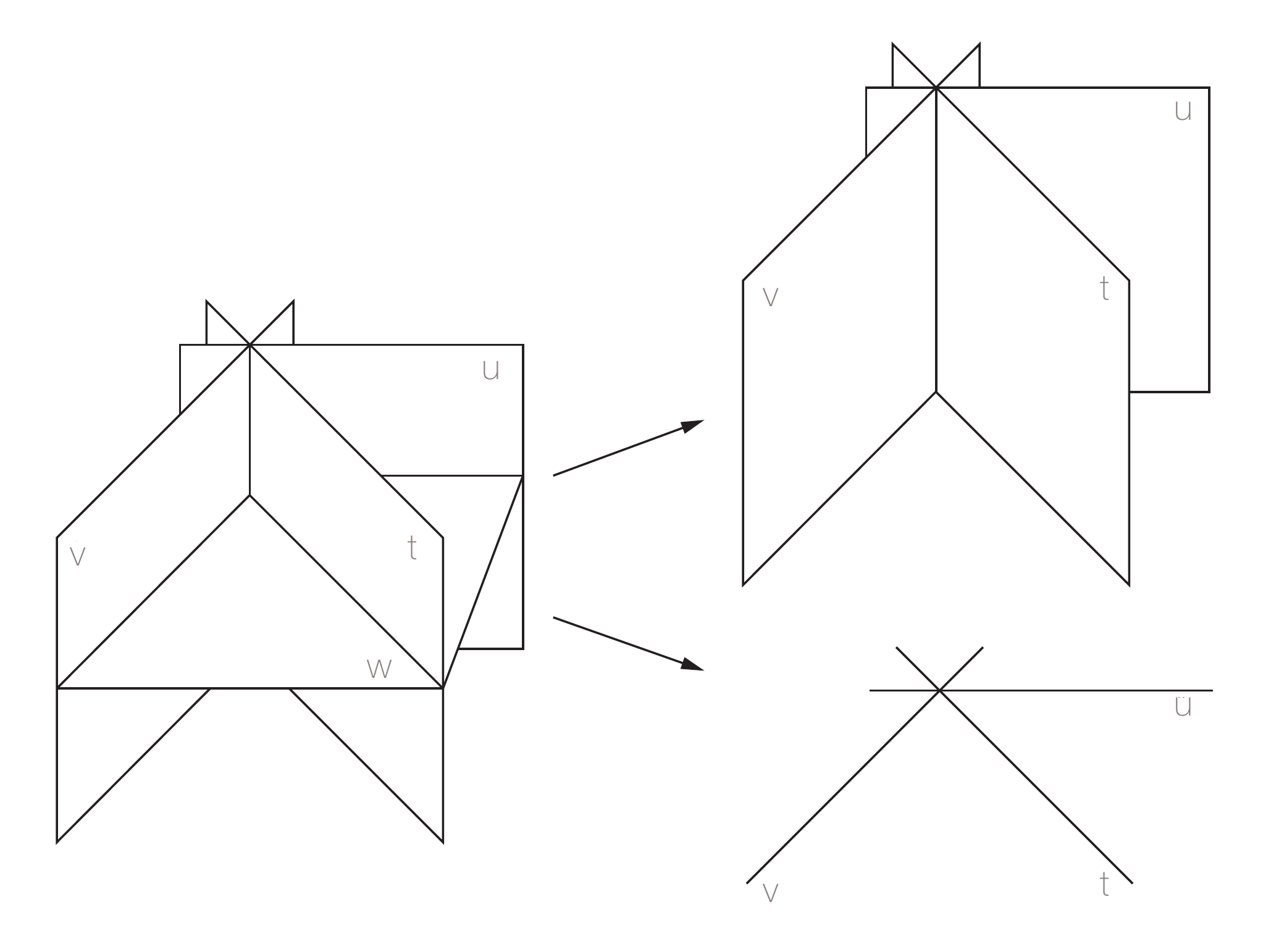}
  \caption{ \label{f.fig:deletioncontraction}
  An arrangement ${\mathcal{A}}$ and its deletion ${\mathcal{A}} \backslash w$ and contraction ${\mathcal{A}}/w$.}
 \end{center}
\end{figure}

\begin{proposition}(Deletion-Contraction)
If ${\mathcal{A}}$ is an arrangement and $H$ is a hyperplane in ${\mathcal{A}}$ then
\[
\chi_{\mathcal{A}}(q)= \chi_{{\mathcal{A}} \backslash H}(q) - \chi_{{\mathcal{A}} / H}(q)
\] 
\end{proposition}

\begin{proof} Whitney's formula gives
\[
\chi_{{\mathcal{A}}}(q) = \sum_{\stackrel{H \notin {\mathcal{B}} \subseteq {\mathcal{A}}}{{\mathcal{B}} \textrm{ central}}} (-1)^{|{\mathcal{B}}|} \, q^{n - r({\mathcal{B}})}  +  \sum_{\stackrel{H \in {\mathcal{B}} \subseteq {\mathcal{A}}}{{\mathcal{B}} \textrm{ central}}} (-1)^{|{\mathcal{B}}|} \, q^{n - r({\mathcal{B}})} =  \chi_{{\mathcal{A}} \backslash H}(q) - \chi_{{\mathcal{A}} / H}(q),
\]
where we use the fact that if $H \in {\mathcal{B}}$ then $r({\mathcal{B}}) = r_{{\mathcal{A}}/H}({\mathcal{B}}\backslash H)+1$.
\end{proof}

Unfortunately hyperplane arrangements are not closed under contraction; in the example of Figure \ref{f.fig:deletioncontraction}, the image of
$t$ in $({\mathcal{A}}/u)/v$
%the image of 3 in $({\mathcal{A}}/1)/2$ 
is not a hyperplane. 
Strictly speaking, this deletion-contraction formula applies only when there is only one copy of $H$ in ${\mathcal{A}}$.
%In particular, we only apply the deletion-contraction formula when there is only one copy of $H$ in ${\mathcal{A}}$. 
These small but annoying difficulties are elegantly solved by working in the wider class of \emph{matroids} (Section \ref{f.sec:matroids}) for central arrangements, or \emph{semimatroids} \cite{f.Ardilasemimatroids} for affine arrangements.
%\comment{We will say more about matroids in Section \ref{f.sec:matroids}; see \cite{f.semimatroids, f.WachsWalker} for the affine case.}

%We leave it to the reader to check that t
The characteristic polynomial of the cone $c{\mathcal{A}}$ can be expressed as follows. 
\begin{proposition} If $c{\mathcal{A}}$ is the cone of arrangement ${\mathcal{A}}$ then
\[
\chi_{c{\mathcal{A}}}(q) = (q-1)\chi_{\mathcal{A}}(q).
\]
\end{proposition}

%\comment{Prove using Whitney?}

\subsubsection{{\textsf{Sign alternation and unimodality}}}
 More can be said about the individual coefficients of the characteristic polynomial.

\begin{theorem}
The characteristic polynomial of an arrangement
\[
\chi_{\mathcal{A}}(q) = q^n - a_{n-1}q^{n-1} + a_{n-2}q^{n-2} - \cdots + (-1)^n a_n q^0
\]
has coefficients alternating signs, so $a_i \geq 0$. Furthermore, the coefficients are unimodal and even log-concave, that is,
\[
a_1 \leq a_2 \leq \cdots \leq a_{i-1} \leq a_i \geq a_{i+1} \geq \cdots \geq a_n \qquad \textrm{for some  $i$, and}
\]
\[
a_{j-1}a_{j+1} \leq a_j^2 \qquad \textrm{for all }j.
\]

\end{theorem}

\begin{proof}
The sign alternation is easily proved by induction using deletion-contraction. The second result is much deeper. It was conjectured by Rota in 1970 \cite{f.RotaICM}, and recently proved by Huh \cite{f.Huh} for fields of characteristic $0$ and by Huh and Katz \cite{f.HuhKatz} for arbitrary fields, drawing from toric geometry, tropical geometry, and matroid theory. This result is conjectured to be true for any geometric lattice; this is still open.
\end{proof}

%Rota, G.-C.: Combinatorial theory, old and new, Actes du Congrs International des Mathmaticiens (Nice, 1970), Tome 3, pp. 229?233. Gauthier-Villars, Paris (1971) (58 #21703)

\subsubsection{{\textsf{Graphs and proper colorings}}} There is a special case of interest, corresponding to arrangements coming from graphs. To each graph $G$ on vertex set $[n]$ we associate the \textbf{graphical arrangement} in ${\mathbbm{k}}^n$, consisting of the hyperplanes $x_i = x_j$ for all edges $ij$ of $G$. 

Given $q$ colors, a \textbf{proper $q$-coloring} of $G$ assigns a color to each vertex of $G$ so that two vertices $i$ and $j$ which share an edge $ij$ in $G$ must have different colors. The number of proper $q$-colorings of $G$ is given by the \textbf{chromatic polynomial} $\chi_G(q)$. 

Say an orientation of the edges of $G$ is \textbf{acyclic} if it creates no directed cycles.

\begin{theorem}
The chromatic polynomial of a graph $G$ equals the characteristic polynomial of its graphical arrangement ${\mathcal{A}}_G$
\[
\chi_{{\mathcal{A}}_G(q)} = \chi_G(q).
\]
and the number of acyclic orientations of $G$ equals the number of regions of ${\mathcal{A}}_G$ in ${\mathbb{R}}^n$. 
\end{theorem}

\begin{proof}
A proper $q$-coloring of $G$ is equivalent to a point $\mathbf{a} \in {\mathbb{F}}_q^n$ which is on none of the hyperplanes $x_i = x_j$ for $ij \in G$. Also, the arrangement ${\mathcal{A}}_G$ in ${\mathbbm{k}}^n$ has the same intersection lattice for any field ${\mathbbm{k}}$. (This is not true for every arrangement.) The first equality then follows from Theorem \ref{f.th:charpoly}.3.

Given a region $R$ of ${\mathcal{A}}_G$ we give each edge $ij$ the orientation $i \rightarrow j$ if $x_i > x_j$ in $R$, and $i \leftarrow j$ if $x_i < x_j$ in $R$. This is a bijection between the regions of ${\mathcal{A}}_G$ and the acyclic orientations of $G$.
\end{proof}

In particular, this theorem proves that the chromatic polynomial $\chi_G(q)$ is indeed given by a polynomial in $q$ when $q$ is a positive integer. It also gives us a reciprocity theorem, telling us what happens when we irreverently substitute the negative integer $q=-1$ into this polynomial:

\begin{corollary}
The graph $G$ has $|\chi_G(-1)|$ acyclic orientations of $G$.
\end{corollary}

\subsubsection{{\textsf{Free arrangements}} 
}
Far more often than we might expect, characteristic polynomials of hyperplane arrangements factor as products of linear forms. \cite{f.Saganfactors} 
%Sagan: Why the characteristic polynomial factors http://www.mth.msu.edu/users/sagan/Papers/Old/cpf-ams.pdf
Such factorizations are often a manifestation of the underlying algebraic structure. The theory of free arrangements gives one possible explanation for this phenomenon.

Let ${\mathcal{A}}$ be a real central arrangement and let $R = {\mathbb{R}}[x_1, \ldots, x_n]$ be the polynomial ring in $n$ variables, graded by total degree. A \textbf{derivation} is a linear map $D$ satisfying Leibniz's law
\[
D(fg) = f (Dg) + (Df) g \qquad \textrm{for all } f, g \in R
\] 
The set Der of derivations is an $R$-module; that is, if $p \in R$ and $D \in \textrm{Der}$ then $pD \in \textrm{Der}$. It is a graded module, where $D$ is homogeneous of degree $d$ if it takes polynomials of degree $k$ to polynomials of degree $k+d$. It is a free module with basis $\frac{\partial}{\partial x_1}, \ldots \frac{\partial}{\partial x_n}$; that is,
\[
\textrm{Der} = \left\{p_1\frac{\partial}{\partial x_1} + \cdots + p_n \frac{\partial}{\partial x_n} \, : \, p_i \in R\right\}
\]
Indeed, if $D$ is a derivation with $Dx_i = p_i \in R$ for $1 \leq i \leq n$, then $D =  p_1\frac{\partial}{\partial x_1} + \cdots + p_n \frac{\partial}{\partial x_n}$ by linearity.

Now consider the submodule of \textbf{${\mathcal{A}}$-derivations}:
\[
\textrm{Der}({\mathcal{A}}) = \left\{D \in \textrm{Der} \, : \, \alpha_H \, \textrm{ divides } \, D(\alpha_H) \textrm{ for all } H \in A
\right\}
\]
where $\alpha_H$ is a linear form defining hyperplane $H$, so $H = \{v \in {\mathbbm{k}}^n \, : \, \alpha_H(v)=0\}$. Let the \textbf{defining polynomial} of ${\mathcal{A}}$ be $Q_{\mathcal{A}} = \prod_{H \in {\mathcal{A}}} \alpha_H$.
 
For example, if $\mathcal{H}$ is the arrangement of coordinate hyperplanes, whose defining polynomial is $Q_{\mathcal{H}} = x_1\cdots x_n$, then $E_i = x_i \frac{\partial}{\partial x_i}$ is an $\mathcal{H}$-derivation for $1 \leq i \leq n$ since $E_i(x_i) = x_i$ and $E_i(x_j) = 0$ for $j \neq i$.

We say the arrangement ${\mathcal{A}}$ is \textbf{free} if Der$({\mathcal{A}})$ is a free $R$-module. This notion is not well understood at the moment; it is not even known if it is a combinatorial condition. For example, Ziegler gave an example of a free arrangement over $\mathbb{F}_2$ and a non-free arrangement over $\mathbb{F}_3$ with isomorphic intersection posets.

\begin{conjecture} \label{f.conj:Terao} \cite{f.Terao}
If two real arrangements ${\mathcal{A}}_1$ and ${\mathcal{A}}_2$ have isomorphic intersection posets and ${\mathcal{A}}_1$ is free, then ${\mathcal{A}}_2$ is free.
\end{conjecture}

%MATROID REPRESENTATIONS AND FREE ARRANGEMENTS GÜNTER M. ZIEGLER

Mysterious as it is, freeness is a very useful property for us, thanks to the following theorem.

\begin{theorem} \label{f.th:Terao}(Terao's Factorization Theorem) \cite{f.Terao}
If ${\mathcal{A}}$ is free then Der$({\mathcal{A}})$ has a homogeneous basis $D_1, \ldots, D_n$ whose degrees $d_1, \ldots, d_n$ only depend on ${\mathcal{A}}$, and the characteristic polynomial of ${\mathcal{A}}$ is
%. Furthermore, the Poincar\'e polynomial of $V({\mathcal{A}})$ is $q^n \chi_{\mathcal{A}}(-1/q)$ where
\[
\chi_{\mathcal{A}}(q) = (q-d_1-1) \cdots (q-d_n-1).
\]
\end{theorem}

%H. Terao, Generalized exponents of a free arrangement of hyperplanes and the Shepherd-Todd-Brieskorn formula, Invent. Math. 63 (1981), 159?179.

Freeness is made more tractable thanks to the following two useful criteria.

\begin{theorem}\label{f.th:Saito} (Saito Criterion) \cite{f.Saito}
Let $D_1, \ldots, D_n$ be ${\mathcal{A}}$-derivations and let $Q_{\mathcal{A}}$ be the defining polynomial of ${\mathcal{A}}$. Then Der$({\mathcal{A}})$ is free with basis $D_1, \ldots, D_n$ if and only if
\[
\det (D_i(x_j))_{1 \leq i, j \leq n} = c \cdot Q_{\mathcal{A}}
\]
for some constant $c$.
\end{theorem}

\begin{theorem}\label{f.th:freeind}\cite{f.Terao}
%Arrangements of hyperplanes and their freeness
Let ${\mathcal{A}}$ be an arrangement and $H$ be a hyperplane of ${\mathcal{A}}$. Any two of the following statements imply the third:

$\bullet$ ${\mathcal{A}} / H$ is free with exponents $b_1, \ldots, b_{n-1}$.

$\bullet$ ${\mathcal{A}} \backslash H$ is free with exponents $b_1, \ldots, b_{n-1}, b_n - 1$.

$\bullet$ ${\mathcal{A}}$ is free with exponents $b_1, \ldots, b_{n-1}, b_n$
\end{theorem}
Theorem \ref{f.th:Saito} can be very easy to use if we have the right candidate for a basis. For example, we saw that $E_i = x_i \frac{\partial}{\partial x_i}$ is an $\mathcal{H}_n$-derivation for $1 \leq i \leq n$. The matrix $E_i(x_j)$ is diagonal with determinant $x_1\cdots x_n$, so this must in fact be a basis for Der($\mathcal{H}$) with exponents $0, \ldots, 0$, and 
\[
\chi_{{\mathcal H}_n}(q) = (q-1)^n.
\]

To use Theorem \ref{f.th:freeind} to prove that an arrangement ${\mathcal{A}}$ is free, we need ${\mathcal{A}}$ to belong to a larger family of free arrangements with predictable exponents, which behaves well under deletion and contraction. For example, to prove inductively that $\mathcal{H}_n$ is free, we need the stronger statement that an arrangement of $k$ coordinate hyperplanes in ${\mathbbm{k}}^n$ (where $k \leq n$) has exponents $0$ ($k$ times) and $-1$ ($n-k$ times). A more interesting example is given in Section \ref{f.sec:computechar}.

%K. Saito, Theory of logarithmic differential forms and logarithmic vector fields, J. Fac. Sci. Univ. Tokyo Sec.1A Math. 27 (1980), 265?291.

\subsubsection{{\textsf{Supersolvability}}
}

 There is a combinatorial counterpart to the notion of freeness that produces similar results. Recall that a lattice $L$ is \textbf{supersolvable} if it has an \textbf{M-chain} $C$ such that the sublattice generated by $C$ and any other chain is distributive. 

In this section we are interested in arrangements whose intersection poset is  supersolvable. Since this poset is a geometric lattice, it is semimodular, so the following theorem applies.

\begin{theorem} \cite{f.Stanleysupersolvable}
Let $L$ be a finite supersolvable semimodular lattice, and suppose that  $\widehat{0} = t_0 < t_1 < \cdots < t_n = \widehat{1}$ is an M-chain. Let $a_i$ be the number of atoms $s$ such that $s \leq t_i$ but $s \nleq t_{i-1}$. Then
\[
\chi_L(q) = (q-a_1) \cdots (q-a_n).
\]
\end{theorem}

Graphical arrangements are an interesting special case. Say a graph $G$ is \textbf{chordal} if there exists an ordering of the vertices $v_1, \ldots, v_n$ such that for each $i$, the vertices among $\{v_1, \ldots, v_{i-1}\}$ which are connected to $v_i$ form a complete subgraph. An equivalent characterization is that $G$ has no induced cycles of length greater than $3$. 

It is very easy to compute the chromatic polynomial of a chordal graph. Suppose we wish to assign colors from $[q]$ to $v_1, \ldots, v_n$ in order, to get a proper $q$-coloring. If $v_i$ is connected to $a_i$ vertices among $\{v_1, \ldots, v_{i-1}\}$, since these are all connected pairwise, they must have different colors, so there are exactly $q-a_i$ colors available for $b_i$. It follows that $\chi_G(q) = (q-a_1) \cdots (q-a_n)$. The similarity in these formulas is not a coincidence. 
  
\begin{theorem} \cite{f.Stanleysupersolvable}
The intersection lattice of the graphical arrangement ${\mathcal{A}}_G$ is supersolvable if and only if the graph $G$ is chordal. 
\end{theorem}

\subsection{\textsf{Computing the characteristic polynomial}}\label{f.sec:computechar} 
The results of the previous section, and Theorem \ref{f.th:charpoly} in particular, show  the importance of computing $\chi_{\mathcal{A}}(q)$ for arrangements of interest. In this section we discuss the most common techniques for doing this.

\bigskip

\noindent \textbf{\textsf{Computing the M\"obius function directly.}} In Section \ref{f.sec:computingMobius} we saw many techniques for computing M\"obius functions, and we can use them to compute $\chi_{\mathcal{A}}(x)$.

\begin{enumerate}
\item (Generic arrangement) ${\mathcal{A}}_{n,r}$: \quad $n$ generic hyperplanes in ${\mathbbm{k}}^r$.

Consider a \textbf{generic} arrangement of $n$ hyperplanes in ${\mathbbm{k}}^r$, where any $k \leq r$ hyperplanes have an intersection of codimension $k$. There are ${n \choose m}$ flats of rank $m$, and for each flat $F$ we have $[\widehat{0}, F] \cong 2^{[m]}$, so $\mu(\widehat{0},F) = (-1)^m$. Therefore the characteristic polynomial is
\[
\chi_{{\mathcal{A}}_{n,r}} (x) = \sum_{m=0}^r (-1)^m {n \choose m} x^m
\]
and the number of regions and bounded regions are
\begin{eqnarray*}
a({\mathcal{A}}_{n,r}) &=& {n \choose r} + {n \choose r-1} + \cdots + {n \choose 0}, \\
b({\mathcal{A}}_{n,r}) &=& {n \choose r} - {n \choose r-1} + \cdots \pm {n \choose 0} = {n-1 \choose r}.
%
%\[
%a({\mathcal{A}}_{n,r}) = {n \choose r} + {n \choose r-1} + \cdots + {n \choose 0}, \quad 
%b({\mathcal{A}}_{n,r}) = {n \choose r} - {n \choose r-1} + \cdots \pm {n \choose 0} = {n-1 \choose r}.
%\]
\end{eqnarray*}

\end{enumerate}

This method works in some examples, but when there is a nice formula for the characteristic polynomial, this is usually not  the most efficient technique.

\bigskip

\noindent \textbf{\textsf{The finite field method.}} Theorem \ref{f.th:charpoly}.3 is about arrangements over finite fields, but it may also be used as a powerful technique for computing $\chi_{\mathcal{A}}(x)$ for real or complex arrangements. The idea is simple: most arrangements ${\mathcal{A}}$ we encounter ``in nature" (that is, in mathematics) are given by equations with integer coefficients.

We can use the equations of ${\mathcal{A}}$ to determine an arrangement ${\mathcal{A}}_q$ over a prime $q$. For large enough $q$, the arrangements ${\mathcal{A}}$ and ${\mathcal{A}}_q$ will have the same intersection poset, so
\[
\chi_{\mathcal{A}}(q) = | {\mathbb{F}}_q^n \backslash  \bigcup_{H \in {\mathcal{A}}_q}H|.
\]
We have thus reduced the computation of $\chi_{\mathcal{A}}$ to an enumerative problem in a finite field. By now we are pretty good at counting, and we can solve these problems for many arrangements of interest. 

\bigskip

\begin{enumerate}
\setcounter{enumi}{1}
\item (Coordinate arrangement) $\mathcal{H}_n: \quad x_i = 0  \qquad (1 \leq i \leq n)$

Here $\chi_{\mathcal{H}_n}(q)$ is the number of $n$-tuples $(a_1, \ldots, a_n) \in {\mathbb{F}}_q^n$ with $a_i \neq 0$ for all $i$, so
\[
\chi_{\mathcal{H}_n}(q) = (q-1)^n, \qquad a(\mathcal{H}_n) = 2^n
\]
There is an easy bijective proof for the number of regions: each region $R$ of $\mathcal{H}_n$ is determined by whether $x_i<0$ or $x_i >0$ in $R$.

\item (Braid arrangement) ${\mathcal{A}}_{n-1}: \quad x_i = x_j  \qquad (1 \leq i < j \leq n).$

Here $\chi_{{\mathcal{A}}_{n-1}}(q)$ is the number of $n$-tuples $(a_1, \ldots, a_n) \in {\mathbb{F}}_q^n$ such that $a_i \neq a_j$ for $i \neq j$. Selecting them in order, $a_i$ can be any element of ${\mathbb{F}}_q$ other than $a_1, \ldots, a_{i-1}$, so
\[
\chi_{{\mathcal{A}}_{n-1}}(q) = q(q-1)(q-2)\cdots (q-n+1), \qquad 
a({\mathcal{A}}_{n-1}) = n!
\]
A region is determined by whether $x_i > x_j$ or $x_i < x_j$ for all $i \neq j$; that is, by the relative linear order of $x_1, \ldots, x_n$. This explains why the braid arrangement has $n!$ regions.

%
%
%\item (Graphical arrangement) ${\mathcal{A}}_{G}: \quad x_i = x_j  \qquad (1 \leq i < j \leq n, \quad ij \in G).$
%
%Given a graph $G$ with vertex set $[n]$, and $q$ colors, a \textbf{proper $q$-coloring} assigns a color to each vertex of $G$ so that two vertices $i$ and $j$ which share an edge $ij$ in $G$ must have different colors. The number of proper $q$-colorings of $G$ is given by the \textbf{chromatic polynomial} $\chi_G(q)$.
%
%Note that a proper $q$-coloring is equivalent to a point $\mathbf{a} \in {\mathbb{F}}_q^n$ which is on none of the hyperplanes $x_i = x_j$ for $ij \in G$. 
%Therefore
%\[
%\chi_{{\mathcal{A}}_G(q)} = \chi_G(q).
%\]
%In particular, this proves that $\chi_G(q)$ is indeed given by a polynomial in $q$. Also
%
%Say an orientation of the edges of $G$ is \textbf{acyclic} if it creates no directed cycles. Then
%\[
%a({\mathcal{A}}_G) = a(G)
%\]
%where $a(G)$ is the number of acyclic orientations of $G$. Indeed, given a region $R$ of ${\mathcal{A}}_G$ we give each edge $ij$ the orientation $i \rightarrow j$ if $x_i > x_j$ in $R$, and $i \leftarrow j$ if $x_i < x_j$ in $R$; this is a bijection.
%

\item (Threshold arrangement) $\mathcal{T}_n: \quad x_i + x_j = 0 , \qquad (1 \leq i < j \leq n).$

Let $q$ be an odd prime. We need to count the points $\mathbf{a} \in {\mathbb{F}}_q^n$ on none of the hyperplanes, that is, those satisfying $a_i + a_j \neq 0$ for all $i \neq j$. To specify one such point, we may first choose the partition $[n]=S_0 \sqcup S_1 \sqcup \cdots \sqcup S_{(q-1)/2}$ where $S_i$ is the set of positions $j$ such that $a_j \in \{i, -i\}$. Note that $S_0$ can have at most one element. Then, for each non-empty block $S_i$ with $i>0$ we need to decide whether $a_j = i$ or $a_j = -i$ for all $j \in S_i$. The techniques of Section \ref{f.sec:egfs} then give
\[
\sum_{n \geq 0} \chi_{\mathcal{T}_n}(x) \frac{z^n}{n!} = (1+z)(2e^z-1)^{(x-1)/2}.
\]
The regions of $\mathcal{T}_n$ are in bijection with the \textbf{threshold graphs} on $[n]$. These are the graphs for which there exist vertex weights $w(i)$ for $1 \leq i \leq n$ and a ``threshold" $w$ such that edge $ij$ is present in the graph if and only if $w(i) + w(j) > w.$ Threshold graphs have many interesting properties and applications; see \cite{f.MahadevPeled}. 

%Mahadev, N. V. R.; Peled, Uri N. (1995), Threshold Graphs and Related Topics, Elsevier.
\end{enumerate}

There are many interesting deformations of the braid arrangement, obtained by considering hyperplanes of the form $x_i-x_j=a$ for various constants $a$. 
The left panel of Figure \ref{f.fig:arrangements} shows the braid arrangement ${\mathcal{A}}_2$. This is really an arrangement in ${\mathbb{R}}^3$, but since all hyperplanes contain the line $x=y=z$, we draw its essentialization by intersecting it with the plane $x+y+z=0$. Similarly, the other panels show the Catalan, Shi, Ish, and Linial arrangements, which we now discuss.

\begin{figure}[ht]
 \begin{center}
  \includegraphics[scale=.8]{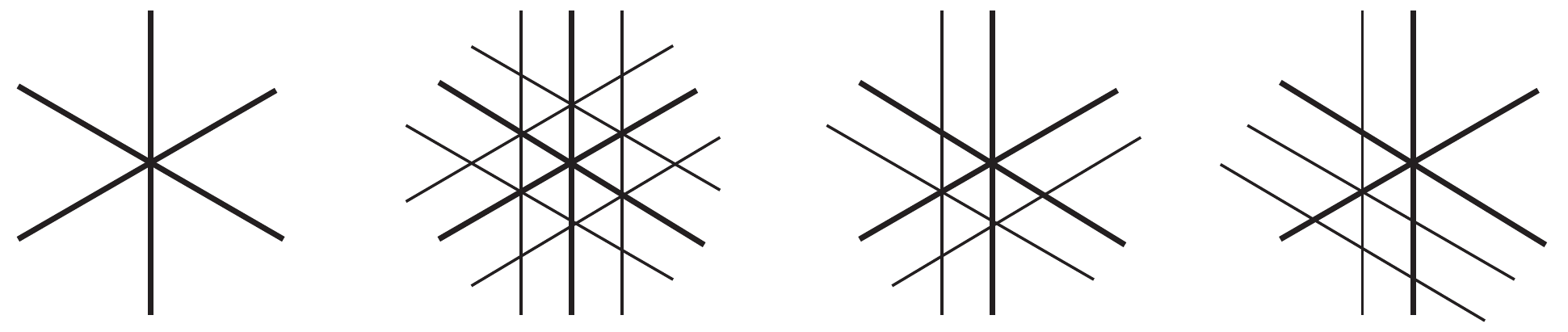}
  \caption{ \label{f.fig:arrangements}
  The arrangements ${\mathcal{A}}_{2}, \mathrm{Cat}_2, \mathrm{Shi}_2,$ and $\mathrm{Ish}_2.$}
  \end{center}
\end{figure}

\begin{enumerate}
\setcounter{enumi}{5}
\item (Catalan arrangement) $\mathrm{Cat}_{n-1}: \quad x_i - x_j = -1, 0, 1  \qquad (1 \leq i < j \leq n).$

To compute $\chi_{\mathrm{Cat}_{n-1}}(q)$, we need to count the $n$-tuples $\mathbf{a}=(a_1, \ldots, a_n) \in {\mathbb{F}}_q^n$ where $a_i$ and $a_j$ are never equal or adjacent modulo $q$. There are $q$ choices for $a_1$, and once we have chosen $a_1$ we can ``unwrap" ${\mathbb{F}}_q - \{a_1\}$ into a linear sequence of $q-1$ dots. The set $A=\{a_2, \ldots, a_n\}$ consists of $n-1$ non-adjacent dots, and choosing them is equivalent to choosing a partition of $q-n$ into $n$ parts, corresponding to the gaps between the dots; there are ${q-n-1 \choose n-1}$ choices. Finally there are $(n-1)!$ ways to place $A$ in a linear order in $\mathbf{a}$. Therefore 
\[
\chi_{\mathrm{Cat}_{n-1}}(q) = q(q-n-1)(q-n-2)\cdots (q-2n+1)
\]
\[
a(\mathrm{Cat}_{n-1}) = n! C_n, \qquad b(\mathrm{Cat}_{n-1}) = n! C_{n-1},
\]
where $C_n$ is the $n$th Catalan number. It is not too difficult to show bijectively that each region of the braid arrangement ${\mathcal{A}}_{n-1}$ contains $C_n$ regions of the Catalan arrangement, $C_{n-1}$ of which are bounded. \cite[Section 5.4]{f.Stanleyhyparr}

\item (Shi arrangement) $\textrm{Shi}_{n-1}: \quad x_i - x_j  =0, 1  \qquad (1 \leq i < j \leq n).$

Consider a point $\mathbf{a}=(a_1, \ldots, a_n) \in {\mathbb{F}}_q^n$ not on any of the Shi hyperplanes. Consider dots $0, 1, \ldots, q-1$ around a circle, and mark dot $a_i$ with the number $i$ for $1 \leq i \leq n$. Mark the remaining dots $\bullet$. Now let $w$ be the word of length $q$ obtained by reading the labels clockwise, starting at the label $1$. 
 %Now write a word $w$ of length $q$ by starting at the dot marked $1$ and writing the marking of each dot, or $\bullet$ if the dot is unmarked. 
 Note that each block of consecutive numbers must be listed in increasing order. 
By recording the sets between adjacent $\bullet$s, and dropping the initial $1$, we obtained an ordered partition $\Pi$ of $\{2, 3, \ldots, n\}$ into $q-n$ parts. There are $(q-n)^{n-1}$ such partitions. 
Figure \ref{f.fig:Shi} shows an example of this construction.

\begin{figure}[ht]
 \begin{center}
  \includegraphics[scale=.8]{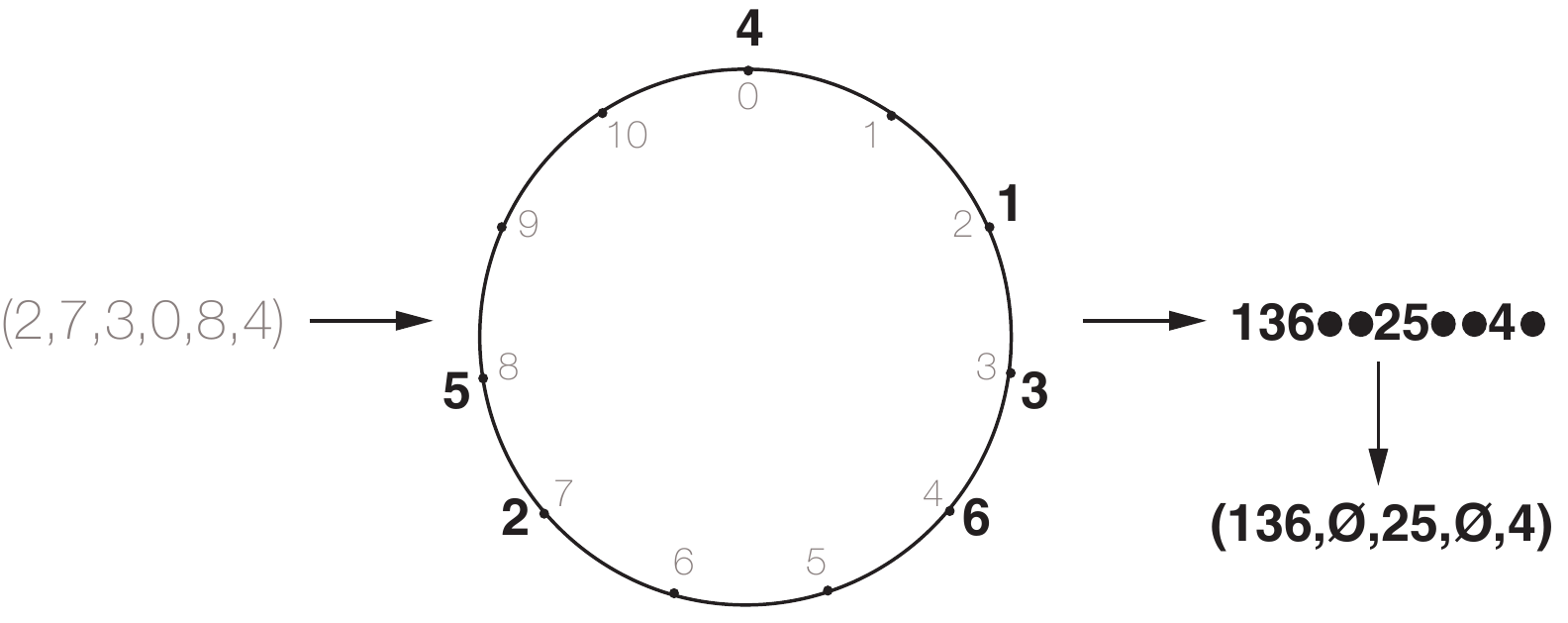}
  \caption{ \label{f.fig:Shi}
$\mathbf{a} = (2,7,3,0,8,4) \in {\mathbb{F}}_{11}^6$,\,\,  $w=136 \bullet \bullet  \,25 \bullet\bullet \,4 \,\bullet$,  \,\,$\Pi=(\{3,6\}, \emptyset, \{2,5\}, \emptyset, \{4\})$.}
  \end{center}
\end{figure}

To recover $a$ from this partition, we only need to know where to put marking $1$ on the circle; that is, we need to know $a_1$. It follows that
\[
\chi_{\mathrm{Shi}_{n-1}}(q) = q(q-n)^{n-1}
\]
\[
a(\mathrm{Shi}_{n-1}) = (n+1)^{n-1}, \qquad b(\mathrm{Shi}_{n-1}) = (n-1)^{n-1}
\]

\item (Ish arrangement) $\mathrm{Ish}_{n-1}: \quad x_i = x_j , \quad  x_1-x_j = i\qquad (1 \leq i < j \leq n).$

To choose a point $\mathbf{a}$ in the complement of the Ish arrangement over ${\mathbb{F}}_q$, choose $a_1 \in {\mathbb{F}}_q$, and then choose $a_n, a_{n-1}, \ldots, a_2$ subsequently. At each step we need to choose $a_i \notin \{a_1, a_1-1, \ldots, a_1-i+1, a_n, a_{n-1}, \ldots, a_{i+1}\}$. These $n$ forbidden values are distinct, so
\[
\chi_{\mathrm{Ish}_{n-1}}(q) = q(q-n)^{n-1},
\]
\[
a(\mathrm{Ish}_{n-1}) = (n+1)^{n-1}, \qquad b(\mathrm{Ish}_{n-1}) = (n-1)^{n-1}.
\]

The Shi and Ish arrangements share several features, which are generally easier to verify for the Ish arrangement. Together, they give a nice description of the ``$q,t$-Catalan numbers". \cite{f.Armstrong}

\item (Linial arrangement) $\mathcal{L}_{n-1}: \quad x_i - x_j = 1 , \qquad (1 \leq i < j \leq n).$

The Linial arrangement has characteristic polynomial
\[
\chi_{\mathcal{L}_{n-1}}(q) = \frac{q}{2^{n}} \sum_{k=0}^n {n \choose k} (q-k)^{n-1}
\]
The number of regions of $\mathcal{L}_{n-1}$ equals the number of \textbf{alternating trees} with vertex set $[n+1]$, where each vertex is either smaller than all its neighbors or greater than all its neighbors. There are two different proofs, using Whitney's Theorem \cite{f.PostnikovStanley} and the finite field method \cite{f.AthanasiadisLinial} respectively. Both are somewhat indirect, combining  combinatorial tricks and algebraic manipulations. To date, there is no known bijection between the regions of the Linial arrangement and alternating trees.

%Deformations of Coxeter Hyperplane Arrangements1 Alexander Postnikov2 and Richard P. Stanley3

%ATHANASIADIS
%EXTENDED LINIAL HYPERPLANE ARRANGEMENTS FOR ROOT SYSTEMS AND A CONJECTURE OF POSTNIKOV AND STANLEY

\end{enumerate}

Here are a few other nice examples.

\begin{enumerate}

\item[10.] (Coxeter arrangement) $\mathcal{BC}_{n}: \quad x_i \pm x_j = 0, \quad x_i = 0  \qquad (1 \leq i < j \leq n).$

To specify an $n$-tuple $(a_1, \ldots, a_n) \in {\mathbb{F}}_q^n \backslash \mathcal{BC}_{n}$, we can
%We need to count the $n$-tuples $(a_1, \ldots, a_n) \in {\mathbb{F}}_q^n$ with $a_i \neq a_j$ for all $i \neq j$. If we 
choose $a_1, \ldots, a_n$ successively. We find there are $q-2i+1$ choices for $a_i$, namely, any number other than $0, \pm a_1, \ldots, \pm a_{i-1}$ (which are all distinct). Therefore
\[
\chi_{\mathcal{BC}_{n}}(q) = (q-1)(q-3)\cdots (q-2n+3)(q-2n+1), \qquad 
r(\mathcal{BC}_{n}) = 2^n \cdot n!.
\]

\item[11.] (Coxeter arrangement) $\mathcal{D}_{n}: \quad x_i \pm x_j = 0 \qquad (1 \leq i < j \leq n).$

Here there are $(q-1)(q-3)(q-5) \cdots (1-2n+3)(q-2n+1)$ $n$-tuples in ${\mathbb{F}}_q^n \backslash \mathcal{D}_{n}$
with no $a_i$ equal to $0$, and $n(q-1)(q-3) \cdots (q-2n+3))$ $n$-tuples with one (and necessarily only one) $a_i$ equal to $0$. Therefore
\[
\chi_{\mathcal{D}_{n}}(q) = (q-1)(q-3)\cdots (q-2n+3)(q-n+1), \qquad 
r(\mathcal{D}_{n}) = 2^{n-1} \cdot n!
\]

\item[12.] (Finite projective space) ${\mathcal{A}}(p,n)$: all linear hyperplanes in ${\mathbb{F}}_p^n$

The equations $\sum_i c_i x_i=0$ (where $c_i \in {\mathbb{F}}_p$) define an arrangement over ${\mathbb{F}}_q$ where $q=p^k$, which has the same intersection poset. Recall that ${\mathbb{F}}_q$ is a $k$-dimensional vector space over ${\mathbb{F}}_p$. Now let us count the points $(a_1, \ldots, a_n)$ in ${\mathbb{F}}_q$ which are not on any hyperplane.We choose $a_1, \ldots, a_n$ subsequently and at each step we need $a_i \notin \textrm{span}_{{\mathbb{F}}_p}(a_1, \ldots, a_{n-1})$. Therefore
\[
\chi_{{\mathcal{A}}(p,n)}(q) = (q-1)(q-p)(q-p^2) \cdots (q-p^{n-1}).
\]
See \cite{f.BaranyReiner, f.Mphako}.

\item[13.] (All-subset arrangement) $\textrm{All}_{n}: \quad \sum_{i \in A} x_i = 0 \qquad (A \subseteq [n], \,\, A \neq \emptyset, [n]).$

This arrangement appears naturally in combinatorics, representation theory, algebraic geometry, and physics, among many other contexts. 
% \cite{f.Billeraetal} %, f.Cavalierietal, f.Evans, f.Kamiyaetal}
To date, we do not know a simple formula for the characteristic polynomial, though we do have a nice bound. \cite{f.Billeraetal}. 
 
 The reduction of  
$\textrm{All}_{n}$ modulo $2$ is the arrangement ${\mathcal{A}}(2,n)$ in ${\mathbb{F}}_2^n$ considered above. The map from $\textrm{All}_{n}$ to ${\mathcal{A}}(2,n)$ changes the combinatorics; for example, the subset of hyperplanes $x_1+x_2=0, x_2+x_3=0, x_3+x_1=0$ decreases from rank $3$ to rank $2$. However, this map is a rank-preserving \emph{weak map}, in the sense that the rank of a subset never increases, and the total rank stays the same. This implies \cite[Cor. 9.3.7]{f.KungNguyen}
 %Joseph P.S. Kung and Hien Q. Nguyen. Weak maps. Chapter 9 in Neil White, editor, Theory of Matroids, pages 254?271. Encyclopedia of Mathematics and it Applications, Vol. 26, Cambridge University Press, Cambridge, 1986.
that the coefficients of $\chi_{\textrm{All}_{n}}(q)$ are greater than the respective coefficients of $\chi_{{\mathcal{A}}(2,n)}(q)$ in absolute value, so
\[
r(\textrm{All}_{n}) = |\chi_{\textrm{All}_{n}}(-1)| >  |\chi_{{\mathcal{A}}(2,n)}(-1)| = \prod_{i=0}^{n-2}(2^i + 1) > 2^{n-1 \choose 2}
\]
It is also known that $r(\textrm{All}_{n}) < 2^{(n-1)^2}$; see \cite{f.Billeraetal}.

%Hidehiko Kamiya, Akimichi Takemura, and Hiroaki Terao. Ranking patterns of unfolding models of codimension one. Adv. in Appl. Math., 47(2), 379?400, 2011.

% Wall Crossings for Double Hurwitz Numbers (with Hannah Markwig and Paul Johnson) In ?Advances in Mathematics?

%T. S. Evans. What is being calculated with Thermal Field Theory? In A. Astbury, B. A. Campbell, W. Israel, F. C. Khanna, D. Page, and J. L. Pinfold, editors, Particle Physics and Cosmology - Proceedings of the Ninth Lake Louise Winter Institute, pages 343?352. World Scientific, 1995.

\end{enumerate}

\bigskip

\noindent \textbf{\textsf{Whitney's formula.}} We can sometimes identify combinatorially the terms in Whitney's Theorem \ref{f.th:Whitney} to obtain a useful formula for the characteristic polynomial. 

\begin{enumerate}
\setcounter{enumi}{13}
\item (Coordinate arrangement) $\mathcal{H}_n: \quad x_i = 0  \qquad (1 \leq i \leq n)$

Whitney's formula gives
\[
\chi_{\mathcal{H}_n}(q) = \sum_{A \subseteq [n]} (-1)^{|A|} q^{n-{|A|}} = (q-1)^n.
\]

\item (Generic deformation of ${\mathcal{A}}_n$) $\mathcal{G}_n: \quad x_i -x_j = a_{ij}  \qquad (1 \leq i<j  \leq n, \quad a_{ij} \textrm{ generic })$

Let $H_{ij}$ represent the hyperplane $x_i-x_j = a_{ij}$. By the genericity of the  $a_{ij}$s, a subarrangement $H_{i_1j_1}, \ldots, H_{i_kj_k}$ is central if and only if the graph with edges $i_1j_1, \ldots, i_nj_n$ has no cycles; that is, it is a forest. Therefore
\[
\chi_{\mathcal{G}_n}(q) = \sum_{F \textrm{ forest on } [n]} (-1)^{|F|} q^{n-{|F|}}, \qquad 
r(\mathcal{G}_n) = \textrm{forests}(n)
\]
where forests$(n)$ is the number of forests on vertex set $[n]$. There is no simple formula for this number, though we can use the techniques of Section \ref{f.sec:egfs} to compute the exponential generating function for $\chi_{\mathcal{G}_n}(q)$.

\item (Other deformations of ${\mathcal{A}}_n$) The same approach can work for any arrangement consisting of hyperplanes of the form $x_i -x_j = a_{ij}$, which correspond to edges marked $a_{ij}$. If we have enough control over the $a_{ij}$s to describe combinatorially which subarrangements are central, we will obtain a combinatorial formula for the characteristic polynomial. 

For simple arrangements like the Shi and Catalan arrangement, the finite field method gives slicker proofs than Whitney's formula. However, this unified approach is also very powerful; for many interesting examples, see \cite{f.PostnikovStanley, f.ArdilaTutte}.

\end{enumerate}

\bigskip
\noindent \textbf{\textsf{Freeness.}} Terao's Factorization Theorem \ref{f.th:Terao} is a powerful algebraic technique for computing characteristic polynomials which factor completely into linear factors.

\begin{enumerate}
\setcounter{enumi}{16}
\item (Braid arrangement, revisited) It is not difficult to describe the ${\mathcal{A}}_{n-1}$-derivations for the braid arrangement ${\mathcal{A}}_{n-1}$. 
Note that $F_d=x_1^d\frac{\partial}{\partial x_1} + \cdots + x_n^d \frac{\partial}{\partial x_n}$ is an ${\mathcal{A}}_{n-1}$-derivation for $d=0, 1, \ldots, n-1$ because $F_d(x_i-x_j) = x_i^d - x_j^d$. By Saito's criterion, since
\[
\det (F_i(x_j))_{0 \leq i, j \leq n-1} = \det (x_j^i)_{1 \leq i, j \leq n} = \prod_{1 \leq i < j \leq n} (x_i-x_j) =  Q_{\mathcal{A}}, 
\]
Der$({\mathcal{A}})$ is free with basis $F_0, \ldots, F_{n-1}$ of degrees $-1, 0, 1, \ldots, n-2$. Therefore
\[
\chi_{{\mathcal{A}}_n}(q) = q(q-1)\cdots (q-n+1).
\]

\item (Coxeter arrangements $\mathcal{BC}_{n}$ and $\mathcal{D}_{n}$, revisited) 

%For $\mathcal{D}_{n}$, the function $F_d$ above satisfies $F_d(x_i+x_j) = x_i^d + x_j^d$, so it is a $\mathcal{D}_n$-derivation only for $n$ odd. 

For $\mathcal{D}_{n}$, the function $F_d$ above satisfies $F_d(x_i+x_j) = x_i^d + x_j^d$, so it is a $\mathcal{D}_{n}$-derivation only for $d$ odd. Fortunately,
\[
\det(F_{2i-1}(x_j))_{1 \leq i, j, \leq n} = \prod_{1 \leq i < j \leq n}(x_i-x_j)(x_i+x_j),
\]
as can be seen by identifying linear factors as in Section \ref{f.sec:dets2}. It follows that  Der$(\mathcal{D}_n)$ is free with basis $F_1, F_3, F_5, \ldots, F_{2n-1}$ and degrees $0, 2, 4, \ldots, 2n-2$.

For $\mathcal{BC}_n$ we need the additional derivation $F = x_1 \cdots x_n \sum_{i} \frac1{x_i} \frac{\partial}{\partial x_i}$. This is indeed a $\mathcal{BC}_n$-derivation because $F(x_i\pm x_j) = (x_j \pm x_i) x_1 \cdots x_n/x_ix_j$. Once again it is easy to check Saito's criterion to see that Der$(\mathcal{BC}_n)$ is free with basis $F_1, F_3, \ldots, F_{2n-3}, F$ and degrees $0, 2, 4, \cdots, 2n-4, n-2$.

\item (Coxeter arrangements, in general) A \textbf{finite reflection group} is a finite group generated by reflections through a family of hyperplanes. For example the reflections across the hyperplanes $x_i=x_j$ of the braid arrangement ${\mathcal{A}}_{n-1}$ correspond to the transpositions $(ij)$, and generate the symmetric group $S_n$. Every finite reflection group is a direct product of irreducible ones. Here we focus on the \textbf{crystallographic} ones, which can be written with integer coordinates. The irreducible crystallographic finite reflection groups $\Phi$ come in three infinite families $A_n, BC_n, D_n$ and five exceptional groups $G_2, F_4, E_6, E_7, E_8$.  
The corresponding hyperplane arrangements ${\mathcal{A}}_\Phi$  are the following.

\begin{center}
\begin{tabular}{|c|c|l|}
\hline
$\Phi$ & ${\mathbbm{k}}^n$ & equations \\
\hline
$\mathcal{A}_{n-1}$ & ${\mathbbm{k}}^n$ & $x_i = x_j $ \\
$\mathcal{BC}_n$ & ${\mathbbm{k}}^n$ & $x_i = \pm x_j , \quad x_i = 0$ \\
$\mathcal{D}_n$ & ${\mathbbm{k}}^n$  & $x_i = \pm x_j $ \\
$\mathcal{G}_2$ & ${\mathbbm{k}}^3$ & $x_i = \pm x_j,  \quad  2x_i = x_j+x_k$ \\
$\mathcal{F}_4$ & ${\mathbbm{k}}^4$ & $x_i = \pm x_j,  \quad  x_i = 0,  \quad \displaystyle  \pm x_1 \pm x_2 \pm x_3 \pm x_4 = 0$.\\% \sum_{i=1}^4 \pm x_i = 0$. \\
$\mathcal{E}_6$ & ${\mathbbm{k}}^9$ & $x_i = x_j, \quad \displaystyle \sum_{i \in \{a,b,c\}} 2x_i = \sum_{i \notin \{a,b,c\}} x_i \quad  % (1 \leq a \leq 3,\,\, 4 \leq b \leq 6,\,\, 7 \leq c \leq 9) $. \\
\left( \substack{1 \leq a \leq 3 \\ 4 \leq b \leq 6 \\ 7 \leq c \leq 9}\right)$. \\
$\mathcal{E}_7$ & ${\mathbbm{k}}^8$ & $x_i = x_j ,\quad \displaystyle \sum_{i \in A} x_i = \sum_{i \notin A} x_i \quad ( \substack{A \subseteq [8] \\ |A|=4})$ \\
$\mathcal{E}_8$ & ${\mathbbm{k}}^8$ & $x_i = \pm x_j, \quad \displaystyle  \sum_{i=1}^8  \epsilon_i x_i = 0 \quad \left( \substack{\epsilon_i = \pm1, \\ \epsilon_1 \cdots \epsilon_8 = 1} \right)$\\
\hline
\end{tabular}
\end{center}

Note that arrangement $\Phi_r$ has rank $r$, but we have embedded some of them them in ${\mathbbm{k}}^d$ for higher $d$ in order to obtain nicer equations. Here the indices range over the corresponding ${\mathbbm{k}}^n$; for instance, $\mathcal{G}_2$ includes the hyperplanes $2x_1 = x_2+x_3, \, 2x_2 = x_1 + x_3, \, 2x_3 = x_1 + x_2$.
We now list the characteristic polynomials and number of regions for these arrangements.

\bigskip

\begin{tabular}{|c|l|l|}
\hline
$\Phi$ & characteristic polynomial & regions \\
\hline
$\mathcal{BC}_n$ & $(q-1)(q-3)(q-5) \cdots (q-2n+3)(q-2n+1$) &  $2^nn!$ \\
$\mathcal{D}_n$ & $(q-1)(q-3)(q-5) \cdots (q-2n+3)(q-n+1)$ & $2^{n-1}n!$ \\
$\mathcal{G}_2$ & $(q-1)(q-5)$ & $12$ \\
$\mathcal{F}_4$ & $(q-1)(q-5)(q-7)(q-11)$ & $1,152$ \\
$\mathcal{E}_6$ & $(q-1)(q-4)(q-5)(q-7)(q-8)(q-11)$ & $51,840$ \\
$\mathcal{E}_7$ & $(q-1)(q-5)(q-7)(q-9)(q-11)(q-13)(q-17)$ & $2,903,040$ \\
$\mathcal{E}_8$ & $(q-1)(q-7)(q-11)(q-13)(q-17)(q-19)(q-23)(q-29)$ &  $696,729,600$ \\
\hline
\end{tabular}

\bigskip

It follows from the theory of reflection groups that the number of regions equals the order of the reflection group:
\[
r({\mathcal{A}}_\Phi) = |\Phi|.
\]

We already computed in two ways the characteristic polynomials of the classical Coxeter groups $\mathcal{A}_n, \mathcal{BC}_n,$ and $\mathcal{D}_n$.
We first gave simple finite field proofs; it would be interesting to do this for the other (``exceptional") groups. We also computed explicit bases for the modules of ${\mathcal{A}}$-derivations; doing this for the other Coxeter groups is not straightforward, but it can be done. \cite[Appendix B]{f.OrlikTerao} 

However, these beautiful formulas clearly illustrate that there are deeper things at play, and they make it very desirable to have a unified explanation. Indeed, invariant theory produces an elegant case-free proof that all Coxeter arrangement are free; for details, see \cite[Theorem 6.60]{f.OrlikTerao}.

\item (Shi arrangement) Given its characteristic polynomial, it is natural to guess that the Shi arrangement $\mathcal{S}_{n-1}$ is free. Since we only defined freeness for central arrangements, we consider the cone $c\mathcal{S}_{n-1}$ instead. Following \cite{f.Athanasiadisfree}, we may prove that this is a free arrangement inductively using Theorem \ref{f.th:freeind}. The strategy is to remove the hyperplanes of $\mathcal{S}_{n-1}$ one at a time, by finding in each step a hyperplane $H$ such that ${\mathcal{A}} \backslash H$ and $A / H$ are free with predictable exponents. This leads us to consider a more general family of arrangements.

We claim that for any $m \geq 0$ and $2 \leq k \leq n+1$, the arrangement $\mathcal{S}_{n-1}^{m,k}$
\begin{eqnarray*}
x_i - x_j &=& 0, 1 \qquad \textrm{ for }  2 \leq i < j \leq n \\
x_1 - x_j &=& 0, 1, \ldots, m \qquad \textrm{ for }  2 \leq j < k \\ 
x_1 - x_j &=& 0, 1, \ldots, m, m+1  \qquad \textrm{ for }  k \leq j \leq n
\end{eqnarray*}
is free with exponents $n+m-1, \ldots, n+m-1$ ($n-k+1$ times)
, $n+m-2, \ldots, n+m -2$ ($k-2$ times), and $-1$.

Obviously it takes some care to assemble this family of free arrangements and their exponents. However, once we have done this, it is straightforward to prove this more general statement by induction using Theorem \ref{f.th:freeind}, by removing one hyperplane at a time while staying within the family of arrangements $\mathcal{S}_{n-1}^{m,k}$. Setting $k=2$ and $m=0$ we get that the Shi arrangement is free. 

Incidentally, Athanasiadis also characterized all the arrangements ${\mathcal{A}}$ with ${\mathcal{A}}_{n-1} \subseteq {\mathcal{A}} \subseteq \mathcal{S}_{n-1}$ such that $c{\mathcal{A}}$ is free \cite{f.Athanasiadisfree}. These arrangements have Ish counterparts which have the same characteristic polynomial, although their cones are not free.
\cite[Corollary 3.3]{f.ArmstrongRhoades}

\end{enumerate}

\subsection{\textsf{The ${\mathbf{c}}{\mathbf{d}}$-index of an arrangement}} 
\label{f.sec:arrcdindex} 

We now discuss a vast strengthening of Zaslavsky's Theorem \ref{f.th:charpoly}.1
due to Billera, Ehrenborg, and Readdy. \cite{f.BilleraEhrenborgReaddy} They showed that the enumeration of chains of faces of a real hyperplane arrangement ${\mathcal{A}}$ (or, equivalently, of a zonotope) depends only on the intersection poset of $L_{\mathcal{A}}$. The equivalence between real arrangements and zonotopes is illustrated in Figure \ref{f.fig:vectorshypszonotope} and explained by the following result.

\begin{proposition}
Let $A$ be an arrangement of non-zero real vectors, and let ${\mathcal{A}}$ be the arrangement of hyperplanes perpendicular to the vectors of $A$. Then there is an order-reversing bijection between the faces of the zonotope $Z(A)$ and the faces of the real arrangement ${\mathcal{A}}$.
\end{proposition}

\begin{figure}[ht]
 \begin{center}
   \includegraphics[height=2.5cm]{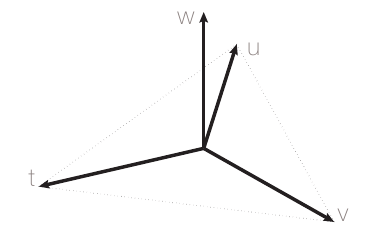} \qquad  
  \includegraphics[height=4cm]{./Pictures/arrangement} \qquad   \quad 
  \includegraphics[height=4cm]{./Pictures/hexagonalprism}
  \caption{ \label{f.fig:vectorshypszonotope}
A real vector arrangement, the normal hyperplane arrangement, and its dual zonotope.}
  \end{center}
\end{figure}

From this result, it follows easily that the ${\mathbf{c}}{\mathbf{d}}$-indices of the zonotope $Z(A)$ and of the face poset $F({\mathcal{A}})$ of ${\mathcal{A}}$ are reverses of each other; that is, they are related by the linear map $^*: \mathbb{Z}\langle{\mathbf{c}},{\mathbf{d}}\rangle \rightarrow \mathbb{Z}\langle{\mathbf{c}},{\mathbf{d}}\rangle$ that reverses any word in ${\mathbf{c}}$ and ${\mathbf{d}}$, so that $(v_1\ldots v_k)^* = v_k \ldots v_1$. 
To describe these ${\mathbf{c}}{\mathbf{d}}$-indices, we introduce the linear map  $\omega: \mathbb{Z}\langle{\mathbf{a}},{\mathbf{b}}\rangle \rightarrow \mathbb{Z}\langle{\mathbf{c}},{\mathbf{d}}\rangle$ obtained by first replacing each occurrence of ${\mathbf{a}}{\mathbf{b}}$ with $2{\mathbf{d}}$, and then replacing every remaining letter with a ${\mathbf{c}}$. 

\begin{theorem} \label{f.th:cd-indexzonotope} \cite{f.BilleraEhrenborgReaddy}
The flag $f$-vector of a real hyperplane arrangement (or equivalently, of the dual zonotope) depends only on its intersection poset. 

More explicitly, let ${\mathcal{A}}$ be a real hyperplane arrangement. Let $F({\mathcal{A}})$ be its face poset and $L_{{\mathcal{A}}}$ be its intersection poset. The ${\mathbf{c}}{\mathbf{d}}$-indices of the 
face poset $F({\mathcal{A}})$ and the dual zonotope $Z(A)$ are given in terms of the ${\mathbf{a}}{\mathbf{b}}$-index of the intersection lattice $L_{\mathcal{A}}$ by the formula:
\[
\Psi_{F({\mathcal{A}})}({\mathbf{c}}, {\mathbf{d}}) = \left[\omega({\mathbf{a}}  \, \Phi_{L_{\mathcal{A}}}({\mathbf{a}},{\mathbf{b}}))\right]^*, \qquad 
\Psi_{Z(A)}({\mathbf{c}}, {\mathbf{d}}) = \omega({\mathbf{a}}  \, \Phi_{L_{\mathcal{A}}}({\mathbf{a}},{\mathbf{b}})), \qquad 
\]
\end{theorem}

\bigskip

Billera, Ehrenborg, and Readdy \cite{f.BilleraEhrenborgReaddy} 
extended this result to \emph{orientable matroids}. They also asked for an interesting interpretation of $\omega({\mathbf{a}}  \, \Phi_{L}({\mathbf{a}},{\mathbf{b}}))$ for an arbitrary geometric lattice $L$; this is still open.

\bigskip

To illustrate this result we revisit Example \ref{f.ex:prism}, where we computed the ${\mathbf{c}}{\mathbf{d}}$-index of a hexagonal prism. This is the dual zonotope to the arrangement of Figure \ref{f.fig:vectorshypszonotope}, whose intersection poset $L_{{\mathcal{A}}}$ is Figure \ref{f.fig:intposet}. The flag $f$ and $h$-vectors of $L_{{\mathcal{A}}}$ are $(f_\emptyset, f_{\{1\}}, f_{\{2\}}, f_{\{1,2\}}) = (1, 4, 4, 9)$ and $(h_\emptyset, h_{\{1\}}, h_{\{2\}}, h_{\{1,2\}}) = (1, 3, 3, 2)$, so its ${\mathbf{a}}{\mathbf{b}}$-index is $\Phi_{L_{\mathcal{A}}}({\mathbf{a}},{\mathbf{b}}))={\mathbf{a}}{\mathbf{a}}+3{\mathbf{a}}{\mathbf{b}}+3{\mathbf{b}}{\mathbf{a}}+2{\mathbf{b}}{\mathbf{b}}$. Therefore
\[
\Psi_{Z(A)}({\mathbf{c}}, {\mathbf{d}}) = \omega({\mathbf{a}}{\mathbf{a}}{\mathbf{a}}+3{\mathbf{a}}{\mathbf{a}}{\mathbf{b}}+3{\mathbf{a}}{\mathbf{b}}{\mathbf{a}}+2{\mathbf{a}}{\mathbf{b}}{\mathbf{b}}) = {\mathbf{c}}{\mathbf{c}}{\mathbf{c}}+3{\mathbf{c}}(2{\mathbf{d}})+3(2{\mathbf{d}})c+2(2{\mathbf{d}}){\mathbf{c}} = {\mathbf{c}}^3+6{\mathbf{c}}{\mathbf{d}}+10{\mathbf{d}}{\mathbf{c}}
\]
and
\[
\Psi_{F({\mathcal{A}})}({\mathbf{c}},{\mathbf{d}}) = {\mathbf{c}}^3 + 6 {\mathbf{d}}{\mathbf{c}} + 10{\mathbf{c}}{\mathbf{d}}.
\]
in agreement with the computation in Section \ref{f.sec:Eulerian}.

\newpage

\section{\textsf{Matroids}}\label{f.sec:matroids}

Matroid theory is a combinatorial theory of independence which has its roots in linear algebra and graph theory, but which turns out to have deep connections with many fields, and numerous applications in pure and applied mathematics. Rota was a particularly enthusiastic ambassador:

\begin{quote}
\emph{``It is as if one were to condense all trends of present day mathematics onto a single finite structure, a feat that anyone would a priori deem impossible, were it not for the fact that matroids do exist."} \cite{f.Rotaindiscrete}
\end{quote}

There are natural notions of independence in linear algebra, graph theory, matching theory, the theory of field extensions, and the theory of routings, among others. Matroids capture the combinatorial essence that those notions share. 

In this section we will mostly focus on enumerative aspects of matroid theory. For a more complete introduction, see \cite{f.Oxley, f.Welsh} and the three volume series \cite{f.White2, f.White1, f.White3}.

In Section \ref{f.sec:matroidmotivation} we discuss some combinatorial aspects of vector configurations and linear independence, and introduce some terminology. This example and terminology strongly motivates the definition(s) of a matroid, which we give in Section \ref{f.sec:matroiddefinitions}. Section \ref{f.sec:matroidexamples} discusses many important families of matroids arising in algebra, combinatorics, and geometry. Section \ref{f.sec:matroidconstructions} introduces some basic constructions, and Section \ref{f.sec:matroidstructure} gives a very brief overview of structural matroid theory. The Tutte polynomial is our main enumerative tool; it is introduced in Section \ref{f.sec:Tutte}. In Section \ref{f.sec:Tutteevaluations} we answer many enumeration problems in terms of Tutte polynomials, and in Section \ref{f.sec:computingTutte} we see how we can actually compute the polynomial in some examples of interest. Section \ref{f.sec:Tuttegeneralizations} is devoted to two generalizations: the multivariate and the arithmetic Tutte polynomials. 
%Finally, we give a brief overview of matroid subdivisions
in Section \ref{f.sec:matroidsubdivisions} we discuss matroid subdivisions, 
%, including a discussion on 
matroid valuations and the Derksen--Fink invariant.

%
%
%
%A set of vectors, a graph, a bipartite matching problem, a set of elements in a field extension, and a routing problem all give rise to matroids. When one proves a theorem about matroids, one obtains a result about each one of these structures.
%
%The theory of matroids is deep and broad on its own, and has intimate connections with and applications in many fields. In this chapter, we will focus on enumerative aspects of matroids, many of which revolve around the Tutte polynomial. 
%

\subsection{{\textsf{The main motivating example: Vector configurations and linear matroids}}} \label{f.sec:matroidmotivation}

Let $E \subset {{\mathbbm{k}}}^d$ be a finite set of vectors. For simplicity, we assume that $E$ does not contain $\mathbf{0}$, does not contain repeated vectors, and spans ${{\mathbbm{k}}}^d$. The \textbf{matroid} of $E$ is given by any of the following definitions:

\medskip

$\bullet$ The \textbf{independent sets} of $E$ are the subsets of $E$ which are linearly independent over ${{\mathbbm{k}}}$. 

$\bullet$ The \textbf{bases} of $E$ are the subsets of $E$ which are bases of ${{\mathbbm{k}}}^d$.

$\bullet$ The \textbf{circuits} of $E$ are the minimal linearly dependent subsets of $E$.

$\bullet$ The \textbf{rank function} $r:2^E \rightarrow {\mathbb{N}}$ is $r(A) = \dim({\mathrm{span }} \, A)$ for $A \subseteq E$.

$\bullet$ The \textbf{flats} $F$ of $E$ are the subspaces of ${{\mathbbm{k}}}^d$ spanned by subsets of $E$. We identify a flat with the set of vectors of $E$ that it contains.

$\bullet$ The \textbf{lattice of flats} $L_M$ is the poset of flats ordered by containment.

$\bullet$ The \textbf{matroid (basis) polytope} is $P_M={\mathrm{conv }}\{{\mathbf{e}}_{b_1} + \cdots + {\mathbf{e}}_{b_r} \, : \, \{b_1, \ldots, b_r\} \textrm{ is a basis}\} \subset {\mathbb{R}}^E$.

%We denote the collections of independent sets, bases, and flats by ${\mathcal{I}}, {\mathcal{B}}, {\mathcal{C}}$ and ${\mathcal{F}}$, respectively. If we know the ground set, any one of $r, {\mathcal{I}}, {\mathcal{B}}, {\mathcal{F}}$ is sufficient to determine all the other ones. We call any of $(E, r)$, $(E, {\mathcal{I}})$, $(E, {\mathcal{B}})$, $(E, {\mathcal{C}})$and $(E, {\mathcal{F}})$ the \textbf{matroid} of $E$. 

\medskip

More concretely, if we know one of the following objects, we can determine them all: the list ${\mathcal{I}}$ of independent sets, the list ${\mathcal{B}}$ of bases, the list ${\mathcal{C}}$ of circuits, the rank function $r$, the list ${\mathcal{F}}$ of flats, the lattice $L_M$ of flats, or the matroid polytope $P_M$. For this reason, we define the \textbf{matroid} of $E$ to be any one (and all) of these objects.
A matroid that arises in this way from a set of vectors in a vector space (over ${{\mathbbm{k}}}$) is called a \textbf{linear} or \textbf{representable matroid} (over ${{\mathbbm{k}}}$).

\begin{example}
Let ${\mathbf{t}}=(1,-1,0), {\mathbf{u}}=(0, 1, -1), {\mathbf{v}} =(-1,0,1), {\mathbf{w}},=(1,1,1)$ be the vector configuration shown in the left panel of Figure \ref{f.fig:vectorshypszonotope}. Then\footnote{We are omitting brackets for clarity, so for example we write ${\mathbf{t}}{\mathbf{u}}{\mathbf{v}}$ for $\{{\mathbf{t}},{\mathbf{u}},{\mathbf{v}}\}$.}
\begin{eqnarray*}
{\mathcal{I}} &=& \{\emptyset, {\mathbf{t}}, {\mathbf{u}}, {\mathbf{v}}, {\mathbf{w}}, {\mathbf{t}}{\mathbf{u}}, {\mathbf{t}}{\mathbf{v}}, {\mathbf{t}}{\mathbf{w}}, {\mathbf{u}}{\mathbf{v}}, {\mathbf{u}}{\mathbf{w}}, {\mathbf{v}}{\mathbf{w}}, {\mathbf{t}}{\mathbf{u}}{\mathbf{w}}, {\mathbf{t}}{\mathbf{v}}{\mathbf{w}}, {\mathbf{u}}{\mathbf{v}}{\mathbf{w}}\},\\
{\mathcal{B}} &=& \{{\mathbf{t}}{\mathbf{u}}{\mathbf{w}}, {\mathbf{t}}{\mathbf{v}}{\mathbf{w}}, {\mathbf{u}}{\mathbf{v}}{\mathbf{w}}\},\\
{\mathcal{C}} &=& \{{\mathbf{t}}{\mathbf{u}}{\mathbf{v}}\} \\
{\mathcal{F}} &=& \{\emptyset, {\mathbf{t}}, {\mathbf{u}}, {\mathbf{v}}, {\mathbf{w}}, {\mathbf{t}}{\mathbf{u}}{\mathbf{v}}, {\mathbf{t}}{\mathbf{w}}, {\mathbf{u}}{\mathbf{w}}, {\mathbf{v}}{\mathbf{w}}, {\mathbf{t}}{\mathbf{u}}{\mathbf{v}}{\mathbf{w}}\}
\end{eqnarray*}
The rank function is $r({\mathbf{t}}{\mathbf{u}}{\mathbf{v}}{\mathbf{w}})=3, r({\mathbf{t}}{\mathbf{u}}{\mathbf{v}})=2$, and $r(A) = |A|$ for all other sets $A$. 
The matroid polytope is the equilateral triangle with vertices $(1, 1, 0, 1), (1, 0, 1, 1), (0, 1, 1, 1)$ in ${\mathbb{R}}^4$. 
\end{example}

A vector arrangement $E \subset {{\mathbbm{k}}}^d$ determines a hyperplane arrangement ${\mathcal{A}}$ in ${{\mathbbm{k}}}^d$ consisting of the normal hyperplanes $\{{\mathbf{x}} \in {{\mathbbm{k}}}^d \, : {\mathbf{a}} \cdot {\mathbf{x}} = 0\}$ for ${\mathbf{a}} \in E$. The rank function and lattice of flats of $E$ are the same as the rank function and intersection poset of ${\mathcal{A}}$, as defined in Section \ref{f.sec:charpoly}. For instance, the lattice of flats of the example above is the same as the intersection poset of the corresponding arrangement, which is illustrated in Figure \ref{f.fig:intposet}.

\subsection{{\textsf{Basic definitions}}}\label{f.sec:matroiddefinitions}

A simple but important insight of matroid theory is that many of the properties of linear independence (and many other notions of independence in mathematics) are inherently combinatorial. That combinatorial structure is unexpectedly rich, and matroids provide a general combinatorial framework that is ideally suited for exploring it.

A \textbf{matroid} $M=(E, {\mathcal{I}})$ consists of a finite set $E$ and a collection ${\mathcal{I}}$ of subsets of $E$ such that

 (I1) $\emptyset \in {\mathcal{I}}$

 (I2) If  $J \in {\mathcal{I}}$ and $I \subseteq J$ then $I \in {\mathcal{I}}$.

 (I3) If $I, J \in {\mathcal{I}}$ and $|I|<|J|$ then there exists $j \in J-I$ such that $I \cup j \in {\mathcal{I}}$.

\noindent The sets in ${\mathcal{I}}$ are called \textbf{independent}.

\medskip

\begin{proposition}\label{f.prop:bases}
In a matroid $M=(E,{\mathcal{I}})$, all the \textbf{bases} (the maximal elements of ${\mathcal{I}}$) have the same size, called the \textbf{rank} of the matroid.
%Let $(E,{\mathcal{I}})$ be a matroid. Then all the maximal elements of ${\mathcal{I}}$, called the \textbf{bases}, have the same size.
\end{proposition}

This easy result
%, which is an immediate consequence of (I3), 
shows that this simple definition already captures the notion of dimension.
% gives a small glimpse of the power of matroids.

\medskip

Now we can extend the definitions of Section \ref{f.sec:matroidmotivation} to any matroid % of the previous section.
$M=(E, {\mathcal{I}})$:

$\bullet$ An \textbf{independent set} is a set in ${\mathcal{I}}$.

$\bullet$ A \textbf{basis} is a maximal independent set.

$\bullet$ The \textbf{rank function} $r:2^E \rightarrow {\mathbb{N}}$ is $r(A) = $(size of the largest independent subset of $A$). 
 
$\bullet$ A \textbf{circuit} is a minimal dependent set; i.e., a minimal set not in ${\mathcal{I}}$.
 
$\bullet$ A \textbf{flat} is a set $F$ such that $r(F \cup e) > r(F)$ for all $e \notin F$.

$\bullet$ The \textbf{lattice of flats} $L_M$ is the poset of flats ordered by containment.

$\bullet$ The \textbf{matroid (basis) polytope} $P_M={\mathrm{conv }}\{{\mathbf{e}}_{b_1} + \cdots + {\mathbf{e}}_{b_r} \, : \, \{b_1, \ldots, b_r\} \textrm{ is a basis}\}.$

\medskip

Again, we let ${\mathcal{B}}$, $r$, ${\mathcal{C}}$  and ${\mathcal{F}}$ denote the set of bases, the rank function, the set of circuits, and the set of flats of $M$. If we know the ground set $E$ and any one of ${\mathcal{B}}$, $r$, ${\mathcal{C}}$, ${\mathcal{F}}$, $L_M$, or $P_M$\footnote{assuming that we know the labels of the elements of $L_M$ and the embedding of the polytope $P_M$ in ${\mathbb{R}}^E$.}, we know them all.
%Knowing $E$, we can recover the collection ${\mathcal{I}}$ of independent sets from ${\mathcal{B}}$, $r$, ${\mathcal{C}}$, or ${\mathcal{F}}$. 
For that reason, we call any one (and all) of $(E, {\mathcal{I}}), (E, {\mathcal{B}}), (E,r), (E, {\mathcal{C}})$ and $(E, {\mathcal{F}})$ ``the matroid $M$". Each one of these points of view has its own axiomatization:

\begin{proposition} 
We have the following characterizations of the possible bases ${\mathcal{B}}$, rank function $r$,  circuits ${\mathcal{C}}$, lattice of flats $L$, independent sets ${\mathcal{I}}$, and matroid polytopes $P_M$  of a matroid.

%flats ${\mathcal{F}}$, and matroid polytopes $P_M$  of a matroid.
\begin{enumerate}
\item
The collection ${\mathcal{B}} \subseteq 2^E$ is the set of bases of a matroid if and only if

(B1) ${\mathcal{B}}$ is nonempty.

(B2) If $B_1, B_2 \in {\mathcal{B}}$ and $b_1 \in B_1 - B_2$, there exists $b_2 \in B_2-B_1$ such that $(B_1 - b_1) \cup b_2 \in {\mathcal{B}}$.

\item
The function $r: 2^E \rightarrow {\mathbb{N}}$ is the rank function of a matroid if and only if:

(R1) $0 \leq r(A) \leq |A|$ for all $A \subseteq E$.

(R2) If $A \subseteq B \subseteq E$ then $r(A) \leq r(B)$.

(R3) $r(A) + r(B) \geq r(A \cup B) + r(A \cap B)$ for all $A, B \subseteq E$.

\item
The collection ${\mathcal{C}} \subseteq 2^E$ is the set of circuits of a matroid if and only if

(C1) $\emptyset \notin {\mathcal{C}}$.

(C2) If $C \in {\mathcal{C}}$ and $C \subsetneq D$ then $D \notin {\mathcal{C}}$.

(C3) If $C_1, C_2 \in {\mathcal{C}}$ and $e \in C_1 \cap C_2$ then there exists $C \in {\mathcal{C}}$ with $C \subseteq C_1 \cup C_2 - e$.

\item 
The poset $L$ is the lattice of flats of a matroid if and only if it is a geometric lattice. 

\item 
The collection ${\mathcal{I}} \subseteq 2^E$ is the set of independent sets of a matroid if and only if

(I1) $\emptyset \in {\mathcal{I}}$.

(I2) If  $J \in {\mathcal{I}}$ and $I \subseteq J$ then $I \in {\mathcal{I}}$.

(I3') For every weight function $w: E \rightarrow {\mathbb{R}}$, the following \emph{greedy algorithm}:
\begin{quote}
Start with $B=\emptyset$. Then, at each step, add to $B$ an element $e \notin B$ of minimum weight $w(e)$ such that $B \cup e \in {\mathcal{I}}$. Stop when $B$ is maximal in ${\mathcal{I}}$.
\end{quote}
produces a maximal element $B$ of ${\mathcal{I}}$ of minimum weight $w(B) = \sum_{b \in B} w(b)$.

\item \cite{f.GGMS}
The collection ${\mathcal{B}} \subseteq 2^E$ is the set of bases of a matroid if and only if every edge of the polytope ${\mathrm{conv }}\{{\mathbf{e}}_{b_1} + \cdots + {\mathbf{e}}_{b_r} \, : \, \{b_1, \ldots, b_r\} \in {\mathcal{B}}\}$ in ${\mathbb{R}}^E$ is a translate of ${\mathbf{e}}_i-{\mathbf{e}}_j$ for some $i, j \in E$. 

%\item \comment{Should I include closure? Greedy algorithm? Circuits?}
\end{enumerate}
\end{proposition}

This proposition allows us to give six other definitions of matroids in terms of bases, ranks, circuits, flats, independent sets, and polytopes. The first three axiom systems are ``not too far" from each other, but it is useful to have them all. For instance, to prove that linear matroids are indeed matroids, it is  easier to check the circuit axioms (C1)--(C3) than the independence axioms (I1)--(I3). The last three axiom systems suggest deeper, fruitful connections to posets, optimization, and polytopes.

In fact, there are several other equivalent (and useful) definitions of matroids. It is no coincidence that matroid theory gave birth to the notion of a \textbf{cryptomorphism}, which refers to the equivalence of different axiomatizations of the same object. This feature of matroid theory can be frustrating at first; but as one goes deeper, it becomes indispensable and extremely powerful. 
A very valuable resource is \cite{f.Brylawskicrypto}, a multilingual cryptomorphism dictionary that translates between these different points of view.

\medskip

We say $(E, {\mathcal{B}})$ and $(E', {\mathcal{B}}')$ are \textbf{isomorphic} if there is a bijection between $E$ and $E'$ that induces a bijection between ${\mathcal{B}}$ and ${\mathcal{B}}'$. 
A \textbf{loop} of $M$ is an element $e \in E$ of rank $0$, so $\{e\}$ is dependent. A \textbf{coloop} is an element $e$ which is independent of $E-e$, so $r(E-e)=r-1$. 
Non-loops ${\mathbf{e}}$ and ${\mathbf{f}}$ are \textbf{parallel} if $r({\mathbf{e}}{\mathbf{f}})=1$; they are indistinguishable inside the matroid. A matroid is \textbf{simple} if it contains no loops or parallel elements. The map $M \mapsto L_M$ is a bijection between simple matroids and geometric lattices.

\subsection{{\textsf{Examples}}}\label{f.sec:matroidexamples}

We now describe a few contexts where matroids arise naturally. These statements are nice (and not always trivial) exercises in their respective fields. Most proofs may be found, for example, in \cite{f.Oxley}. We give references for the rest.

\begin{enumerate}
\item[0.] (Uniform matroids) For $n\geq k \geq 1$, the \textbf{uniform matroid} $U_{k,n}$ is the matroid on $[n]$ where every $k$-subset is a basis. 
It is the linear matroid of $n$ generic vectors in a $k$-dimensional space. 
%It has $n$ elements, and the bases are all the $k$-subsets of the ground set.

\item[1a.] (Vector configurations) Let $E$ be a set of vectors in ${{\mathbbm{k}}}^d$ and let ${\mathcal{I}}$ be the set of linearly independent subsets of $E$. Then $(E, {\mathcal{I}})$ is a matroid. Such a matroid is called \textbf{linear} or \textbf{representable} (over ${{\mathbbm{k}}}$). \cite{f.MacLane, f.Whitney} There are at least three other ways of describing the same family of matroids, listed in 1b, 1c, and 1d below.

\item[1b.] (Point configurations) Let $E$ be a set of points in ${{\mathbbm{k}}}^d$ and let ${\mathcal{I}}$ be the set of affinely independent subsets of $E$; that is, the subsets $\{{\mathbf{a}}_1, \ldots, {\mathbf{a}}_k\}$ whose affine span is $k$-dimensional. Then $(E, {\mathcal{I}})$ is a matroid. 

\item[1c.] (Hyperplane arrangements) Let $E$ be a hyperplane arrangement in ${{\mathbbm{k}}}^d$ and let $r(\{H_1, \ldots, H_k\}) = d - \dim(H_1 \cap \cdots \cap H_k)$ for $\{H_1, \ldots, H_k\} \subseteq E$. Then $(E, r)$ is a matroid. 

\item[1d.] (Subspaces) Let $V \cong {{\mathbbm{k}}}^E$ be a vector space with a chosen basis, and let $U \subseteq V$ be a subspace of codimension $r$. Consider the coordinate subspaces $V_S = \{{\mathbf{v}} \in V \, : \, v_s = 0 \textrm{ for } s \in S\}$. Say an $r$-subset $B \subseteq [d]$ is a basis if $U \cap  V_B = \{0\}$. If ${\mathcal{B}}$ is the set of bases, $(E, {\mathcal{B}})$ is a matroid.

Given a matrix over ${{\mathbbm{k}}}$, the matroid on the columns of $A$ (in the sense of 1a) is the same as the matroid of the rowspace of $A$ (in the sense of 1d).

\setcounter{enumi}{1}

\item (Graphs) Let $E$ be the set of edges of a connected graph $G$ and let ${\mathcal{C}}$ be the set of cycles of $G$ (where each cycle is regarded as the list of its edges). Then $(E, {\mathcal{C}})$ is a matroid. Such a matroid is called \textbf{graphical}. The bases are the spanning trees of $G$. 

\item (Field extensions) Let $\mathbb{F} \subseteq \mathbb{K}$ be two fields, and let $E$ be a finite set of elements in the extension field $\mathbb{K}$. Say $I=\{i_1, \ldots, i_k\} \subseteq E$ is independent if it is algebraically independent over $\mathbb{F}$; that is, if there does not exist a non-trivial polynomial $P(x_1, \ldots, x_k)$ with coefficients in $\mathbb{F}$ such that $P(i_1, \ldots, i_k)=0$. If ${\mathcal{I}}$ is the collection of independent sets,  $(E, {\mathcal{I}})$ is a matroid. Such a matroid is called \textbf{algebraic}. \cite{f.VanderWaerden}

\item (Matchings) Let $G=(S \cup T, E)$ be a bipartite graph, so every edge in $E$ joins a vertex of $S$ and a vertex of $T$. Recall that a \textbf{partial matching} of $G$ is a set of edges, no two of which have a common vertex. Let ${\mathcal{I}}$ be the collection of subsets $I \subseteq T$ which can be matched to $S$; that is, those for which there exists a partial matching whose edges contain the vertices in $I$. Then $(T, {\mathcal{I}})$ is a matroid. Such a matroid is called \textbf{transversal}. \cite{f.EdmondsFulkerson}

\item (Routings) Let $G=(V,E)$ be a directed graph and $B_0$ be an $r$-subset of the vertices. Let ${\mathcal{B}}$ be the collection of $r$-subsets of $V$ for which there exists a routing (a collection of $r$ vertex-disjoint paths) starting at $B$ and ending at $B_0$. Then $(V,{\mathcal{B}})$ is a matroid. Such a matroid is called \textbf{cotransversal}. \cite{f.Mason}

\item (Lattice paths) For each Dyck path $P$ of length $n$, let $B \subseteq [2n]$ be the upstep set of $P$; that is, the set of integers $i$ for which the $i$th step of $P$ is an upstep. Let ${\mathcal{B}}$ be the collection of upstep sets of Dyck paths of length $n$. Then $([2n], {\mathcal{B}})$ is the set of basis of a matroid. This matroid is called a \textbf{Catalan matroid} $\mathbf{C}_n$. 
More generally, we may consider the lattice paths of $2n$  steps $(1,1)$ and $(1, -1)$ which stay between an upper and a lower border; their upstep sets form the bases of a \textbf{lattice path matroid}. \cite{f.ArdilaCatalan, f.BonindeMierNoy}

\item (Schubert matroids) Given $n \in {\mathbb{N}}$ and a set of positive integers $I=\{i_1 < \cdots < i_k\}$, the sets $\{j_1 < \cdots < j_k\} \subseteq [n]$ such that $j_1 \geq i_1, \ldots, j_k \geq i_k$ are the bases of a matroid. These matroids are called  \textbf{Schubert matroids} due to their connection to the Schubert cells in the Grassmannian Gr$(k,n)$. 
%In the stratification of the Grassmannian of $k$-planes in $n$-space into Schubert cells, the matroid of a generic point on a Schubert cell is a Schubert matroid.
This simple but important family of matroids has been rediscovered many times under  names such as freedom matroids \cite{f.CrapoSchmitt} and shifted matroids \cite{f.ArdilaCatalan, f.Klivans}.

\item (Positroids) A \textbf{positroid} is a matroid on $[n]$ which can be represented by the columns of a full rank $d \times n$ matrix such that all its maximal minors are nonnegative. \cite{f.PostnikovTNN} These matroids arise in the study of the totally nonnegative part of the Grassmannian, and have recently found applications in the physics of scattering amplitudes. Remarkably, these matroids can be described combinatorially, and they are in bijection with several interesting classes of objects with elegant enumerative properties. \cite{f.ArdilaRinconWilliams, f.Oh, f.PostnikovTNN, f.WilliamsLe}

\end{enumerate}

\begin{figure}[ht]
 \begin{center}
  \includegraphics[scale=1]{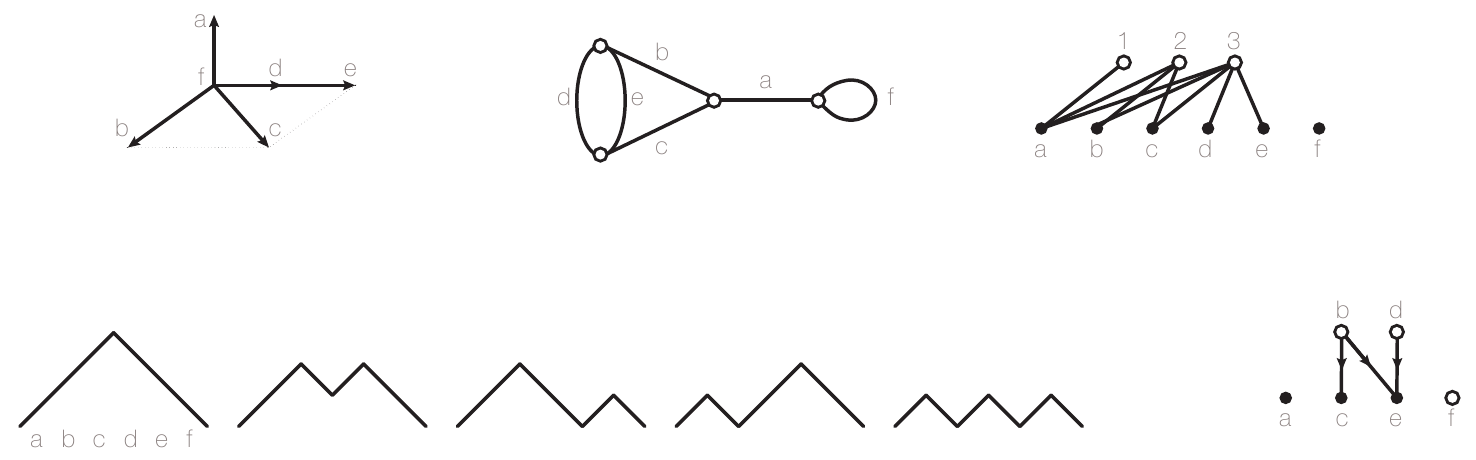}
  \caption{ \label{f.fig:exmatroid}
Many manifestations of the same matroid.}
 \end{center}
\end{figure}

\begin{example}
The list above gives eight different manifestations of the matroid on $\{a,b,c,d,e,f\}$ with bases ${\mathcal{B}}=\{abc, abd, abe, acd, ace\}$. Figure \ref{f.fig:exmatroid} shows a vector configuration in ${\mathbb{R}}^3$ (with $f=\mathbf{0}$) which realizes $M$ as a linear matroid, a graph which realizes it as a graphical matroid, a bipartite graph which realizes it as a transversal matroid, the five Dyck paths of length $3$ which show that $M$ is the Catalan matroid $\mathbf{C}_3$, and a directed graph with sinks $B_0=\{a,c,e\}$ which realize $M$ as a cotransversal matroid. It can  be realized as an algebraic matroid in the field extension ${\mathbb{R}} \subset {\mathbb{R}}(x,y,z)$ with $a=z^3, \, b=x+y, \, c=x-y, \, d=xy, \, e=x^2y^2, \, f=1$. It is also isomorphic to the Schubert matroid on $[6]$ for $I=\{2,4,6\}$. Finally, it is a positroid, since in the realization shown, all five bases are positive: they all satisfy the ``right-hand rule" for vectors in ${\mathbb{R}}^3$.
\end{example}

To give a small illustration of the power of matroids, let us revisit the almost trivial observation that all bases of a matroid have the same size. This result has many interesting consequences: 

$\bullet$ All bases of a vector space $V$ have the same size, called the ``dimension" of $V$. 

$\bullet$  All spanning trees of a connected graph have the same number ($v-1$) of edges, where $v$ is the number of vertices of the graph.

$\bullet$ In a bipartite graph $G=(S \cup T, E)$, all the maximal subsets of $T$ that can be matched to $S$ have the same size.

$\bullet$ All transcendence bases of a field extension have the same size, called the ``transcendence degree" of the extension. 

\noindent Naturally, once we work harder to obtain more interesting results about matroids, we will obtain more impressive results in all of these different fields.

\subsection{{\textsf{Basic constructions}}}\label{f.sec:matroidconstructions}

We now discuss the notions of duality, minors, direct sums, and connected components in matroids; in Section \ref{f.sec:matroidstructure} we explain how they generalize similar notions in linear algebra and graph theory.

\begin{propdef} (Duality)
Let $M=(E, {\mathcal{B}})$ be a matroid, and let ${\mathcal{B}}^* = \{E-B \, : \, B \in {\mathcal{B}}\}$. Then $M^* =  (E, {\mathcal{B}}^*)$ is also a matroid, called the \textbf{dual matroid} of $M$.
\end{propdef}

\begin{proposition}
For any matroid $M$ we have $(M^*)^* = M$.
\end{proposition}

Note that the loops of a matroid $M$ are the coloops of $M^*$.

\begin{propdef} (Deletion, contraction)
For a matroid $M=(E, {\mathcal{I}})$ and $s \in E$, let
\[
{\mathcal{I}}' = \{I \in {\mathcal{I}} \, : \, s \notin I\}, \qquad \qquad {\mathcal{I}}'' = 
\begin{cases}
 \{I \subseteq E-s \, : \, I \cup s \in I\}, \quad & \textrm{ if $s$ is not a loop}\\
{\mathcal{I}}, \quad & \textrm{ if $s$ is a loop}
\end{cases}
\]
Then $M \backslash s = (E-s, {\mathcal{I}}')$ and $M / s = (E-s, {\mathcal{I}}'')$ are also matroids. They are called the \textbf{deletion} of $s$ from $M$ (or the \textbf{restriction} of $M$ to $E-s$) and %$M / e = (E-e, {\mathcal{I}}'')$ is a matroid called 
the \textbf{contraction} of $s$ from $M$.
\end{propdef}

\begin{propdef} (Minors) Let $M$ be a matroid on $E$.
\begin{enumerate}
\item
Deletion and contraction commute; that is, for all $s \neq t$ in $E$,
\[
(M \backslash s) \backslash  t = (M\backslash t)\backslash s, \quad (M/s)/t = (M/t)/s, \quad (M/s)\backslash t = (M\backslash t)/s.
\]
\item
Let the \textbf{deletion} $M \backslash S$ (resp. \textbf{contraction} $M/S$) of a subset $S \subseteq E$ from $M$ be the successive deletion (resp. contraction) of the individual elements of $S$. 
Their rank functions are
\[
r_{M\backslash S}(A) = r(A), \qquad \qquad r_{M/S}(A) = r(A \cup S) - r(S) \qquad \qquad (A \subseteq E-S)
\]
Let a \textbf{minor} of $M$ be a matroid of the form $M / S \backslash T$ for disjoint $S, T \subseteq E$.
\item
Deletion and contraction are dual operations; that is, for all $S \subseteq E$,
\[
(M\backslash S)^* = M^*/S.
\]
\end{enumerate}
\end{propdef}

\begin{propdef} (Direct sum, connected components)
\begin{enumerate}
\item
If $M=(E_1, {\mathcal{I}}_1)$ and $M_2 = (E_2, {\mathcal{I}}_2)$ are matroids on disjoint ground sets, then ${\mathcal{I}}=\{I_1 \cup I_2 \, : \, I_1 \in {\mathcal{I}}_1, I_2  \in {\mathcal{I}}_2\}$ is the collection of independent sets of a matroid $M_1 \oplus M_2 = (E_1 \cup E_2, {\mathcal{I}})$, called the \textbf{direct sum} of $M_1$ and $M_2$.
\item
Any matroid $M$ decomposes uniquely as a direct sum $M=M_1 \oplus \cdots \oplus M_c$. The ground sets of $M_1, \ldots, M_c$ are called the \textbf{connected components} of $M$. If $c=1$, $M$ is \textbf{connected}.
\end{enumerate}
\end{propdef}

The matroid polytope of $M$ has $\dim(P_M) = n-c$ where $n$ is the number of elements of $M$ and $c$ is the number of connected components.

\bigskip

\subsection{\textsf{{A few structural results}}} \label{f.sec:matroidstructure}
Although they will not be strictly necessary in our discussion of enumerative aspects of matroids, it is worthwhile to mention a few important structural results. In particular, in this section we explain the meaning of duality and minors for various families of matroids.

\bigskip
\noindent \textsf{\textbf{Duality.}}

$\bullet$
The dual of a linear matroid is linear. 
When $M$ is the matroid of a subspace $V$ of a vector space $W$, $M^*$ is the matroid of its orthogonal complement $V^\perp$. For that reason, $M^*$ is sometimes called the \textbf{orthogonal matroid} of $M$.  

 $\bullet$
The dual of a graphical matroid is not necessarily graphical; however, matroid duality is a generalization of graph duality.
We say a graph $G$ is  \textbf{planar} if it can be drawn on the plane so that its edges intersect only at the endpoints. Such a drawing is called a \textbf{plane} graph. The dual graph $G^*$ is obtained by putting a vertex $v^*$ inside each region of $G$ (including the ``outside region"), and joining $v^*$ and $w^*$ by an edge labelled $e^*$ if the corresponding regions of $G$ are separated by an edge $e$ in $G$. This construction is exemplified in Figure \ref{f.fig:dualgraph}. The graph $G^*$ depends on the drawing of $G$, but its matroid does not: if $G$ is planar, we have $M(G^*) = M(G)^*$.

\begin{figure}[ht]
 \begin{center}
  \includegraphics[scale=.8]{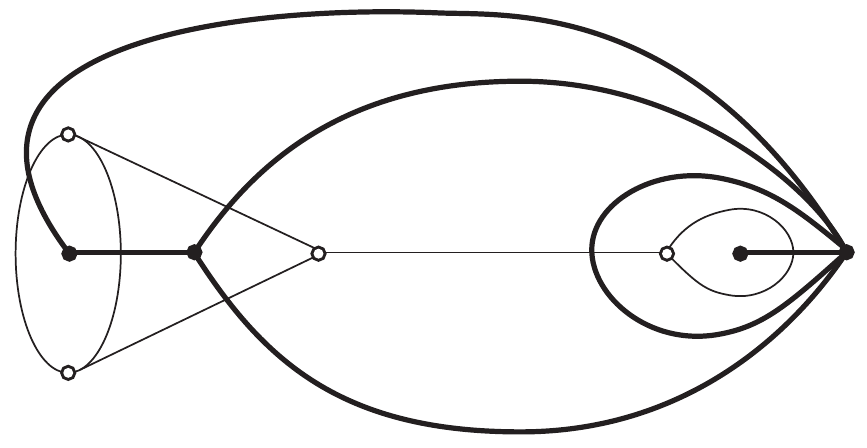}
  \caption{ \label{f.fig:dualgraph}
A plane graph $G$ and its dual $G^*$.}
 \end{center}
\end{figure}

 $\bullet$
Embarrassingly, no one knows whether the dual of an algebraic matroid is always algebraic. This illustrates how poorly we currently understand this family of matroids.

 $\bullet$
 Cotransversal matroids are precisely the duals of  transversal matroids.

 $\bullet$
The Catalan matroid is self-dual. The dual of a lattice path matroid is a lattice path matroid.

$\bullet$
Linear programming duality can be framed and generalized in the context of \emph{oriented matroids}; see \cite{f.OMs} for details. 

\begin{corollary}
If a connected planar graph has $v$ vertices, $e$ edges, and $f$ faces, then $v-e+f=2$.
\end{corollary}

\begin{proof}
A basis of our graph $G$ has $v-1$ elements and a basis of its dual graph $G^*$ has $f-1$ elements; so if $M$ is the matroid of $G$, we have  $(v-1)+(f-1)=r(M) + r(M^*) = e$. 
\end{proof}

\bigskip

\noindent \textsf{\textbf{Minors.}}

$\bullet$
The minors of a linear matroid are linear. 
If $M$ is the matroid of a vector configuration $E \subset V$ and ${\mathbf{v}} \in E$, then $M \backslash {\mathbf{v}}$ is the matroid of $E \backslash {\mathbf{v}}$, and $M / {\mathbf{v}}$ is the matroid of $E \backslash {\mathbf{v}}$ modulo ${\mathbf{v}}$; that is, the matroid of $\pi(E \backslash {\mathbf{v}})$, where $\pi: V \rightarrow V/({\mathrm{span }} \, {\mathbf{v}})$ is the canonical projection map.

 $\bullet$
The minors of a graphical matroid are graphical.
If $M$ is the matroid of a graph $G$ and $e=uv$ is an edge, then $M \backslash e$ is the matroid of the graph $G \backslash e$ obtained by removing the edge $e$, and $M/e$ is the matroid of the graph $G/e$ obtained by removing the edge $e$ and identifying vertices $u$ and $v$.

 $\bullet$ \cite[Cor. 6.7.14]{f.Oxley}
The minors of an algebraic matroid are algebraic.

 $\bullet$
Transversal matroids are closed under contraction, but not deletion. Cotransversal matroids are closed under deletion, but not contraction. A \textbf{gammoid} is a deletion of a transversal matroid or, equivalently, a contraction of a cotransversal matroids. Gammoids are the smallest minor-closed class of matroids containing transversal (or cotransversal) matroids.

 $\bullet$ \cite{f.BonindeMierNoy}
Catalan matroids are not closed under deletion or contraction. Lattice path matroids are closed under deletion and contraction.

\bigskip

\noindent \textsf{\textbf{Direct sums and connectivity.}}

$\bullet$
If $M_1$ and $M_2$ are the linear matroids of configurations $E_1 \subset V_1$ and $E_2 \subset V_2$ over the same field ${{\mathbbm{k}}}$, then $M_1 \oplus M_2$ is the matroid of $\{({\mathbf{v}}_1,\textbf{0}) \, : \, {\mathbf{v}}_1 \in E_1\}  \cup \{(\textbf{0}, {\mathbf{v}}_2) \, : \, {\mathbf{v}}_2 \in E_2\}$ in $V_1 \oplus V_2$.

$\bullet$
If $M_1$ and $M_2$ are the matroids of graphs $G_1$ and $G_2$, then $M_1 \oplus M_2$ is the matroid of the disjoint union of $G_1$ and $G_2$. The matroid  of a loopless graph $G$ is connected if and only if $G$ is $2$-connected; that is, $G$ is connected, and there does not exist a vertex of $G$ whose removal disconnects it.

\bigskip
\noindent \textsf{\textbf{Non-representability.}} Although it is conjectured that almost all matroids are not linear and not algebraic, it takes some effort to construct examples of non-linear and non-algebraic matroids. We describe the simplest examples.

\begin{figure}[ht]
 \begin{center}
  \includegraphics[scale=.7]{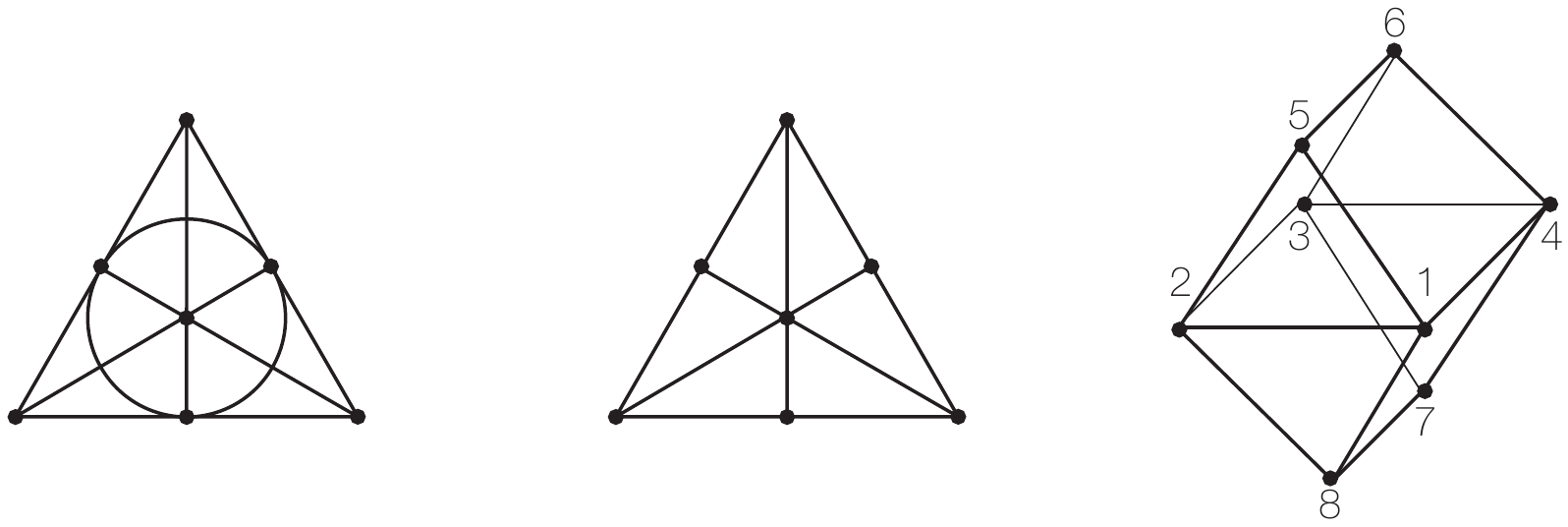}
  \caption{ \label{f.fig:FanoyVamos}
The Fano matroid, the non-Fano matroid $F_7^-$, and the V\'amos matroid $V_8$.}
 \end{center}
\end{figure}

$\bullet$ The \textbf{Fano matroid} $F_7$ is the linear matroid determined by the seven non-zero vectors in ${\mathbb{F}}_2^3$. This matroid is illustrated in the left panel of Figure \ref{f.fig:FanoyVamos}, where an element is denoted by a vertex, and a circuit is denoted by a line or circle joining the points in question. The Fano matroid is representable over ${{\mathbbm{k}}}$ if and only if $\textrm{char}({{\mathbbm{k}}}) = 2$.

$\bullet$ The \textbf{non-Fano matroid} $F_7^-$ is the linear matroid determined by the seven non-zero 0-1 vectors $(0,0,1), (0,1,0), (1,0,0), (0,1,1), (1,0,1), (1,1,0), (1,1,1)$ in ${\mathbb{R}}^3$. It is representable over ${{\mathbbm{k}}}$ if and only if $\textrm{char}({{\mathbbm{k}}}) \neq 2$. 

$\bullet$ The direct sum $F_7 \oplus F_7^-$ is not representable over any field.

$\bullet$ The \textbf{V\'amos matroid} $V_8$  is a matroid on $[8]$ of rank $4$ whose only non-trivial circuits are $1234$, $1456$, $1478$, $2356$, $2378$, as illustrated in the right panel of Figure \ref{f.fig:FanoyVamos}. This matroid is not representable over any field, because in any affine configuration of points with these coplanarities, $5678$ would also be coplanar. The V\'amos matroid is also not algebraic over any field. 
%It is conjectured that asymptotically every matroid has $V_8$ as a minor, and hence is not representable or algebraic over any field. \cite{f.Mayhewetal}

\bigskip
\noindent \textsf{\textbf{Relating various classes.}} Most matroids that we encounter naturally in combinatorics are linear, although this  is usually not clear from their definition. 

$\bullet$ All the examples of Section \ref{f.sec:matroidexamples} are linear, with the possible exception of algebraic matroids.

$\bullet$ Graphical, transversal, cotransversal matroids, and positroids are linear.

$\bullet$ Linear matroids over ${{\mathbbm{k}}}$  are algebraic over ${{\mathbbm{k}}}$.

$\bullet$ Algebraic matroids over ${{\mathbbm{k}}}$ of characteristic $0$ are linear over a field extension of ${{\mathbbm{k}}}$.
%
%$\bullet$ The Fano matroid $F_7$ is linear over ${{\mathbbm{k}}}$ if and only if $\textrm{char}({{\mathbbm{k}}}) = 2$. The non-Fano matroid $F_7^-$ is linear over ${{\mathbbm{k}}}$ if and only if $\textrm{char}({{\mathbbm{k}}}) \neq 2$. Therefore $F_7 \oplus F_7^-$ is not linear.
%
%$\bullet$ The V\'amos matroid is not algebraic. 

$\bullet$ Catalan matroids are Schubert matroids.

$\bullet$ Schubert matroids are lattice path matroids.

$\bullet$ Lattice path matroids are transversal.

$\bullet$ Positroids are gammoids.

\bigskip
\noindent \textsf{\textbf{Obstructions to representability.}} A general question in structural matroid theory is to try to describe a family ${\mathcal{F}}$ of matroids by understanding the ``obstructions" to belonging to ${\mathcal{F}}$. These results are inspired by Kuratowski's Theorem in graph theory, which asserts that a graph is planar if and only if it does not contain the complete graph $K_5$ or the complete bipartite graph $K_{3,3}$ as a minor. Similarly, one seeks a (hopefully finite) list of minors that a matroid must avoid to be in ${\mathcal{F}}$. 

The question that has received the most attention has been that of characterizing the matroids that are representable over a particular field. Here are some results:

 $\bullet$ $M$ is representable over ${\mathbb{F}}_2$ if and only if it has no minor isomorphic to $U_{2,4}$.

 $\bullet$ $M$ is representable over ${\mathbb{F}}_3$ if and only if it has no minor isomorphic to $U_{2,5}, U_{3,5}, F_7,$ or $F_7^*$.

 $\bullet$ $M$ is representable over every field if and only if it has no minor isomorphic to $U_{2,4}, F_7,$ or $F_7^*$.

 $\bullet$ $M$ is graphical if and only if it has no minor isomorphic to $U_{2,4}, F_7, F_7^*, M(K_5)^*,$ or $M(K_{3,3})^*$.

 $\bullet$ For any field ${\mathbb{F}}$ of characteristic zero, there are infinitely many minor-minimal matroids that are not representable over ${\mathbb{F}}$.

 $\bullet$ (Rota's Conjecture) For any finite field ${\mathbb{F}}_q$, there is a finite set $\mathcal{S}_q$ of matroids such that $M$ is representable over ${\mathbb{F}}_q$ if and only if it contains no minor isomorphic to a matroid in $\mathcal{S}_q$.

In 2013, Geelen, Gerards, and Whittle announced a proof of Rota's conjecture. 
In \cite{f.GeelenGerardsWhittle} they write: \emph{``We are now immersed in the lengthy task of writing up our results. Since that process will take a few years, we have written this article offering a high-level preview of the proof".}

\bigskip
\noindent \textsf{\textbf{Asymptotic enumeration.}} The asymptotic growth of various classes of matroids is not well understood. Let $P$ be a matroid property, and let $p(n)$ %(resp. $p_u(n)$) 
be the number of matroids on $[n]$
% labelled 
%(resp. unlabelunlabeleded) elements 
having property $P$. Let $m(n)$ %(resp. $u(n)$) 
be the number of matroids on $[n]$.
 %labelled (resp. unlabelled) elements. 
If  $\lim_{n \rightarrow \infty} p(n)/m(n) = 1$, we say that asymptotically almost every matroid has property $P$. 
%The asymptotic proportion of matroids having property $P$ is $\lim_{n \rightarrow \infty} p(n)/m(n)$, if it exists.
 %or $\lim_{n \rightarrow \infty} p_u(n)/u(n)$,  respectively.  %-- when one of these limits exists, the other one also exists, and they are equal. 

\bigskip

We list some conjectures: \cite{f.Mayhewetal}

$\bullet$ Asymptotically almost every matroid is connected.

$\bullet$  Asymptotically almost every matroid $M$ is \textbf{paving}; that is, 
%A matroid of rank $r$ is \textbf{paving} if 
all its circuits have size $r$ or $r+1$.

$\bullet$ Asymptotically almost every matroid has a trivial automorphism group.

$\bullet$ Asymptotically almost every matroid is not linear.

$\bullet$ Asymptotically almost every matroid of $n$ elements on rank $r$ satisfies $(n-1)/2 \leq r \leq (n+1)/2$.

\bigskip

We also list some results:

$\bullet$  \cite{f.BansalPendavinghVanderpol, f.Knuthmatroids} The number $m_n$ of matroids on $[n]$ satisfies 
\[
n - \frac32 \log n + \frac12 \log \frac2\pi - o(1) \leq 
\log \log m_n \leq n - \frac32 \log n + \frac12 \log \frac2\pi + 1 + o(1)
\]

$\bullet$ \cite{f.Mayhewetal} Asymptotically almost every matroid is loopless and coloopless.

$\bullet$ \cite{f.Mayhewetal} Asymptotically, the proportion of matroids that are connected is at least $1/2$. (It is conjectured to equal $1$.)

$\bullet$ \cite{f.ArdilaRinconWilliams} Asymptotically, the proportion of \textbf{positroids} that are connected is $1/e^2$.

\subsection{\textsf{{The Tutte polynomial}}}\label{f.sec:Tutte}

Throughout this section, let $R$ be an arbitrary commutative ring. A function $f:\textrm{Matroids} \rightarrow R$ is a \textbf{matroid invariant} if $f(M) = f(N)$ whenever $M \cong N$.  Let $L$ and $C$ be the matroids consisting of a single loop and a single coloop, respectively. 

\subsubsection{\textsf{{Explicit definition}}}

The invariant that appears most often in enumerative questions related to matroids is the \textbf{Tutte polynomial}
\begin{equation}\label{f.th:Tutteformula}
T_M(x,y) = \sum_{A \subseteq M} (x-1)^{r-r(A)} \, (y-1)^{|A|-r(A)}.
\end{equation}
When we meet a new matroid invariant, a good first question to ask is whether it is an evaluation of the Tutte polynomial.

\subsubsection{\textsf{{Recursive definition and universality property}}} The ubiquity of the Tutte polynomial is not an accident: it is \emph{universal} in a large, important family of matroid invariants. Let us make this precise. Let $R$ be a ring. Say $f:\textrm{Matroids} \rightarrow R$ is a \textbf{generalized Tutte-Grothendieck invariant} if for every matroid $M$ and every element $e \in M$, we have
\begin{equation} \label{f.eq:T-G}
f(M) =
\begin{cases} 
a f(M \backslash e) +  b f(M / e)  & \textrm{ if $e$ is neither a loop nor a coloop} \\
f(M \backslash e) f(L)  &  \textrm{ if $e$ is a loop}\\
f(M / e) f(C)  & \textrm{ if $e$ is a coloop}
\end{cases}
\end{equation}
for some non-zero constants $a,b \in R$.We say $f(M)$  is a \textbf{Tutte-Grothendieck invariant} when $a=b=1$.

\newpage

\begin{theorem} \label{f.th:Tutterecursion} The Tutte polynomial is a universal Tutte-Grothendieck invariant; namely:
\begin{enumerate}
\item For every matroid $M$ and every element $e \in M$
\[
T_M(x,y) =
\begin{cases} 
 T_{M \backslash e}(x,y) +   T_{M / e}(x,y)  & \textrm{ if $e$ is neither a loop nor a coloop}, \\
y  \, T_{M \backslash e}(x,y)   &  \textrm{ if $e$ is a loop},\\
x \, T_{M / e}(x,y)  & \textrm{ if $e$ is a coloop}.
\end{cases}
\]
\item Any generalized Tutte-Grothendieck invariant is a function of 
 %can be expressed in terms of 
 the Tutte polynomial. Explicitly, if $f$ satisfies (\ref{f.eq:T-G}), then
\[
f(M) = a^{n-r} \, b^r  \, T_M\left(\frac{f(C)}{b}, \frac{f(L)}{a}\right).
\]
where $n$ is the number of elements and $r$ is the rank  of $M$.\footnote{We do not need to assume $a$ and $b$ are invertible; when we multiply by $a^{n-r} \, b^r$, we cancel all denominators.}
\end{enumerate}
\end{theorem}

\begin{proof}[Sketch of Proof] This is a very powerful result with a very simple proof. The first part is a straightforward computation from the definitions. The second statement then follows easily by induction on the number of elements of $M$. 
\end{proof}

More generally, write $T(x,y) = \sum_{i,j} t_{ij}x^iy^j$. If $f$ satisfies $f(M) = f(M \backslash e) + f(M / e)$ when $e$ is neither a loop nor a coloop, then $f(M) = \sum_{i,j} t_{ij} f(C^i \oplus L^j)$, where $C^i \oplus L^j$ denotes the matroid consisting of $i$ coloops and $j$ loops.

\medskip

%\subsubsection{ \textsf{\textbf{Recursive definition.}}} 

%The Tutte polynomial behaves well with respect to duality and direct sums, as follows easily from the definitions:

\begin{proposition} The Tutte polynomial behaves well with respect to duality and direct sums:
\begin{enumerate}
\item For any matroid $M$, $T_{M^*}(x,y) = T_M(y,x)$.
\item For any matroids $M$ and $N$ on disjoint ground sets, $T_{M \oplus N}(x,y) = T_M(x,y)T_N(x,y)$.
\end{enumerate}
\end{proposition}

\subsubsection{\textsf{{Activity interpretation}}} An unexpected consequence of Theorem \ref{f.th:Tutterecursion} is that the Tutte polynomial has non-negative coefficients; this is not at all apparent from the explicit formula (\ref{f.th:Tutteformula}). As usual, the natural question for a combinatorialist is: what do these coefficients count? The natural question for an algebraist is: what  vector spaces have these coefficients as their dimensions? At the moment, the first question has a nice answer, while the second one does not. 

Fix a linear order $<$ on the elements of $E$. Say an element $i \in B$ is \textbf{internally active} if there is no basis $B-i \cup j$ with $j<i$. Say an element $j\notin B$ is \textbf{externally active} if there is no basis $B-i \cup j$ with $i<j$. These are dual notions: $i$ is internally active with respect to  basis $B$ in $M$ if and only if it is externally active with respect to basis $E-B$ in $M^*$.

\begin{theorem}\cite{f.Crapo, f.Tuttecontribution}
For any linear order $<$ on the ground set of a matroid, let $I(B)$ and $E(B)$ be, respectively, the set of internally active and externally active elements with respect to $B$. Then 
\[
T_M(x,y) = \sum_{B \textrm{ basis}} x^{|I(B)|}y^{|E(B)|}.
\]
\end{theorem}

\begin{proof}[Sketch of Proof.] It is possible to give a standard deletion-contraction proof, by showing that the right hand side is a Tutte-Grothendieck invariant, and applying Theorem \ref{f.th:Tutterecursion}. However, there is a much more enlightening explanation. One may prove \cite{f.Crapo} that the intervals $[B \backslash I(B), B \cup E(B)]$ partition the Boolean lattice $2^E$. In other words, every subset of the ground set can be written uniquely in the form $(B  \backslash J)  \cup F$ where $B$ is a basis, $J \subseteq I(B),$ and $F \subseteq E(B)$. Furthermore, under these assumptions $r((B  \backslash J)  \cup F) = r-|J|$. 
Then (\ref{f.th:Tutteformula}) becomes
\[
T_M(x,y) = \sum_{B \textrm{ basis }} \sum_{J \subseteq I(B)}  \sum_{F \subseteq E(B)} (x-1)^{|J|} (y-1)^{|F|} = \sum_{B \textrm{ basis}} x^{|I(B)|}y^{|E(B)|}
\]
as desired.
\end{proof}

This combinatorial interpretation of the non-negative coefficients of the Tutte polynomial is very interesting, but it would be even nicer to find an interpretation that does not depend on choosing a linear order on the ground set.

In a different direction, Procesi \cite{f.DeConciniProcesiBjorner} asked the following question: Given a matroid $M$, is there a natural bigraded algebra whose bigraded Hilbert polynomial is $T_M(x,y)$? This question is still open. It may be more tractable (and still very interesting) when $M$ is representable; we will see some approximations in Section \ref{f.sec:Tutteevaluations}.

\subsubsection{\textsf{{Finite field interpretation}}} The \textbf{coboundary polynomial} $\overline{\chi}_M(X,Y)$ is the following simple transformation of the Tutte polynomial:
\[
\overline{\chi}_M(X,Y) = (Y-1)^rT_M\left(\frac{X+Y-1}{Y-1}, Y\right). 
\]
%\[
%\overline{\chi}_M(X,Y) = (Y-1)^rT_M\left(\frac{X+Y-1}{Y-1}, Y\right), \qquad 
%\]
%These polynomials are equivalent; i
It is clear how to recover $T(x,y)$ from $\overline{\chi}(X,Y)$. We have the following interpretation of the coboundary (and hence the Tutte) polynomial.

\begin{theorem}\label{f.th:Tuttefinitefield} (Finite Field Method) \cite{f.ArdilaTutte, f.CrapoRota, f.WelshWhittle} Let ${\mathcal{A}}$ be a hyperplane arrangement of rank $r$ in ${\mathbb{F}}_q^d$. For each point $p \in {\mathbb{F}}_q^d$ let $h(p)$ be the number of hyperplanes of ${\mathcal{A}}$ containing $p$. Then
\[
\sum_{p \in {\mathbb{F}}_q^d} t^{h(p)} = q^{d-r} \overline{\chi}(q,t).
\]
\end{theorem}

\begin{proof}
%Let $q$ be a power of a large enough prime, so that ${\mathcal{A}}$ reduces correctly over ${\mathbb{F}}_q$. 
%For each ${\mathcal{B}} \subseteq {\mathcal{A}}$, let ${\mathcal{B}}_q$ be the subarrangement of ${\mathcal{A}}_q$ induced by it. 
For each $p \in {\mathbb{F}}_q^n$, let $H(p)$ be the set of hyperplanes
of ${\mathcal{A}}$ that $p$ lies on. Then 
\begin{eqnarray*}
q^{n-r}\,\overline{\chi}_{{\mathcal{A}}}(q,t) &=& 
\sum_{{{\mathcal{B}} \subseteq {\mathcal{A}}}} q^{n-r({\mathcal{B}})} (t-1)^{|{\mathcal{B}}|} 
 =  \sum_{{{\mathcal{B}} \subseteq {\mathcal{A}}}} q^{\dim \cap {\mathcal{B}}} (t-1)^{|{\mathcal{B}}|}
=  \sum_{{{\mathcal{B}} \subseteq {\mathcal{A}}}} |\cap {\mathcal{B}} \,| \, (t-1)^{|{\mathcal{B}}|} 
 \\
& = & \sum_{{{\mathcal{B}} \subseteq {\mathcal{A}}}} \sum_{p \, \in \, \cap {\mathcal{B}}} \, (t-1)^{|{\mathcal{B}}|} 
  = \sum_{p \, \in \, {\mathbb{F}}_q^n} \sum_{{\mathcal{B}} \subseteq H(p)} (t-1)^{|{\mathcal{B}}|} 
 =  \sum_{p \, \in \, {\mathbb{F}}_q^n} (1+(t-1))^{h(p)},
\end{eqnarray*}
as desired. 
\end{proof}

Computing Tutte polynomials is extremely difficult (\#P-complete \cite{f.Welshcomplexity}) for general matroids, but it is still possible in some cases of interest. Theorem \ref{f.th:Tuttefinitefield} is one of the most effective methods for computing Tutte polynomials of (a few) particular arrangements ${\mathcal{A}}$ in ${{\mathbbm{k}}}^d$, as follows.

If the hyperplanes of ${\mathcal{A}}$ have integer coefficients (as most arrangements of interest do), we may use the same equations to define an arrangement ${\mathcal{A}}_q$ over ${\mathbb{F}}_q^d$. If $q$ is a power of a large enough prime, then ${\mathcal{A}}$ and ${\mathcal{A}}_q$ have isomorphic matroids, and hence have the same Tutte polynomial. Then Theorem \ref{f.th:Tuttefinitefield} reduces the computation of $T_{\mathcal{A}}(x,y)$ to an enumerative problem over ${\mathbb{F}}_q^d$, which can sometimes be solved. \cite{f.ArdilaTutte}

\subsection{\textsf{{Tutte polynomial evaluations}}}\label{f.sec:Tutteevaluations} 

Many important invariants of a matroid are generalized Tutte-Grothendieck invariants, and hence are evaluations of the Tutte polynomial. In fact, many results outside of matroid theory fit naturally into this framework. In this section we collect, without proofs, results of this sort in many different areas of mathematics and applications.

One can probably prove every statement in this section by proving that the quantities in question satisfy a deletion--contraction recursion; many of the results also have more interesting and enlightening explanations. The wonderful surveys \cite{f.BrylawskiOxley}, \cite{f.EllisMerino}, and \cite{f.Welshcomplexity} include most of the results mentioned here; we provide references for the other ones. 

\subsubsection{\textsf{{General evaluations}}}  Let $M$ be any matroid of rank $r$.

$\bullet$ The number of independent sets is $T(2, 1)$.

$\bullet$ The number of spanning sets is $T(1,2)$.

$\bullet$ The number of bases is $T(1,1)$.

$\bullet$ The number of elements of the ground set is $\log_2 T(2,2)$.

$\bullet$ The \textbf{M\"obius number} is $\mu(M) = \mu_{L_M}({\widehat{0}}, {\widehat{1}}) = (-1)^r T(1,0)$. 

$\bullet$ We have the formula $\chi_M(q) = (-1)^{r(M)}T(1-q,0)$ for the \textbf{characteristic polynomial}:\footnote{Note that if ${\mathcal{A}}$ is a hyperplane arrangement of rank $r$ in ${{\mathbbm{k}}}^n$ and $M$ is its matroid, then the characteristic polynomial $\chi_{\mathcal{A}}(q)$  defined in Section \ref{f.sec:charpoly} is given by $\chi_{\mathcal{A}}(q) = q^{d-r}\chi_M(q)$.}
\[
\chi_M(q) = \sum_{F \in L_M} \mu(\widehat{0}, F) q^{r - r(F)}.
\]

$\bullet$ The \textbf{beta invariant} of $M$ is defined to be $\beta(M)= [x^1y^0]T(x,y) = [x^0y^1]T(x,y)$.
It has some useful properties. We always have $\beta(M) \geq 0$. We have $\beta(M)=0$ if and only if $M$ is disconnected. We have $\beta(M) = 1$ if and only if $M$ is the graphical matroid of a \textbf{series-parallel} graph; that is, a graph obtained from a single edge by repeatedly applying series extensions (convert an edge $uv$ into two edges $uw$ and $wv$ for a new vertex $w$) and parallel extensions (convert one edge $uv$ into two edges joining $u$ and $v$). There are similar characterizations of the matroids with $\beta(M) \leq 4$. 

\bigskip

The \textbf{independence complex} of $M$ is the simplicial complex consisting of the independent sets of $M$. Recall that the $f$-vector $(f_0, \ldots, f_r)$ and $h$-vector $(h_0, \ldots, h_r)$ are defined so that $f_i$ is the number of sets of size $i$ and $\sum h_ix^{r-i} = \sum f_i(x-1)^{r-i}$. Then:

$\bullet$ 
The reduced Euler characteristic of the independence complex is $T(0,1)$.

$\bullet$ 
 The $f$-polynomial of the independence complex is $\sum_i f_ix^i = x^rT(1+ \frac1x,1)$.

$\bullet$ The $h$-polynomial of the independence complex is $\sum_i h_ix^i = x^rT(\frac1x,1)$.

\noindent The polynomial $\sum_i h_i x^{r-i} = T(x,1)$ is also known as the \textbf{shelling polynomial} of $M$.

There are many other interesting geometric/topological objects associated to a matroid; we mention two. Bj\"orner \cite{f.Bjornerhomology} described the \textbf{order complex of the (proper part of the) lattice of flats} $\Delta(L_M-\{{\widehat{0}},{\widehat{1}}\})$, showing it is a wedge of $T(1,0)$ spheres. 
The \textbf{Bergman fan} $\textrm{Trop}(M) = \{w \in {\mathbb{R}}^E \, : \, \textrm{for every circuit } C,\,  \max_{c \in C} w_c \textrm{ is achieved at least twice}\}$ is the tropical geometric analog of a linear space. %It is $\textrm{Trop}(M) = \{w \in {\mathbb{R}}^E \, : \, \textrm{for every circuit } C,\,  \min_{c \in C} w_c \textrm{ is achieved at least twice}\}$. 
The \textbf{Bergman complex} $\mathcal{B}(M)$ is its intersection with the hyperplane $\sum_i x_i=0$ and the unit sphere $\sum_i x_i^2=1$; it captures all the combinatorial structure of Trop$(M)$. Ardila and Klivans showed $\mathcal{B}(M)$ and $\Delta(L_M-\{{\widehat{0}},{\widehat{1}}\})$ are homeomorphic. \cite{f.ArdilaKlivans} Then:

$\bullet$ \cite{f.Bjornerhomology}
The reduced Euler characteristic of the order complex $\Delta(L_M-\{{\widehat{0}},{\widehat{1}}\})$  is $T(1,0)$. 
 
$\bullet$ \cite{f.ArdilaKlivans}
The reduced Euler characteristic of the Bergman complex $\mathcal{B}(M)$  is $T(1,0)$. 
 
$\bullet$ \cite{f.Zharkov}
The Poincar\'e polynomial of the \emph{tropical cohomology} of $\textrm{Trop}(M)$ is $q^rT(1+1/q,0)$.

\bigskip

One of the most intriguing conjectures in matroid theory is \textbf{Stanley's $h$-vector conjecture} \cite{f.Stanleyh-vector} which states that for any matroid $M$, there exists a set $X$ of monomials such that:

- if $m$ and $m'$ are monomials such that $m \in X$ and $m'|m$, then $m' \in X$,

- all the maximal monomials in $X$ have the same degree,

- there are exactly $h_i$ monomials of degree $i$ in $X$.

\noindent This conjecture has been proved, using rather different methods, for several families: duals of graphic matroids,  \cite{f.Merinocographic}, lattice path matroids \cite{f.Schweig} cotransversal matroids \cite{f.Ohcotransversal}, paving matroids \cite{f.Merinoetal}, and matroids up to rank $4$ or corank $2$ \cite{f.DeLoeraKemperKlee,  KleeSamper}. The general case remains open.

%$\bullet$ \cite{f.ArdilaTutte} Suppose each element of matroid $M$ is deleted with probability $p$, independently of the other elements. Then the expected characteristic polynomial of  $\chi(x)$ of the resulting matroid is $x^{n-r}\overline{chi}(x,p)$.
%
%

\subsubsection{\textsf{{Graphs}}} Let $G=(V,E)$ be a graph with $v$ vertices, $e$ edges, and $c$ connected components. Let $M(G)$ be the matroid of $G$ and $T(x,y)$ be its Tutte polynomial. 

\bigskip

\noindent \textbf{\textsf{Colorings.}}
Recall that the \textbf{chromatic polynomial} $\chi_G(q)$ of a graph, when evaluated at a positive integer $q$, counts the proper colorings of $G$ with $q$ colors.

$\bullet$ We have 
\[
\chi_G(q) = (-1)^{v-c}q^c \, T(1-q,0).
\]

$\bullet$ More generally, for every $q$-coloring $\chi$ of the vertices of $G$, let $h(\chi)$ be the number of improperly colored edges, that is, those whose endpoints have the same color. Then
\begin{equation} \label{f.eq:dichromatic}
\sum_{\chi: V \rightarrow [q]} t^{h(\chi)} = (t-1)^{v-c} q^c \, T\left(\frac{q+t-1}{t-1},t\right)
\end{equation}

\bigskip

\noindent \textbf{\textsf{Flows.}} Fix an orientation of the edges of $G$ and a finite Abelian group $H$ of $t$ elements, such as ${\mathbb{Z}}/t{\mathbb{Z}}$, for some $t \in {\mathbb{N}}$. An \textbf{$H$-flow} is an assignment $f:E \rightarrow H$ of a ``flow" (an element of $H$) to each edge such that at every vertex, the total inflow equals the outflow as elements of $H$. We say $f$ is \textbf{nowhere zero} if $f(e) \neq 0$ for all edges $e$.

$\bullet$ The number of nowhere zero $t$-flows is given by the \textbf{flow polynomial}
\[
\chi^*_G(t) = (-1)^{e-v+c}T(0, 1-t).
\]
In particular, this number is independent of the orientation of $G$. It also does not depend on the particular group $H$, but only on its size.

$\bullet$ More generally, for every $H$-flow $f$ on the edges of $G$, let $h(v)$ be the number of edges having flow equal to $0$. Then
\[
\sum_{f: E \rightarrow H} t^{h(f)} = (t-1)^{e-v+c}T\left(t, \frac{q+t-1}{t-1}\right).
\]

\bigskip

\noindent \textbf{\textsf{Acyclic orientations.}} An \textbf{acyclic orientation} of $G$ is an orientation of the edges that creates no directed cycles. 

$\bullet$ The number of acyclic orientations is $T(2,0)$.

$\bullet$ A \textbf{source} (resp. a \textbf{sink}) of an orientation is a vertex with no incoming (resp. outgoing) edges. For any fixed vertex $w$, the number of acyclic orientations whose unique source is $w$ is $(-1)^{v-c}\mu(M(G)) = T(1,0)$. In particular, it does not depend on the choice of $w$.

$\bullet$ For any edge $e=uv$, the number of acyclic orientations whose unique source is $u$ and whose unique sink is $v$ is the beta invariant $\beta(M(G))$. In particular, it is independent of the choice of $e$.

\bigskip

\noindent \textbf{\textsf{Totally cyclic orientations.}}
A \textbf{totally cyclic orientation} of $G$ is an orientation of the edges such that every edge is contained in a cycle. If $G$ is connected, this is the same as requiring that there are directed paths in both directions between any two vertices.

$\bullet$ The number of totally cyclic orientations is $T(0,2)$.

$\bullet$ Given an orientation $o$ of $G$, the outdegree of a vertex is the number of outgoing edges, and the outdegree sequence of $o$ is $(\textrm{outdeg}(v) \, : \, v \in V)$. The number of distinct outdegree sequences among the orientations of $G$ is $T(2,1)$ 

$\bullet$ For any edge $e$, the number of totally cyclic orientations of $G$ such that every directed cycle contains $e$ is equal to $2\beta(M(G))$. In particular, it is independent of $e$. 

\bigskip

\noindent \textbf{\textsf{Eulerian orientations.}} Recall that an orientation of $G$ is \textbf{Eulerian} if each vertex has the same number of incoming and outgoing edges.
Say $G$ is 4-regular if every vertex has degree $4$. 

$\bullet$ A 4-regular graph $G$ has $(-1)^{v+c}T_G(0,-2)$ Eulerian orientations or \textbf{ice configurations} . (They are in easy bijection with the nowhere-zero ${\mathbb{Z}}_3$-flows.)

\bigskip

\noindent \textbf{\textsf{Chip firing.}} 
Let $G=(V,E)$ be a graph and let $q \in V$ be a vertex called the \textbf{bank}. A \textbf{chip configuration} is a map $\theta: V \rightarrow {\mathbb{Z}}$ with $\theta(v) \geq 0$ for $v \neq q$ and $\theta(q) = -\sum_{v \neq q} \theta(v)$. We think of vertex $v$ as having $\theta(v)$ chips. We say $v \neq q$ is \textbf{ready} if $\theta(v) \geq \deg(v)$, and $q$ is \textbf{ready} if no other vertex is ready. In each step of the \textbf{chip firing game}, a ready vertex $v$ gives away $\deg(v)$ chips, giving $e$ chips to $w$ if there are $e$ edges connecting $v$ and $w$. 

We say a configuration $\theta$ is \textbf{stable} if the only vertex that is ready is $q$. We say $\theta$ is \textbf{recurrent} if it is possible to fire a (non-empty) sequence of vertices subsequently, and return to $\theta$. We say $\theta$ is \textbf{critical} if it is stable and recurrent. One can check that for any starting chip configuration, the chip firing game leads us to a unique critical configuration. The following result is the key to Merino's proof of Stanley's $h$-vector conjecture in the special case of cographic matroids:

$\bullet$ \cite{f.Merinocographic} The generating function for critical configurations is 
\[
\sum_{\theta \textrm{ critical } } y^{-\theta(q)} = y^{|E|-\deg(q)}T_G(1,y).
\]
Chip-firing games have recently gained prominence as that they are intimately connected to the study of divisors on tropical algebraic curves. \cite{f.BakerNorine, f.GathmannKerber}

%
%We define the \textbf{level} of a critical configuration $\theta$ to be $\textrm{level}(\theta) = \sum_{v \neq q} \theta(v) - |E| + \deg(q)$, and we let $c_i$ be the number of critical configurations of level $i$. Then
%$\sum_{i \geq 0} c_iy^i = y^{\deg(q)-|E|}T_G(1,y)$.

\bigskip

%$\bullet$ Suppose each element of matroid $M$ is deleted with probability $p$, independently of the other elements. The probability that the resulting matroid has the same rank as $M$ is $p^{n-r}(1-p)^rT(1,1/p)$.

\noindent \textbf{\textsf{Plane graphs and Eulerian partitions.}} Let $G$ be a connected plane graph. The \textbf{medial graph} $H$ has a vertex on each edge of $G$, and two vertices of $H$ are connected if the corresponding edges of $G$ are neighbors in some face of $G$. This construction is shown in Figure \ref{f.fig:medialgraph}.

\begin{figure}[ht]
 \begin{center}
  \includegraphics[scale=.8]{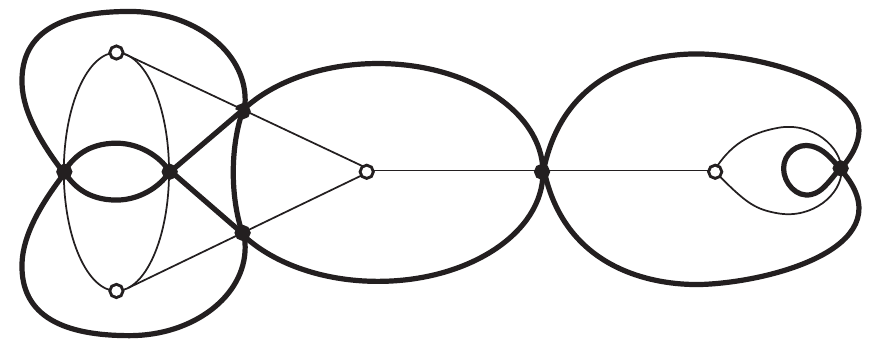}
  \caption{ \label{f.fig:medialgraph}
A connected graph $G$ and its medial graph $H$ in bold.}
 \end{center}
\end{figure}
An \textbf{Eulerian partition} $\pi$ of $H$ is a partition of the edges into closed paths. A closed path may revisit vertices, but it may not revisit edges; and different starting points and orientations are considered to be the same path.
Let $\gamma(\pi)$ be the number of paths in $\pi$. If the edges of a vertex $v$ of $H$ are $e_1,e_2,e_3,e_4$ in clockwise order, say $\pi$ has a \textbf{crossing} at $v$ if one of its paths uses edges $e_1$ and $e_3$ consecutively.

%\noindent o 

$\bullet$ \cite{f.LasVergnas33}
 We have $T_G(-1,-1) = (-1)^e(-2)^{\gamma(\pi_c)-1}$, where $\pi_c$ is the unique \textbf{fully crossing} Eulerian partition of $H$ with a crossing at every vertex.

%\noindent o I
$\bullet$ \cite{f.LasVergnas33}
If $\Pi_{nc}(H)$ is the set of Eulerian partitions of $H$ with no crossings, then 
\[
T_G(x,x) = \sum_{P \in \Pi_{nc}(H)} (x-1)^{\gamma(P)-1}.
\]

In the example above, note that the fully crossing Eulerian partition of $G$ consists of one Eulerian walk, in agreement with $T_G(-1,-1) = 1 = (-1)^6(-2)^{1-1}$. We have
$T_G(x,x) = 3x^4+2x^3$.
%o Let $\Pi_{nc}(H)$ be the set of Eulerian partitions of $H$ consisting of noncrossing paths. Then $T_G(x,x) = \sum_{P \in \Pi_{nc}(H)} (x-1)^{\gamma(P)-1}$.

%o Note that every vertex of $H$ has degree $4$; let $O(H)$ be the set of its \textbf{Eulerian orientations}. In an orientation $o \in O(H)$, call vertex $v$ a saddle if its incident edges are oriented ``in, out, in, out" in cyclic order, and let $s(o)$ be the number of saddles. Then $\sum_{o \in O(H)} 2^{s(o)} = 2T_G(3,3)$

\subsubsection{\textsf{{Hyperplane arrangements}}} \label{f.sec:Tuttearrangements} Let ${\mathcal{A}}$ be an arrangement of $n$ hyperplanes in ${{\mathbbm{k}}}^d$ of rank $r$. Let $V({\mathcal{A}})$ be the complement of ${\mathcal{A}}$ in ${{\mathbbm{k}}}^d$. We restate some theorems from Section \ref{f.sec:charpoly} and present new ones.

$\bullet$ If ${{\mathbbm{k}}} = {\mathbb{R}}$, the number of regions of $V({\mathcal{A}})$ is $T(2,0)$.

$\bullet$ If ${{\mathbbm{k}}}={\mathbb{C}}$, the Poincar\'e polynomial of $V({\mathcal{A}})$ is $q^rT(1+\frac1q,0)$.

$\bullet$ If ${{\mathbbm{k}}}={\mathbb{F}}_q$, the number of elements of $V({\mathcal{A}})$ is $(-1)^rq^{n-r}T(1-q,0)$

$\bullet$ Suppose ${{\mathbbm{k}}} = {\mathbb{R}}$, and consider an affine hyperplane $H$ which is in general position with respect to ${\mathcal{A}}$. Then the number of regions of ${\mathcal{A}}$ which  have a bounded (and non-empty) intersection with $H$ is $(-1)^r|\mu(M)| = T(1,0)$. In particular, it is independent of $H$.

$\bullet$ Suppose ${{\mathbbm{k}}} = {\mathbb{R}}$, and add to ${\mathcal{A}}$ an affine hyperplane $H'$ which is a translation of $H \in {\mathcal{A}}$. The number of bounded regions of ${\mathcal{A}} \cup H'$ is $\beta(M)$. In particular, it is independent of $H$.

\subsubsection{\textsf{{Algebras from vector and hyperplane arrangements}}} There are several natural algebraic spaces related to the Tutte polynomial arising in commutative algebra, hyperplane arrangements, box splines, and index theory; we discuss a few. For each hyperplane $H $ in a hyperplane arrangement ${\mathcal{A}}$ in ${{\mathbbm{k}}}^d$ let $l_H$ be a linear function such that $H$ is given by the equation $l_H(x)=0$. 

$\bullet$ \cite{f.Wagner} 
Let $C_{{\mathcal{A}},0} = {\mathrm{span }} \{\prod_{H \in {\mathcal{B}}} l_H \, : \, {\mathcal{B}} \subseteq {\mathcal{A}}\}$. This is a subspace of a polynomial ring in $d$ variables, graded by degree. Its dimension is $T(2,1)$ and its Hilbert series is 
\[
\textrm{Hilb}(C_{{\mathcal{A}}, 0}; q) = \sum_{j \geq 0} \dim (C_{{\mathcal{A}},0})_j \, q^j = 
q^{n-r}T\left(1+q, \frac1q\right).
\]
 
$\bullet$ \cite{f.ArdilaPostnikov, f.DahmenMicchelli, f.HoltzRon, f.PostnikovShapiro, f.PostnikovShapiroShapiro} More generally, let $C_{{\mathcal{A}},k}$ be the vector space of polynomial functions such that the restriction of $f$ to any line $l$ has degree at most $\rho_{{\mathcal{A}}}(h)+k$, where $\rho_{{\mathcal{A}}}(h)$ is the number of hyperplanes of ${\mathcal{A}}$ not containing $h$. It is not obvious, but true, that this definition of $C_{{\mathcal{A}},0}$ matches the one above. Then %In fact, there are similar (but more complicated) descriptions for $k \geq -2$. We have 
\[
\textrm{Hilb}(C_{{\mathcal{A}}, -1}; q) = q^{n-r}T\left(1, \frac1q\right), \qquad \textrm{Hilb}(C_{{\mathcal{A}}, -2}; q) = q^{n-r}T\left(0, \frac1q\right)
\]
and similar formulas hold for any $k \geq -2$.

$\bullet$ \cite{f.BrionVergne, f.ProudfootSpeyer, f.Teraoalgebras} Let $R({\mathcal{A}})$ be the vector space of rational functions whose poles are in ${\mathcal{A}}$. It may be described as the ${{\mathbbm{k}}}$-algebra generated by the rational functions $\{1/l_H \, : \, H \in {\mathcal{A}}\}$; we grade it so that $\deg(1/l_H)=1$. Then
\[
\textrm{Hilb}(R({\mathcal{A}}); q) = \frac{q^d}{(1-q)^d} T\left( \frac1q, 0 \right).
\]

\subsubsection{\textsf{{Error-correcting codes}}} Suppose we wish to transmit a message over a noisy channel. We might then encode our message in a redundant way, so that we can correct small errors introduced during transmission. 
An \textbf{error-correcting code} is a set $C \subset A^n$ of \textbf{codewords} of length $n$ over an alphabet $A$. The sender encodes each word into a redundant codeword, which they transmit. If the channel is not too noisy and the codewords in $C$ are sufficiently different from each other, the recipient will succeed in recovering the original message. Of course it is useful to have many codewords available under these constraints.

A common kind of error-correcting code is a \textbf{linear code} $C$, which is a $k$-dimensional subspace of a vector space ${\mathbb{F}}_q^n$. The codewords have length $n$ and alphabet ${\mathbb{F}}_q$. The \textbf{support} of a word is the set of non-zero entries. The \textbf{distance} $d({\mathbf{u}},{\mathbf{v}})$ between two words ${\mathbf{u}}$ and ${\mathbf{v}}$ is the number of coordinates where ${\mathbf{u}}$ and ${\mathbf{v}}$ differ. Since ${\mathbf{u}}-{\mathbf{v}} \in {\mathcal{C}}$, we have $d({\mathbf{u}},{\mathbf{v}}) = |{\mathrm{supp }}({\mathbf{u}}-{\mathbf{v}})|$. The minimum \textbf{distance} $d$ in the code is $d=\min_{c \in C} |{\mathrm{supp }}(c)|$. The code $C$ is said to be of type $[n,k,d]$.

%The \textbf{dual code} of $C$ is $C^\perp \ \{{\mathbf{x}} \in {\mathbb{F}}_q^n \, : \, ({\mathbf{x}},\c)=0 \textrm{ for all } \c \in C\}$. 

\medskip

$\bullet$ \cite{f.Greene}
The \textbf{weight enumerator} of a linear code 
%$C$ of rank $k$ and length $n$  
over ${\mathbb{F}}_q$  is $W_C(q,z) = \sum_{c \in C} z^{|{\mathrm{supp }}(c)|}$. Greene  discovered that it can be expressed in terms of the Tutte polynomial; in fact, Theorem \ref{f.th:Tuttefinitefield} is equivalent to the equation
\begin{equation}\label{f.eq:Greene}
W_C(q,z) = (1-z)^k z^{n-k} \,  T_C\left(\frac{1+(q-1)z}{1-z}, \frac1z\right). 
\end{equation}

The \textbf{dual code} of $C$ is $C^\perp = \{{\mathbf{x}} \in {\mathbb{F}}_q^n \, : \, ({\mathbf{x}},{\mathbf{c}})=0 \textrm{ for all } {\mathbf{c}} \in C\}$. An important example is the \textbf{Hamming code} $H$. Let $\mathcal{A}(2,n)$ be the rowspace of the $n \times (2^n-1)$ matrix whose columns are all the nonzero vectors in ${\mathbb{F}}_2^n$, and let $H = \mathcal{A}(2,n)^\perp$. This is a largest possible code (consisting of $2^{2^n-n-1}$ words) of length $2^n-1$ and distance $3$. 

$\bullet$  \cite{f.MacWilliams} MacWilliams's identity elegantly relates the weight enumerators of $C$ and $C^\perp$:
\[
W_{C^\perp}(q,z) = \frac{(1+(q-1)z)^n}{q^k} W_C\left(q, \frac{1-z}{1+(q-1)z}\right)
\]
In light of (\ref{f.eq:Greene}), this is equivalent to Tutte polynomial duality: $T_{C^\perp}(x,y) = T_C(y,x)$. As an application, it is easy to show 
%not too difficult to use (\ref{f.eq:Greene}) to show 
that the weight enumerator of $\mathcal{A}(2,n)$ is $1+(q^n-1)z^{q^{n-1}}$. \cite[Example 3.4]{f.Greene} MacWilliams's identity then gives us the weight enumerator of the Hamming code.

$\bullet$ Another nice result \cite{f.Jaeger} is that if $C$ is a linear code in ${\mathbb{F}}_2^n$ then $T_C(-1,-1) = (-1)^n|C \cap C^\perp|$.

\subsubsection{\textsf{{Probability and statistical mechanics}}} %The Tutte polynomial arises in several standard processes in probability and statistical physics. %We consider three related models. \cite{f.BrylawskiOxley, f.SokalTutte, f.Welshcomplexity}

$\bullet$ \cite{f.BrylawskiOxley}
Suppose each edge of a connected graph $G$ is white with probability $p$ and black with probability $1-p$ (for fixed $0 \leq p \leq 1$), independently of the other edges. The probability that the white graph is connected is given by the \textbf{reliability polynomial} 
\[
R_G(p) = \sum_{\stackrel{A \subseteq E}{A \textrm{ spanning}}} p^{|A|} (1-p)^{|E-A|} = 
 (1-p)^{e-v+1}p^{v-1} \, T_G\left(1,\frac1{1-p}\right).
\]

%$\bullet$ The random-cluster model on $G$ has parameters $0 \leq p \leq 1$ and $q>0$. The partition function is $Z_G(q,p) := \sum_{F \subset E} q^{c(F)} p^|F| =  q^c \overline{\chi}(q,1+p)$  
%

$\bullet$ \cite{f.SokalTutte} The \textbf{random cluster model} depends on parameters $0 \leq p \leq 1$ and $q > 0$. Now we choose a white set of edges at random, and the probability of choosing $A \subset E$ is $p^{|A|} (1-p)^{|E-A|} q^{c(A)}/Z$, where $c(A) = |V|-r(A)$ is the number of components of $A$, and $Z$ is a scaling constant. To know the probability of a particular state, it is fundamental to know the scaling constant $Z$, which is the \textbf{partition function}
\[
Z(p,q) = \sum_{A \subseteq E}  p^{|A|} (1-p)^{|E-A|} q^{c(A)} = p^{v-c}(1-p)^{e-v+c}q^c \, T_G\left(1+\frac{q(1-p)}p, \frac1{1-p}\right).
\]

$\bullet$ \cite{f.Welshcomplexity}
In the \textbf{$q$-state Potts model}, a graph $G=(V,E)$ models a set $V$ of ``atoms" and a set $E$ of ``bonds" between them. (When $q=2$, this is the \textbf{Ising model}.)  Each atom can exist in one of $q$ states or ``spins". Each edge $e=uv$ has an associated \emph{interaction energy} $J_e$ between $u$ and $v$. The energy (or Hamiltonian)  $H$ of a configuration is the sum of $-J_e$ over all edges $e$ whose vertices have the same spin.\footnote{The case $J_e \geq 0$ is called ferromagnetic, as it favors adjacent spins being equal. The case $J_e \leq 0$ is antiferromagnetic.
Here we are assuming that there is no external magnetic field. If there were such a field, it would contribute an additional term to the Hamiltonian, and the direct connection with the Tutte polynomial is no longer valid.} 
A configuration of energy $H$ has Boltzmann weight $e^{-\beta H}$ where  
 $\beta = 1/kT > 0$, $T$ is the temperature, and $k$ is Boltzmann's constant. 

The \textbf{partition function} is the sum of the Boltzmann weights of all configurations:
\[
Z_G(q,\mathbf{w}) = \sum_{\sigma: V \rightarrow [q]} \prod_{\stackrel{e=ij \in E}{\sigma(i) = \sigma(j)}} e^{\beta J_e}=  \sum_{\sigma: V \rightarrow [q]} \prod_{\stackrel{e=ij \in E}{\sigma(i) = \sigma(j)}} (1+w_e)
\]
where $w_e = e^{\beta J_e} - 1$.
If all $J_e$s are equal to $J$ and $w=e^{\beta J}-1$, then (\ref{f.eq:dichromatic}) gives
\[
Z_G(q,w) = w^{v-c}q^c \,  T\left(\frac{q}{w}+1, w+1\right).
\]
In general, $Z_G(q, \mathbf{w})$ is essentially the
%the partition function can be expressed in terms of the %
\textbf{multivariate Tutte polynomial} $\widetilde{Z}_G(q, \mathbf{w})$  
of Section \ref{f.sec:multivariateTutte}:

%a simple computation shows that
\[
q^{-v} Z_G(q, \mathbf{w}) = \sum_{A \subseteq E} q^{-r(A)} \prod_{e \in A} w_e =  \widetilde{Z}_G(q, \mathbf{w}).
\]
%where $\widetilde{Z}_G(q, \mathbf{w})$ is the \textbf{multivariate Tutte polynomial} of Section \ref{f.sec:multivariateTutte}.

%\[
%Z_G(q,J) = \left(e^{\beta J} -1\right)^{v-c}q^c T\left(\frac{q+e^{\beta J}-1}{e^{\beta J}-1}, e^{\beta J}\right)
%\]

\subsubsection{\textsf{{Other applications}}} 

$\bullet$ 
 \cite{f.Jones, f.Welshcomplexity}
A \textbf{knot} is an embedding of a circle in ${\mathbb{R}}^3$. It is a difficult, important question to determine whether a knot can be deformed smoothly (without cutting or crossing segments) to obtain another knot; or even to determine whether a given knot is actually knotted or not. Let \textbf{O} be the unknotted circle in ${\mathbb{R}}^3$.

One common approach is to construct a function $f(K)$ of a knot $K$ which does not change under smooth deformation. If $f(K) \neq f(\textbf{O})$, then $K$ is actually knotted. One such function is the \textbf{Jones polynomial }$V(K)$. If $K$ is an \textbf{alternating knot}, there is a graph $G$ associated to it such that the Jones polynomial of $K$ is an evaluation of the Tutte polynomial of $G$. 
$\bullet$ \cite{f.ReinerTutte} There is a more symmetric finite field interpretation for the Tutte polynomial. Let $M$ be an integer matrix, and let $p$ and $q$ be prime powers such that the matroid of $M$ does not change when $M$ is considered as a matrix over ${\mathbb{F}}_p$ or over ${\mathbb{F}}_q$. Then
\[
T_M(1-p, 1-q) = (-1)^{r(M)} \sum_{\stackrel{{\mathbf{x}} \in \textrm{row}(M), {\mathbf{y}} \in \textrm{ker}(M)}{\textrm{supp}({\mathbf{x}}) \sqcup \textrm{supp}({\mathbf{y}}) = E}} (-1)^{|\textrm{supp}(y)|},
\]
where the rowspace $\textrm{row}(M)$ is considered as a subspace of ${\mathbb{F}}_p^n$ and the kernel $\textrm{ker}(M)$ is considered as a subspace of ${\mathbb{F}}_q^n$, and $\sqcup$ denotes a disjoint union.

$\bullet$ \cite{f.ArdilaCatalan, f.BonindeMierNoy} The Tutte polynomial of the Catalan matroid is 
\[
T_{\mathbf{C}_n}(x,y) = \sum_{P \textrm{ Dyck}} x^{a(P)}y^{b(P)},
\]
where $a(P)$ is the number of upsteps before the first downstep of $P$, and $b(P)$ is the number of times that $P$ returns to the $x$-axis. Since ${\mathbf{C}_n}$ is self-dual, this polynomial is symmetric in $x$ and $y$. 

In fact, the coefficient of $x^iy^j$ in $T_{\mathbf{C}_n}(x,y)$ depends only on $i+j$. It would be interesting to find other families of matroids with this unusual property.

$\bullet$ \cite{f.BonindeMierNoy} More generally, the Tutte polynomial of a lattice path matroid is the sum over the bases $P$ of $x^{a(P)}y^{b(P)}$, where $a(P)$ and $b(P)$ are the numbers of returns to the upper and lower boundary paths, respectively. 

$\bullet$ \cite{f.KornPak} A T-tetromino is a T-shape made of four unit squares. Figure \ref{f.fig:tetris} shows a T-tetromino tiling of an $8 \times 8$ square.
The number of \textbf{T-tetromino tilings} of a $4m \times 4n$ rectangle equals $2T_{L_{m,n}}(3,3)$, where $L_{m,n}$ is the $m \times n$ grid graph. This extends to T-tetromino tilings of many other shapes and, more generally, to coverings of many graphs with copies of the \emph{claw} graph $K_{1,3}$.

\begin{figure}[ht]
 \begin{center}
  \includegraphics[scale=.34]{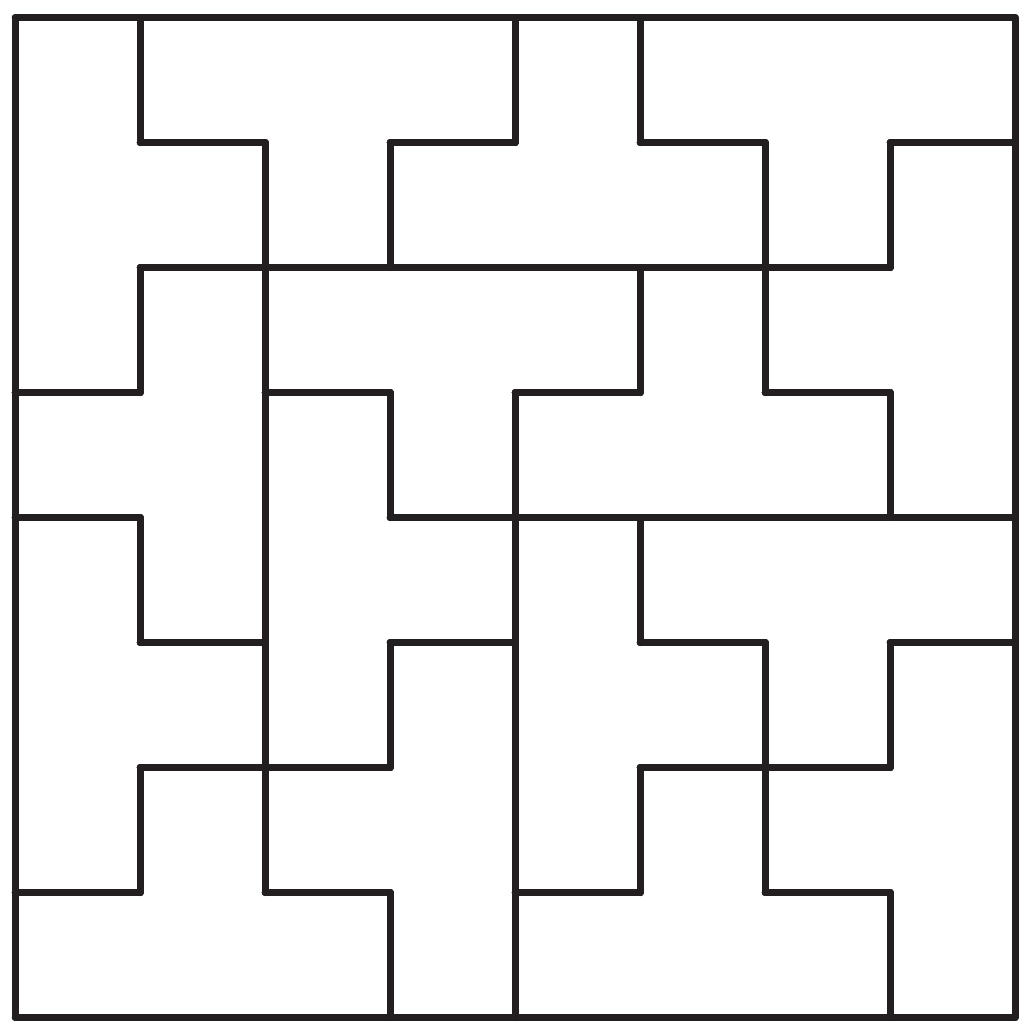}
  \caption{ \label{f.fig:tetris}
A tiling of an $8 \times 8$ rectangle by T-tetrominoes.}
 \end{center}
\end{figure}

%$\bullet$ Let $K_{n+1}$ be the complete graph on $[n+1]$. The Euler number $E_n$, which counts the \textbf{alternating permutations} of $[n]$ (see Section \ref{f.sec:ogfexamples}) is given by $E_n = T_{K_{n+1}}(1,-1)$.

\subsection{\textsf{{Computing the Tutte polynomial}}} 
\label{f.sec:computingTutte}

As we already mentioned, computing the Tutte polynomial of an arbitrary graph or matroid is not a tractable problem. In the language of complexity theory, this is a \#P-complete problem. \cite{f.Welshcomplexity}  However, the results of Section \ref{f.sec:Tutte} allow us to compute the Tutte polynomial of \textbf{some} matroids of interest. We now survey some of the most interesting examples; see \cite{f.MerinoRamirezetal} for others. Some of these formulas are best expressed in terms of the \textbf{coboundary polynomial} 
\[
\overline{\chi}_{{\mathcal{A}}}(X,Y) = (y-1)^{r({\mathcal{A}})} T_{{\mathcal{A}}}(x,y), \qquad \textrm{where } x = \frac{X+Y-1}{Y-1}, \quad y=Y. 
\]
Almost all of them are most easily proved using the Finite Field Method (Theorem \ref{f.th:Tuttefinitefield}) or its graph version (\ref{f.eq:dichromatic}).

$\bullet$ For the uniform matroid $U_{k,n}$ we have $T_{U_{k,n}}(x,y) = \sum_{i=1}^r{n - i - 1 \choose n-r-1} x^i +  \sum_{j=1}^{n-r} {n-j-1 \choose r-1} y^j$.

$\bullet$ If $M^{(k)}$ is the matroid obtained from $M$ by replacing each  element by $k$ copies of itself, then
%parallel elements, and replacing each loop by $k$ loops. Then
\[
T_{M^{(k)}}(x,y) = (y^{k-1} + y^{k-2} + \cdots + y+1)^r T_M\left(\frac{y^{k-1} + y^{k-2} + \cdots + y+x}{y^{k-1} + y^{k-2} + \cdots + y+1}, y^k\right)
\]
This formula is straightforward in terms of coboundary polynomials: $\overline{\chi}_{M^{(k)}}(X,Y) = \overline{\chi}_{M}(X,Y^K)$. For an extensive generalization, see Section \ref{f.sec:multivariateTutte}.

$\bullet$ \cite{f.ArdilaTutte, f.Mphako} Root systems are arguably the most important vector configurations; these highly symmetric arrangements play a fundamental role in many branches of mathematics. For the general definition and properties, see for example \cite{f.Humphreys}; we focus on the four infinite families of \textbf{classical root systems}:
\begin{eqnarray*}
A_{n-1} &=& \{e_i-e_j,\, : \, 1\leq i < j\leq n\} \\
B_n &=& \{e_i -  e_j, e_i + e_j \, : \,  1\leq i <  j\leq n\} \cup \{e_i \, : \, 1 \leq i \leq n\} \\
C_n &=& \{e_i -  e_j, e_i + e_j \, : \,  1\leq i <  j\leq n\} \cup \{2e_i \, : \, 1 \leq i \leq n\} \\
D_n &=& \{e_i -  e_j, e_i + e_j \, : \,  1\leq i <  j\leq n\} 
\end{eqnarray*}
Figure \ref{f.fig:rootsystems} illustrates the two-dimensional examples.

\begin{figure}[ht]
 \begin{center}
  \includegraphics[scale=.6]{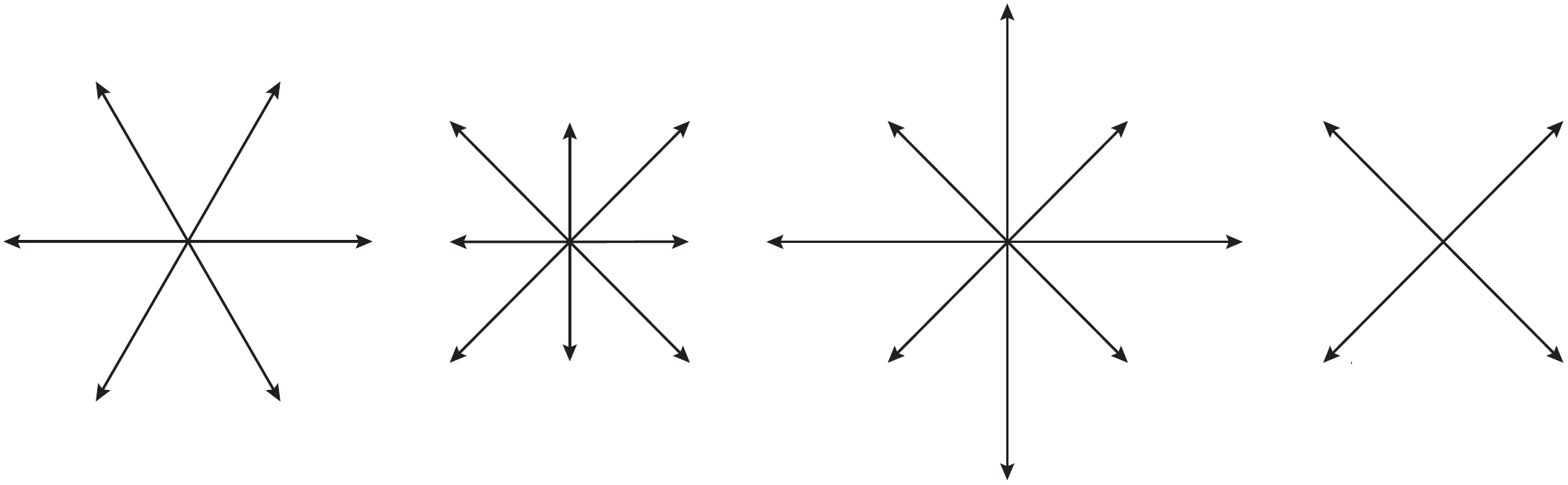}
  \caption{ \label{f.fig:rootsystems}
The root systems $A_2, B_2, C_2,$ and $D_2$, respectively.} \end{center}
\end{figure}

Let the \textbf{deformed exponential function} be 
%$F(\alpha, \beta) = \sum_{n \geq 0} \frac{\alpha^n \, \beta^{n \choose 2}}{n!}$. 
$F(\alpha, \beta) = \sum_{n \geq 0} {\alpha^n \, \beta^{n \choose 2}}/{n!}$. 
Then
 the \textbf{Tutte generating functions} of the infinite families $A$ and $\Phi=B,C,D$:
\[
{T}_A(X,Y,Z) = 1+X \sum_{n \geq 1} \overline{\chi}_{A_{n-1}}(X,Y) \frac{Z^n}{n!}, \quad
{T}_\Phi(X,Y,Z) = \sum_{n \geq 0} \overline{\chi}_{\Phi_n}(X,Y) \frac{Z^n}{n!}
\]
are given by:

\begin{eqnarray*}
T_A &=& F(Z,Y)^X,\\
T_B &=& F(2Z,Y)^{(X-1)/2}F(YZ,Y^2),\\
T_C &=& F(2Z,Y)^{(X-1)/2}F(YZ,Y^2),\\
T_D &=& F(2Z,Y)^{(X-1)/2}F(Z,Y^2).
\end{eqnarray*}
Aside from the four infinite families, there is a small number of exceptional root systems, which are also very interesting objects. Their Tutte polynomials are computed in \cite{f.DeConciniProcesi.Tutte}.

$\bullet$ 
Since the matroid of the complete graph $K_n$ is isomorphic to the matroid of the vector configuration $A_{n-1}$, the first formula above is a formula for the Tutte polynomials of the complete graphs, proved originally by Tutte \cite{f.Tuttedichromatic}. 

Similarly, the coboundary polynomials of the complete graphs $K_{m,n}$ are given by
\[
1 + X \sum_{\stackrel{m,n \geq 0}{(m,n) \neq (0,0)}} \overline{\chi}_{K_{m,n}}(X,Y) \frac{Z_1^m}{m!}\frac{Z_2^n}{n!} = \left(\sum_{m, n \geq 0} Y^{mn} \frac{Z_1^m}{m!} \frac{Z_2^n}{n!}\right)^X.
\]

$\bullet$ \cite{f.BaranyReiner, f.Mphako}.
The Tutte polynomial of the arrangement ${\mathcal{A}}(p,n)$ of all linear hyperplanes in ${\mathbb{F}}_p^n$ is best expressed in terms of a ``$p$-exponential generating function":
\[
\sum_{n \geq 0} \overline{\chi}_{{\mathcal{A}}(p,n)}(X,Y) \frac{u^n}{(p;p)_n}
= 
\frac{(u;p)_\infty}{(Xu;p)_\infty}
\sum_{n \geq 0} Y^{1+p+\cdots + p^{n-1}} \frac{u^n}{(p;p)_n}
\]
where $(a;p)_\infty = (1-a)(1-pa)(1-p^2a)\cdots$ and  $(a;p)_n = (1-a)(1-pa) \cdots (1-p^{n-1}a)$.

$\bullet$ 
Equation (\ref{f.eq:dichromatic}) expresses the Tutte polynomial of a graph $G$ as the number of $q$-colorings of $G$, weighted by the number of improperly colored edges. When $G$ has a path-like or cycle-like structure, one may use the transfer-matrix method 
of Section \ref{f.sec:transfermatrix}
to carry out that enumeration. This was done for the grid graphs $L_{m,n}$ for small fixed $m$ \cite{f.CalkinMerinoetal} and for the wheel graphs $W_n$ \cite{f.ChangShrock} shown in Figure \ref{f.fig:gridwheel}. For example, the Tutte polynomial of the wheel graph is
\[
T_{W_n}(x,y) = \frac1{2^n}\left(b+\sqrt{b^2-4a}\right)^n + \frac1{2^n}\left(b-\sqrt{b^2-4a}\right)^n +xy-x-y-1
\]
where $b=1+x+y$ and $a=xy$.

\begin{figure}[ht]
 \begin{center}
  \includegraphics[scale=1]{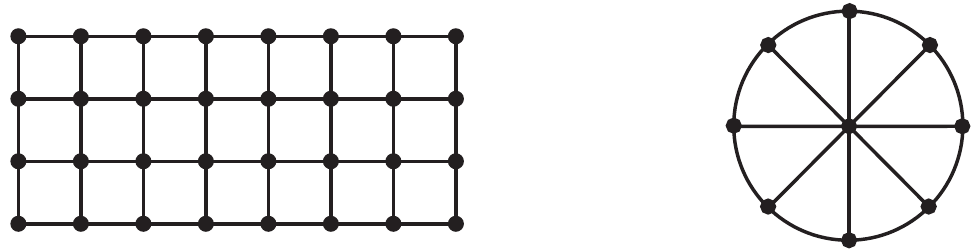}
  \caption{ \label{f.fig:gridwheel}
   The grid graph $L_{4,8}$ and the wheel graph $W_8$.} \end{center}
\end{figure}

\subsection{\textsf{{Generalizations of the Tutte polynomial}}} \label{f.sec:Tuttegeneralizations}

\noindent {\textsf{\textbf{The multivariate Tutte polynomial.}}} \label{f.sec:multivariateTutte}
The \textbf{multivariate Tutte polynomial} of a matroid $M$ is
\[
\widetilde{Z}_M(q; \mathbf{w}) = \sum_{A \subseteq E} q^{-r(A)} \prod_{e \in A} w_e 
\]
where $q$ and $(w_e)_{e \in E}$ are indeterminates. When all $w_e=w$, we get $\widetilde{Z}(q,w) = w^{v-c}q^cT(\frac{q}{w}+1, w+1)$.

Notice that $\widetilde{Z}_M(q; \mathbf{w})$ determines $E$ and $r$, and hence $M$. In fact, among the many definitions of a matroid, we could define $M$ to be $\widetilde{Z}_M(q; \mathbf{w})$. This is a useful encoding of the matroid. 

$\bullet$ We saw in Section \ref{f.sec:Tutteevaluations} that for a graph $G$ and a positive integer $q$, the multivariate Tutte polynomial is equal to the partition function of the $q$-state Potts model on $G$.

$\bullet$ For a vector ${\mathbf{a}} \in {\mathbb{N}}^n$, let $M({\mathbf{a}})$ be the matroid $M$ where each element $e$ is replaced by $a_e$ copies of $e$. It is natural to ask for the Tutte polynomials of the various matroids $M({\mathbf{a}})$. For each ${\mathbf{a}}$,
\[
T_{M({\mathbf{a}})}(x,y) = (x-1)^{r({\mathrm{supp }}({\mathbf{a}}))} \widetilde{Z}_M \left((x-1)(y-1); y^{a_1-1}, \ldots y^{a_n-1}\right).
\]
The generating function for 
%We can also encode 
the Tutte polynomials of \textbf{all} the matroids $M({\mathbf{a}})$ turns out to be equivalent to the multivariate Tutte polynomial, disguised under a change of variable:
\[
\sum_{{\mathbf{a}} in {\mathbb{N}}^n} \frac{T_{M({\mathbf{a}})}(x,y)}{(x-1)^{r({\mathrm{supp }}({\mathbf{a}}))}} w_1^{a_1}\cdots w_n^{a_n}
=
\frac{1}{\prod_{i=1}^n(1-w_i)} \widetilde{Z}_M\left((x-1)(y-1); \frac{(y-1)w_1}{1-yw_1}, \ldots,  \frac{(y-1)w_n}{1-yw_n}\right)
\]

$\bullet$ For an algebraic interpretation of the multivariate Tutte polynomial, see \cite{f.ArdilaPostnikov}.

\bigskip

\noindent {\textsf{\textbf{The arithmetic Tutte polynomial}}} \label{f.sec:arithmeticTutte}
When we have a collection $A \subseteq {\mathbb{Z}}^n$ of integer vectors, there is a variant of the Tutte polynomial that is quite useful. The \textbf{arithmetic Tutte polynomial} is
\[
M_{A}(x,y)= \sum_{B \subseteq {A}} m(B)(x-1)^{r(A)-r(B)}(y-1)^{|B|-r(B)}
\]
where, for each $B \subseteq A$, the \emph{multiplicity} $m(B)$ is the index of ${\mathbb{Z}} B$ as a sublattice of $({\mathrm{span }} \, B) \cap {\mathbb{Z}}^n$. If we use the vectors in $B$ as the  columns of a matrix, then $m(B)$ equals the greatest common divisor of the minors of full rank. This polynomial is related to the zonotope of $A$ \cite{f.Stanleyzonotope, f.D'AdderioMoci.Ehrhart} as follows:

$\bullet$
The volume of the zonotope $Z(A)$ is $M_A(1,1)$.

$\bullet$
The zonotope $Z(A)$ contains $M_A(2,1)$ lattice points, $M_A(0,1)$ of which are in its interior.

$\bullet$
The Ehrhart polynomial of the zonotope $Z(A)$ is $q^rM(1+\frac1q,1)$.

\bigskip

Let $T=\mathrm{Hom}({\mathbb{Z}}^n,G)$ be the group of homomorphisms from ${\mathbb{Z}}^n$ to a multiplicative group $G$, such as the unit circle $\mathbb{S}^1$ or ${\mathbb{F}}^*={\mathbb{F}}\backslash \{0\}$ for a field ${\mathbb{F}}$. 
%We might also consider the unitary characters $T=\Hom(\Lambda,\SS^1)$ where $\SS^1$ is the unit circle in ${\mathbb{C}}$. It is easy to check that $T$ is isomorphic to ${\mathbb{F}}^*$ and to $\SS^1$, respectively. 
Each element $a \in A$ determines a (hyper)torus $T_a = \{t \in T \, : \, t(a) = 1\}$ in $T$. For instance $a=(2,-3,5)$ gives the torus $x^2y^{-3}z^5=1$. Let 
\[
{\mathcal{T}}(A) = \{T_a \, : \, a \in A\}, \qquad R(A) = T \, \setminus \bigcup_{a \in {\mathcal{T}}(A)} T_a
\]
be the \emph{toric arrangement} of $A$ and its complement, respectively. The following results are toric analogs to the theorems about hyperplane arrangements in Section \ref{f.sec:Tuttearrangements}:

$\bullet$ \cite{f.EhrenborgReaddySlone, f.Moci.toric} If $G = \mathbb{S}^1$, the number of regions of $R({{A}})$ in the torus $(\mathbb{S}^1)^r$ is $M_A(1,0)$.

$\bullet$ \cite{f.DeConciniProcesi.toric, f.Moci.toric} If $G={\mathbb{C}}^*$, the Poincar\'e polynomial of $R({{A}})$ is $q^rM_A(2+\frac1q,0)$.

$\bullet$ \cite{f.BrandenMoci, f.ArdilaCastilloHenley} If $G={\mathbb{F}}_{q+1}^*$ where $q+1$ is a prime power, then the number of elements of $R({{A}})$ is $(-1)^rq^{n-r}M_A(1-q,0)$. Furthermore,
\[
\sum_{p \in {\mathbb{F}}_{q+1}^*} t^{h(p)} = (t-1)^r q^{n-r} M_A \left( \frac{q+t-1}{t-1}, t \right)
\]
where $h(p)$ is the number of tori of ${\mathcal{T}}(A)$ that $p$ lies on.

\bigskip

As with ordinary Tutte polynomials, this last result may be used as a \textbf{finite field method} to compute arithmetic Tutte polynomials for some vector configurations $A$. However, at the moment there are very few results along these lines. 
One exception is the case of root systems; the study of their geometric properties motivates much of the theory of arithmetic Tutte polynomials. % was motivated by geometric questions about root systems
Explicit formulas for the arithmetic Tutte polynomials of the classical root systems $A_n, B_n, C_n,$ and $D_n$ are given in \cite{f.ArdilaCastilloHenley}, based on the computation of Example 15 in Section \ref{f.sec:egfs}.

\bigskip

\noindent 

\subsection{\textsf{{Matroid subdivisions, valuations, and the Derksen-Fink invariant}}} \label{f.sec:matroidsubdivisions}

A very interesting recent developments in matroid theory has been the study of matroid subdivisions and valuative matroid invariants. We offer a brief account of some key results and some pointers to the relevant bibliography.

% The Tutte polynomial, and many other matroid invariants, have the remarkable valuative property. 

A \textbf{matroid subdivision} is a polyhedral subdivision $\mathcal{P}$ of a matroid polytope $P_M$ where every polytope $P \in \mathcal{P}$ is itself a matroid polytope. Equivalently, it is a subdivision of $P_M$ whose only edges are the edges of $P_M$.
In the most important case, $M$ is the uniform matroid $U_{d,n}$ and  $P_M$ is the hypersimplex $\Delta(d,n)$. 

\begin{figure}[ht]
 \begin{center}
  \includegraphics[scale=.8]{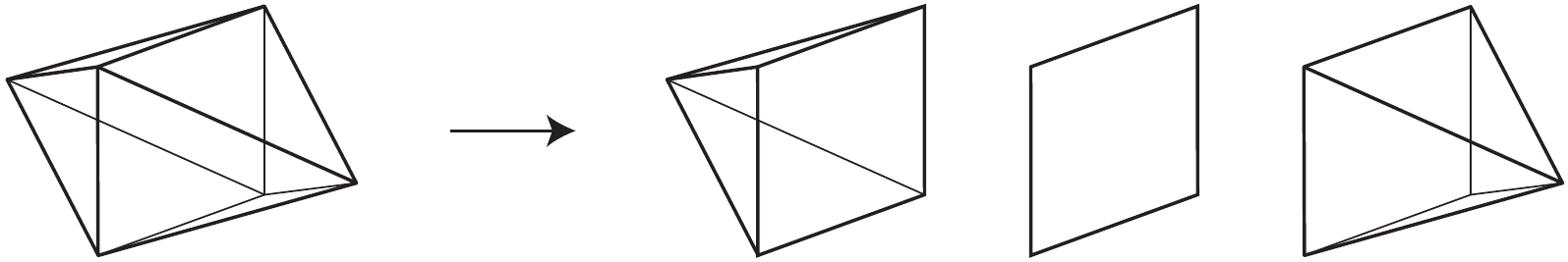}
  \caption{ \label{f.fig:matroidsubdiv}
A matroid subdivision of $\Delta(2,4)$.}
 \end{center}
\end{figure}

Matroid subdivisions arose in algebraic geometry \cite{f.HackingKeelTevelev, f.Kapranov, f.Lafforgue}, in the theory of valuated matroids \cite{f.DressWenzel, f.Murota}, and in tropical geometry \cite{f.Speyer1}. For instance, Lafforgue showed that if a matroid polytope $P_M$ has no nontrivial matroid subdivisions, then the matroid $M$ has (up to trivial transformations) only finitely many realizations over a fixed field ${\mathbb{F}}$. This is one of very few results about realizability of matroids over arbitrary fields.

The connection with tropical geometry is a rich source of examples. Let $K$ be the field of Puiseux series $a_0x^{m/N} + a_1x^{(m+1)/N} + \cdots$ (where $a_0, a_1, \ldots \in {\mathbb{C}}$, $m,N \in {\mathbb{Z}}$ and $N>0$). Roughly speaking, every subspace $L$ of $K^n$ may be \emph{tropicalized}, and there is a canonical way of decomposing the resulting tropical linear space into Bergman fans of various matroids. These matroids give a matroid subdivision of $\Delta(d,n)$. For some nice choices of $L$, the corresponding matroid subdivisions can be described explicitly. For details, see for example \cite{f.ArdilaKlivans, f.Rincon, f.Speyer1}. 

There is also a useful connection with the subdivisions of the product of simplices $\Delta_{d-1} \times \Delta_{n-d-1}$, which are much better understood, as we discussed in Section \ref{f.sec:Ehrhart}. Note that the vertex figure of any vertex of $\Delta(d,n)$ is $\Delta_{d-1} \times \Delta_{n-d-1}$. Every matroid subdivision of $\Delta(d,n)$ then induces a ``local" polyhedral subdivision of $\Delta_{d-1} \times \Delta_{n-d-1}$ at each one of its vertices. Conversely, we can ``cone" any subdivision $\mathcal{S}$ of $\Delta_{d-1} \times \Delta_{n-d-1}$ to get a matroid subdivision that looks like $\mathcal{S}$ at a given vertex. \cite{f.Herrmannetal, f.Rincon}

\bigskip

\noindent {\textbf{\textsf{Enumerative aspects of matroid subdivisions.}}}
Say a face $P$ of a subdivision $\mathcal{P}$ of $P_M$ is \emph{internal} if it is not on the boundary of $P_M$. Let $\mathcal{P}^{\textrm{int}}$ be the set of internal faces of $\mathcal{P}$.
Currently, the most interesting enumerative question on matroid subdivisions is the following.

\begin{conjecture} \label{f.conj:fvector} (Speyer's $f$-vector conjecture \cite{f.Speyer1})
A matroid subdivision of the hypersimplex $\Delta(d,n)$ has at most $\frac{(n-c-1)!}{(d-c)!(n-d-c)!(c-1)!}$ interior faces of dimension $n-c$, with equality if and only if all facets correspond to series-parallel matroids. 
\end{conjecture}

Speyer constructed a subdivision simultaneously achieving the conjectural maximum number of interior faces for all $c$. In attempting to prove this conjecture, he pioneered the study of valuative matroid invariants. We say that a matroid invariant $f$ is \textbf{valuative} if for any matroid subdivision $\mathcal{P}$ of any matroid polytope $P_M$ we have
\[
f(M) = \sum_{P_{M_i} \in \mathcal{P^\textrm{int}}} (-1)^{\dim P_M - \dim P_{M_i}} f(M_i).
\]
%where $\mathcal{P}^o$ denotes the set of faces of $\mathcal{P}$ that are not on the boundary of $P_M$.

There are obvious matroid valuations, such as the volume or (thanks to Ehrhart reciprocity) the Ehrhart polynomial of $P_M$. Much more remarkably, we have the following result.

\begin{theorem}\label{f.th:Tuttevaluative}\cite{f.Speyer1} 
The Tutte polynomial $T_M(x,y)$ is a matroid valuation.
\end{theorem}

\begin{corollary}
The $f$-vector conjecture is true for $c=1$.
\end{corollary}

\begin{proof}
By Theorem \ref{f.th:Tuttevaluative} the beta invariant  $\beta(M) = [x^1y^0]T_M(x,y)$ is also a matroid valuation.
Recall that $\beta(N)=0$ if  and only if $N$ is not connected (or, equivalently, if $P_N$ is not a facet of $\mathcal{P}$), $\beta(N) = 1$ if and only if $N$ is series-parallel, and $\beta(N) \geq 2$ otherwise. Also $\beta(U_{d,n}) = {n-2 \choose d-1}$. Therefore, in a matroid subdivision of $\Delta(d,n)$ we have
\[
{n-2 \choose d-1} = \beta(U_{d,n}) = \sum_{\stackrel{P_N \in \mathcal{P}}{P \textrm{ facet}}}\beta(N) \geq (\textrm{number of facets of } \mathcal{P}).  
\]
with equality if and only if every facet is a series-parallel matroid.
\end{proof} 

For matroids $M$ realizable over some field of characteristic $0$, Speyer \cite{f.Speyer2} used the K-theory of the Grassmannian to construct another polynomial invariant $g_M(t)$, which he used to prove the f-vector conjecture for matroid subdivisions whose matroids are realizable in characteristic $0$. 

\begin{theorem}\cite{f.Speyer2}
The $f$-vector conjecture is true for matroid subdivisions consisting of matroids which are representable over a field of characteristic $0$.
\end{theorem}

Speyer's proof of this result relies on the nonnegativity of $g_M(t)$, which is proved geometrically for matroids realizable in characteristic $0$. The nonnegativity of $g_M(t)$ for all $M$, which would  prove Conjecture \ref{f.conj:fvector} in full generality, remains open.
 
\bigskip

\noindent {\textbf{\textsf{The Derksen-Fink invariant.}}}
Many other natural matroid functions were later discovered to be valuative. %Several other valuative invariants were constructed
 \cite{f.ArdilaFinkRincon, f.BilleraJiaReiner, f.Derksen}. 
An example of a very general valuation on matroid polytopes from \cite{f.ArdilaFinkRincon} is the formal sum $R(M) = \sum_{A \subseteq E} R_{A, r(A)}$ of symbols of the form $R_{S,k}$ where $S$ is a subset of $E$ and $k$ is an integer. Many other nice valuations can be obtained from this one. In fact, the matroid $M$ can clearly be recovered from $R(M)$.\footnote{However, note that $R$ is not a matroid isomorphism invariant.}

Eventually, Derksen \cite{f.Derksen} constructed a valuative matroid invariant, which he and Fink proved to be universal. \cite{f.DerksenFink} We call it the Derksen--Fink invariant. They gave several versions; we present one which is particularly simple to define.

Let $\textrm{Matroids}_{n,r}$ be the set of matroids on $[n]$ of rank $r$.  
Let $\mathcal{C}_n$ be the set of $n!$ complete chains $\mathcal{S}: \emptyset = S_0 \subsetneq S_1 \subsetneq \cdots \subsetneq S_{n-1} \subsetneq S_n = [n]$. For  a complete chain $\mathcal{S}$ and a matroid $M=([n],r)$, let $r[S_{i-1}, S_i] = r(S_i) - r(S_{i-1})$ for $1 \leq i \leq n$; this is always $0$ or $1$. The \textbf{Derksen--Fink invariant} of $M$ is 
\[
\mathcal{G}(M) = 
\sum_{\mathcal{S}  \in \mathcal{C}_n} 
U_{(r[S_0,S_1], r[S_1,S_2], \ldots, r[S_{n-1},S_n])}
%\sum_{i_1 < \cdots < i_n}
%x_{i_1}^{r[F_0,F_1]} x_{i_2}^{r[F_1,F_2]} \cdots x_{i_n}^{r[F_{n-1},F_n]}.
\]
in the vector space $\mathbf{U}$ spanned by the ${n \choose r}$ formal symbols $U_{w_1 \cdots w_n}$ where $w_1, \ldots, w_n \in \{0,1\}$ and $w_1+\cdots+w_n=r$.

\begin{theorem}\cite{f.DerksenFink}
The Derksen--Fink invariant $\mathcal{G}: \textrm{Matroids}_{n,r} \rightarrow \mathbf{U}$ is a universal valuative matroid invariant; that is, for any valuative matroid invariant $f: \textrm{Matroids}_{n,r} \rightarrow V$ there is a linear map $\pi: \mathbf{U} \rightarrow V$ such that $f = \pi \circ \mathcal{G}$.
\end{theorem}

\bigskip

\noindent{\textsf{\textbf{Acknowledgments.}}} I am extremely grateful to my combinatorics teachers, Richard Stanley, Gian-Carlo Rota, Sergey Fomin, Sara Billey, and Bernd Sturmfels (in chronological order); their influence on the way that I  understand combinatorics is apparent in these pages. I've also learned greatly from my wonderful collaborators and students.

Several superb surveys have been instrumental in the preparation of this manuscript; they are mentioned throughout the text, but it is worth acknowledging a few that I found especially helpful:
Aigner's \emph{A Course in Enumeration},
Beck and Robins's \emph{Computing the continuous discretely}, 
Brylawski and Oxley's \emph{The Tutte polynomial and its applications},
Flajolet and Sedgewick's \emph{Analytic Combinatorics}, 
Krattenthaler's \emph{Advanced Determinantal Calculus}, and
Stanley's \emph{Enumerative Combinatorics} and \emph{Lectures on hyperplane arrangements}.

%In the middle of this project, I had the fortune of spending a week in Colombia with Lou Billera and Richard Stanley, and hearing many illuminating stories about the origins and motivations behind algebraic and geometric combinatorics. I hope I was able to transmit some of the spirit of those lessons here.

I am very grateful to Mikl\'os B\'ona for the invitation to contribute to this project. Writing this survey has not been an easy task, but I have enjoyed it immensely, and have learned an enormous amount. I also thank Alex Fink, Alicia Dickenstein, Alin Bostan, David Speyer, Federico Castillo, Mark Wildon,  Richard Stanley, and the anonymous referee for their helpful comments.

Finally, and most importantly, I am so very thankful to May-Li for her patience, help, and support during my months of hard work on this project.

\newpage

\begin{footnotesize}
\bibliographystyle{amsalpha}
\bibliography{algmethods.bib}

\end{footnotesize}

\end{document}